\documentclass[reqno, 10pt]{memo-l}

\usepackage{amsthm}
\usepackage{amsfonts}
\usepackage{amsmath}
\usepackage{amssymb}
\usepackage{mathrsfs}
\usepackage{epsfig}
\usepackage{graphicx}
\usepackage{epstopdf}
\usepackage{hyperref}
\usepackage{color}
\usepackage{ifthen}
\usepackage{float}
\usepackage[active]{srcltx}
\usepackage{mathbbol}

\numberwithin{section}{chapter}
\numberwithin{equation}{chapter}

\makeatletter
\@namedef{subjclassname@2010}{%
\textup{2010} Mathematics Subject Classification}
\makeatother

\newtheorem{theorem}{Theorem}[section]
\newtheorem{thm}{Theorem}[section]
\newtheorem{proposition}[theorem]{Proposition}%[section]
\newtheorem{corollary}[theorem]{Corollary}%[section]
\newtheorem{lemma}[theorem]{Lemma}%[section]
\newtheorem{nono-lemma}{Lemma}[]

\theoremstyle{remark}
\newtheorem{remark}[theorem]{Remark}%[section]

\theoremstyle{definition}
\newtheorem{definition}[theorem]{Definition}%[section]
\newtheorem{example}[theorem]{Example}

\newcommand{\ran}{\mbox{ran}\,}

\newcommand{\bq}{\begin{equation}}
\newcommand{\eq}{\end{equation}}
\newcommand{\beqn}{\begin{eqnarray*}}
\newcommand{\eeqn}{\end{eqnarray*}}
\newcommand{\beq}{\begin{eqnarray}}
\newcommand{\eeq}{\end{eqnarray}}

\newcommand{\rar}{\rightarrow}

\newcommand{\bc}{\begin{centre}}
\newcommand{\ec}{\end{centre}}

\newcommand{\mf}{\mathfrak}

\newcommand{\sgn}{{\rm sgn\,}}
\newcommand{\ba}{\begin{array}}
\newcommand{\ea}{\end{array}}

\newcommand{\inp}[2]{\langle{#1},\,{#2} \rangle}

\newcommand*{\sib}[1]{\mathsf{sib}(#1)}

\newcommand*{\child}[1]{\mathsf{Chi}(#1)}
\newcommand*{\childn}[2]{{\mathsf{Chi}}^{\langle#1\rangle}(#2)}

\newcommand*{\childnt}[2]{{\mathsf{Chi}}^{\ll#1\gg}(#2)}

\newcommand*{\childi}[2]{{\mathsf{Chi}}_{#1}(#2)}
\newcommand*{\childki}[3]{{\mathsf{Chi}}_{#1}^{\langle#2\rangle}(#3)}
\newcommand*{\parentn}[2]{{\mathsf{par}}^{\langle#1\rangle}(#2)}

\newcommand*{\parentnt}[2]{{\mathsf{par}}^{\ll#1\gg}(#2)}

\newcommand*{\parenti}[2]{{\mathsf{par}}_{#1}(#2)}
\newcommand*{\parentki}[3]{{\mathsf{par}}_{#1}^{\langle#2\rangle}(#3)}

\newcommand*{\Ge}{\geqslant}
\newcommand*{\lambdab}{\boldsymbol\lambda}
\newcommand*{\thetab}{\boldsymbol\theta}

\newcommand*{\deltab}{\boldsymbol \delta}
\newcommand*{\Le}{\leqslant}
\newcommand*{\parent}[1]{\mathsf{par}(#1)}

\newcommand*{\rootb}{{\mathsf{root}}}

\makeindex

\begin{document}

\frontmatter

\title[Multishifts on Product of Directed Trees]
{Multishifts on Directed Cartesian Product of Rooted Directed Trees}

 %  \author[S. Chavan \\ D. Pradhan \\ S. Trivedi]{Sameer Chavan \\ Deepak Kumar Pradhan \\ Shailesh Trivedi}

\author{Sameer Chavan \and Deepak Kumar Pradhan \and Shailesh Trivedi
}

%\author{Sameer Chavan and Deepak Kumar Pradhan \\
%Department of Mathematics and Statistics\\
%Indian Institute of Technology Kanpur, India\\
%chavan@iitk.ac.in and dpradhan@iitk.ac.in
%\and
%%\author{
%Shailesh Trivedi\\
%School of Mathematics \\ Harish-Chandra Research Institute\\ 
%Chhatnag Road, Jhunsi, Allahabad 211019, India\\
%shaileshtrivedi@hri.res.in
%}

%\address[2][]{\addresses{\address{A}{B}}

\address{Department of Mathematics and Statistics\\
Indian Institute of Technology Kanpur, India}
   \email{chavan@iitk.ac.in}
   
 % \author[D. Pradhan]{Deepak Kumar Pradhan}  
   \email{dpradhan@iitk.ac.in}
   
 %  \author[S. Trivedi]{Shailesh Trivedi}
  
 \address{School of Mathematics \\ Harish-Chandra Research Institute\\ 
Chhatnag Road, Jhunsi, Allahabad 211019, India}
\email{shaileshtrivedi@hri.res.in}

\thanks{The work of the second author is supported through the NBHM Research Fellowship}

%   
%   \thanks{The research of the
%third and fourth authors was supported by the NCN
%(National Science Center), decision No.
%DEC-2013/11/B/ST1/03613.}
%   \subjclass[2010]{Primary 47B37, 47A10; Secondary 46E22, 47B38}
%\keywords{weighted shift, directed
%tree, multiplication operator, reproducing kernel of finite bandwidth, Hilbert space of holomorphic functions}

\date{}

%    Recognition of the 2010 edition of the Mathematics Subject
%    Classification requires a version of amsbook.cls from July 2009
%    or later.  If "2010" is not recognized, please upgrade.


\maketitle

\newpage

\dedicatory{\noindent {\it raah-e-mazm$\overline{\mbox{u}}$n-e-taaza band nahi$\tilde{\mbox{n}}$ \\ t$\overline{\mbox{a}}$ qay$\overline{\mbox{a}}$mat khul$\overline{\mbox{a}}$ hai baab-e-su$\d{\mbox{k}}$han} \\ \\ \hspace{35pt} {The pathways to new ideas will never close; the doors of wisdom (literature) will
always be open, till the very end.}
\footnote{Translated by Anant Dhavale}
\\ \\
\dedicatory{\\ \hspace{35pt}  -  \small{Wali Deccani (1667-1707)} } 
%\tableofcontents

%\setcounter{tocdepth}{3} 
\tableofcontents

\begin{abstract} 
We systematically develop the multivariable counterpart of the theory of weighted shifts on rooted directed trees. 
Capitalizing on the theory of product of directed graphs, we introduce and study the notion of multishifts on directed Cartesian product of rooted directed trees. This framework unifies the theory of weighted shifts on rooted directed trees and that of classical unilateral multishifts. Moreover, this setup brings into picture some new phenomena such as the appearance of system of linear equations in the eigenvalue problem for the adjoint of a multishift. In the first half of the paper, we focus our attention mostly on the multivariable spectral theory and function theory including finer 
analysis of various joint spectra and wandering subspace property for multishifts. In the second half,
we separate out two special classes of multishifts, which we refer to as torally balanced and spherically balanced multishifts. The classification of these two classes is closely related to toral and spherical polar decompositions of multishifts. Furthermore, we exhibit a family of spherically balanced multishifts on $d$-fold directed Cartesian product $\mathscr T$ of rooted directed trees. These multishifts turn out be
multiplication $d$-tuples $\mathscr M_{z, a}$ on certain reproducing kernel Hilbert spaces $\mathscr H_a$ of vector-valued holomorphic functions defined on the unit ball $\mathbb B^d$ in $\mathbb C^d$, which can be thought of as
tree analogs of the multiplication $d$-tuples acting on the reproducing kernel Hilbert spaces associated with the kernels $\frac{1}{(1-\inp{z}{{w}})^a}~(z, w \in \mathbb B^d, a \in \mathbb N).$ 
Indeed, the reproducing kernels associated with $\mathscr H_a$ are certain operator linear combinations of $\frac{1}{(1-\inp{z}{{w}})^a}$ and multivariable hypergeometric functions ${}_2F_1(\alpha_v+a+1, 1, \alpha_v+2, \cdot)$ defined on $\mathbb B^d \times \mathbb B^d$, where $\alpha_v$ denotes the depth of a branching vertex $v$ in $\mathscr T$.
We also classify joint subnormal and joint hyponormal multishifts within the class of spherically balanced multishifts.
\end{abstract}

\mainmatter

\chapter{Introduction}

The investigations in the present work are related to the idea of shifts associated with discrete structures (e.g. directed trees), recently boosted in the theory of Hilbert space operators (\cite{JJS}, \cite{BJBS-1}, \cite{JBS-1}, \cite{BJBS-2}, \cite{JBS-2}, \cite{JBS-3}, \cite{BJBS-3}, \cite{BJBS-4}, \cite{BJBS-5}, \cite{CT},\cite{MS}). The significantly large class of weighted shift operators on directed trees contains all classical weighted shifts and has an overlap with that of composition operators. This interplay of graph theory and operator theory provides illuminating examples exhibiting subtle phenomena such as existence of non-hyponormal operators generating Stieltjes moment sequences and triviality of the domain of integral powers of densely defined subnormal operators (\cite{JBS-1}, \cite{JBS-3}, \cite{BJBS-4}, \cite{BJBS-5}; refer also to \cite{KK}, \cite{KLP} for a systematic study of operator algebras associated with directed graphs). Further, this framework turns out to be a rich source of $k$-diagonal reproducing kernel Hilbert spaces of vector-valued holomorphic functions (\cite{AM}, \cite{CT}). 

The motivation for the present work primarily comes from the multivariable operator theory, where the main objectives of study have been function theory and spectral theory of classical multishifts (\cite{JL}, \cite{Cu-00}, \cite{Cu-0}, \cite{CuS}, \cite{At-0}, \cite{At-00}, \cite{Cu}, \cite{Cu-1}, \cite{At}, \cite{At-1}, \cite{MuS}, \cite{MV}, \cite{BM}, \cite{Ar}, \cite{AZ}, \cite{GHX}, \cite{GRS}, \cite{GR}, \cite{AT}, \cite{RS}, \cite{DE},  \cite{DE-1}, \cite{BV}, \cite{Kad}, \cite{Ka}, \cite{CY}, \cite{Hz}). This work is an effort to develop the theory of weighted shifts on directed trees in several variables by implementing the methods of graph theory (\cite{Sa}, \cite{We}, \cite{Mc}, \cite{Ha}, \cite{Feig}, \cite{IK}; refer also to \cite{ARS1}, \cite{ADV}, \cite{AV}, \cite{ARS2}, \cite{ARS3}, where Hardy-Besov spaces associated with certain trees have been introduced and studied). 
The well-established theory of product of directed graphs provides a foundation for the study of multishifts. Various notions of product of directed graphs  (e.g. Cartesian product, tensor product) lead naturally to interesting counterparts of classical shifts in one and several variables. One peculiar aspect of the directed Cartesian product of directed trees is that although its not a directed tree, it admits a directed semi-tree structure. Interestingly, there is a natural shift operator on any directed semi-tree \cite{MS}. Thus a single discrete structure gives rise to at least two distinct notions of shifts. This is not so easy to reveal in the classical case. 
%We would like to emphasize here that the set up of directed trees is quite in contrast with that of undirected trees. For example, tensor product of two undirected trees has exactly two components (Remark \ref{t-components}), whereas in the context of directed trees, we get countably many components, each of which itself is a directed tree (Theorem \ref{tensor-prop}). 
One of the advantages of this setup is that any disjoint decomposition of the set $V$ of vertices induces a natural decomposition of the associated unweighted Lebesgue space $l^2(V)$, which in turn decomposes the multishift $S_{\lambdab}$ into known objects like tuples of compact operators and classical multishifts. 
On the other hand, there are numerous ways of decomposing $V$ by defining equivalence relations on $V$ in terms of siblings, generations etc.
As evident, every set up has its own set of problems. Here also in the context of subnormality of multishifts, one can ask for finite and minimal subset $W$ of set $V$ of vertices with the following property: A multishift is joint subnormal if and only if its moments are completely monotone at every vertex from $W.$ The notion of joint branching index plays an important role in the affirmative answer of this problem.
%Although there are known examples of classical multishifts with nowhere dense Taylor spectra \cite{Cu-2}, the decompositions of multishifts obtained in this way provide a recipe to construct operator tuples of this kind in abundance.

In the remaining part of this section, we set some standard notations and also collect various notions and facts, which are central to the present text. 
For a set $X$, $\mbox{card}(X)$ denotes the cardinality of $X$. 
\index{$\mbox{card}(X)$}
We next recall that the symbol ${\mathbb N}$ stands for the set of nonnegative
integers and that ${\mathbb N}$ forms a semigroup under addition. For a positive integer $d,$ let
${\mathbb{N} }^d$ denote the $d$-fold Cartesian product ${\mathbb{N}} \times
\cdots \times { \mathbb{N}}$ of $\mathbb N$. 
\index{$\mathbb N^d$}
Then, for $\alpha =
(\alpha_1, \cdots, \alpha_d)$ and $\beta = (\beta_1, \cdots, \beta_d)$ in
${\mathbb{N}}^d$, we write $\alpha \leqslant \beta$
if $\alpha_j \leqslant \beta_j$ for all $j=1, \cdots, d$ 
and we also use $\alpha!:=\prod_{j=1}^d \alpha_j!$ and $|\alpha|:=\sum_{j=1}^d
\alpha_j$. 
Throughout this paper, we follow the conventions below: 
\beqn
\sum_{i=1}^{n-1}x_i = 0~\mbox{if~}n=1,\quad \prod_{j=0}^{n-1}y_j=1~\mbox{if~}n=0.
\eeqn  
More generally, the sum over empty set is understood to be $0$ while the product over empty set is always $1.$

Let $\mathbb C$ denote the field of complex numbers and let $\mathbb C^d$ denote the $d$-fold Cartesian product of $\mathbb C$.
\index{$\mathbb C^d$} 
We set $tz:=(tz_1, \cdots, tz_d)$ for $t \in \mathbb C$
and $z=(z_1, \cdots, z_d) \in \mathbb C^d.$ 
Whenever $\alpha \in \mathbb N^d$ or $z \in \mathbb C^d$, it is understood that $\alpha=(\alpha_1, \cdots, \alpha_d)$ and $z=(z_1, \cdots, z_d).$ The complex conjugate of $z \in \mathbb C^d$ is defined by $\overline{z}:=(\overline{z}_1, \cdots, \overline{z}_d)$.
We denote by $\mathbb B^d_r$ the open ball in $\mathbb C^d$ centered at the origin and of
radius $r > 0$:
$$
\mathbb B^d_r :=
\{ z= (z_1, \cdots, z_d)
\in \mathbb C^d : ~ \|z\|_2 < r\},$$ where $\|z\|_2 :=\sqrt{|z_1|^2 + \cdots + |z_d|^2}$ denotes the Euclidean norm of $z=(z_1, \cdots, z_d)$ in $\mathbb C^d.$
\index{$\mathbb B^d_r$}
The sphere $\{z \in \mathbb C^d : \|z\|_2=r\}$
centered at the origin and of radius $r > 0$ is denoted by $\partial
\mathbb B^d_r.$ 
\index{$\partial \mathbb B^d_r$}
For simplicity, the unit ball $\mathbb B^d_1$ and the
unit sphere $\partial \mathbb B^d_1$ are denoted respectively by
$\mathbb B^d$ and $\partial \mathbb B^d.$ 
\index{$\mathbb B^d$}
\index{$\partial \mathbb B^d$}
We denote by $\mathbb D^d_r$ the open polydisc centered at the origin and of
polyradius $r=(r_1, \cdots, r_d)$ with $r_1, \cdots, r_d > 0$:
\index{$\mathbb D^d_r$}
$$
\mathbb D^d_r :=
\{ z= (z_1, \cdots, z_d)
\in \mathbb C^d : ~ |z_1| < r_1, \cdots, |z_d| < r_d\}.$$ 
The $d$-torus $\{z \in \mathbb C^d : |z_1|=r_1, \cdots, |z_d|=r_d\}$
centered at the origin and of polyradius $r$ is denoted by $
\mathbb T^d_r.$ 
\index{$\mathbb T^d_r$}
Again, for simplicity, the unit polydisc $\mathbb D^d_1$ \index{$\mathbb D^d$}
and the
unit $d$-torus $\mathbb T^d_1$ \index{$\mathbb T^d$} are denoted respectively by
$\mathbb D^d$ and $\mathbb T^d.$ For a subset $X$ of $\mathbb C^d,$ the closure of $X$ in $\|\cdot\|_2$ is denoted by 
\index{$\mbox{cl}(X)$} 
$\mbox{cl}(X)$.

Let $\mathcal H$ be a complex Hilbert space.
The inner product  on $\mathcal H$ will be denoted by $\inp{\cdot}{\cdot}_{\mathcal H}$.
If no confusion is likely, then we suppress the suffix, and simply write the inner product as $\inp{\cdot}{\cdot}$.  
By a {\it subspace} of $\mathcal H$, we mean a closed linear manifold.
Let $W$ be a subset of $\mathcal H.$ Then $\mbox{span}\,W$ stands for the smallest linear manifold generated by $W.$
\index{$\mbox{span}\,W$}
In case $W$ is singleton $\{w\},$ we use the convenient notation $[w]$ in place of  $\mbox{span}\,\{w\}$.
\index{$[w]$}
By $\bigvee \{w : w \in W\},$ 
we understand the subspace generated by $W$.
\index{$\bigvee \{w : w \in W\}$}
The orthogonal complement of a subspace $\mathcal M$ in $\mathcal H$ will be denoted either by $\mathcal M^{\perp}$ or $\mathcal H \ominus \mathcal M$. \index{$\mathcal M^{\perp}=\mathcal H \ominus \mathcal M$}
The Hilbert space dimension of the subspace $\mathcal M$ is denoted by $\dim \mathcal M.$ \index{$\dim \mathcal M$}
For a subspace $\mathcal M$ of $\mathcal H,$ we use $P_{\mathcal M}$ to denote the orthogonal
projection of $\mathcal H$ onto $\mathcal M.$ 
\index{$P_{\mathcal M}$}
We use $I$ to denote the identity operator on $\mathcal H.$ If $\mathcal M$ is a subspace of $\mathcal H$, then we use $I|_{\mathcal M}$ to denote the identity operator on $\mathcal M.$
\index{$I\vert_{\mathcal M}$}
%For vectors $x, y \in \mathcal H,$ we use the notation $x \otimes y$ to denote the rank one operator given by
% $$x \otimes y (h) = \inp{h}{y}x,~h \in \mathcal H.$$ 

Unless stated otherwise, all the Hilbert spaces occurring in this paper are complex 
infinite-dimensional separable and for any such Hilbert space 
$\mathcal H$, ${B}({\mathcal H})$ denotes the Banach algebra of all bounded linear operators on $\mathcal H$ endowed with the operator norm. 
\index{${B}({\mathcal H})$}
For $A \in B(\mathcal H),$ the symbols $\ker A$ \index{$\ker A$}
and $\mbox{ran}\,A$ will stand for the kernel and 
the range of $A$ respectively.
\index{$\mbox{ran}\,A$}
The Hilbert space adjoint of $A$ will be denoted by $A^*.$ 
%By an $A$-invariant subspace $\mathcal M$, we mean a subspace $\mathcal M$ of $\mathcal H$ such that $A(\mathcal M) \subseteq \mathcal M.$ 

By a {\it commuting
$d$-tuple} $T$ on ${\mathcal H},$ we mean a tuple $(T_1, \cdots,
T_d)$ of commuting bounded linear operators $T_1, \cdots, T_d$ on
$\mathcal H.$ A commuting $d$-tuple $T$ on $\mathcal H$ is said to be {\it doubly commuting} if  
$T_iT_j^*=T_j^*T_i$ for all $i,j = 1,\cdots,d$ with $i \neq j$.
If $T =
(T_1, \cdots, T_d)$ is a commuting $d$-tuple on $\mathcal H$, then we set $T^*$
to denote $(T^*_1, \cdots, T^*_d)$ \index{$T^*$}
while $T^\alpha$ represents $T^{\alpha_1}_1 \cdots
T^{\alpha_d}_d$ for $\alpha = (\alpha_1,
\cdots, \alpha_d) \in {\mathbb{N}}^d$, where we use the convention that $A^0=I$ for $A \in B(\mathcal H)$. 
Also, we find it convenient to introduce the following operators on $B(\mathcal H)$. 
Given a commuting $d$-tuple $T$ on $\mathcal H,$ 
%define $Q_j : B(\mathcal H) \rar B(\mathcal H)$ by
%$$Q_j(X) := T^*_jXT_j~(X \in B(\mathcal H), ~j=1, \cdots, d).$$
define $Q_T :  B(\mathcal H) \rar B(\mathcal H)$ by $Q_T(X) :=\sum_{j=1}^d T^*_jXT_j$ for $X \in B(\mathcal H).$ Further,
the operator $Q^n_T$
is inductively defined for all $n \in \mathbb N$ through
the relations $Q^0_T(X):=X$ and $Q^n_T(X):=Q_T(Q^{n-1}_T(X))~(n \geqslant  1)$ for $X \in B(\mathcal H).$ 
\index{$Q^n_T$}
It is easy to see
that \beq \label{sp-gen-powers} Q^n_T(I)={\sum_{|\alpha|=n}\frac{n!}{\alpha!}{T^*}^{\alpha}T^{\alpha}}~(n \in \mathbb N).\eeq
We collect below some definitions required throughout this text. 

Let $T=(T_1, \cdots, T_d)$ be a commuting $d$-tuple on $\mathcal H.$ 
We say that $T$ is a 
\begin{enumerate}
\item[(i)] {\it toral contraction} if $Q_{T_j}(I) \leqslant I$ for $j=1, \cdots, d.$
\item[(ii)] {\it toral isometry} if $Q_{T_j}(I) = I$ for $j=1, \cdots, d.$
\item[(iii)] {\it toral left invertible} $d$-tuple if $Q_{T_j}(I)$ is invertible for $j=1, \cdots, d.$
\item[(iv)] {\it joint contraction} if 
$Q_T(I) \leqslant I.$ 
\item[(v)] {\it joint isometry} if 
$Q_T(I) = I.$ 
\item[(vi)] {\it joint left invertible} $d$-tuple if 
$Q_T(I)$ is invertible. 
\end{enumerate}

A commuting $d$-tuple $T=(T_1, \cdots, T_d)$ on $\mathcal H$ is {\it joint subnormal}
if there exist a Hilbert space ${\mathcal K}$ containing ${\mathcal
H}$ and a commuting $d$-tuple $N=(N_1, \cdots, N_d)$ of normal
operators $N_1, \cdots, N_d$ in ${B}({\mathcal K})$ such that $N_jh =
T_jh$ for every $h \in {\mathcal H}$ and $j=1, \cdots, d.$
It is well-known that toral isometry and joint isometry are joint subnormal \cite[Propositions 1 and 2]{At}.
The notion of joint subnormal tuples is closely related to a classical notion from abstract harmonic analysis, namely, the notion of completely monotone functions. 

Let $\phi$ be a real-valued map on $\mathbb{N}^d$. For $1 \leqslant j \leqslant d,$
define the difference operators $\nabla_j$ by $(\nabla_j{\phi})(\alpha):={\phi}(\alpha)-
{\phi}(\alpha + \epsilon_j) \ (\alpha \in \mathbb{N}^d),$ where $\epsilon_j$ is the $d$-tuple with $1$ in the $j^{\mbox{\tiny th}}$ entry and zeros elsewhere. 
\index{$\epsilon_j$}
The operator ${
\nabla}^\beta$ is inductively defined for every $ \beta \in \mathbb{N}^d$ through the
relations ${\nabla}^0{\phi}:=\phi$, ${\nabla}^{\beta +\epsilon_j} \phi :=\nabla_j (\nabla^\beta \phi)$.
\index{${\nabla}^\beta$}
A real-valued map $\phi
$ on $\mathbb{N}^d$ is said to be $\mathit{completely \; monotone}$ if $({\nabla}^\beta \phi)(\alpha) \geqslant  0$ for all $\alpha, \beta \in \mathbb{N}^d$. 
\begin{remark}
A toral contractive $d$-tuple $T$ on ${\mathcal H}$ is joint subnormal if and only if $\phi(\alpha):=\|T^{\alpha}h\|^2~(\alpha \in \mathbb N^d)$ is completely monotone for every $h \in \mathcal H$ (\cite[Theorem 4.4]{At-0}).
\end{remark}

A commuting $d$-tuple $T=(T_1, \cdots, T_d)$ is {\it joint hyponormal} if the $d
\times d$ matrix $([T^*_j, T_i])_{1 \leqslant i, j \leqslant d}$ is positive
definite, where $[A, B]$ stands for the commutator $AB-BA$ for $A$
and $B$ in $B(\mathcal H)$. 
\index{$[A, B]$}
A joint subnormal tuple is always
joint hyponormal \cite{At-00}, \cite{Cu-1}.

{\it For all notions introduced above, we skip the prefixes toral or joint in case the dimension $d$ is $1.$
Although it has been a common practice to use interchangeably joint isometry with spherical isometry, we do not follow this practice.}
%The reason for this is that it is likely to get mixed with somewhat constrained notion of spherical tuples as introduced in \cite{CY}.}

%Unless stated otherwise, all $d$-tuples $T$ consists of commuting, %bounded linear operators on $\mathcal H.$

%We recall that a $d$-tuple $T=(T_1, \cdots, T_d)$ is
%{\it commuting} if $T_i T_j = T_j T_i$ for all $i,j = 1, \cdots, d$.
%\begin{enumerate}
%\item[(i)] and
%\item[(ii)] {\it doubly commuting} if $T$ is commuting and
%$T_iT_j^*=T_j^*T_i$ for all $i,j = 1,\cdots,d$ with $i \neq j$.
%\end{enumerate} 

We briefly recall from \cite{CC} the definitions of toral and spherical Cauchy dual tuples.
Let $T = (T_1, \cdots, T_d)$ be a commuting $d$-tuple on $\mathcal H$.
Assume that $T$ is {toral left invertible}. 
%Let $Q_j(X) = T^*_jXT_j~(X \in B(\mathcal H), ~j=1, \cdots, d).$ 
We refer to the $d$-tuple $T^{\mathfrak{t}}=
(T^{\mathfrak{t}}_1, \cdots, T^{\mathfrak{t}}_d)$ as the \textit{
toral Cauchy dual} of $T$, where \index{$T^{\mathfrak{t}}$}
\beq \label{toral-dual} T^{\mathfrak{t}}_j :=
T_j(Q_{T_j}(I))^{-1}~(j=1, \cdots, d).\eeq
Note that $(T^{\mathfrak{t}})^{\mf t}=T.$ 
%For a commuting $d$-tuple $T=(T_1, \cdots, T_d)$ on $\mathcal H,$ let
%\beqn \label{sp-gen}
%Q_T(X) := \sum_{j=1}^dT^*_jXT_j~(X \in B(\mathcal H)).
%\eeqn
%The operator $Q^n_T$
%is inductively defined for all $n \in \mathbb N$ through
%the relations $Q^0_T(X)=X$ and $Q^n_T(X)=Q_T(Q^{n-1}_T(X))~(n \geqslant  1)$ for $X \in B(\mathcal H).$ 
%It is easy to see
%that \beq \label{sp-gen-powers} Q^n_T(I)={\sum_{|\alpha|=n}\frac{n!}{\alpha!}{T^*}^{\alpha}T^{\alpha}}~(n \in \mathbb N).\eeq

Assume that $T$ is {joint left-invertible}. We refer to the $d$-tuple $T^{\mathfrak{s}} = (T^{\mathfrak{s}}_1,\cdots,T^{\mathfrak{s}}_d)$ as the \textit{spherical Cauchy dual} of $T$, where \index{$T^{\mathfrak{s}}$}
\beq \label{spherical-dual} T^{\mathfrak{s}}_j :=
T_j(Q_T(I))^{-1}~(j=1, \cdots, d).\eeq
Note that $(T^{\mathfrak{s}})^{\mf s}=T.$

%For the definitions and the basic theory of 
%various joint spectra, the reader is referred to \cite{Cu}. For a commuting $d$-tuple $T = (T_1, %\cdots, T_d)$ on $\mathcal H$, we reserve 
%the symbols $\sigma(T)$, $\sigma_{p}(T)$, $\sigma_{l}(T)$, and $\sigma_e(T)$
%for the Taylor spectrum, point spectrum, left spectrum, and essential spectrum of $T$, respectively. 

\section{Multivariable Spectral Theory} \label{Spec-th}

In what follows, the multivariable spectral theory will be central to the investigations in this paper. Thus we find it necessary to include a brief account of Taylor's notion of invertibility and related concepts. We have freely drawn on \cite{Cu} and \cite{Ar1} throughout this paper, particularly, in the following discussion.

The 
classical spectral theory deals with the problem of finding $x \in \mathcal H$ such that $Tx=y$ for any given $y \in \mathcal H$, where $T \in B(\mathcal H)$. Note that
$T$ is (boundedly) invertible if and only if the above problem is solvable for every $y \in \mathcal H$ with unique $x \in \mathcal H$. Let us formulate an analog of the problem above for two operators $T_1, T_2 \in B(\mathcal H)$ such that $T_1T_2 = T_2T_1.$ One is interested in the notion of invertibility which will give ``unique" solution $(x_1, x_2)$ of the problem of finding $x_1, x_2 \in \mathcal H$ such that $T_1 x_1+ T_2 x_2 =y$ for every given $y \in \mathcal H.$ 
In this case we can not hope for uniqueness. Indeed, if $(x_1, x_2)$ is a solution of this problem, then for any $h \in \mathcal H$, if we set $x'_1 := x_1 - T_2h$ and $x'_2 := x_2 - T_1h$, then it is easy to check $(x'_1, x'_2)$ is a also solution. 
Following \cite{Ar1}, we will refer to $(x'_1, x'_2)$ as the {\it tautological perturbation} of the solution $(x_1, x_2).$ 
%Let us consider the above problem in several variables. Let $T_1, \cdots, T_d \in B(\mathcal H).$
\begin{remark}
There are no non-trivial tautological perturbations in case of single operator.
\end{remark}
One needs to determine what happens modulo tautological perturbations. 
This is where homology enters into the picture.

Given  a Hilbert space $\mathcal H,$ consider 
\beqn
\Lambda_0 &:=& \mathcal H, \\ \Lambda_1 &:= & \lbrace (h_1,h_2) : h_1, h_2 \in \mathcal H \rbrace,\\
\Lambda_2 &:=& \lbrace (h_{ij}) : h_{ij} \in \mathcal H~\mbox{for~}1\leqslant i,j \leqslant 2, (h_{ij})~\mbox{is skew symmetric} \rbrace. \eeqn Note that $\Lambda_2$ is  isometrically isomorphic to $\mathcal H \: \text{via} \: \left(\begin{array}{cc}
 0 & h\\
-h & 0
\end{array}\right) \rightsquigarrow h$. Consider the following short sequence 
\beq \label{K} K : \lbrace 0 \rbrace \xrightarrow{0} \Lambda_2 \xrightarrow{B_2} \Lambda_1 \xrightarrow{B_1} \Lambda_0 \xrightarrow{0} \lbrace 0 \rbrace, \eeq
where
\beqn B_2((h_{ij})) &:=& (h_{ij}) \begin{pmatrix}
T_1\\
T_2\\
\end{pmatrix} = (T_2 h_{12}, -T_1 h_{12}), \\ B_1(h_1,h_2) &:=& (h_1,h_2) \begin{pmatrix}
T_1\\
T_2\\
\end{pmatrix} = T_1 h_{1}+T_2 h_{2}.\eeqn
Note that $K$ is a complex, that is, $ B_1 \circ B_2 = 0.$

Let us examine the complex $K$ as given in \eqref{K}.
\begin{enumerate}
\item $T_1x_1 +T_2x_2 = y$ has a solution if and only if $B_1(x_1,x_2) = y$ has a solution if and only if $B_1$ is surjective.
\item Given a solution $(x_1,x_2) \in \Lambda_1$ of $T_1x_1+T_2x_2=y$ for a given $y \in \mathcal H$,  $(x_1',x_2')$ is also a solution if and only if $(x_1 - x_1', x_2-x_2') \in \ker B_1$.
\item $ \ker B_2 = \ker T_1 \cap \ker T_2$.
\item \eqref{K} is exact if and only if $T_1\mathcal H +T_2\mathcal H = \mathcal H$, $\ker T_1 \cap \ker T_2 = \lbrace 0 \rbrace$ and solution of $T_1x_1+T_2x_2=y$ for a given $y \in \mathcal H$ is unique upto tautological perturbations.
\end{enumerate}

%\subsection{Invertibility in $d$-dimensions}
Let us now see the Taylor invertibility in the general case.
For that purpose, consider the co-ordinate linear functionals $e_1,e_2, \cdots ,e_d$ on $\mathbb{C}^d$ with respect to the standard basis. Let $\Lambda^0(\mathbb{C}^d) := \mathbb{C}$ and let $\Lambda^1(\mathbb{C}^d)$ be the vector space with basis  $\{e_1, e_2, \cdots, e_d\}$. Given $ w,w' \in \Lambda^1(\mathbb{C}^d)$ 
let us define $w \wedge w'$ by  $$w \wedge w' (v_1, v_2):=  w(v_1)w'(v_2)-w(v_2)w'(v_1)~(v_1, v_2 \in \mathbb C^d).$$
Note that $w \wedge w = 0$ and $w \wedge w'= - w' \wedge w$.
Also note that any $2$-form is a linear combination of $e_1 \wedge e_2,\, e_1 \wedge e_3,\, \cdots, e_{d-1} \wedge e_d$ ($\binom{d}{d-2}$ elements).
One may now define inductively all higher ordered forms with the help of following definition of wedge product: Let $w$ (resp. $w'$) denote a $p$-form on $p$-fold Cartesian product ${\mathbb C^d}^{(p)}$ of $\mathbb C^d$ (resp. a $q$-form on ${\mathbb C^d}^{(q)}$). We define
the {\it wedge product} $w \wedge w'$ as the $(p+q)$-form on ${\mathbb C^d}^{(p+q)}$ given by
\beqn w \wedge w' (v) := \frac{1}{p!q!} \sum_{\sigma \in
\mathbb S_{p+q}} \sgn(\sigma)w\left(v_{\sigma(1)}, \cdots,
v_{\sigma(p)}\right) w'\left(v_{\sigma(p+1)}, \cdots,
v_{\sigma(p+q)}\right), \eeqn
where $\mathbb S_n$ denotes the group of permutations on $\{1, \cdots, n\}.$
For $i=1, \cdots, d,$ 
let $\Lambda^{i}(\mathbb{C}^d)$ be the vector space generated by $i$-forms.
We  define $\Lambda(\mathbb C^d)$ as the algebra over $\mathbb C$ consisting of $\Lambda^{i}(\mathbb{C}^d)~(i=1, \cdots, d)$ with identity $e_0$ defined by $e_0 \wedge w= w$, where the multiplication is the wedge product.
The vector space $\Lambda(\mathbb C^d)$ is $2^d$ dimensional, which can be endowed with an inner product $\inp{\cdot}{\cdot}_{\Lambda}$ so that $$\lbrace e_0 \rbrace \cup \lbrace e_{i_1 } \wedge e_{i_2 } \wedge e_{i_3 } \wedge \cdots \wedge e_{i_k } :  1 \Le i_1 \leqslant i_2 \leqslant i_3 \leqslant \cdots \leqslant i_k \Le d \rbrace$$ forms an orthonormal basis.

The finite dimensional Hilbert space $\Lambda(\mathbb C^d)$ admits natural operators, to be referred to as, 
{\it creation operators} $E_i : \Lambda \longrightarrow \Lambda $ defined by $E_i(w) := e_i \wedge w~(i=1, \cdots d)$, and $ E_0(w) := w$. It satisfies the
{\it anti-commutation relations} : $$E_iE_j +E_jE_i= 0~(1 \Le i, j \Le d).$$ 
%which follows from $$E_iE_j (w) = e_i \wedge  e_j \wedge w = - e_j \wedge e_i \wedge w= E_j E_i(w).$$
Let $T=(T_1, \cdots, T_d)$ be a commuting $d$-tuple on $\mathcal H$ and
let $\Lambda(\mathcal H) := \mathcal H \otimes_\mathbb{C} \Lambda(\mathbb C^d)$ be a Hilbert space endowed with the inner product $$\inp{x \otimes w}{y \otimes w'}_{\Lambda^d(\mathcal H)} := \inp{x}{y}_{\mathcal H}\inp{w}{w'}_{\Lambda}.$$
We set $\Lambda^i(\mathcal H) := \mathcal H \otimes_\mathbb{C} \Lambda^i(\mathbb C^d)$ for $i=0, \cdots, d.$
Consider the {\it boundary operator} $\partial_T : \Lambda(\mathcal H) \longrightarrow \Lambda (\mathcal H) $ given by 
$$ \partial_T(h\otimes w) := \sum_{i=1}^{d} T_i(h) \otimes E_i(w).$$
\index{$\partial_T$}
Note that $\partial_T$ is a bounded linear operator on $\Lambda(\mathcal H).$
Since $T$ is commuting and $E_1, \cdots, E_d$ are anti-commuting, $\partial_T^2=0.$ 
%This follows from $E_iE_j=-E_jE_i$ and
%\begin{equation*}
%\begin{split}
%\partial_T^2(h \otimes w)  
%                        =  \sum_{j=1}^{d}\sum_{i=1}^{d} T_j T_i (h) \otimes %E_jE_i(w).
%\end{split}
%\end{equation*}
%insert equation for adjoint of boundary operator 
This allows us to define the
{\it Koszul complex $K(T)$} associated with $T$ as
\beqn 
K(T): \lbrace 0 \rbrace \xrightarrow {0} \Lambda^0(\mathcal H) \xrightarrow {\partial_{T,0}} \Lambda^1 (\mathcal H) \xrightarrow {\partial_{T,1}} \Lambda^2(\mathcal H) \cdots  \Lambda^{d-1}(\mathcal H) \xrightarrow {\partial_{T,d-1}}
\Lambda^d(\mathcal H) \xrightarrow {~~0~~} \lbrace 0 \rbrace \eeqn
where $\partial_{T, i} := \partial_T |_{\Lambda^{i}(\mathcal H)}$ for $i=0, \cdots, d-1$.
\index{$K(T)$}
\begin{remark}
If $K(T)$ is exact, then
\begin{enumerate}
\item $\ker \partial_{T,0}=\{0\},$ that is, $\cap_{i=1}^d \ker T_i = \{0\}.$ 
\item $\mbox{ran}\, \partial_{T,p} $ is closed ($\Rightarrow \mbox{ran}\, \partial_{T,p}^*$ is closed).
%\item  $\partial_{T,p-1}\partial_{T,p-1}^* |_{ran(\partial_{T,p-1})}$ is one-one ($\because ran(\partial_{T,p-1}) = ker(\partial_{T,p-1}^*)^{\perp}$).?
%\item \marginpar{domain of $\partial^{*}_{T,p} is \Lambda_{p+1}?$} $D_{T,p} :=  \partial_{T,p}\partial_{T,p}^* + \partial_{T,p}^*\partial_{T,p} \hspace{5pt} : \Lambda^p(H) \longrightarrow \Lambda^p(H) $, then $D_{T,P}$ is one-one and has dense range.?
\item $\mbox{ran}\,\partial_{T,d-1}=\Lambda^d(\mathcal H),$ that is, $T_1\mathcal H + \cdots + T_d \mathcal H = \mathcal H.$
\end{enumerate}
\end{remark}

The {\it Taylor spectrum (or joint spectrum)} of $T$ is defined as 
\index{$\sigma(T)$}
$$ \sigma(T) := \lbrace \lambda \in \mathbb{C}^d : K(T - \lambda ) \text{\: is not exact}\rbrace.$$
We also define {\it point spectrum} of $T$ as \index{$\sigma_p(T)$} 
$$\sigma_p(T) := \{\lambda \in \mathbb C^d : \partial_{T-\lambda,0}~\mbox{is not one-to-one}\},$$
and {\it left spectrum} of $T$ as \index{$\sigma_l(T)$} 
$$\sigma_{l}(T) := \{\lambda \in \mathbb C^d : \partial_{T-\lambda,0}~\mbox{is not bounded from below}\}.$$ 
\begin{remark}
Note that $\sigma_p(T) \subseteq \sigma_{l}(T) \subseteq \sigma(T).$
\end{remark}

It turns out that the Taylor spectrum of $T$ is a nonempty compact subset of $\mathbb C^d$, which has spectral mapping property for polynomial mappings $p$ from $\mathbb C^d$ into $\mathbb C^{d'}$ for any positive integer $d'.$

The {\it spectral radius for the Taylor spectrum} $\sigma(T)$ of a commuting $d$-tuple $T$ on $\mathcal H$ is defined as \index{$r(T)$} 
\beqn
\label{def-sp-rad} 
r(T):=\max \{\|z\|_2 : z \in \sigma(T)\}.
\eeqn 
We recall 
the spectral radius formula for the Taylor spectrum $\sigma(T)$ of a commuting $d$-tuple $T$ (\cite{CZ}, \cite{MuS}) : 
\beq \label{sp-rad} r(T) = \lim_{n \rar \infty} \|Q^n_{T}(I)\|^{1/2n}. \eeq
In particular, $\sigma(T) \subseteq \{w \in \mathbb C^d : \|w\|_2 \leqslant r(T)\}.$ 
It is also known from \cite[Lemma 3.6]{CY} that the {\it inner radius} $m_{\infty}(T)$ for the left spectrum $\sigma_l(T)$ of $T$  is given by
\index{$m_{\infty}(T)$} 
\beq \label{l-sp-rad}
m_{\infty}(T) =  \sup_{n \geqslant  1} \inf_{\underset{\|h\|=1}{h \in \mathcal H}} \inp{Q^n_{T}(I)h}{h}^{1/2n}. 
\eeq
Here by inner radius $m_{\infty}(T)$, we mean the largest nonnegative number $r$ for which $$\sigma_l(T) \subseteq \{w \in \mathbb C^d : r \leqslant \|w\|_2 \leqslant r(T)\}.$$
For $k=0, \cdots, d$, let $H^k(T)$ denote the $k^{\mbox{\tiny th}}$ cohomology group appearing in the Koszul complex $K(T)$.
\index{$H^k(T)$} 
We say that $T$ is {\it Fredholm} if $H^k(T)$ is finite dimensional for every $k=0, \cdots, d.$
The {\it Fredholm index} $\mbox{ind}(T)$ of a Fredholm $d$-tuple $T$ is the Euler characteristic of $K(T)$ given by \index{$\mbox{ind}(T)$} 
\beq
\label{index} \mbox{ind}(T):= \sum_{k=0}^d (-1)^k \dim H^k(T).
\eeq
The {\it essential spectrum} of $T$ is defined as 
\index{$\sigma_e(T)$}
$$ \sigma_e(T) := \lbrace \lambda \in \mathbb{C}^d : T-\lambda\text{\: is not Fredholm}\rbrace.$$ 
Clearly, essential spectrum is a subset of the Taylor spectrum.
By Atkinson-Curto Theorem \cite{Cu-0}, $\sigma_e(T)=\sigma(\pi(T))$, where $\pi$ is Calkin map and $\pi(T):=(\pi(T_1), \cdots, \pi(T_d)).$ 
\index{$\pi(T)$}
In particular, essential spectrum is a nonempty compact set with polynomial spectral mapping property.

\section{Classical Multishifts}

For a given multisequence ${\bf w}=\left\{{w^{(j)}_\alpha} : 1 \leqslant j \leqslant d, ~\alpha \in {\mathbb
N}^d\right\}$ of complex numbers and an orthonormal basis $\{e_\alpha\}_{\alpha \in \mathbb N^d}$ of a Hilbert space $\mathcal H$, we define {\it $d$-variable
weighted shift} $S_{\bf w} = (S_1, \cdots, S_d)$ as
\index{$S_{\bf w}$}
\beqn S_je_\alpha \mathrel{\mathop:}= 
w^{(j)}_\alpha e_{\alpha + \epsilon_j}~(1 \leqslant j \leqslant d). \eeqn
For convenience, we refer to $S_{\bf w}$ as the {\it classical multishift}. 
%The notation $T : \{{w^{(j)}_n}\}_{ n \in \mathbb N^m}$ will mean
%that $T$ is the $m$-variable weighted shift tuple with weight
%multi-sequence $\left\{{w^{(j)}_n} : 1 \leqslant j \leqslant m, n \in {\mathbb
%N}^m\right\}$.
Notice that $S_j$ commutes with $S_k$ if and only if
$w^{(j)}_\alpha w^{(k)}_{\alpha + \epsilon_j}=w^{(k)}_\alpha w^{(j)}_{\alpha +
\epsilon_k}$ for all $\alpha \in {\mathbb N}^d.$ Moreover, $S_1, \cdots, S_d$ are bounded if and only if
\beq \label{bd-cla-s}
\sup \big\{|w^{(j)}_\alpha| : 1 \leqslant j
  \leqslant d, \alpha \in {\mathbb N}^d \big\} < \infty.
\eeq
In this text, we always assume that the multisequence
$\bf w$
consists of positive numbers and satisfies \eqref{bd-cla-s}. 

Let $S_{\bf w}$ be a classical multishift. Define $\gamma_\alpha := \|S^\alpha_{\bf w} e_{0}\|~(\alpha \in \mathbb N^d)$, where $0$ is the $d$-tuple in $\mathbb N^d$ with all entries being zero. Consider the Hilbert space $H^2(\gamma)$ of formal power series \index{$H^2(\gamma)$}
\beqn
f(z)=\sum_{\alpha \in \mathbb N^d}a_\alpha z^\alpha
\eeqn
such that
\beqn
\|f\|^2_{H^2(\gamma)}:=\sum_{\alpha \in \mathbb N^d}|a_\alpha|^2 \gamma^2_\alpha
< \infty.
\eeqn
It is worth noting that $S_{\bf w}$ is unitarily equivalent to the
$d$-tuple $M_z=(M_{z_1}, \cdots, M_{z_d})$ of multiplication by the
co-ordinate functions $z_1, \cdots, z_d$ on the corresponding space
$H^2(\gamma)$ (\cite[Proposition 8]{JL}).

%Let $S_{\bf w}$ be a toral left invertible classical multishift.
%Then the operator tuple $S^{\mathfrak{t}}_{\bf w}$ toral Cauchy dual of $S_{\bf w}$ is given
%by 
%\beqn 
%S^{\mathfrak{t}}_j e_\alpha := \frac{1}{w^{(j)}_\alpha} e_{\alpha + \epsilon_j}~(1 \leqslant j \leqslant
%d). 
%\eeqn 
%Note that $S^{\mathfrak{t}}_{\bf w}$ is also a commuting $d$-variable weighted shift with
%weight  multisequence 
%\begin{equation*}
%\Big\{\frac{1}{w^{(j)}_\alpha} : 1 \leqslant j  \leqslant d, ~\alpha \in {\mathbb{N}}^d\Big\}.
%\end{equation*}

%The spherical Cauchy dual $S^{\mathfrak s}_{\bf w}$ of a joint left invertible $S_{\bf w} :
%\{{w^{(j)}_\alpha}\}_{\alpha \in \mathbb N^d}$ is the $d$-variable weighted
%shift given by \beqn S^{\mathfrak{s}}_je_\alpha :=
%\frac{w^{(j)}_\alpha}{\delta_{\alpha, S_{\bf w}}} \;\, e_{\alpha + \epsilon_j}~(1 \leqslant j
%\leqslant d), \eeqn where
%\beqn
%\label{delta_n} \delta_{\alpha, S_{\bf w}} := {\sum_{j=1}^d
%\left(w^{(j)}_{\alpha}\right)^2} \enspace ~(\alpha \in {\mathbb{N}}^d).
%\eeqn
%It is easily seen that that $S^{\mathfrak{s}}_{\bf w}$ is commuting if and
%only if $\delta_{\alpha + \epsilon_j, S_{\bf w}} = \delta_{\alpha + \epsilon_k, S_{\bf w}}$
%for all $1 \leqslant j, k \leqslant d$ \cite[Section 6]{CC}.

Let us discuss some basic examples of classical multishifts.
%In view of the preceding discussion, it suffices to describe the reproducing kernel of the corresponding Hilbert spaces.
\begin{example} \label{1.1.1}
For integers $a, d>0$, let $\mathcal H_{a, d}$ be the reproducing kernel Hilbert space \index{$\mathcal H_{a, d}$}  of
holomorphic functions on the open unit ball $\mathbb B^d$ with 
\index{$\kappa_{\mathcal H_{a, d}}$} reproducing
kernel 
\beqn
\kappa_{\mathcal H_{a, d}}(z, w)=\frac{1}{(1- \inp{z}{w})^a}~(z, w \in
\mathbb B^d).
\eeqn
The multiplication $d$-tuple $M_{z, a}$ on $\mathcal H_{a, d}$ is unitarily
equivalent to the weighted shift $d$-tuple $S_{{\bf w}, a}$ with weight
\index{$S_{{\bf w}, a}$} multisequence
\beqn
\label{Sz-Be-Dru}
w^{(j)}_{\alpha, a}=\sqrt{\frac{\alpha_j + 1}{|\alpha|+a}}~(\alpha \in \mathbb
N^d, j=1, \cdots, d),
\eeqn
(see \cite[Proof of Lemma 4.4]{GR}).
The spaces $\mathcal H_{d, d}, \mathcal H_{d+1, d},
\mathcal H_{1, d}$ are commonly known as the {\it Hardy space} $H^2(\partial \mathbb B^d)$,
the {\it Bergman space} $A^2(\mathbb B^d)$, the {\it Drury-Arveson
space} $H^2_d$ respectively.
The associated classical multishifts $S_{\bf w, d},
S_{\bf w, d+1}, S_{\bf w, 1}$
are referred to as the {\it
Szeg\"o $d$-shift}, the {\it Bergman $d$-shift}, the
{\it Drury-Arveson $d$-shift} respectively.
\end{example}

%We will refer to the multishifts $M_{z, a}$ as {\it Agler-type multishifts}. 
For ready reference, we record the following proposition about various spectral parts of $S_{{\bf w}, a}$ (see \cite[Proposition 2.6]{GRS} and \cite[Theorem 3.4]{CY}).

\begin{proposition} \label{sp-th-classical}
Let $S_{{\bf w}, a}$ be as defined in Example \ref{1.1.1}. Then
\beqn \sigma(S_{{\bf w}, a})=\mbox{cl}({\mathbb B}^d), ~\sigma_p(S_{{\bf w}, a})=\emptyset, ~\sigma_p(S^*_{\bf w, a})=\mathbb B^d,~\sigma_e(S_{{\bf w}, a})=\partial{\mathbb B}^d=\sigma_l(S_{{\bf w}, a}). \eeqn
\end{proposition}

We will investigate later the so-called tree analogs of $S_{{\bf w}, a}$
(refer to Section 1.4).

\section{Weighted Shifts on Directed Trees}

In this section, we recall some basic concepts from the theory of directed graphs which will be frequently used in the subsequent chapters. The reader is referred to R. Diestel \cite{Di} for a detailed exposition on graph theory (refer also to \cite{JJS} for a brief account of the theory of directed trees).

A {\it directed graph} is a pair $\mathscr T= (V,\mathcal E)$, where  $V$ is a nonempty set and $\mathcal E$ is a nonempty subset of $V \times V \setminus \{(v,v): v \in V\}$. An element of $V$ (resp. $\mathcal E$) is called a {\it vertex} (resp. an {\it edge}) of $\mathscr T$.  A finite sequence $\{v_i\}_{i=1}^n$ of distinct vertices is said to be a {\it circuit} in $\mathscr T$ if $n \geqslant  2$, $(v_i,v_{i+1}) \in \mathcal E$ for all $1 \leqslant i \leqslant n-1$ and $(v_n,v_1) \in \mathcal E$. 
We say that two distinct vertices $u$ and $v$ of $\mathscr T$ are {\it connected by a path} if there exists a finite sequence $\{v_i\}_{i=1}^n$ of distinct vertices of $\mathscr T$ $(n \geqslant  2)$ such that $u=v_1$, $v_n=v$ and $(v_i,v_{i+1})$ or $(v_{i+1},v_i) \in \mathcal E$ for all $1 \leqslant i \leqslant n-1$.
A directed graph $\mathscr T$ is said to be {\it connected} if any two distinct vertices of $\mathscr T$ can be connected by a path in $\mathscr T.$  For a
subset $W$ of $V$, define 
$$\child{W} := \bigcup_{u\in W} \{v\in V
\colon (u,v) \in \mathcal E\}.$$ 
One may define inductively $\childn{n}{W}$ \index{$\childn{n}{W}$} for 
$n \in \mathbb N$ as follows: 
\beqn
\childn{n}{W}:= 
\begin{cases} W & ~\mbox{if }~n=0, \\
\child{\childn{n-1}{W}} & ~\mbox{if~} n \geqslant 
1. \end{cases}
\eeqn
%Let
%$$
%\des{W}=\bigcup_{n=0}^{\infty} \childn{n}{W}.$$
Given $v\in V$, we write $\child{v}:=\child{\{v\}}$,
$\childn{n}{v}:=\childn{n}{\{v\}}$. A member of $\child{v}$ is called a {\it child} of $v.$ 
%For a given vertex $v\in V$, if there exists a unique vertex
%$u \in V$ such that $(u,v)\in \mathcal E$, we say that $v$ has a {\em parent} $u$ and denote %it by $\parent{v}$. 
The {\it descendants} of a vertex $v \in V$ \index{$\mathsf{Des}(v)$} is given by
\beqn
\mathsf{Des}(v):=\bigcup_{n=0}^{\infty} \childn{n}{v}.
\eeqn
For a given vertex $v \in V,$ consider the set $\mathsf{Par}(v):=\{u \in V : (u, v) \in \mathcal E\}$ (set of ``generalized" parents). 
\index{$\mathsf{Par}(v)$}
If $\mathsf{Par}(v)$ is singleton, then the unique vertex in $\mathsf{Par}(v)$ is called the {\it parent} of $v$, 
\index{$\parent{v}$}
which we denote by $\parent{v}.$
Let the subset $\mathsf{Root}(\mathscr T)$ 
\index{$\mathsf{Root}(\mathscr T)$}
of $V$ be defined as
$$\mathsf{Root}(\mathscr T) := \{v \in V : \mathsf{Par}(v)  = \emptyset\}.$$
%Let the subset $\boldsymbol{\mathsf{Root}}(\mathscr T)$ of $V$ is defined as
%$$\mathsf{Root}(\mathscr T) := \{v \in V : \mathsf{Par}(v) = \emptyset\}.$$
Then an element of $\mathsf{Root}(\mathscr T)$ is called a {\it root} of $\mathscr T$. If $\mathsf{Root}(\mathscr T)$ is singleton, then its unique element 
is denoted  by $\mathsf{root}$. 
We set $V^\circ:=V \setminus \mathsf{Root}(\mathscr T)$. \index{$V^{\circ}$}
A directed graph $\mathscr T= (V,\mathcal E)$ is called a {\it directed tree} if 
$\mathscr T$ has no circuits, $\mathscr T$ is connected and
each vertex $v \in V^\circ$ has a unique parent. 
\begin{remark}
It is well-known that every directed tree has at most one root \cite[Proposition 2.1.1]{JJS} (see Figure 1.1). 
\end{remark}

The following example is borrowed from \cite[Chapter 6]{JJS}.

\begin{example} \label{T-n-k} For a positive integer $n_0$ and $k_0 \in \mathbb N,$ we define the directed tree 
\index{$\mathscr T_{n_0, k_0}$}
$\mathscr T_{n_0, k_0}=(V, \mathcal E)$ as follows:
\beqn
V &=& \{-1, \cdots, -k_0\} \cup \mathbb N, \\
\mathcal E &=& \{(j, j+1) : j=-k_0, \cdots, -1\} \cup \{(0, j) : j =1, \cdots, n_0\} \\ & \cup & \cup_{j=1}^{n_0} \{(j + (l-1)n_0, j+ln_0) : l \Ge 1\}. 
\eeqn
(see Figures 1.2 and 1.3 for the cases $(n_0, k_0)=(1, 0)$ and $(n_0, k_0)=(2, 0)$ respectively). 
\end{example}

A directed graph $\mathscr T$ is said to be
\begin{enumerate}
\item[(i)] {\it rooted} if it has a unique root.
%\item[(ii)] {\it rootless} if it has no root.
\item[(ii)]
{\it locally finite} if $\mbox{card}(\child u)$ is finite for all $u \in V.$ 
\item[(iii)]
{\it leafless} if every vertex has at least one child.
\end{enumerate}

\begin{figure}
\includegraphics[scale=.5]{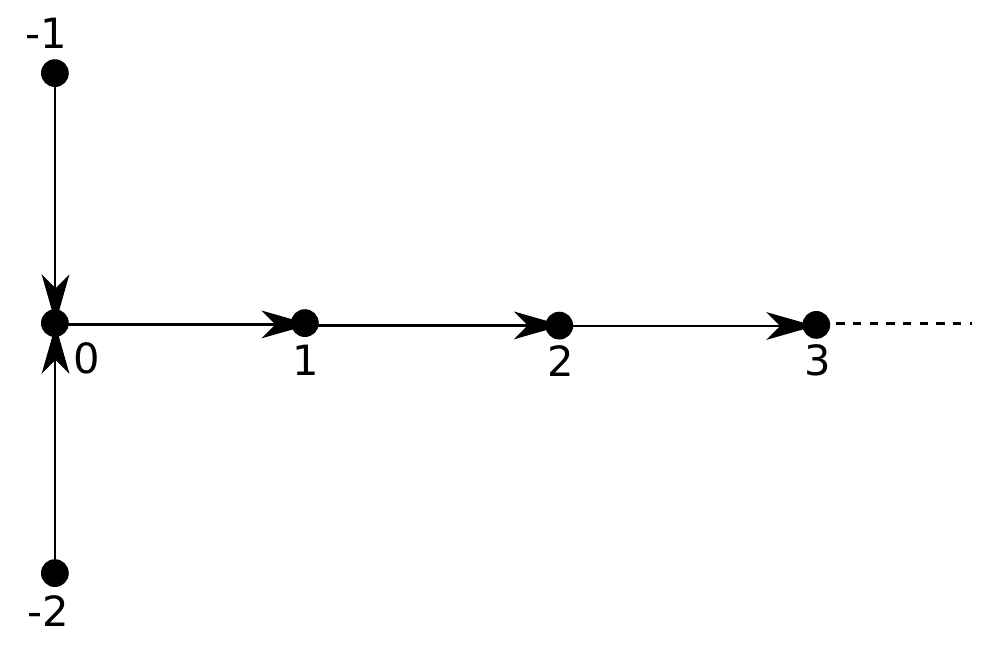} \caption{A Directed Graph which is not a Directed Tree}
\end{figure}

\begin{figure}
\includegraphics[scale=.5]{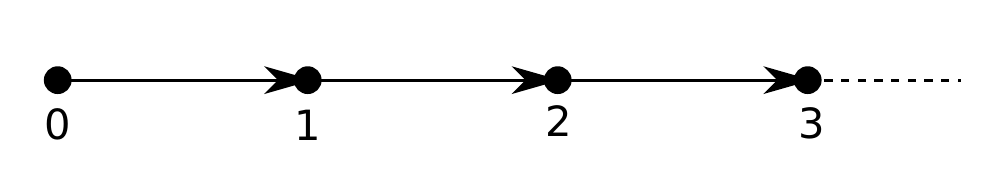} \caption{The Directed Tree $\mathscr T_{1, 0}$}
\end{figure}

\begin{figure}
\includegraphics[scale=.5]{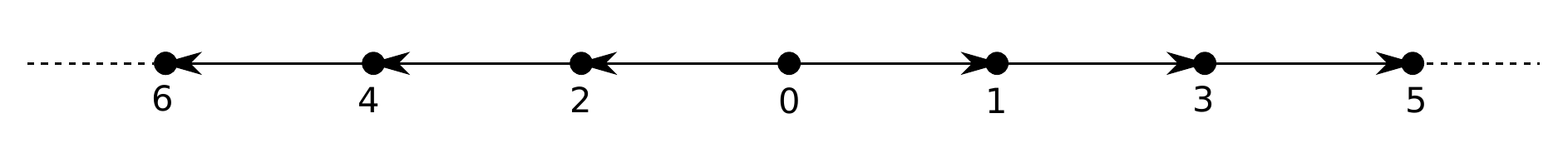} \caption{The Directed Tree $\mathscr T_{2, 0}$}
\end{figure}

Let $\mathscr T= (V,\mathcal E)$ be a directed tree and
let $l^2(V)$ stand for the Hilbert
space of square summable complex functions on $V$
equipped with the standard inner product. 
Note that
the set $\{e_u\}_{u\in V}$ is an
orthonormal basis of $l^2(V)$, where $e_u \in l^2(V)$
is the indicator function of $\{u\}$. Given a system
$\lambda = \{\lambda_v\}_{v\in V^{\circ}}$ of nonzero complex numbers, 
we define the {\em weighted shift operator} $S_{\lambda}$ 
\index{$S_{\lambda}$}
on ${\mathscr T}$
with weights $\lambda$ by
   \begin{align*}
   \begin{aligned}
{\mathscr D}(S_{\lambda}) & := \{f \in l^2(V) \colon
\varLambda_{\mathscr T} f \in l^2(V)\},
   \\
S_{\lambda} f & := \varLambda_{\mathscr T} f, \quad f \in {\mathscr
D}(S_{\lambda}),
   \end{aligned}
   \end{align*}
where $\varLambda_{\mathscr T}$ is the mapping defined on
complex functions $f$ on $V$ by
   \begin{align*}
(\varLambda_{\mathscr T} f) (v) :=
   \begin{cases}
\lambda_v \cdot f\big(\parent v\big) & \text{if } v\in
V^\circ,
   \\
   0 & \text{if } v \text{ is a root of } {\mathscr T}.
   \end{cases}
   \end{align*}
Unless stated otherwise, $\{\lambda_v\}_{v\in V^{\circ}}$ consists of nonzero complex numbers and $S_{\lambda}$ belongs to $B(l^2(V)).$ It may be concluded from \cite[Proposition 3.1.7]{JJS} that $S_\lambda$ is an injective weighted shift on $\mathscr T$ if and only if $\mathscr T$ is leafless. 

{\it In what follows, we always assume that all the directed trees considered in the remaining part of this paper are leafless.} 

Let $\mathscr T=(V, \mathcal E)$ be a rooted directed tree with root $\mathsf{root}$. Then
\beq \label{disjoint-0}
V = \bigsqcup_{n = 0}^{\infty} \childn{n}{\mathsf{root}}~(\mbox{disjoint union})
\eeq
(\cite[Corollary 2.1.5]{JJS}).
%We set $n_{\mathsf{root}}:=0$ and for
For $u\in V$, let $\alpha_u$ denote the unique integer in $\mathbb N$ 
(to be referred to as the {\it depth of $u$ in $\mathscr T$}) such that
$u \in \mathsf{Chi}^{\langle \alpha_u\rangle}(\mathsf{root})$. We use the convention that 
$\mathsf{Chi}^{\langle{j}\rangle}(\mathsf{root})=\emptyset$ if $j < 0.$ Similar convention holds for $\mathsf{par}$.
The {\it branching index} $k_{\mathscr T} \in \mathbb N \cup \{\infty\}$ of a rooted directed tree $\mathscr T$ is defined as  
\[k_\mathscr{T}:=\begin{cases}
 1+\sup\{\alpha_w:w\in V_{\prec}\}& \text{if} ~V_{\prec}~ \text{is nonempty},\\
 0 & \text{if $V_{\prec}$ is empty},
\end{cases}
\]
where $V_{\prec}:=\{u\in V: \mbox{card}(\mathsf{Chi}(u)) \geqslant 2\}$.
If $V_{\prec}$ is finite, then $k_\mathscr T$ is necessarily finite but converse is not true in general \cite[Remark 2]{CT}.

\begin{remark} Let $\mathscr T_{n_0, k_0}$ be as discussed in Example \ref{T-n-k}. Note that $\mathscr T_{n_0, k_0}$ is a locally finite, rooted directed tree with branching index 
%$k_0+1.$ 
\[k_\mathscr{T}=\begin{cases}
k_0+1& \text{if}~ n_0 \Ge 2,\\
 0 & \text{otherwise}.
\end{cases}
\]
\end{remark}

If $S_\lambda$ is a left-invertible weighted shift on a rooted directed tree, then $S_\lambda$ has wandering subspace property. 
This fact was recorded in \cite[Theorem 2.7(iii)]{CT}. But it turns out that this is a general nature of a (bounded) weighted shift on a rooted directed tree and the left-invertibility is no longer required. We illustrate this fact in the following proposition.  

\begin{proposition}\label{wandering-one}
Let $\mathscr T = (V, \mathcal E)$ be a rooted directed  tree and $S_\lambda \in B(l^2(V))$ be a weighted shift on $\mathscr T$. Set $E:= \ker S^*_{\lambda}$. Then
\beq \label{wandering-one-eq}
\bigvee_{k \in \mathbb N} S^k_\lambda (E) = l^2(V).
\eeq	
\end{proposition}

\begin{proof}
Set $M:= \bigvee_{k \in \mathbb N} S^k_\lambda (E)$. We claim that $e_v \in M$ for all $v \in \childn{n}{\mathsf{root}}$ and for all $n \in \mathbb N$. We prove this by induction on $n$. Recall from the \cite[Proposition 3.5.1(ii)]{JJS}, that
\beq \label{formula-k}
E=[e_{\mathsf{root}}] \oplus \bigoplus_{v \in 
V}\left(l^2(\mathsf{Chi}(v)) \ominus [\Gamma_v] \right),
\eeq
where
$\Gamma_v : \mathsf{Chi}(v) \rar \mathbb C$ 
\index{$\Gamma_v$}
is given by 
$\Gamma_v = \sum_{u \in \mathsf{Chi}(v)} \lambda_u e_u = S_\lambda e_v$. 
Clearly, $e_{\mathsf{root}} \in M$. Thus the claim holds true for $n = 0$. Suppose it is true for some $n \in \mathbb N$. That is, $e_u \in M$ for all $u \in \childn{n}{\mathsf{root}}$. Let $v \in \childn{n+1}{\mathsf{root}}$. Then $v \in \child{u}$ for some $u \in \childn{n}{\mathsf{root}}$. By the induction hypothesis, $e_u \in M$. Since $M$ is $S_\lambda$-invariant, $S_\lambda e_u = \Gamma_u \in M$ and hence $[\Gamma_u] \subseteq M$. Further, as $E \subseteq M$, $l^2(\child{u}) \ominus [\Gamma_u] \subseteq M$. Thus $l^2(\child{u}) \subseteq M$, which in turn implies that $e_v \in M$. Thus the claim stands verified. By \eqref{disjoint-0}, it follows that $e_v \in M$ for all $v \in V$, and hence, $l^2(V) \subseteq M$. Thus \eqref{wandering-one-eq} stands true.
\end{proof}
The argument above relies completely on the formula \eqref{formula-k} for the kernel of $S^*_{\lambda}.$ Clearly, this formula is associated with a system of linear equations corresponding to vertices from the branching set. In case of several variables, this correspondence becomes highly involved. This is one of the difficulties in the derivation of the wandering subspace property in several variables (see Theorem \ref{wandering}).

%\section{Tree Analog of the multiplication $d$-tuple $M_{z, a}$}
\section{Overture}

In this section, we briefly discuss some important aspects of this work. The exposition here is far from being complete, but it conveys some of the essential ideas presented in this text. 
%It is evident that distinct discrete structures $\mathscr T$ yield different multishifts.
Motivated by \cite[Question 4.7]{CY} about the classification of so-called spherical tuples of higher multiplicity, we construct tree analogs of the multiplication $d$-tuples $M_{z, a}$ as discussed in Example \ref{1.1.1}. We outline this construction as follows.

Consider the directed Caretsian product $\mathscr T=(V, \mathcal E)$ of locally finite, rooted directed trees $\mathscr T_1, \cdots, \mathscr T_d$ of finite joint branching index and let $S_{\lambdab_{\mf C_a}}$ denote the multishift on $\mathscr T$ with
weights given by
\beqn
\lambda^{(j)}_w = \sqrt{\frac{1}{{\mbox{card}(\childi{j}{v})}}} \sqrt{\frac{\alpha_{v_j}  + 1}{|\alpha_v| + a}}~\mbox{for~}w \in \childi{j}{v},~ v \in V~\mbox{and~} j=1, \cdots, d.
\eeqn
Here $a$ is a positive integer and $\alpha_u \in \mathbb N^d$ denotes the depth of $u \in V$ in $\mathscr T$.
It turns out that the multishift $S_{\lambdab_{\mf C_a}}$ is unitarily equivalent to multiplication $d$-tuple $\mathscr M_{z, a}$ acting on reproducing kernel Hilbert space $\mathscr H_{a, d}$ of $E$-valued holomorphic functions on the open unit ball in $\mathbb C^d$, where $E$ denotes the joint kernel of $S^*_{\lambdab_{\mf C_a}}$.  The associated reproducing kernel $\kappa_{\mathscr H_{a, d}} : \mathbb B^d \times \mathbb B^d \rar B(E)$ is given by
\beqn
 \kappa_{\mathscr H_{a, d}}(z, w) = \frac{1}{(1-\inp{z}{w})^a}\, P_{[e_\rootb]} + \sum_{\underset{F \neq \emptyset}{F  \in \mathscr{P}}} \sum_{u \in \Omega_F} \kappa_{u, F}(z, w),
\eeqn
where 
\beqn 
\kappa_{u, F}(z, w)=\sum_{\alpha \in \mathbb N^d}  \left(\frac{\alpha_{u}!}{(\alpha_{u}+\alpha)!} \right)\left({\prod_{j=0}^{|\alpha|-1}(|\alpha_u|+a + j)} \right) z^{\alpha} \overline{w}^{\alpha}\, P_{\mathcal L_{u, F}} 
\eeqn
with $P_{\mathcal M}$ being the orthogonal projection of $\mathcal H$ onto a subspace $\mathcal M$ of $\mathcal H$.
We refer the reader to Theorem \ref{S-c-a-kernel} for a precise statement.
\begin{remark}
In case $\mathscr T_j=\mathscr T_{1, 0}$ for $j=1, \cdots, d$, only first series in $\kappa_{\mathscr H_{a, d}}(z, w)$ survives, and hence we obtain the kernel $$\frac{I_{E}}{(1-\inp{z}{w})^a} ~(z, w \in \mathbb B^d).$$
\end{remark}
Let us try to understand the above formula for $\kappa_{\mathscr H_{a, d}}(z, w)$. The following decomposition of the joint kernel $E$ of $S^*_{\lambdab_{\mf C_a}}$ is useful in this regard:
%understanding the expression of $\kappa_{\mathscr H_{a, d}}(z, w)$.
\beqn 
E = [e_\rootb] \oplus \bigoplus_{\underset{F \neq \emptyset}{F  \in \mathscr{P}}} \bigoplus_{u \in \Omega_{F}} \mathcal L_{u, F}.
\eeqn 
Here $\mathscr P$ denotes the power set of $\{1, \cdots, d\}$ and $\Omega_F$ is a certain indexing set corresponding to $F \in \mathscr P.$ In particular, the first series appearing in $\kappa_{\mathscr H_{a, d}}(z, w)$ corresponds to $[e_{\rootb}]$ while $\kappa_{u, F}$ corresponds to $\mathcal L_{u, F}$, where $\mathcal L_{u, F}$ is a subspace associated with certain system of linear equations related to $S_{\lambdab_{\mf C_a}}$ (the reader is referred to Chapters 4 and 5 for a detailed discussion). 
It is worth noting that the spaces $\mathscr H_{a, d}$ are unnoticed even in dimension $d=1.$ Indeed, in this case, the reproducing kernel $\kappa_{\mathscr H_a} = \kappa_{\mathscr H_{a, 1}}$ takes a concrete form:
\beqn
\kappa_{\mathscr H_a}(z, w) &=&  \frac{1}{(1-z\overline{w})^a}\, P_{[e_\rootb]} \\  &+& 
\sum_{v \in V_{\prec}} \sum_{n=0}^{\infty}  \frac{(\alpha_v + n + a)! (\alpha_v +1)!}{(\alpha_v +a)! (\alpha_v + n+1)!}~{z^n \overline{w}^n}\, P_{l^2(\child{v}) \ominus [\Gamma_v]} ~(z, w \in \mathbb D),
\eeqn
where $V_{\prec}$ denotes the set of branching vertices of $V.$
One can rewrite this formula using the hypergeometric funcion ${}_2F_1(a, b, c, t)$ \cite[Pg 217]{SSV}:
\beqn
\kappa_{\mathscr H_a}(z, w) &=& {}_2F_1(a, 1, 1, z\overline{w}) ~P_{[e_\rootb]} \\ & + & \sum_{v \in V_{\prec}} {}_2F_1(\alpha_v+a+1, 1, \alpha_v+2, z\overline{w})~P_{l^2(\child{v}) \ominus [\Gamma_v]} ~(z, w \in \mathbb D).
\eeqn
An alternative verification of this formula (based on Shimorin's analytic model) will be given in Chapter 5.
\begin{remark}
We analyze below the cases in which $a=1$ and $a=2$ in the one-dimensional case. Note that $\kappa_{\mathscr H_1}$ is the Cauchy kernel $\frac{I_E}{1-z\overline{w}}$ while $\kappa_{\mathscr H_2}$ is given by
\beqn 
%\label{repkernel-1-dim} 
\kappa_{\mathscr{H}_2}(z,w) &=& \sum_{n=0}^\infty  (n+1)P_{[e_\rootb]}\, z^n \overline{w}^n  \\ & + &  \sum_{v \in V_{\prec}}  \sum_{n=0}^\infty \Big(\frac{\alpha_v + n + 2}{\alpha_v + 2}\Big)  z^n \overline{w}^n \, P_{l^2(\child{v}) \ominus [\Gamma_v]} ~  (z, w \in \mathbb{D}),
\eeqn
%where $\mathcal B_v$ is an orthonormal basis of $l^2(\child{v}) \ominus [\lambdab^v]$ for $v \in V_{\prec}$ (see \eqref{formula-k}).
Note that $\kappa_{\mathscr H_2}=\frac{1}{(1-z\overline{w})^2}$ in case $\mathscr T=\mathscr T_{1, 0}.$ 
%Further, it can be seen from \eqref{s-ca-moment-f-2} that $\mathscr H$ contains copies of weighted Bergman spaces
\end{remark}

%Let $\mathscr T$ be a directed Cartesian product of locally finite, rooted directed trees with finite joint branching index and let $S_{c_a, \lambdab}$ be the multishift on $\mathscr T.$
It turns out that $S_{\lambdab_{\mf C_a}}$ is finitely multicyclic, essentially normal $d$-tuple with Taylor spectrum being equal to the closed unit ball $\mbox{cl}(\mathbb B^d).$ 
However, we would like to emphasize here that $S_{\lambdab_{\mf C_a}}$ are, in general, not unitarily equivalent to orthogonal direct sums of any number of copies of the classical multishifts $S_{{\bf w}, a}$. For instance, in case $d=1$ and $a=2,$ the defect operator $I - 2 S_{{\bf w}, a}S^*_{{\bf w}, a} + S^2_{{\bf w}, a}S^{*2}_{{\bf w}, a}$ is always an orthogonal projection of rank $1$ \cite[Pg 618]{H}. On the other hand, if $v \in V^{\circ}$ is such that $\mbox{card}(\mathsf{sib}(\mathsf{par}(v))) =1$ and $s_v:=\frac{1}{\mbox{card}(\mathsf{sib}(v))} < 1$, then
\beqn
\inp{(I - 2 S_{\lambdab_{\mf C_a}}S^*_{\lambdab_{\mf C_a}} + S^2_{\lambdab_{\mf C_a}}S^{*2}_{\lambdab_{\mf C_a}})^je_v}{e_v} = (1-s_v)^j~\mbox{for~}j=1, 2,
\eeqn
which shows that $I - 2 S_{\lambdab_{\mf C_a}}S^*_{\lambdab_{\mf C_a}} + S^2_{\lambdab_{\mf C_a}}S^{*2}_{\lambdab_{\mf C_a}}$ is not even idempotent.

%\section{Outline of the Paper}

We conclude this chapter with a brief description of the layout
of the present work. In Chapter 2, we discuss the theory of product of directed graphs in the context of directed trees. The motivation for this chapter comes from the theory of multishifts with which we are primarily concerned. In particular, we pay attention to two important notions, namely, directed Cartesian product and tensor product of directed trees. We will see the significance of the notion of tensor product of directed trees in the context of so-called spherically balanced multishifts later in Chapter 5. 

In Chapter 3, we formally introduce the notion of multishifts $S_{\lambdab}$ on directed Cartesian product $\mathscr T$ of finitely many rooted directed trees. Apart from various elementary properties of multishifts, we reveal its relation with the shift operator arising from directed semi-tree structure of $\mathscr T$ as ensured in Chapter 2. 
The later half of this chapter deals with spectral properties of multishifts $S_{\lambdab}$ on $\mathscr T$.
A particular attention is given to circularity and analyticity of $S_{\lambdab}$. Indeed, $S_{\lambdab}$ turns out to be strongly circular and separately analytic.
These properties are then used to show that the point spectrum of $S_{\lambdab}$ is empty and the Taylor spectrum is Reinhardt. Further, we obtain a matrix decomposition of $2$-variable multishifts and discuss some of its consequences to spectral theory. In particular, we compute essential spectra for a family of multishifts.

Chapter 4 is devoted to the description of the joint kernel of $S^*_{\lambdab}$. This in turn relies on decompositions of vertex set of product of directed trees and that of the underlying Hilbert space. It turns out that the problem of computing the joint kernel $\ker S^*_{\lambdab}$ of $S^*_{\lambdab}$ is equivalent to solving a system of (possibly infinitely many) linear equations. We illustrate this with the help two instructive examples in which $\ker S^*_{\lambdab}$ is explicitly computed.
The description of $\ker S^*_{\lambdab}$ enables to derive the wandering subspace property for $S_{\lambdab}$ on $\mathscr T$. It is to be noted that the situation gets far simpler in case of either one variable weighted shifts or classical multishifts. As a consequence, we obtain a multivariable counterpart of Shimorin's model in this context, and use it to show that these multishifts belong to the Cowen-Douglas class.

In Chapter 5, we discuss two notions of balanced multishifts, namely, spherical and toral. We use the classification of torally balanced multishifts to obtain a local analog of von Neumann's inequality. 
The classification of spherically balanced multishifts is given in terms of certain integral representations. Unlike the classical case \cite{CK}, several Reinhardt measures appear in this characterization. 
In the classification of spherically balanced multishifts, the notion of tensor product $\mathscr T^{\otimes}$ of directed trees appears naturally. Indeed, various properties of $S_{\lambdab}$ on $\mathscr T$ are reflected in the corresponding properties of the one variable shift on the component of $\mathscr T^{\otimes}$ containing root. This correspondence allows us, in particular, to compute the spectral radius of Taylor spectrum and the inner spectral radius of left spectrum for $S_{\lambdab}.$  
In this chapter, we also discuss special classes of joint subnormal and joint hyponormal multishifts $S_{\lambdab}$ on $\mathscr T.$ In particular, we characterize these classes within the class of spherically balanced multishifts.
We illustrate these results with a family of examples which can be thought of as tree analogs of the multiplication tuples on the reproducing kernel Hilbert spaces associated with the kernels $\frac{1}{(1-\inp{z}{{w}})^a}~(z, w \in \mathbb B^d, a > 0).$ 
%We conclude the paper with some possible directions for future work and some unresolved problems pertaining to multishifts $S_{\lambdab}$ on $\mathscr T.$

\chapter{Product of Directed Trees}

In this chapter, we discuss two well-studied notions of product of directed trees, namely, the directed Cartesian product and the tensor product (\cite{Sa}, \cite{We}, \cite{Mc}, \cite{Ha}, \cite{Feig}). These notions can certainly be introduced in the general context of (directed) graphs. However, since the main objects of the present study are multishifts on product of directed trees, we confine ourselves to directed trees.

\section{Directed Cartesian Product of Directed Trees}

The definition of the directed Cartesian product of two directed graphs has been introduced and studied by G. Sabidussi \cite{Sa} (refer also to \cite{Feig}). This notion readily generalizes to the case of finitely many directed trees as given below. 

\begin{definition}
Let $d$ be a positive integer
and let $\mathscr T_j = (V_j, \mathcal E_j)~(j=1, \cdots, d)$ be a collection of directed trees. 
The {\it directed Cartesian product of $\mathscr T_1,$ $\cdots, \mathscr T_d$} is a directed graph $\mathscr T=(V, \mathcal E),$ 
\index{$\mathscr T=(V, \mathcal E)$}
where $V:=V_1 \times \cdots \times V_d$ and 
\beqn
\mathcal E := \Big\{(v, w) \in V \times V: 
~\mbox{there is a positive integer~} k \in \{1, \cdots, d\} \\ \mbox{such that}~ v_j = w_j~ \mbox{for}~ j \neq k~\mbox{and the edge}~ (v_k, w_k) \in \mathcal E_k \Big\},
\eeqn 
where we adhere to the convention that $v \in V = V_1 \times \cdots \times V_d$ is always understood as $v = (v_1, \cdots, v_d)$ with $v_j \in V_j$ for $j = 1, \cdots, d$. We sometimes use the notation $\mathscr T_1 \times \cdots \times \mathscr T_d$ for the directed Cartesian product $\mathscr T$ of $\mathscr T_1, \cdots, \mathscr T_d$.
\end{definition}
\begin{remark} \label{edge}
Note that $\mathcal E$ is precisely the collection of edges $(v, w)$ such that $w_k$ is a child of $v_k$ for some $k$ and $w_j=v_j$ for all $j \neq k.$
\end{remark}

\begin{remark}
In case $d \Ge 2$, $\mathscr T=(V, \mathcal E)$ is never a directed tree. Indeed, $\mbox{card}(\child{u} \cap \child{v}) \Ge  1$ for some $u, v \in V$ with $u \neq v$. This can be seen as follows.
Since $\mathscr T_1, \cdots, \mathscr T_d$ are leafless, for any $w=(w_1, \cdots, w_d) \in V$, consider $u=(u_1, w_2, \cdots, w_d)$ and $v=(w_1, u_2, w_3, \cdots, w_d)$, where  
 $u_j \in \child{w_j}$ for $j=1, 2.$ Note that $$(u_1, u_2, w_3, \cdots, w_d) \in \child{u} \cap \child{v}.$$
\end{remark}

\begin{remark}
%For $j=1, \cdots, d$, let $\mathscr T_j = (V_j, \mathcal E_j)$ be a rooted directed tree with root denoted by $\mathsf{root}_j$.
For $j=1, \cdots, d$, let $\mathscr T_j$ be a rooted directed tree with root denoted by $\mathsf{root}_j$.  Then the directed Cartesian product $\mathscr T$ of $\mathscr T_1, \cdots, \mathscr T_d$ is a rooted direct graph with root given by \index{$\mathsf{root}$}
${\mathsf{root}} =(\mathsf{root}_1, \cdots, \mathsf{root}_d) \in V$.
\end{remark}

%Recall that $\mathbb N^d = \mathbb N \times \cdots \times \mathbb N$ denotes the $d$-fold Cartesian product of the set of non-negative integers $\mathbb N$. 

We discuss below three basic examples of directed Cartesian product.
\begin{example} \label{classical} 
%Consider the directed tree $\mathscr T_{1, 0}$ with the set of vertices 
%$\mathbb{N}$ and the set of edges $\{(j, j+1) : j \in \mathbb N\$. 
Let $\mathscr T_{1, 0}$ be as discussed in Example \ref{T-n-k} and
let $\mathscr T_j = \mathscr T_{1, 0}$ for all $j=1, \cdots, d.$
The directed Cartesian product $\mathscr T=\mathscr T^d_{1, 0}$ of $\mathscr T_1, \cdots, \mathscr T_d$ is given by $\mathscr T=(V, \mathcal E)$, where $V = \mathbb N^d$ and
\beqn
\mathcal E &=& \Big\{(\alpha, \beta) \in \mathbb N^d \times \mathbb N^d: 
~\mbox{there is a positive integer~} k \in \{1, \cdots, d\} \\        & & \mbox{such that} ~ \alpha_j = \beta_j~ \mbox{for}~ j \neq k~\mbox{and}~ \beta_k= \alpha_k + 1\Big\} \\
&=& \big\{(\alpha, \alpha + \epsilon_j) \in \mathbb N^d \times \mathbb N^d: j=1, \cdots, d \big\}
\eeqn
(see Figure 2.1). 
\end{example}

The directed graph $\mathscr T$ discussed above is the {\it $d$-finite Bargmann graph} in disguise. The later one was introduced in \cite[Section 3]{MS}.

\begin{figure}
\includegraphics[scale=.5]{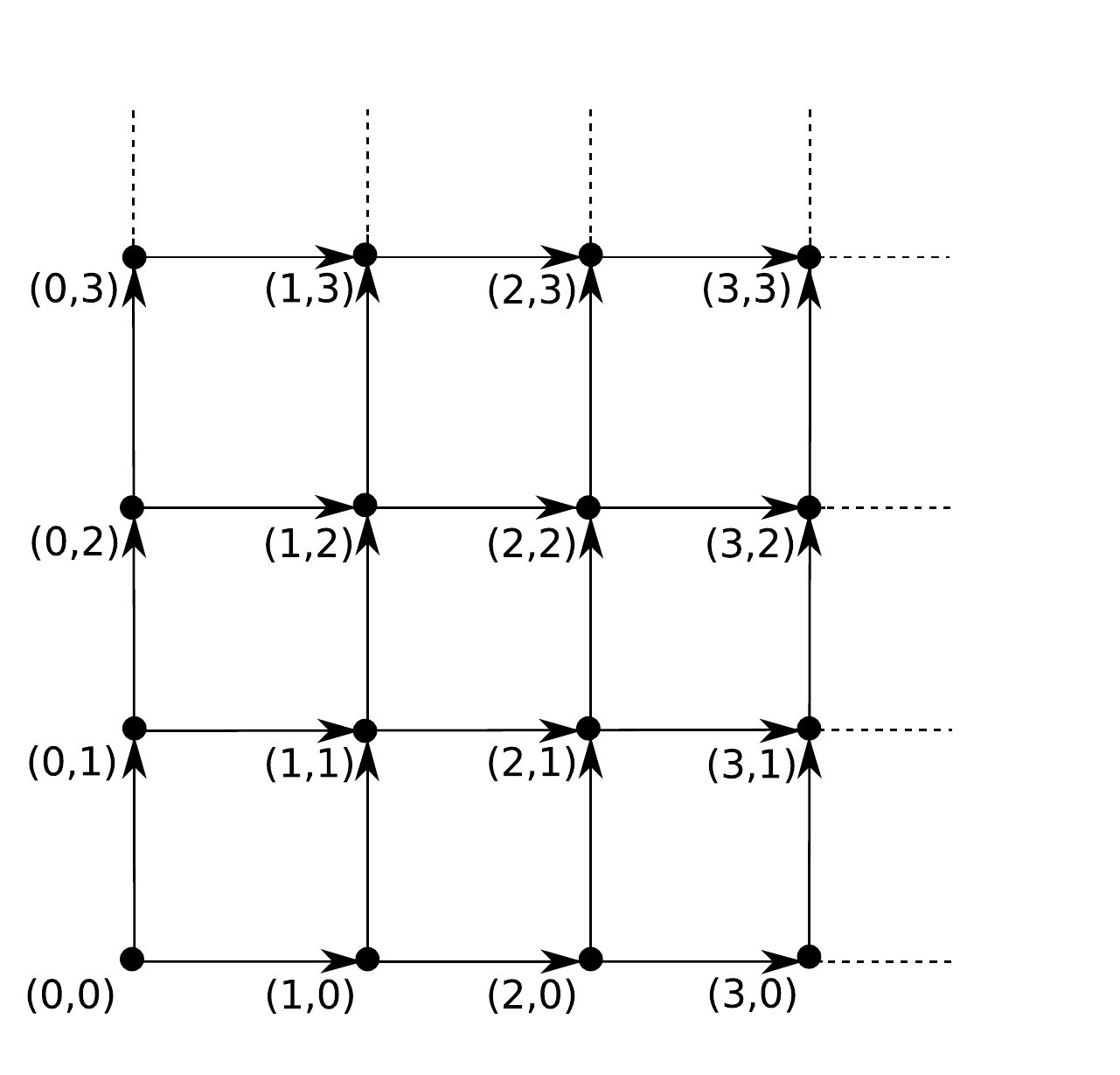} \caption{Directed Cartesian Product $\mathscr T = \mathscr T_{1, 0} \times \mathscr T_{1, 0}$}
\end{figure}

\begin{example}\label{classical-mix}
Let $\mathscr T_{1, 0}, \mathscr T_{2, 0}$ be as discussed in Example \ref{T-n-k}.
%Let $\mathscr T_{2, 0} = (V_1, \mathcal E_1)$ be a rooted directed tree, where $V_1 = \mathbb{N}$ and $\mathcal E_1 = \{(0,1), (0,2)\} \cup \{(n, n+2) : n \geqslant  1\}$. 
Then the directed Cartesian product $\mathscr T=\mathscr T_{2, 0} \times \mathscr T_{1, 0}$ of $\mathscr T_{2, 0}$ and $\mathscr T_{1, 0}$ is given by $\mathscr T = (V, \mathcal E)$, where $V = \mathbb N \times \mathbb N$ and $\big( (m,n), (k,l) \big) \in \mathcal E$ if and only if either $m = k$ and $l = n+1$, or $n = l$ and $$k = \begin{cases} m+2 &~ \mbox{if~} m \neq 0, \\ 1 \mbox{~or~} 2 &~ \mbox{otherwise} \end{cases}$$ (see Figure 2.2).
\end{example}

\begin{figure}
\includegraphics[scale=.5]{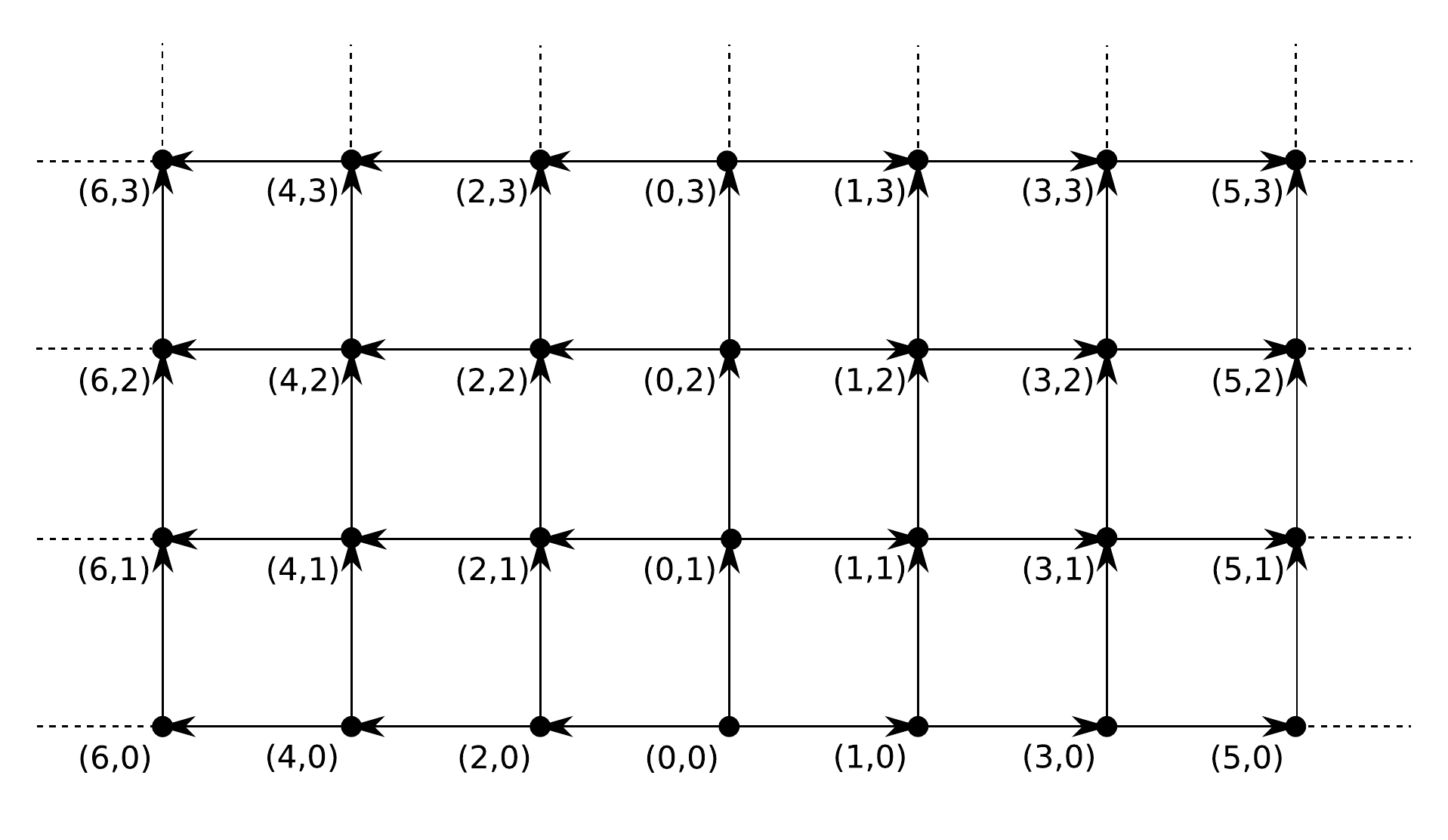} \caption{Directed Cartesian Product $\mathscr T = \mathscr T_{2, 0} \times \mathscr T_{1, 0}$}
\end{figure}

\begin{figure}
\includegraphics[scale=.4]{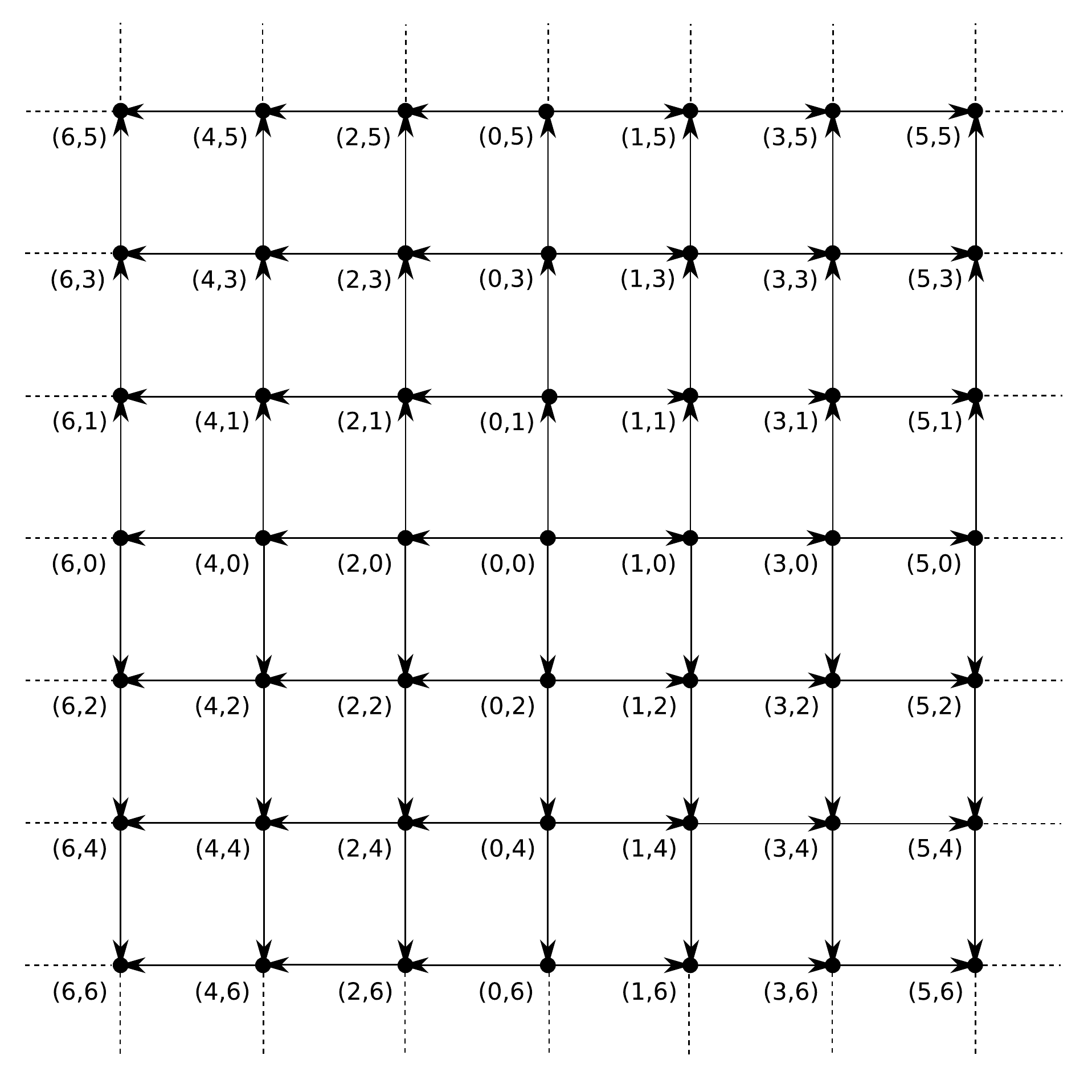} \caption{Directed Cartesian Product $\mathscr T = \mathscr T_{2, 0} \times \mathscr T_{2, 0}$} 
\end{figure}

\begin{example} \label{T1-T1}
Let $\mathscr T_{2, 0}$ be as discussed in Example \ref{T-n-k}. Then the directed Cartesian product of $\mathscr T_{2, 0}$ with itself is given by $\mathscr T = (V, \mathcal E)$, where $V = \mathbb N \times \mathbb N$ and $\big( (m,n), (k,l) \big) \in \mathcal E$ if and only if either $m = k$ and $$l = \begin{cases} n+2 &~ \mbox{if~} n \neq 0, \\ 1 \mbox{~or~} 2 &~ \mbox{otherwise,} \end{cases}$$ or $n = l$ and $$k = \begin{cases} m+2 &~ \mbox{if~} m \neq 0, \\ 1 \mbox{~or~} 2 &~ \mbox{otherwise} \end{cases}$$
(see Figure 2.3).
\end{example}

 %Set $V^{\circ} = V \setminus \{\rootb\}$. 

\begin{definition}
Let $\mathscr T=(V, \mathcal E)$ be the directed Cartesian product of directed trees $\mathscr T_1, \cdots, \mathscr T_d$. 
For $j = 1,\cdots, d$ and $v \in V$, we set 
\beqn \label{child-i}
\childi{j}{v}:=\big\{w \in V : w_j \in \child{v_j}\ \text{and}\ w_k=v_k\ \text{for}\ k \neq j\big\}.\eeqn
Further, for $W \subseteq V$, we define 
$$\mathsf{Chi}_j(W):=\bigcup_{w \in W} \childi{j}{w}.$$
For $k \in \mathbb N,$ we denote $\underbrace{\mathsf{Chi}_j \cdots \mathsf{Chi}_{j}}_{{k}~ \mbox{{\tiny times}}}(W)$ by 
\index{$\childki{j}{k}{W}$}
$\childki{j}{k}{W}$,
where we understand that $\childki{j}{0}{W} = W$.
Further, for $\alpha = (\alpha_1, \cdots, \alpha_d) \in \mathbb N^d$ and $W \subseteq V$, we define \index{$\childnt{\alpha}{W}$}
\beqn \label{child-m}
\childnt{\alpha}{W} := \mathsf{Chi}_1^{\langle \alpha_1 \rangle} \cdots \childki{d}{\alpha_d}{W}.\eeqn 
If $W=\{v\}$ for some $v \in V$, then we use the simpler notation $\childnt{\alpha}{v}$ for $\childnt{\alpha}{\{v\}}.$
\end{definition}

\begin{remark} \label{chi-j}
It may be concluded from Remark \ref{edge}
that $$\child v = \{w \in V : (v, w) \in \mathcal E\}  =
\bigcup_{j=1}^d \childi{j}{v}.$$ Further, for $j=1, \cdots, d$,
$\childnt{\epsilon_j}{v}=\childi{j}{v}$. 
\end{remark}

The following lemma enriches the directed graph $\mathscr T = (V,\mathcal E)$ with a tree-like structure.

\begin{lemma} \label{disjoint}
Let $\mathscr T = (V, \mathcal E)$ be the directed Cartesian product of directed trees $\mathscr T_1, \cdots, \mathscr T_d$. Then we have the following:
\begin{enumerate}
\item[(i)] 
%In particular,
$\mathsf{Chi}_j\childi{i}{v}=\mathsf{Chi}_i\childi{j}{v}$ for all $v \in V$ and $i,j =1, \cdots, d$. 
\item[(ii)] 
For each $\alpha \in \mathbb N^d$ and $v, w \in V$, $$\childnt{\alpha}{v} \cap \childnt{\alpha}{w} = \emptyset~\mbox{if~} v \neq w.$$
\item[(iii)]
For each $\alpha, \beta \in \mathbb N^d$ and $v \in V$, $$\childnt{\alpha}{v} \cap \childnt{\beta}{v} = \emptyset~\mbox{if}~ \alpha \neq \beta.$$
\item[(iv)] For any $n \in \mathbb N$ and $v \in V,$ $$\childn{n}{v} = \bigsqcup_{\underset{|\alpha|=n}{\alpha \in \mathbb N^d}} \childnt{\alpha}{v}.$$
\item[(v)]
For each $m, n \in \mathbb N$ and $v \in V$, $$\childn{m}{v} \cap \childn{n}{v} = \emptyset~\mbox{if}~ m \neq n.$$
\item[(vi)] If $\mathscr T_1, \cdots, \mathscr T_d$ are rooted directed trees, then $$\bigsqcup_{j \in \mathbb N} \childn{j}{\rootb} =V= \bigsqcup_{\alpha \in \mathbb N^d} \childnt{\alpha}{\rootb}.$$
\end{enumerate}
\end{lemma}

\begin{proof}
Note that 
the conclusion in (i) follows from 
\beqn \mathsf{Chi}_j\childi{i}{v} = \begin{cases} 
\big\{u \in V: u_i \in \child{v_i}, u_j \in \child{v_j}\ \text{and}\ u_k=v_k\ \text{for}\ k \neq i,j\big\} & {i \neq j}, \\
\big\{u \in V: u_i \in \childn{2}{v_i}\ \text{and}\ u_k=v_k\ \text{for}\ k \neq i\big\}, & {i=j}. \end{cases}
\eeqn
for $i, j=1, \cdots, d$. 
To see (ii), let $u \in \childnt{\alpha}{v} \cap \childnt{\alpha}{w}$. Then $u_j \in \childn{\alpha_j}{v_j} \cap \childn{\alpha_j}{w_j}$ for all $j= 1, \cdots, d$. In other words, $\parentn{\alpha_j}{u_j} = v_j$ and $\parentn{\alpha_j}{u_j} = w_j$ for all $j= 1, \cdots, d$. Since parent is unique, it follows that $v_j = w_j$ for all $j= 1, \cdots, d$. Thus $v = w$. This proves (ii).

Let $u \in \childnt{\alpha}{v} \cap \childnt{\beta}{v}$. Then $u_j \in \childn{\alpha_j}{v_j} \cap \childn{\beta_j}{v_j}$ for all $j= 1, \cdots, d$. In other words, $\parentn{\alpha_j}{u_j} = v_j$ and $\parentn{\beta_j}{u_j} = v_j$ for all $j= 1, \cdots, d$. It may be now concluded from \eqref{disjoint-0} that $\alpha_j = \beta_j$ for all $j= 1, \cdots, d$. Thus $\alpha = \beta$. This proves (iii).

We prove (iv) by induction on $n$. By Remark \ref{chi-j}, $$\child{v} = \bigsqcup_{j=1}^d \childi{j}{v} = \bigsqcup_{j=1}^d \childnt{\epsilon_j}{v} = \bigsqcup_{\underset{|\alpha|=1}{\alpha \in \mathbb N^d}} \childnt{\alpha}{v}.$$ Thus the result holds for $n = 1$. Suppose it is true for some $n \geqslant  1$. Then
\beqn
\childn{n+1}{v} &=& \child{\bigsqcup_{\underset{|\alpha|=n}{\alpha \in \mathbb N^d}} \childnt{\alpha}{v}} = \bigsqcup_{\underset{|\alpha|=n}{\alpha \in \mathbb N^d}} \child{\childnt{\alpha}{v}}\\
&=& \bigsqcup_{\underset{|\alpha|=n}{\alpha \in \mathbb N^d}} \bigcup_{j=1}^d \childnt{\epsilon_j}{\childnt{\alpha}{v}} \\ &=& \bigcup_{j=1}^d \bigsqcup_{\underset{|\alpha|=n}{\alpha \in \mathbb N^d}} \childnt{\alpha + \epsilon_j}{v}
= \bigsqcup_{\underset{|\alpha|=n+1}{\alpha \in \mathbb N^d}} \childnt{\alpha}{v},
\eeqn
where the last union is disjoint in view of part (iii). Part (v) is now immediate from parts (iii) and (iv).

To see the last part, let $U:=\bigsqcup_{\alpha \in \mathbb N^d} \childnt{\alpha}{\rootb}$, $W:=\bigsqcup_{j\in \mathbb N} \childn{j}{\rootb}$. Clearly, $U$ and $W$ are subsets of $V.$ By (iv), $U = W$ and the unions in $U$ and $W$ are disjoint. Thus it suffices to check that $V \subseteq U.$
To this end, suppose that $v = (v_1, \cdots, v_d) \in V$. Then there exists $\alpha=(\alpha_1, \cdots, \alpha_d) \in \mathbb N^d$ such that $v_j \in \childn{\alpha_j}{\mathsf{root}_j}$ for $j=1, \cdots, d$. This implies that $v \in \mathsf{Chi}_1^{\langle \alpha_1 \rangle} \cdots \childki{d}{\alpha_d}{\rootb}= \childnt{\alpha}{\rootb}$. 
\end{proof}

%In this paper, we will be interested in the tree counterpart of classical unilateral multishifts. Hence all the directed tree discussed in the remaining part of this text will be rooted.

 \begin{definition} \label{depth-gen}
Let $\mathscr T = (V, \mathcal E)$ be the directed Cartesian product of rooted directed trees $\mathscr T_1, \cdots, \mathscr T_d$ and 
let $v \in V.$ The unique $\alpha_v \in \mathbb N^d$ 
\index{$\alpha_v$}
such that $$v \in \childnt{\alpha_v}{\rootb}$$ will be referred to as the {\it depth of $v$ in $\mathscr T$} (see Lemma \ref{disjoint}(vi)). The {\it $t^{\mbox{\tiny th}}$ generation} $\mathcal G_t$ of $\mathscr T$ is 
\index{$\mathcal G_t$}
defined as 
\beqn
\mathcal G_t :=\{v \in V : |\alpha_v|=t\}.
\eeqn
\end{definition}
\begin{remark}
Note that $\alpha_v = (\alpha_{v_1}, \cdots, \alpha_{v_d}),$ where $\alpha_{v_j}$ denotes the depth of the vertex $v_j \in V_j$ in the directed tree $\mathscr T_j$ for $j=1, \cdots, d.$ Note further that by Lemma \ref{disjoint}(iv),
$\mathcal G_t = \childn{t}{\rootb}.$
\end{remark}

Let us briefly discuss the notion of co-ordinate parent of a vertex in the directed Cartesian product of directed trees. 

\begin{definition}
Let $\mathscr T=(V, \mathcal E)$ be the directed Cartesian product of directed trees $\mathscr T_1, \cdots, \mathscr T_d$. 
For $j = 1,\cdots, d$ and $v \in V$, we 
set
\beqn \label{par-i}
\parenti{j}{v}:=\begin{cases} \big\{w \in V : w_j = \parent{v_j}\ \text{and}\ w_k=v_k\ \text{for}\ k \neq j\big\} &~ \mbox{if}~v_j \neq \mathsf{root}_j, \\
\emptyset &\mbox{otherwise}.\end{cases}\eeqn
Further, for $W \subseteq V$, we define 
$$\mathsf{par}_j(W):=\bigcup_{w \in W} \parenti{j}{w}.$$
For a positive integer $k,$ we denote 
\index{$\parentki{j}{k}{W}$}
$\underbrace{\mathsf{par}_j \cdots \mathsf{par}_j }_{{k}~ \mbox{{\tiny times}}}(W)$ by $\parentki{j}{k}{W}$.
Moreover, we understand $\parentki{j}{0}{W} = W$.

For $\alpha = (\alpha_1, \cdots, \alpha_d) \in \mathbb N^d$ and $W \subseteq V$, we define \index{$\parentnt{\alpha}{W}$}
$$\parentnt{\alpha}{W} := \mathsf{par}_1^{\langle \alpha_1 \rangle} \cdots \parentki{d}{\alpha_d}{W}.$$
In case $W=\{v\}$ for some $v \in V,$ we use $\parentnt{\alpha}{v}$ for $\parentnt{\alpha}{\{v\}}.$ 
\end{definition}

\begin{remark}
Note that
\beq \label{semi-parent}
\mathsf{Par}(v) = \bigsqcup_{j=1}^d \parenti{j}{v}.
\eeq 
In particular, $\mbox{card}(\mathsf{Par}(v))$ is at most $d.$
\end{remark}

\begin{remark} \label{child-par}
Note that $\mathsf{par}_j\parenti{i}{v}=\mathsf{par}_i\parenti{j}{v}$ for all $v \in V$. 
Further, for $i, j =1, \cdots, d$ with $i \neq j,$
$$\childi{i}{\parenti{j}{v}} = \parenti{j}{\childi{i}{v}}.$$
\end{remark}

Although the directed Cartesian product of directed trees need not be a directed tree, the following result shows that it has many structural similarities of directed tree. In particular, it always admits a directed semi-tree structure in the sense of \cite[Definition 2.6]{MS}. This has been observed in the special context of Example \ref{classical} in \cite[Lemmas 3.1, 3.2, 3.5(ii)]{MS}.

\begin{theorem} \label{semi-tree}
Let $\mathscr T=(V, \mathcal E)$ be the directed Cartesian product of directed trees $\mathscr T_1, \cdots, \mathscr T_d$. Then \begin{itemize}
\item[(i)] $\mathscr T$ has no circuits.
\item[(ii)] $\mathscr T$ is connected.
\item[(iii)] $\mathscr T$ can have at most one root.
\item[(iv)] $\mbox{card}(\child{u} \cap \child{v}) \leqslant 1$ for every $u, v \in V$ with $u \neq v$.
\end{itemize}
\end{theorem}

\begin{proof}
To see (i), on contrary suppose that $\{v^{(i)}\}_{i=1}^n \subseteq V$ is a circuit in $\mathscr T$. Hence $(v^{(1)}, v^{(2)}) \in \mathcal E$ implies that $v^{(2)} \in \child{v^{(1)}}$. Similarly, $(v^{(2)}, v^{(3)}) \in \mathcal E$ implies that $v^{(3)} \in \child{v^{(2)}} \subseteq \childn{2}{v^{(1)}}$. Consequently, $v^{(n)} \in \childn{n-1}{v^{(1)}}$. Finally, $(v^{(n)}, v^{(1)}) \in \mathcal E$ implies that $v^{(1)} \in \childn{n}{v^{(1)}}$. By Lemma \ref{disjoint}(iv), $v^{(1)} \in \childnt{\alpha}{v^{(1)}}$ for some $\alpha \in \mathbb N^d$ with $|\alpha|=n.$ This means that $v^{(1)}_j \in \childn{\alpha_j}{v^{(1)}_j}$ for all $j=1, \cdots, d.$ Also, since $|\alpha|=n,$ at least one $\alpha_j$ is nonzero.
This is a contradiction to the fact that $\mathscr T_j$ being directed tree has no circuits. Thus $\mathscr T$ has no circuits.

To prove (ii), let $v, w \in V$. Since each $\mathscr T_j$ ($1 \leqslant j \leqslant d$) is connected, so for each $v_j$ and $w_j$ there is a finite sequence of vertices $\{u_{j, k}\}_{k=1}^{\alpha_j} \subseteq V_j$ such that $u_{j, 1}= v_j$, $u_{j, \alpha_j}= w_j$ and, $\big(u_{j, k}, u_{j, k+1} \big)$ or $\big(u_{j, k+1}, u_{j, k} \big) \in \mathcal E_j~(k=1, \cdots, \alpha_j-1)$. Let $\alpha = (\alpha_1, \cdots, \alpha_d)$. We construct a sequence $\{u^{(k)}\}_{k=1}^{|\alpha|}$ of vertices in $V$ as follows. Put 
\beqn u^{(1)} &=& v, ~u^{(2)} = \big(u_{1, 2},  v_2, \cdots, v_d \big), \cdots, u^{(\alpha_1)} = \big(u_{1, \alpha_1}, v_2, \cdots, v_d \big), \\ u^{(\alpha_1+1)} &=& \big(u_{1, \alpha_1}, u_{2, 2}, v_3, \cdots, v_d \big), \cdots, u^{(\alpha_1+\alpha_2)} = \big(u_{1, \alpha_1}, u_{2, \alpha_2}, v_3, \cdots, v_d \big), \cdots, \\ u^{(|\alpha|)} &=& \big(u_{1, \alpha_1}, u_{2, \alpha_2}, \cdots, u_{d, \alpha_d} \big) = w. \eeqn It is evident from the construction that $\big(u^{(j)}, u^{(j+1)} \big)$ or $\big(u^{(j+1)}, u^{(j)} \big) \in \mathcal E$. This proves that $\mathscr T$ is connected.

The part (iii) follows immediately from Lemma \ref{disjoint}(vi).
Indeed, suppose that $\mathscr T$ has a root $v \neq \rootb.$ By  Lemma \ref{disjoint}(vi), there exists some $\alpha \in \mathbb N^d \setminus \{0\}$ such that $v \in \childnt{\alpha}{\rootb}.$ But then $\mathsf{Par}(v) \neq \emptyset,$ which contradicts the definition of the root. 

%Then for all $u \in V$, $(u,v) \notin \mathcal E$ and $(u,w) \notin \mathcal E$. In particular, if $u := (u_1, v_2, \cdots, v_d)$, then $(u,v) \notin \mathcal E$ implies that $(u_1, v_1) \notin \mathcal E_1$. Similarly, if $u := (u_1, w_2, \cdots, w_d)$, then $(u,w) \notin \mathcal E$ implies that $(u_1, w_1) \notin \mathcal E_1$. 
%Since this holds for all $u_1 \in V_1$, it follows that both $v_1$ and $w_1$ are roots of $\mathscr T_1$. As $\mathscr T_1$ can have at most one root, we must have $v_1 = w_1 = \mathsf{root}_1$. In a similar manner, we can prove that $v_j = w_j = \mathsf{root}_j$ for $j = 2, \cdots, d$. Thus $v = w = \rootb$. This proves (iii).

Let $u, v \in V$ and $u \neq v$. Without loss of generality, we may assume that $u_1 \neq v_1$. Suppose that $s, w \in \child{u} \cap \child{v}$. Then $s \in \childi{i}{u} \cap \childi{k}{v}$ and $w \in \childi{j}{u} \cap \childi{l}{v}$ for some $1 \leqslant i,j,k,l \leqslant d$. Note that $i \neq k$. For if $i = k$, then $u_1 \neq v_1$ would imply that $\childi{i}{u} \cap \childi{k}{v} = \emptyset$. Similarly, $j \neq l$.
We treat only the case in which $i < k$ and $j < l.$
Since $s \in \childi{i}{u}$, $s = (u_1, \cdots, u_{i-1}, s_i, u_{i+1}, \cdots, u_d)$, and that $w \in \childi{j}{u}$ implies that $w = (u_1, \cdots, u_{j-1}, w_j, u_{j+1}, \cdots, u_d)$. 
%Further we have, $\parenti{k}{s} = v=\parenti{l}{w}$. 
Let $\parent{u_k} =  \hat{u}_k$ and $\parent{u_l} = \hat{u}_l$. Since $i \neq k$ and $j \neq l$, it follows from $\parenti{k}{s} = v=\parenti{l}{w} $ that
\beq\label{semi-eq}
(u_1, \cdots, s_i, \cdots, \hat{u}_k, \cdots, u_d) = v=(u_1, \cdots, w_j, \cdots, \hat{u}_l, \cdots, u_d).
\eeq
In case $i = j$, we obtain $s_i=w_i$ in view of \eqref{semi-eq}. But then as $s, w \in \childi{i}{u}$, we must have $s = w$. Let $i \neq j$. Then from \eqref{semi-eq}, either $s_i = u_i$ or $s_i = \hat{u}_l$. As $s \in \childi{i}{u}$, we must have $s_i = \hat{u}_l$, and hence $i=l$. Therefore, $s = \parenti{i}{u}$. Thus $s \in \childi{i}{u} \cap \parenti{i}{u}$. This is not possible. Hence the case $i \neq j$ can not occur. This proves that $s = w$ and hence (iv) stands verified.
\end{proof}

%\begin{example} In case 
%$\mbox{card}(\child{u} \cap \child{v}) \leqslant 1$ for every $u, v \in V,$ $\mathscr T=(V, \mathcal E)$ is a directed semi-tree in the sense of \cite{MS}. For instance, this happens if $\mathscr T$ is the Cartesian product as discussed in Example \ref{classical}. Indeed, let $n = (n_1, \cdots, n_d), m = (m_1, \cdots, m_d) \in V$ and $n \neq m$. Without loss of generality we may assume that $n_1 \neq m_1$. Suppose that $n+\epsilon_i, n+\epsilon_j \in \child{n} \cap \child{m}$ for $i \neq j$. Then $n+\epsilon_i = m+\epsilon_k$ and $n+\epsilon_j = m+\epsilon_l$ for some $k,l$, $1 \leqslant k,l \leqslant d$. Since $i \neq j$, we must have $k \neq l$. Further, $m+\epsilon_k-\epsilon_i = m+\epsilon_l-\epsilon_j$. This is possible if and only if $i=j$ and $k =l$. Thus we arrive at a contradiction. Hence $\mbox{card}(\child{m} \cap \child{n}) \leqslant 1$ for all $m,n \in V$.
%\end{example}

\begin{definition}\label{branching-vertex} 
Let $\mathscr T = (V, \mathcal E)$ be the directed Cartesian product of directed trees $\mathscr T_1, \cdots, \mathscr T_d$. A vertex $v=(v_1, \cdots, v_d) \in V$ is called a {\it branching vertex} of $\mathscr T$ if $\text{card}(\child {v_j}) \geqslant  2$ for all $j=1, \cdots, d$. The set of all branching vertices of $\mathscr T$ is denoted by $V_\prec$.
\end{definition}
\begin{remark}
If $V_\prec^{(j)}$ is the set of branching vertices 
\index{$V_\prec$}
of $\mathscr T_j$, then 
$$V_\prec = V_\prec^{(1)} \times \cdots \times V_\prec^{(d)}.$$
\end{remark}

\begin{proposition}\label{Vprec}
Let $\mathscr T = (V, \mathcal E)$ be the directed Cartesian product of rooted directed trees $\mathscr T_1, \cdots, \mathscr T_d$. 
If $\mathscr T_j$ has finite branching index $k_{\mathscr T_j}$ for $j=1, \cdots, d$, then for $k_{\mathscr T} = (k_{\mathscr T_1}, \cdots, k_{\mathscr T_d})$, one has
$$\childnt{k_{\mathscr T}}{V_\prec} \cap V_\prec = \emptyset.$$
\end{proposition}

\begin{proof}
Assume that each $\mathscr T_j$ has finite branching index $k_{\mathscr T_j}$ and let $v \in \childnt{k_{\mathscr T}}{V_\prec}$. Then $v \in \mathsf{Chi}_1^{\langle k_{\mathscr T_1} \rangle} \cdots \childki{d}{k_{\mathscr T_d}}{w}$ for some $w \in V_\prec$. That is, $v = (v_1, \cdots, v_d)$ such that $v_j \in \childn{k_{\mathscr T_j}}{w_j}$ with $w_j \in V_\prec^{(j)}$. But then $v_j \notin V_\prec^{(j)}$. Hence $v \notin V_\prec$. This shows that 
$\childn{k_{\mathscr T}}{V_\prec} \cap V_\prec = \emptyset.$
\end{proof}

The multiindex $k_{\mathscr T} \in \mathbb N^d$ appearing in Proposition \ref{Vprec} will be referred to as the {\it joint branching index} of $\mathscr T$. 
\index{$k_{\mathscr T}$}
Also, we say that {\it $\mathscr T$ has finite joint branching index} if $|k_{\mathscr T}|$ is finite.

We conclude this section with a brief discussion on the notion of siblings of a vertex. 

Let $\mathscr T = (V, \mathcal E)$ be the directed Cartesian product of rooted directed trees $\mathscr T_1, \cdots, \mathscr T_d$.
For $u \in V$ and $j = 1, \cdots, d$, we set \beqn \mathsf{sib}_j(u) := \begin{cases} \childi{j}{\parenti{j}{u}} & \mbox{if~}u_j \neq \mathsf{root}_j, \\ 
\emptyset & \mbox{otherwise}.
\end{cases} \eeqn
For $W \subseteq V$, we define 
\index{$\mathsf{sib}_j(W)$}
$\mathsf{sib}_j(W) := \displaystyle\bigcup_{u \in W} \mathsf{sib}_j(u)$.
\begin{remark} \label{rmk-sib}
Let $1 \leqslant i,  j \leqslant d$ and $v \in V.$ Then $\mathsf{sib}_i \mathsf{sib}_j (v) = \mathsf{sib}_j \mathsf{sib}_i (v)$. Further,
$\mathsf{sib}_i \mathsf{sib}_i (v) = \mathsf{sib}_i (v).$
\end{remark}

For future reference, we record the following simple yet useful observation.
\begin{proposition} \label{sib-id}
Let $\mathscr T = (V,\mathcal E)$ be the directed Cartesian product of directed trees $\mathscr T_1, \cdots, \mathscr T_d$. Then 
\beqn 
\mbox{card}(\mathsf{sib}_i(v)) \mbox{card}(\mathsf{sib}_j(\parenti{i}{v})) =  \mbox{card}(\mathsf{sib}_j(v)) \mbox{card}(\mathsf{sib}_i(\parenti{j}{v}))
\eeqn
for every $v \in \childi{i}{\childi{j}{w}}, w \in V, \mbox{and~}i, j = 1, \cdots, d.$ 
\end{proposition}
\begin{proof}
Clearly, the identity holds for $i=j.$ In case $i \neq j,$
the conclusion follows from the fact that $\mbox{card}(\mathsf{sib}_j(\parenti{i}{v}))=\mbox{card}(\mathsf{sib}(v_j))=\mbox{card}(\mathsf{sib}_j(v))$.
\end{proof}

\section{Tensor Product of Rooted Directed Trees}

In this section, we discuss another notion of the product of two directed trees to be referred to as {tensor product}. This notion was introduced by P. Weichsel \cite{We} for undirected graphs, and later, it was extended to directed graphs by M. McAndrew \cite{Mc} (refer also to \cite{Ha}). 
In the literature, the tensor product is also known as categorical product, Kronecker product, cardinal product, weak direct product and even Cartesian product. We would like to emphasize that the notions of Cartesian product and tensor product, as discussed in this text, are conceptually different. 

The definition of the tensor product of two directed graphs extends naturally in context of finitely many directed trees.

\begin{definition}
Let $d$ be a positive integer
and let $\mathscr T_j = (V_j, \mathcal E_j)~(j=1, \cdots, d)$ be a collection of rooted directed trees. 
The {\it tensor product of $\mathscr T_1, \cdots, \mathscr T_d$} is a directed graph 
\index{$\mathscr T^{\otimes}=(V, \mathcal E^{\otimes})$}
$\mathscr T^{\otimes}=(V, \mathcal E^{\otimes}),$ where $V:=V_1 \times \cdots \times V_d$ and 
\beqn
\mathcal E^{\otimes} := \big\{(\mf v, \mf w) \in V \times V: 
(\mf v_j, \mf w_j) \in \mathcal E_j \mbox{~for all~}j = 1, \cdots, d \big\}.
\eeqn 
\end{definition}
{\bf Caution.} Since the set of vertices of the tensor product is same as that of directed Cartesian product, we use scripted letters throughout the text to distinguish the vertices (except the root) of the tensor product from those of directed Cartesian product.
\begin{remark} \label{edge-tensor}
Note that $(\mf v, \mf w) \in \mathcal E^{\otimes}$ if and only if $\mf w \in \mathsf{Chi}_1 \cdots \childi{d}{\mf v}$. Thus $\child{\mf v} = \mathsf{Chi}_1 \cdots \childi{d}{\mf v}$ (with reference to the directed graph $\mathscr T^{\otimes}$). This should be compared with the expression for $\child{\cdot}$ (with reference to the directed graph $\mathscr T$) as given in Remark \ref{chi-j}.
\end{remark}

Let $\mathscr T^{\otimes} = (V, \mathcal E^{\otimes})$ be the tensor product of rooted directed trees $\mathscr T_1, \cdots, \mathscr T_d$. Define two vertices $\mf v, \mf w \in V$ to be equivalent if either $\mf v=\mf w$ or $\mf v$ and $\mf w$ can be connected by a path in $\mathscr T^{\otimes}.$ Note that this defines an equivalence relation.
A {\it component} of $\mathscr T^{\otimes}$ is an equivalence class corresponding to the above equivalence relation.
It turns out that each component of $\mathscr T^{\otimes}$ is a rooted directed tree. We illustrate this in more detail in the following theorem.

\begin{theorem}\label{tensor-prop}
Let $\mathscr T^{\otimes} = (V, \mathcal E^{\otimes})$ be the tensor product of rooted directed trees $\mathscr T_1, \cdots, \mathscr T_d$ and let $\mathscr T^{\otimes}_{\rootb}=(V^{\otimes}, \mathcal F)$ denote the (unique) component of $\mathscr T^\otimes$ 
\index{$\mathscr T^{\otimes}_{\rootb}$}
that contains $\rootb$. Let $\mathscr T$ be the directed Cartesian product of $\mathscr T_1, \cdots, \mathscr T_d$.
Set \index{$\mathsf{Root}^\otimes$}
\beqn
\mathsf{Root}^\otimes := \{\mf v \in V : \mf v_j = \mathsf{root}_j~\mbox{for at least one~} j = 1, \cdots, d\}.
\eeqn
Then the following statements are true.
\begin{enumerate}
\item[(i)] If $\mf v \in \mathsf{Root}^\otimes$, then there does not exist any $\mf u \in V$ such that $(\mf u,\mf v) \in \mathcal E^{\otimes}$.
\item[(ii)] For each $\mf v \in V \setminus \mathsf{Root}^\otimes$, there is a unique $\mf u \in V$ such that $(\mf u,\mf v) \in \mathcal E^{\otimes}$. In other words, each $\mf v \in V \setminus \mathsf{Root}^\otimes$ has a parent.
\item[(iii)] $\mathscr T^{\otimes}$ has no circuits.
\item[(iv)] No two distinct vertices in $\mathsf{Root}^\otimes$ can be connected by a path in $\mathscr T^{\otimes}$.  
\item[(v)] There is a bijective correspondence between the collection of components of $\mathscr T^{\otimes}$ and the elements of $\mathsf{Root}^\otimes$. In particular, $\mathscr T^{\otimes}$ contains countably many components.
\item[(vi)] Each component is a rooted directed tree with root coming from $\mathsf{Root}^\otimes$.
\item[(vii)] $\mathscr T^{\otimes}$ is locally finite if and only if $\mathscr T$ is locally finite.
\item[(viii)] $\mathscr T^{\otimes}$ is leafless. 
%if and only if each $\mathscr T_j$, $1 \leqslant j \leqslant d$, is leafless.
\item[(ix)] If $\mathscr T$ is of finite joint branching index $k_{\mathscr T}=(k_{\mathscr T_1}, \cdots, k_{\mathscr T_d})$, then the branching index $k_{\mathscr T^{\otimes}_{\rootb}}$ of $\mathscr T^{\otimes}_{\rootb}$ is given by $k_{\mathscr T^{\otimes}_{\rootb}} = \max\{k_{\mathscr T_j} : 1 \leqslant j \leqslant d\}$. 
\item[(x)] For each $v \in V$, there exists $\mf v \in V^{\otimes}$ such that $|\alpha_v| = \alpha_{\mf v},$ where $\alpha_v$ is the depth of $v$ in $\mathscr T$ and $\alpha_{\mf v}$ is the depth of $\mf v$ in the directed tree $\mathscr T^{\otimes}_{\rootb}$. 
\end{enumerate}
\end{theorem}
\begin{remark} \label{t-components}
The conclusion of (v) above is in contrast with the situation occurring in the case of undirected trees, where the tensor product of two undirected trees has exactly two components (see \cite[Theorem 5.29]{IK}).
\end{remark}
\begin{proof}
Let $\mf v \in \mathsf{Root}^\otimes$. Then for some $j = 1, \cdots, d$, $\mf v_j = \mathsf{root}_j$. Now, if $\mf u \in V$ such that $(\mf u,\mf v) \in \mathcal E^{\otimes}$, then $\mf u_j = \parent{\mf v_j}$, which is not possible. This proves (i).

To see (ii), let $\mf v \in V \setminus \mathsf{Root}^\otimes$. Then $\mf v_j \neq \mathsf{root}_j$ for all $j = 1, \cdots, d$. Consider $\mf u \in V$ with $\mf u_j = \parent{\mf v_j}$ for all $j = 1, \cdots, d$. Then $(\mf u,\mf v) \in \mathcal E^{\otimes}$. This proves the existence as well as the uniqueness of $\mf u$.  
%let there exist $u, w \in V$ such that both $(u,v)$ and $(w,v)$ belong to $\mathcal E^{\otimes}$. Then $u_j = \parent{v_j} = w_j$ for each $j = 1, \cdots, d$. This implies that $u = w$. Thus (ii) stands verified.

To see (iii), suppose there is a finite sequence $\{\mf w^{(i)}\}_{i=1}^n ~ (n \geqslant  2)$ of distinct vertices such that $(\mf w^{(i)}, \mf w^{(i+1)}) \in \mathcal E^\otimes$ for all $1 \leqslant i \leqslant n-1$ and $(\mf w^{(n)}, \mf w^{(1)}) \in \mathcal E^\otimes$. Then $\mf w^{(1)}_j \in \childn{n}{\mf w^{(1)}_j}$ for all $1 \leqslant j \leqslant d$. This is a contradiction to the fact that $\mathscr T_j$ has no circuits.

Let $\mf u$ and $\mf v$ be two distinct vertices in $\mathsf{Root}^\otimes$. Suppose there exists a finite sequence $\{\mf w^{(i)}\}_{i=1}^n \subseteq V~(n \geqslant  2)$ of distinct vertices such that $\mf w^{(1)} = \mf u$, $\mf w^{(n)} = \mf v$ and $(\mf w^{(i)}, \mf w^{(i+1)})$ or $(\mf w^{(i+1)}, \mf w^{(i)}) \in \mathcal E^\otimes~(i=1, \cdots, n-1)$. By (i), $(\mf w^{(2)}, \mf w^{(1)})$ can not belong to $\mathcal E^\otimes$. Hence, we must have $(\mf w^{(1)}, \mf w^{(2)}) \in \mathcal E^\otimes$. Thus $\mf w^{(2)}$ belongs to $V \setminus \mathsf{Root}^\otimes$. Next, if 
$(\mf w^{(3)}, \mf w^{(2)}) \in \mathcal E^\otimes$, then by (ii), $\mf w^{(3)} = \mf w^{(1)}$. This contradicts that the vertices $\mf w^{(1)}, \cdots, \mf w^{(n)}$ are distinct. Thus $(\mf w^{(2)}, \mf w^{(3)}) \in \mathcal E^\otimes$. By arguing similarly, one can see that $(\mf w^{(i)}, \mf w^{(i+1)}) \in \mathcal E^\otimes$ for all $1 \leqslant i \leqslant n-1$. Thus we get $\mf v_j \in \childn{n-1}{\mf u_j}$ for each $j = 1, \cdots, d$. Hence $\mf v \notin \mathsf{Root}^\otimes$, which is a contradiction. This proves (iv).

Suppose that $\mf v \in V$ belongs to some component $\mathcal C$. There exist  $k_1, \cdots, k_d \in \mathbb N$ such that $\mf v_j \in \childn{k_j}{\mathsf{root}_j}$ for $j=1, \cdots, d$. Let $k = \min\{k_j : 1 \leqslant j \leqslant d\}$. Then $\mf u = \mathsf{par}_1^{\langle k \rangle} \cdots \parentki{d}{k}{\mf v} \in \mathsf{Root}^\otimes$, and since $\mathcal C$ is connected, $\mf u \in \mathcal C$. Thus each component contains an element from $\mathsf{Root}^\otimes$. Further, (iv) implies that each component contains at most one element from $\mathsf{Root}^\otimes$. Clearly, as each element of $\mathsf{Root}^\otimes$ belongs to some component, it follows that the correspondence between the collection of components of $\mathscr T^{\otimes}$ and the elements of $\mathsf{Root}^\otimes$ is bijective. This completes the verification of (v).

Let $\mathcal C$ be any component of $\mathscr T^{\otimes}$. From (v), there exists a unique $\mf v \in \mathsf{Root}^\otimes$ such that $\mf v \in \mathcal C$. Further, (i) implies that $\mf v$ has no parent, in particular, in the subgraph $\mathcal C$. Thus $\mf v$ is a root for $\mathcal C$. Clearly, $\mathcal C$ is connected, and by (ii), each $\mf u \in \mathcal C$, with $\mf u \neq \mf v$, has a parent. 
Also, by (iii), $\mathcal C$ has no circuits. This proves (vi).   

Let $\mf v$ be a vertex in $V$. By Remark \ref{edge-tensor}, 
$\child{\mf v}=\mathsf{Chi}_1 \cdots \childi{d}{\mf v}$. It follows that
\beqn
\mbox{card}(\child{\mf v})=\prod_{j=1}^d\mbox{card}(\childi{j}{\mf v})=\prod_{j=1}^d\mbox{card}(\child{\mf v_j}).
\eeqn
%$\mbox{card}(\child{v})$ is finite if and only if $\mathsf{Chi}_1 \cdots \childi{d}{v}$ is finite. Since, by Lemma \ref{disjoint}(i),   $\mathsf{Chi}_j\childi{i}{v}=\mathsf{Chi}_i\childi{j}{v}$ for all $i,j =1, \cdots, d$, it follows that $\mathsf{Chi}_1 \cdots \childi{d}{v}$ is finite if and only if $\mbox{card}(\childi{j}{v}) = \mbox{card}(\child{v_j})$ is finite for all $j =1, \cdots, d$. 
Thus $\mathscr T^{\otimes}$ is locally finite if and only if $\mathscr T$ is locally finite.

The proof of (viii) is an obvious consequence of the fact (observed in Remark \ref{edge-tensor}) that $\child{\mf v} = \mathsf{Chi}_1 \cdots \childi{d}{\mf v}$ for all $\mf v \in V$. 

To see (ix), first note that 
$V^{\otimes}=\sqcup_{k=0}^{\infty} \childn{k}{\rootb}=
\sqcup_{k=0}^{\infty} \mathsf{Chi}^{\langle k \rangle}_1 \cdots \childki{d}{k}{\rootb}$. Hence,
if $\mf v$ is any vertex of $\mathscr T^{\otimes}_{\rootb}$, then there exists a unique $k \in \mathbb N$ such that $\mf v_j \in \childn{k}{\mathsf{root}_j}$ for all $1 \leqslant j \leqslant d$. Thus $\alpha_\mf v = \alpha_{\mf v_j}$ for all $1 \leqslant j \leqslant d$. Next, observe that $\mbox{card}(\child{\mf v}) \geqslant  2$ if and only if there exists a positive integer $j~ (1 \leqslant j \leqslant d)$ such that $\mbox{card}(\child{\mf v_j}) \geqslant  2$. With these two observations it is easy to see that $$\sup\{\alpha_\mf v : \mbox{card}(\child{\mf v}) \geqslant  2 \} = \max_{1 \leqslant j \leqslant d} \sup\{\alpha_{\mf v_j} : \mbox{card}(\child{\mf v_j}) \geqslant  2 \}.$$
This implies that $k_{\mathscr T^{\otimes}_{\rootb}} = \max\{k_{\mathscr T_j} : 1 \leqslant j \leqslant d\}$. We leave the verification of the last part to the reader.
\end{proof} 

%The component of $\mathscr T^{\otimes}$ that contains $\rootb$ will be denoted by $\mathscr T^{\otimes}_{\rootb}$. 
We will see later that certain weighted shifts acting on the rooted directed tree $\mathscr T^{\otimes}_{\rootb}$ arise naturally in an integral representation of so-called spherically balanced multishifts originating from directed Cartesian product of directed trees.
\begin{remark}
Note that $\mathscr T^{\otimes}_{\rootb}$ is an isomorphic invariant for $\mathscr T^{\otimes}$. In fact, let $\{\tilde{\mathscr T}_j : 1 \leqslant j \leqslant d\}$ be a collection of rooted directed trees with roots $\tilde{\mathsf{root}_j}$ respectively and let $\tilde{\mathscr T}^\otimes$ be the tensor product of $\tilde{\mathscr T}_1, \cdots, \tilde{\mathscr T}_d$. Let $\tilde{\mathscr T}^{\otimes}_{\tilde{\rootb}}$ be the component of $\tilde{\mathscr T}^\otimes$ containing $\tilde{\rootb}$. Suppose that $\phi_j$ is an isomorphism between $\mathscr T_j$ and $\tilde{\mathscr T}_j$. Then the map $(\mf v_1, \cdots, \mf v_d) \mapsto (\phi_1(\mf v_1), \cdots, \phi_d(\mf v_d))$ defines an isomorphism between $\mathscr T^\otimes$ and $\tilde{\mathscr T}^\otimes$, and hence, $\mathscr T^{\otimes}_{\rootb}$ and $\tilde{\mathscr T}^{\otimes}_{\tilde{\rootb}}$ become isomorphic.      
\end{remark}

For future reference, we find it necessary to describe $\mathscr T^{\otimes}_{\rootb}$ in the context of Examples \ref{classical}, \ref{classical-mix}, \ref{T1-T1} (see Figures 2.4, 2.5, 2.6 respectively).
In this regard,
the reader is advised to recall the definition of $\mathscr T_{n_0, k_0}$ as given in Example \ref{T-n-k}.
\begin{example} \label{tensor-c}
Let $\mathscr T_{1, 0}$ be the tree as discussed in Example \ref{classical}. Let $d = 2$ and $\mathscr T_j = \mathscr T_{1, 0}$ for $j = 1, 2$. Then
$$\mathsf{Root}^\otimes = \big\{(i,0), (0,j) : i,j \geqslant  0 \big\}.$$
For $i \geqslant  1$, the components $\mathcal C_{(i,0)} = (V_{(i,0)}, \mathcal E_{(i,0)})$ containing $(i,0)$ are given by $$V_{(i,0)} = \{(i+k,k) : k \geqslant  0\} \mbox{~and~} \child{(i+k,k)} = \{(i+k+1,k+1)\}$$ for all $k \geqslant  0$. Similar description of $\mathcal C_{(0,j)}$ is obtained for $j \geqslant  1$. Further, the rooted directed tree $\mathscr T^{\otimes}_{\rootb}$, with set of vertices $V^{\otimes}$, is given by
$${V^{\otimes}} = \{(k,k) : k \geqslant  0\} \mbox{~and~} \child{(k,k)} = \{(k+1,k+1)\}$$ for all $k \geqslant  0$. 
Note that $\mathscr T^{\otimes}_{\rootb}$ is isomorphic to $\mathscr T_{1, 0}$ via the isomorphism $(k, k) \mapsto k.$
\end{example}

\begin{figure} \label{ten-1}
\includegraphics[scale=.5]{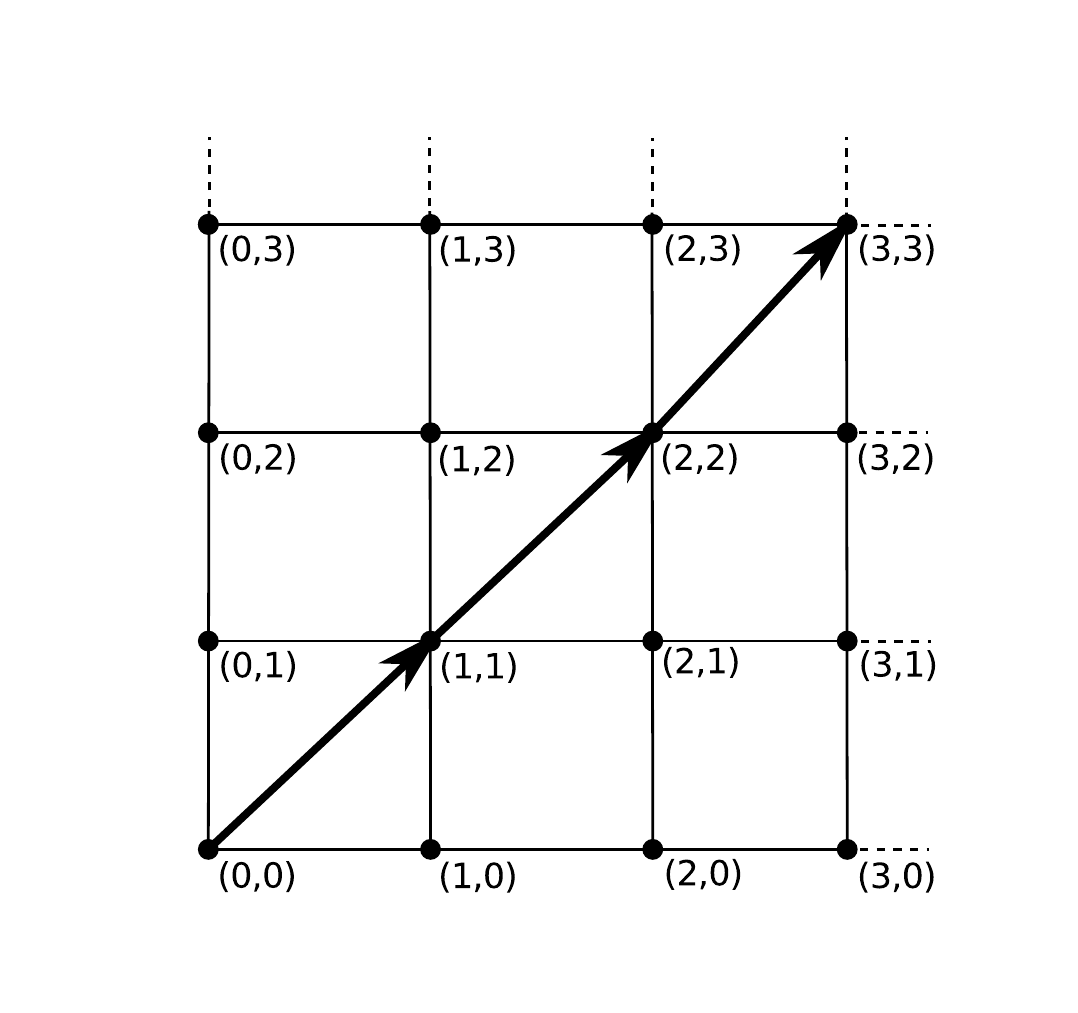} \caption{In $\mathscr T_{1, 0} \otimes \mathscr T_{1, 0}$, $\mathscr T^{\otimes}_{\rootb}$ is represented with bold-faced edges}
\end{figure}

\begin{figure} \label{tensor-2}
\includegraphics[scale=.4]{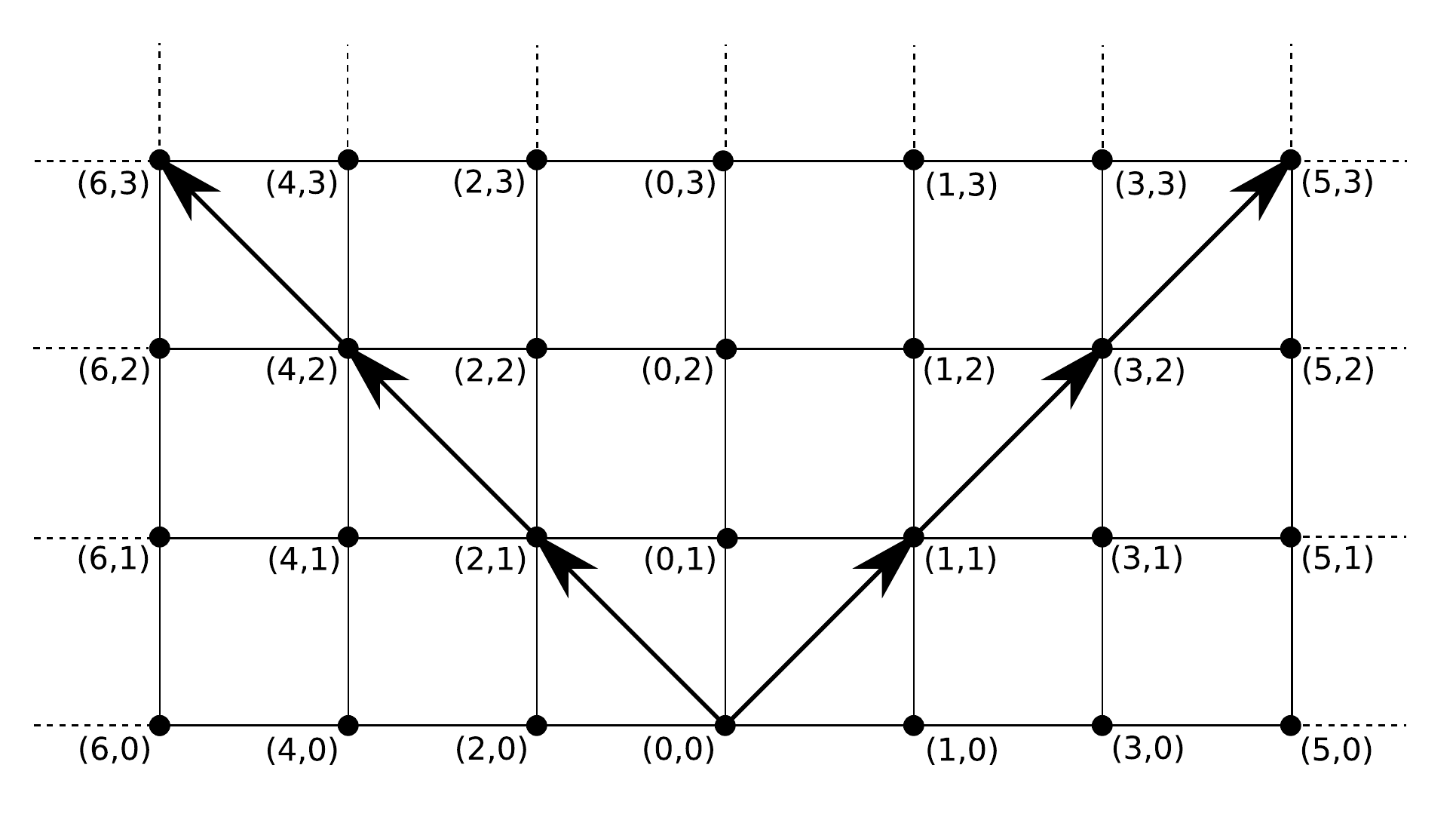} \caption{In $ \mathscr T_{2, 0} \otimes \mathscr T_{1, 0}$, $\mathscr T^{\otimes}_{\rootb}$ is represented with bold-faced edges}
\end{figure}

\begin{figure} \label{tensor-3}
\includegraphics[scale=.4]{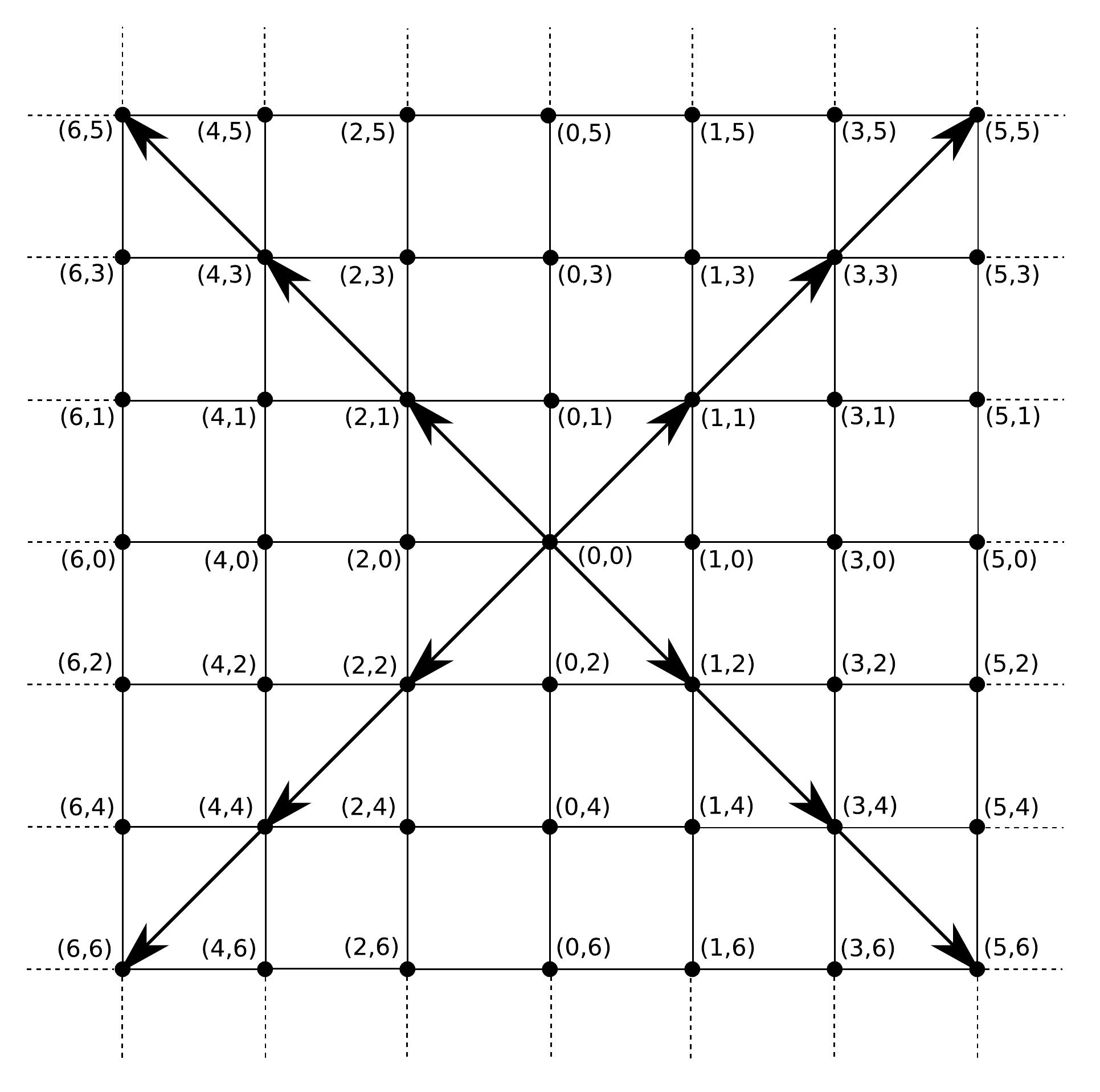} \caption{In $ \mathscr T_{2, 0} \otimes \mathscr T_{2, 0}$, $\mathscr T^{\otimes}_{\rootb}$ is represented with bold-faced edges}
\end{figure}

\begin{example} \label{tensor-m}
Let $\mathscr T_1=\mathscr T_{2, 0}$, $\mathscr T_2=\mathscr T_{1, 0}$ (see Example \ref{classical-mix}). Then
$$\mathsf{Root}^\otimes = \big\{(i,0), (0,j) : i,j \geqslant  0 \big\}.$$
For $i \geqslant  1$, the components $\mathcal C_{(i,0)} = (V_{(i,0)}, \mathcal E_{(i,0)})$ containing $(i,0)$ are given by $$V_{(i,0)} = \{(i+2k,k) : k \geqslant  0\} \mbox{~and~} \child{(i+2k,k)} = \{(i+2k+2,k+1)\}$$ for all $k \geqslant  0$. Further, for $j \geqslant  1$, the components $\mathcal C_{(0,j)} = (V_{(0,j)}, \mathcal E_{(0,j)})$ containing $(0,j)$ are given by 
$$V_{(0,j)} = \{(0,j), (1,j+1), (2, j+1)\} \cup \{(2k+1, j+k+1), (2k, j+k) : k \geqslant  1\}$$
and $\child{(0,j)} = \{(1,j+1), (2, j+1)\}$, $\child{(k,l)} = \{(k+2, l+1)\}$ for all $k, l \geqslant  1$. Moreover, the rooted directed tree $\mathscr T^{\otimes}_{\rootb}$, with set of vertices $V^{\otimes}$, is given by
$$V^{\otimes} = \{(2k+1, k+1), (2k, k) : k \geqslant  0\}$$
and $\child{(0,0)} = \{(1,1), (2, 1)\}$, $\child{(k,l)} = \{(k+2, l+1)\}$ for all $k, l \geqslant  1$.
Note that $\mathscr T^{\otimes}_{\rootb}$ is isomorphic to $\mathscr T_{2, 0}$ via the isomorphism $(k, l) \mapsto k.$
\end{example}

\begin{example} \label{tensor-T1-T1}
Let $\mathscr T_1=\mathscr T_{2, 0}=\mathscr T_2$. Then
$$\mathsf{Root}^\otimes = \big\{(i,0), (0,j) : i,j \geqslant  0 \big\}.$$
For $j \geqslant  1$, the components $\mathcal C_{(0,j)} = (V_{(0,j)}, \mathcal E_{(0,j)})$ containing $(0,j)$ are given by 
$$V_{(0,j)} = \{(0,j), (1,j+2), (2, j+2)\} \cup \{(2k+1, j+2k+2), (2k, j+2k) : k \geqslant  1\}$$
and $\child{(0,j)} = \{(1,j+2), (2, j+2)\}$, $\child{(k,l)} = \{(k+2, l+2)\}$ for all $k, l \geqslant  1$. Similar description for $\mathcal C_{(i,0)}$ is obtained for all $i \geqslant  1$. Further, the rooted directed tree $\mathscr T^{\otimes}_{\rootb}$, with set of vertices $V^{\otimes}$, is given by
$$V^{\otimes} = \{(k,k) : k \geqslant  0\} \cup \{(2k-1, 2k), (2k, 2k-1) : k \geqslant  1\}$$
and $\child{(0,0)} = \{(1,1), (1,2), (2, 1), (2,2)\}$, $\child{(k,l)} = \{(k+2, l+2)\}$ for all $k, l \geqslant  1$.
It can be seen that $\mathscr T^{\otimes}_{\rootb}$ is isomorphic to the rooted directed tree $\mathscr T_{4, 0}$.
\end{example}

We visit the above examples once again in the context of spherically balanced multishifts in Chapter 5.

\chapter{Multishifts on Product of Rooted Directed Trees}

In this chapter, we introduce and study the notion of multishifts on directed Cartesian product of finitely many rooted directed trees. In particular, we discuss some basic properties of multishifts such as boundedness, commutativity, circularity and analyticity. These are then used to describe various spectral parts of $S_{\lambdab}$ including the Taylor spectrum. 

In this paper, we are interested in the tree counterpart of classical unilateral multishifts. Hence all the directed trees discussed in the remaining part of this text will be rooted.

\section{Definition and Elementary Properties}
\label{Section3.2}

Let $\mathscr T_j = (V_j, \mathcal E_j)~(j=1, \cdots, d)$ be rooted directed trees and
let $\mathscr T = (V,\mathcal E)$ be the directed Cartesian product of $\mathscr T_1, \cdots, \mathscr T_d$. 
For a vertex $v \in V,$ let $e_v : V \rar \mathbb C$ denote the indicator function of the set $\{v\}.$
Consider the complex Hilbert space $l^2(V)$ of square summable complex functions on $V$
\index{$l^2(V)$}
equipped with the standard inner product. Note that $l^2(V)$ admits the orthonormal basis $\{e_v : v \in V\}$. We always assume that $\mbox{card}(V)=\aleph_0.$
For a nonempty subset $W$ of $V$, $l^2(W)$ may be considered as a subspace of $l^2(V)$. Indeed, if one sets $\tilde{f}=f$ on $W$ and $0$ otherwise, then the mapping $U : l^2(W) \rar l^2(V)$ given by $Uf = \tilde{f}$ is an isometric homomorphism.
\begin{remark}
Consider the the category $\mathcal T$ of the directed Cartesian products of finitely many directed trees with morphisms being directed graph homomorphisms (or directed graph isomorphisms).
Note that $l^2$ defines a {\it covariant functor} from $\mathcal T$ into the category $\mathcal C$ of Hilbert spaces with bounded linear operators (resp. unitaries) as morphisms. Indeed, any graph homomorphism (resp.  isomorphism) $\phi$ induces a bounded linear operator (resp. unitary) $l^2(\phi)$ given by $$l^2(\phi)(e_v)=e_{\phi(v)},$$
which satisfies $l^2(\phi \circ \psi)=l^2(\phi)\circ l^2(\psi)$.
\end{remark}

\begin{definition} \label{multishift-dfn}
Given a system
$\lambdab=\{\lambda^{(j)}_v : v \in V^{\circ}, ~ j=1, \cdots, d\}$ of nonzero complex numbers, we define
the {\it multishift $S_{\lambdab}$ on $\mathscr T$} 
with weights $\lambdab$ as the $d$-tuple of operators 
\index{$S_{\lambdab}$}
$S_1, \cdots, S_d$ on $l^2(V)$ given by 
   \begin{align*}
   \begin{aligned}
{\mathscr D}(S_{j}) & := \{f \in l^2(V) \colon
\varLambda^{(j)}_{\mathscr T} f \in l^2(V)\},
   \\
S_{j} f & := \varLambda^{(j)}_{\mathscr T} f, \quad f \in {\mathscr
D}(S_{j}),
   \end{aligned}
   \end{align*}
where $\varLambda^{(j)}_{\mathscr T}$ is the mapping defined on
complex functions $f$ on $V$ by
   \begin{align*}
(\varLambda^{(j)}_{\mathscr T} f) (v) :=
   \begin{cases}
\lambda^{(j)}_v \cdot f\big(\parenti{j}{v}\big) & \text{if } v_j \in
V^\circ_j,
   \\
   0 & \text{if } v_j \text{ is a root of } {\mathscr T_j}.
   \end{cases}
   \end{align*}
\end{definition}   
{\it We note here that {not} all weights $\lambda^{(j)}_v$ in the system $\lambdab$ are used in the above definition. For instance, if $v \in \childi{1}{\rootb}$ then $\lambda^{(2)}_v$ will not appear in the definition of $S_{\lambdab}$.}
%However, for the sake of convenience, we use the same system $\lambdab$ to denote the weights of $S_{\lambdab}.$ Further, 
%The weights appearing in the definition of $S_{\lambdab}$ are always assumed to be non-zero.}
\begin{remark}
If $e_v \in {\mathscr D}(S_{j})$, then
\beq\label{Si}
S_j e_v = \sum_{w \in \childi{j}{v}} \lambda^{(j)}_w e_w.
\eeq
\end{remark}

\begin{example}[Classical Multishifts] \label{classical-m} Consider the directed Cartesian product $\mathscr T$ of $d$ copies of $\mathscr T_{1, 0}$ as discussed in Example \ref{classical}. Assume that $S_j$ is bounded for $j=1, \cdots, d.$ Then 
\beqn
S_j e_{\alpha} = \sum_{\beta \in \childi{j}{\alpha}} \lambda^{(j)}_\beta e_\beta = \lambda^{(j)}_{\alpha + \epsilon_j} e_{\alpha + \epsilon_j}.
\eeqn
If one sets $w^{(j)}_{\alpha}:=\lambda^{(j)}_{\alpha + \epsilon_j}$, then $S_{\lambdab}$ is nothing but the classical multishift $S_{\bf w}$ with weight multisequence $\{w^{(j)}_{\alpha} : \alpha \in \mathbb N^d, j=1, \cdots, d\}.$ 
\end{example}

\begin{lemma} \label{bddness}
Let $\mathscr T = (V,\mathcal E)$ be the directed Cartesian product of rooted directed trees $\mathscr T_1, \cdots, \mathscr T_d$ and let
$S_{\lambdab}=(S_1, \cdots, S_d)$ be a multishift on $\mathscr T$. Then, for $j=1, \cdots, d$, the following statements hold:
\begin{enumerate}
\item[(i)]  
$S_j$ is a bounded linear operator on $l^2(V)$ if and only if
\beqn
\sup_{v \in V} \sum_{w \in \childi{j}{v}}  |\lambda^{(j)}_w|^2 < \infty.
\eeqn
\item[(ii)] 
$S_{j}$ is injective. 
%if and only if $\mathscr T_j$ is leafless. 
\end{enumerate}
\end{lemma}
\begin{proof}
By Lemma \ref{disjoint}(ii),  $\{e_w\}_{w \in \childi{j}{v}}$ is orthogonal for every $v \in V$ and $j=1, \cdots, d.$ The first part now follows from \eqref{Si}.
To see (ii), suppose that $S_jf=0$ for some $f \in l^2(V).$ 
Then
\beqn
\sum_{v \in V} |f(v)|^2 \sum_{w \in \childi{j}{v}} |\lambda^{(j)}_w|^2 = \|S_jf\|^2=0.
\eeqn 
Since $\lambdab$ consists of nonzero complex numbers, the above equality holds if and only if either $f(v)=0$ or $\childi{j}{v}=\emptyset.$ However,
 by assumption $\mathscr T_1, \cdots, \mathscr T_d$ are leafless, and hence $f(v) =0$ for all $v \in V.$
\end{proof}

{\it 
Unless stated otherwise, $S_{j}$ belongs to $B(l^2(V))$ for every $j=1, \cdots, d.$ }
%We also assume throughout this text that each $\mathscr T_j$ is leafless.} 

If $S_{\lambdab}=(S_1, \cdots, S_d)$ is the multishift on $\mathscr T$, then the Hilbert space adjoint $S_j^*$ of $S_j$ is given by
\beqn\label{Si*}
S_j^* e_v = \overline{\lambda}_v^{(j)} e_{\parenti{j}{v}}\ \text{for all}\ v \in V.
\eeqn
%where $\overline{z}$ denotes the complex conjugate of $z \in \mathbb C.$
\begin{remark} \label{0-eigen}
Note that $S^*_je_{\rootb}=0$ for all $j=1, \cdots, d.$ In particular, $0$ belongs to the point spectrum of $S^*_{\lambdab}.$
\end{remark}

In the following proposition, we collect several elementary properties of the multishift $S_{\lambdab}$.

\begin{proposition}\label{shift-prop}
Let $\mathscr T = (V,\mathcal E)$ be the directed Cartesian product of rooted directed trees $\mathscr T_1, \cdots, \mathscr T_d$ and 
let $S_{\lambdab}$ be a multishift on $\mathscr T$. For $j=1, \cdots, d$, $w \in V$,  let $\beta (j, w, 0):=1$ 
\index{$\beta (j, w, n)$}
and 
\beqn \beta (j, w, n) := \lambda_w^{(j)} \lambda_{\parenti{j}{w}}^{(j)} \cdots \lambda_{\parentki{j}{n-1}{w}}^{(j)}~(n \geqslant  1).\eeqn 
Also, let $\alpha^{(0)}=0 \in \mathbb N^d$ and $\alpha^{(j)}=(\alpha_1, \cdots, \alpha_{j}, 0, \cdots, 0) \in \mathbb N^d$ for $j = 1, \cdots, d.$ 
Then 
the following statements are true:
\begin{enumerate}
\item[(i)] $S_{\lambdab}$ is commuting if and only if for all $i,j =1, \cdots, d$ and for all $v \in V$,
\beq \label{commuting} \lambda^{(j)}_u \lambda^{(i)}_{\parenti{j}{u}}=\lambda^{(i)}_u \lambda^{(j)}_{\parenti{i}{u}}\ \text{for all}\ u \in \mathsf{Chi}_j\childi{i}{v}.\eeq
\item[(ii)] $S_{\lambdab}$ is doubly commuting if and only if \eqref{commuting} holds and for all $v \in V$ and $i,j=1, \cdots, d$ with $i \neq j$, the following condition holds:
\beq \label{d-commuting} \overline{\lambda}_v^{(j)}\lambda_{\parenti{j}{u}}^{(i)} = \lambda_u^{(i)}\overline{\lambda}_u^{(j)}\
\text{for all}~ u \in \childi{i}{v}.\eeq
\end{enumerate}
If, in addition, $S_{\lambdab}$ is commuting then the following statements hold true:
\begin{enumerate}
\item[(iii)] For all $1 \leqslant i,j \leqslant d$, for all $v \in V$ and for all $n \geqslant  1$,
\beq\label{beta-property}
\beta(j, \parenti{i}{v}, n) \lambda^{(i)}_v = \beta(j, v, n) \lambda^{(i)}_{\parentki{j}{n}{v}}.
\eeq
\item[(iv)] For $\alpha = (\alpha_1, \cdots, \alpha_d) \in \mathbb N^d$ and for all $v \in V$,
\beq \label{S-powers}
S^{\alpha}_{\lambdab} e_v = \sum_{w \in \childnt{\alpha}{v}} \prod_{j=1}^d \beta\big(j, \parentnt{\alpha^{(j-1)}}{w}, \alpha_j\big)e_w.
\eeq
\item[(v)] 
For $\alpha = (\alpha_1, \cdots, \alpha_d) \in \mathbb N^d$ and for all $v \in V$,
\beqn \label{S*-powers}
S^{*\alpha}_{\lambdab} e_v 
= \prod_{j=1}^d \overline{\beta} \big(j, \parentnt{{\alpha}^{(j-1)}}{v}, \alpha_j\big)e_{\parentnt{\alpha}{v}}, 
\eeqn
where $\overline{\beta}(\cdot) = \overline{\beta(\cdot)}.$
%$\tilde{\alpha}^{(j)}=(0, \cdots, 0, \alpha_{j+1}, \cdots, \alpha_{d}) \in \mathbb N^d$ for $1 \leqslant j \leqslant d-1$ and $\tilde{\alpha}^{(d)}=0 \in \mathbb N^d$. 
\item[(vi)] For $\alpha = (\alpha_1, \cdots, \alpha_d) \in \mathbb N^d$ and for all $v \in V$,
\beqn \label{S*S-powers}
S^{*\alpha}_{\lambdab} S^{\alpha}_{\lambdab} e_v = \sum_{w \in \childnt{\alpha}{v}} \prod_{j=1}^d \big|\beta \big(j, \parentnt{\alpha^{(j-1)}}{w}, \alpha_j\big)\big|^2e_v.
\eeqn
\item[(vii)] The multishift $S_{\lambdab}$ is toral left invertible if and only if $$\inf_{1 \leqslant j \leqslant d} \ \inf_{v \in V}\sum_{w \in \childi{j}{v}} |\lambda^{(j)}_w|^2 > 0.$$
\item[(viii)] The multishift  $S_{\lambdab}$ is joint left invertible if and only if
$$\inf_{v \in V}\sum_{j=1}^d \sum_{w \in \childi{j}{v}}|\lambda^{(j)}_w|^2 > 0.$$
\item[(ix)] If $\alpha \neq \beta$ in $\mathbb N^d$, then $\inp{S^{\alpha}_{\lambdab} e_v}{S^{\beta}_{\lambdab}e_v}=0$ for every $v \in V.$ 
\item[(x)] If $v \neq w$ in $V$, then $\inp{S^{\alpha}_{\lambdab} e_v}{S^{\alpha}_{\lambdab}e_w}=0$ for every $\alpha \in \mathbb N^d.$ 
\end{enumerate}
\end{proposition}

\begin{proof}
Let $i,j =1, \cdots, d$ and $v \in V$. Then
\beq\label{SjSi}
S_j S_i e_v = S_j \sum_{w \in \childi{i}{v}}\lambda_w^{(i)} e_w &=& \sum_{w \in \childi{i}{v}} \sum_{u \in \childi{j}{w}}\lambda_w^{(i)}\lambda_u^{(j)}e_u
\nonumber\\
&=& \sum_{w \in \childi{i}{v}} \sum_{u \in \childi{j}{w}}\lambda_{\parenti{j}{u}}^{(i)}\lambda_u^{(j)}e_u.
\eeq
By symmetry,
\beq\label{SiSj}
S_i S_j e_v 
= \sum_{w \in \childi{j}{v}} \sum_{u \in \childi{i}{w}}\lambda_u^{(i)}\lambda_{\parenti{i}{u}}^{(j)}e_u.
\eeq 
%\beq\label{SiSj}
%S_i S_j e_v = \sum_{w \in \childi{j}{v}}\lambda_w^{(j)} S_i e_w &=& \sum_{w \in \childi{j}{v}}\lambda_w^{(j)} \sum_{u \in \childi{i}{w}}\lambda_u^{(i)}e_u \nonumber\\
%&=& \sum_{u \in \mathsf{Chi}_j\childi{i}{v}}\lambda_u^{(i)}\lambda_{\parenti{i}{u}}^{(j)}e_u.
%\eeq 
By Lemma \ref{disjoint}(i), $\mathsf{Chi}_j\childi{i}{v}=\mathsf{Chi}_i\childi{j}{v}$. Hence,
by evaluating at $u \in \mathsf{Chi}_j\childi{i}{v}$, we may infer from equations \eqref{SjSi} and \eqref{SiSj} that $S_{\lambdab}$ is commuting if and only if \eqref{commuting} holds.
To see (ii), note that
$$S_iS_j^* e_v = \sum_{w \in \childi{i}{\parenti{j}{v}}} \overline{\lambda}_v^{(j)} \lambda_w^{(i)} e_w\ \text{and}\ S_j^*S_i e_v = \sum_{w \in \childi{i}{v}} \lambda_w^{(i)} \overline{\lambda}_w^{(j)} e_{\parenti{j}{w}}.$$
By Remark \ref{child-par}, $\childi{i}{\parenti{j}{v}} = \parenti{j}{\childi{i}{v}}$. Therefore, by arguing as above,  $S_iS_j^* e_v = S_j^*S_i e_v$ for all $v \in V$ if and only if \eqref{d-commuting} holds. This proves (i) and (ii). Also, (iii) may be deduced by repeated applications of \eqref{commuting}.
%To see (iii), let $n \geqslant  1$. Then
%\beqn
%\beta(j, \parenti{i}{v}, n) \lambda^{(i)}_v &=& \lambda^{(i)}_v \lambda_{\parenti{i}{v}}^{(j)} \lambda_{\parenti{j}{{\parenti{i}{v}}}}^{(j)} \cdots \lambda_{\parentki{j}{n-1}{{\parenti{i}{v}}}}^{(j)}\\
%&\overset{\eqref{commuting}}=& \lambda^{(j)}_v \lambda_{\parenti{j}{v}}^{(i)} \lambda_{\parenti{j}{{\parenti{i}{v}}}}^{(j)} \cdots \lambda_{\parentki{j}{n-1}{{\parenti{i}{v}}}}^{(j)}\\
%&\overset{\eqref{commuting}}=& \lambda^{(j)}_v \lambda_{\parenti{j}{v}}^{(j)} \lambda^{(i)}_{\parentki{j}{2}{v}}\lambda_{\parentki{j}{2}{{\parenti{i}{v}}}}^{(j)} \cdots \lambda_{\parentki{j}{n-1}{{\parenti{i}{v}}}}^{(j)}\\
%&\cdots&\\
%&=&\lambda_v^{(j)} \lambda_{\parenti{j}{v}}^{(j)} \cdots \lambda_{\parentki{j}{n-1}{v}}^{(j)} \lambda^{(i)}_{\parentki{j}{n}{v}}=\beta(j, v, n) \lambda^{(i)}_{\parentki{j}{n}{v}}.
%\eeqn

We prove (iv) by induction on $|\alpha|$ for $\alpha \in \mathbb{N}^d$. In case $|\alpha| = 0$, it is easily verified that \eqref{S-powers} holds. Let $n \in \mathbb N$ and suppose that \eqref{S-powers} holds for all $\alpha \in \mathbb N^d$ with $|\alpha| = n$. Let $\alpha \in \mathbb N^d$ and $|\alpha| = n+1$. Then $\alpha =\gamma + \epsilon_i$ for some $1 \leqslant i \leqslant d$ and some $\gamma \in \mathbb N^d$ with $|\gamma| = n$. Therefore, for $v \in V$,
\beqn
S^{\alpha}_{\lambdab} e_v &=& S^{\epsilon_i}_{\lambdab} S^{\gamma}_{\lambdab} e_v = S_i \sum_{w \in \childnt{\gamma}{v}} \prod_{j=1}^d \beta\big(j, \parentnt{\gamma^{(j-1)}}{w}, \gamma_j\big)e_w\\
&=& \sum_{w \in \childnt{\gamma}{v}} \prod_{j=1}^d \beta\big(j, \parentnt{\gamma^{(j-1)}}{w}, \gamma_j\big) \sum_{u \in \childi{i}{w}} \lambda^{(i)}_u e_u\\
&=& \sum_{u \in \childnt{\alpha}{v}} \prod_{j=1}^d \beta\big(j, \parentnt{\gamma^{(j-1)}}{\parenti{i}{u}}, \gamma_j\big) \lambda^{(i)}_u e_u.
\eeqn
In view of $\childi{i}{\childnt{\gamma}{v}}=\childnt{\alpha}{v}$, the last equality may be justified by pointwise evaluation.
It now suffices to check that 
\beq \label{id}
\prod_{j=1}^d \beta\big(j, \parentnt{\gamma^{(j-1)}}{\parenti{i}{u}}, \gamma_j\big) \lambda^{(i)}_u  = \prod_{j=1}^d \beta\big(j, \parentnt{\alpha^{(j-1)}}{{u}}, \alpha_j\big).
\eeq
This follows from repeated applications of \eqref{beta-property}. Indeed,
\beqn
\beta(1, \parenti{i}{u}, \gamma_1)\lambda^{(i)}_u &\overset{\eqref{beta-property}}= & \beta(1, u, \gamma_1)\lambda^{(i)}_{\parentki{1}{\gamma_1}{u}}, \\ 
\beta(2, \parentki{1}{\gamma_1}{\parenti{i}{u}}, \gamma_2)\lambda^{(i)}_{\parentki{1}{\gamma_1}{u}} &\overset{\eqref{beta-property}}= &
\beta(2, \parentki{1}{\gamma_1}{u}, \gamma_2)
\lambda^{(i)}_{\mathsf{par}^{\langle \gamma_2 \rangle}_2 \parentki{1}{\gamma_1}{u}}.
\eeqn
Continuing in this way and using the facts that $\gamma+ \epsilon_i = \alpha$ and $\parenti{i}{\parenti{j}{u}} = \parenti{j}{\parenti{i}{u}}$, we obtain \eqref{id}. This proves (iv).
The proof of (v) is along the lines of (iv) and hence we skip the details. 
Note that (vi) is a consequence of (iv) and (v).
We leave the routine verifications of (vii) and (viii) to the reader. Finally, (ix) may be concluded from \eqref{S-powers} and Lemma \ref{disjoint}(iii) while (x) is an immediate consequence of  \eqref{S-powers} and Lemma \ref{disjoint}(ii). 
\end{proof}

\begin{corollary}
Let $\mathscr T = (V,\mathcal E)$ be the directed Cartesian product of rooted directed trees $\mathscr T_1, \cdots, \mathscr T_d$ and
let $S_{\lambdab}$ be a commuting multishift on $\mathscr T$. Then
$S_{\lambdab}$ is unitarily equivalent to $S_{|\lambdab|}$, where $$|\lambdab|=\{|\lambda^{(j)}_v| : v \in V^{\circ}, ~ j=1, \cdots, d\}.$$ 
\end{corollary}
\begin{proof}
The idea of the proof is a combination of ideas from \cite[Corollary 2]{JL} and \cite[Theorem 3.2.1]{JJS}.
For the sake of simplicity, we treat the case $d=2.$
By Proposition \ref{shift-prop}(i),
for all $i,j =1, 2$ and for all $v \in V$,
\beq \label{arg-c} \arg^{(j)}_u + \arg^{(i)}_{\parenti{j}{u}}=\arg^{(i)}_u + \arg^{(j)}_{\parenti{i}{u}}\ \text{for all}\ u \in \mathsf{Chi}_j\childi{i}{v},\eeq
where $\arg^{(j)}_v$ denotes the principal argument of ${\lambda^{(j)}_v}$.
For a subset $\{\vartheta_v\}_{v \in V}$ of the real line $\mathbb R$, define the unitary operator $U_{\vartheta} : l^2(V) \rar l^2(V)$ 
by
\beqn
U_{\vartheta}e_v = \exp(i\vartheta_v)e_v,~v \in V.
\eeqn
Let $(T_1, T_2)$ denote the commuting $2$-tuple $S_{|\lambdab|}.$ With these notations, the system $S_jU_{\vartheta}=U_{\vartheta}T_j~(j=1, 2)$ is equivalent to the system
\beq \label{mod-system}
\vartheta_w - \vartheta_{\parenti{j}{w}} = \arg^{(j)}_w, ~w \in \childi{j}{V}~\mbox{and} ~j=1, 2.
\eeq
We will show that the above system has a solution. To see that, let $\vartheta_{\rootb}=0.$ 
Its clear that $\vartheta_w$ can be defined recursively using \eqref{mod-system}.
To see that $\vartheta_w$ is well-defined,  
it suffices to check that
\beqn
\arg^{(1)}_w + \vartheta_{\parenti{1}{w}} = \arg^{(2)}_w + \vartheta_{\parenti{2}{w}},~w \in \childi{1}{V} \cap \childi{2}{V}. 
\eeqn
whenever \eqref{mod-system} holds for $\parenti{j}{w}~(j=1, 2).$ Note that
\beqn
\arg^{(2)}_w + \vartheta_{\parenti{2}{w}} &\overset{\eqref{mod-system}}=& 
\arg^{(2)}_w + (\vartheta_{\parenti{1}{\parenti{2}{w}}} + \arg^{(1)}_{\parenti{2}{w}}) \\ &\overset{\eqref{arg-c}}=& \arg^{(1)}_w + (\vartheta_{\parenti{2}{\parenti{1}{w}}} + \arg^{(2)}_{\parenti{1}{w}}) \\ &\overset{\eqref{mod-system}}=& \arg^{(1)}_w + \vartheta_{\parenti{1}{w}}
\eeqn 
This completes the proof.
\end{proof}

{\it From here onwards, we assume that the weights from
$\lambdab$ appearing in the definition of $S_{\lambdab}$ are always positive.}

Given a positive integer $d,$ \index{$\mathcal H^{\oplus{d}}$}
we set $$\mathcal H^{\oplus{d}} := \underbrace{\mathcal H \oplus \cdots \oplus \mathcal H}_{d~\mbox{\tiny times}}.$$
For a commuting $d$-tuple $T=(T_1, \cdots, T_d)$ on $\mathcal H,$ consider the linear transformation 
\index{$D_T$}
$D_T : \mathcal H \rar \mathcal H^{\oplus{d}}$ given by $$D_Th:=(T_1h, \cdots, T_dh)~\mbox{for}~ h \in \mathcal H.$$
Note that the kernel of $D_T$ is precisely the joint kernel 
\index{$\ker T=\ker D_T$}
$\ker T:=\bigcap_{j=1}^d \ker T_j$ of $T.$ As per requirement, we use notations $\ker D_T$ and $\ker T$ interchangeably.

\begin{corollary} \label{dense} 
Let $\mathscr T = (V,\mathcal E)$ be the directed Cartesian product of rooted directed trees $\mathscr T_1, \cdots, \mathscr T_d$ and
let $S_{\lambdab}=(S_1, \cdots, S_d)$ be a commuting multishift on $\mathscr T$. For $i \in \mathbb N$, let $T^{(i)}$ denote the commuting $d$-tuple $(S^{*i}_1, \cdots, S^{*i}_d)$. Then $\bigcup_{i \in \mathbb N} \ker D_{T^{(i)}}$ is dense in $l^2(V).$
\end{corollary}
\begin{proof} Since $\{\ker D_{T^{(i)}}\}_{i \in \mathbb N}$ is a monotonically increasing sequence of subspaces of $l^2(V),$ it suffices to show that $\mathcal M:=\bigcup_{i \in \mathbb N} \ker D_{T^{(i)}}$ contains $e_v$ for every $v \in V.$ For any $v \in V,$ 
by Proposition \ref{shift-prop}(v), $e_v \in \ker D_{T^{(i)}}$ for all $i >  |\alpha_v|,$ where $\alpha_v$ is the depth of $v$ in $\mathscr T.$
\end{proof}

%For any non-empty subset $W$ of $V$, we identify the Hilbert space $l^2(W)$ as a closed linear subspace of $l^2(V)$ by identifying each $f \in l^2(W)$ with the function $\tilde{f} \in l^2(V)$ which extends $f$ and vanishes on the set $V \setminus W$.

The following proposition is motivated by the description of kernel of the adjoint of weighted shift on directed tree as given in \cite{JJS}.
\begin{proposition} \label{kerSi*}
Let $\mathscr T = (V,\mathcal E)$ be the directed Cartesian product of rooted directed trees $\mathscr T_1, \cdots, \mathscr T_d$ and let $S_{\lambdab}=(S_1, \cdots, S_d)$ be a commuting multishift on $\mathscr T$. Then, for $j=1, \cdots, d,$ we have
\beq \label{kerSi*-eq} 
\ker S_j^* = \bigoplus_{\underset{v_j \in V_\prec^{(j)}}{v \in V}} \Big\{l^2(\childi{j}{v}) \ominus [\Gamma^{(j)}_v]  \Big\} \oplus \bigvee \big\{e_v : v \in V\ \text{and}\ v_j = \mathsf{root}_j \big\},\eeq
where 
\index{$\Gamma^{(j)}_v$}
$\Gamma^{(j)}_v : \childi{j}{v} \rar \mathbb{C}$ is given by $\Gamma^{(j)}_v (u) = \lambda^{(j)}_u (= (S_j e_v)(u))$.
\end{proposition}

\begin{proof}
The result follows from \cite[Proposition 3.5.1(ii)]{JJS} and \cite[eq.(4)]{CT}.
\end{proof}
\begin{remark} \label{inf-dim}
Note that $\ker S^*_j$ is infinite dimensional whenever $d > 1$. 
\end{remark}

%From the view point of spectral theory, one would like 
It is desirable to have a description similar to \eqref{kerSi*-eq} for the joint kernel of $S^*_{\lambdab}.$ The following example shows that the situation is more intriguing than it seems.

\begin{example}
Let $\mathscr T=(V, \mathcal E)$ be the directed Cartesian product of two rooted directed trees $\mathscr T_1, \mathscr T_2$. Assume that there exists a vertex
$v= (v_1, v_2) \in V$ such that $v_1$ has two children, say $\dot{v_{1}}$ and $\ddot{v_1}$, and $v_2$ has only one child $\dot{v_2}$. 
Then 
\beqn
\child{v} = \{(\dot{v_{1}},v_2),(\ddot{v_1}, v_2), (v_1,\dot{v_2})\},
\eeqn
and hence $l^2(\child{v})$ is $3$-dimensional.
Let
$f \in l^{2}(\child{v}) \ominus [\Gamma^{(1)}_{v}, \Gamma^{(2)}_{v}]$, and write $$f = \alpha e_{(\dot{v_{1}},v_2)} + \beta e_{(\ddot{v_1}, v_2)} + \gamma e_{(v_1,\dot{v_2})}$$ for some scalars $\alpha, \beta, \gamma \in \mathbb C.$ 
We claim that $$l^{2}(\child{v})  \ominus [\Gamma^{(1)}_{v}, \Gamma^{(2)}_{v}] \nsubseteq E,$$ 
where $E$ denotes the joint kernel of $S^*_{\lambdab}.$ 
Assume to the contrary that $f \in E.$
Note that
$S_{2}^{*}f =0$ implies $$\alpha \lambda_{(\dot{v_{1}},v_2)}^{(2)} e_{(\dot{v_{1}},\mathsf{par} (v_2))}+ \beta \lambda^{(2)}_{({\ddot{v_{1}}},v_2)} e_{(\ddot{v_{1}},\mathsf{par} (v_2))} + \gamma \lambda_{ (v_1,\dot{v_2})}^{(2)} e_{(v_1,v_2)} =0,$$ which is true only if $\alpha=\beta=\gamma=0$, that is, $f=0.$ 
On the other hand, $[\Gamma^{(1)}_{v}, \Gamma^{(2)}_{v}]$ is at most two dimensional (as $\Gamma^{(1)}_{v}, \Gamma^{(2)}_{v}$ could be linearly dependent), and hence  
$$\dim(l^{2}(\child{v}) \ominus [\Gamma^{(1)}_{v}, \Gamma^{(2)}_{v}]) \geqslant  1.$$ Thus the claim stands verified.
\end{example} 

As evident from the preceding discussion, the exact description of the joint kernel of $S^*_{\lambdab}$ is not as simple as in the case $d=1$, and hence
we postpone it to Chapter 4. 
For the time being, let us see that the joint kernel of $S^*_{\lambdab}$ can be finite dimensional in many interesting situations (cf. Remark \ref{inf-dim}).
 
\begin{proposition} \label{joint-k}
Let $\mathscr T = (V,\mathcal E)$ be the directed Cartesian product of rooted directed trees $\mathscr T_1, \cdots, \mathscr T_d$. Let $S_{\lambdab}=(S_1, \cdots, S_d)$ be a commuting multishift on $\mathscr T$ and let $E$ denote the joint kernel of $S^*_{\lambdab}$. Then
\beq
\label{joint-k-eq}
\bigvee_{j=1}^d \bigoplus_{{v \in D_j}} \Big\{l^2(\childi{j}{v}) \ominus [\Gamma^{(j)}_v] \Big\} \oplus [e_{\rootb}] \notag \\  \subseteq E \subseteq \bigvee \big\{e_v : v \in F_1 \times \cdots \times F_d \big\},
\eeq
where $D_j:=\{v \in V : v_j \in V^{(j)}_{\prec}~\mbox{and~}v_i=\mathsf{root}_i~\mbox{for~}i \neq j\}$, 
$\Gamma^{(j)}_v : \childi{j}{v} \rar \mathbb{C}$ is given by $\Gamma^{(j)}_v (u) = \lambda^{(j)}_u$,
and   
$F_j:=\child{V^{(j)}_{\prec}} \cup \{\mathsf{root}_j\}~(j=1, \cdots, d).$
\end{proposition}
\begin{proof}
To see the first inclusion, let $f \in l^2(\childi{j}{v}) \ominus [\Gamma^{(j)}_v]$ for some $v \in D_j$ and for a fixed $j=1, \cdots, d.$
Thus $f = \sum_{u \in \childi{j}{v}}f(u)e_u$ satisfies $\inp{f}{\Gamma^{(j)}_v}=0.$
Now, for any $i \neq j,$
\beqn
S^*_if = \sum_{u \in \childi{j}{v}}f(u)\lambda^{(i)}_u
e_{\parenti{i}{u}} = 0
\eeqn
since $v \in D_j.$ Further, 
\beqn
S^*_jf = \sum_{u \in \childi{j}{v}}f(u)\lambda^{(j)}_u
e_{\parenti{j}{u}} =\sum_{u \in \childi{j}{v}}f(u)\lambda^{(j)}_u
e_{v}=\inp{f}{\Gamma^{(j)}_v}e_v= 0,
\eeqn
where we used $\inp{f}{\Gamma^{(j)}_v}=\sum_{u \in \childi{j}{v}}f(u)\lambda^{(j)}_u.$

To see the second inclusion,
%for $v \in V,$ set $\mathsf{sib}_j(v):=\childi{j}{\parenti{j}{v}}$ for $j=1, \cdots, d.$
let $f \in E$ be such that $f = \sum_{v \in V}f(v)e_v.$ Then, for $j=1, \cdots, d$,
\beqn
S^*_jf &=& \sum_{v \in V}f(v) \lambda^{(j)}_v e_{\parenti{j}{v}} \\
&=& \sum_{\underset{\mbox{\tiny card}(\mathsf{sib}_j(v))=1}{v \in V}}f(v) \lambda^{(j)}_v e_{\parenti{j}{v}} + 
\sum_{\underset{\mbox{\tiny card}(\mathsf{sib}_j(v))\geqslant  2}{v \in V}}f(v) \lambda^{(j)}_v e_{\parenti{j}{v}}.
\eeqn
Note that $e_{\parenti{j}{v}} \perp e_{\parenti{j}{w}}$ if $v \neq w$ and $\mbox{card}(\mathsf{sib}_j(v))=1=\mbox{card}(\mathsf{sib}_j(w))$.  
Since $S^*_jf=0$, we obtain 
$f(v)=0$ for every $v \in V$ such that $\mbox{card}(\mathsf{sib}_j(v))=1.$
Thus $f(v) \neq 0$ implies that either $\mbox{card}(\mathsf{sib}_j(v))$ is $0$ or bigger than $1$ for all $j=1, \cdots, d.$
However, $\mbox{card}(\mathsf{sib}_j(v))\geqslant  2$ if and only if $v_j \in \child{V^{(j)}_{\prec}}$. Further, $\mbox{card}(\mathsf{sib}_j(v))= 0$ if and only if $v_j=\mathsf{root}_j$. 
This completes the proof. 
\end{proof}

\begin{corollary} \label{dimE-bound}
Let $\mathscr T, S_{\lambdab}$ and $E$ be as in the preceding proposition. If $\mathscr T$ is locally finite with finite set of branching vertices, then $E$ is finite dimensional. Moreover, 
\beq \label{dim-formula}
1 + \sum_{j=1}^d \sum_{v_j \in V^{(j)}_\prec}\big( \mbox{card} \big( \child{v_j}\big) - 1 \big) \, \leqslant \, \dim E \, \leqslant \, \prod_{j=1}^d (\mbox{card}\big(\child{V^{(j)}_\prec}\big) + 1). 
\eeq
%where $a_j :=\mbox{card}\big(\child{V^{(j)}_\prec}\big) + 1$ for $j = 1, \cdots, d$.
%$$a_j = \begin{cases} 
%\mbox{card}\big(\child{V^{(j)}_\prec}\big) & \mbox{~if~} \mathsf{root}_j \in V^{(j)}_\prec, \\ 
%\mbox{card}\big(\child{V^{(j)}_\prec}\big) + 1 & \mbox{~otherwise.} \end{cases}$$
\end{corollary}

\begin{proof}
The proof is obvious from \eqref{joint-k-eq}.
\end{proof}
\begin{remark}
In Example \ref{classical}, the formula \eqref{dim-formula} holds with equalities at all places (with $\dim\, E =1$). 
%Similar conclusion holds for Example \ref{classical-mix} (with $\dim\, E =2$). 
On the other hand, in Example \ref{classical-mix}, equality holds only at left end of \eqref{dim-formula} (with $\dim\, E =2$). Further, in  Example \ref{T1-T1}, equality may or may not hold even at left end of \eqref{dim-formula} (with $\dim\, E =3$ or $4$). The last two assertions may be concluded from Examples \ref{j-k-mixed} and \ref{j-k-T1T1}.
\end{remark}

\begin{corollary} \label{dimE-finite} 
Let $\mathscr T = (V,\mathcal E)$ be the directed Cartesian product of rooted directed trees $\mathscr T_1, \cdots, \mathscr T_d$. Let $S_{\lambdab}$ be a commuting multishift on $\mathscr T$ and let $E$ denote the joint kernel of $S^*_{\lambdab}$. Then $E$ is finite dimensional if and only if $\mathscr T$ is locally finite with finite joint branching index.
\end{corollary}
\begin{proof} 
If $E$ is infinite dimensional then by \eqref{joint-k-eq}, the cardinality of $F_j=\child{V^{(j)}_{\prec}} \cup \{\mathsf{root}_j\}$ must be infinite for some $j=1, \cdots, d.$ It follows that either $\mathscr T_j$ is not locally finite or $V^{(j)}_{\prec}$ is infinite. To see the converse, suppose that $\mathscr T_j$ is either not locally finite or of infinite branching index for some $j=1, \cdots, d$.
%Then the infinite dimensional subspace $\mathcal M$ of $l^2(V)$ given by
%\beqn
%\mathcal M=l^2\Big(\cup_{k=0}^{\infty}\childnt{k\epsilon_j}{\rootb}\Big)
%\eeqn
%is invariant for $S^*_{j}$. Also, since $S^*_i|_{\mathcal M}=0$ for $i \neq j$, it follows that $\ker {S^*_j}|_{\mathcal M}=\ker S^*_{\lambdab}|_\mathcal M \subseteq E.$ 
By Proposition \ref{joint-k}, $$\mathcal M:=\bigoplus_{{v \in D_j}} \Big\{l^2(\childi{j}{v}) \ominus [\Gamma^{(j)}_v] \Big\} \subseteq E,$$
where $D_j=\{v \in V : v_j \in V^{(j)}_{\prec}~\mbox{and~}v_i=\mathsf{root}_i~\mbox{for~}i \neq j\}$. 
Note that $l^2(\childi{j}{v}) \ominus [\Gamma^{(j)}_v]$ is nonzero for every $v \in D_j.$ 
If $\mathscr T_j$ is not locally finite, then $l^2(\childi{j}{v})$ is infinite dimensional for some $v \in D_j.$ If $\mathscr T_j$ is of infinite branching index, then $D_j$ is infinite. In any case, $\mathcal M$ and hence $E$ is infinite dimensional.
\end{proof}

%\section{A Single Variable Shift associated with $S_{\lambdab}$}
We have already seen in Theorem \ref{semi-tree} that the directed Cartesian product of directed trees admits a directed semi-tree structure. Since there is a notion of shifts $\mathscr S_{\deltab}$ on directed semi-trees \cite{MS}, it is thus natural to reveal the relation between $\mathscr S_{\deltab}$ on directed semi-tree $\mathscr T$ and the multishift $S_{\lambdab}$. To see this, let us recall the notion of shift on a directed semi-tree from \cite[Section 5]{MS}. 

Let $\mathscr T = (V,\mathcal E)$ be the directed Cartesian product of rooted directed trees $\mathscr T_1, \cdots, \mathscr T_d$. 
Given a system 
${\deltab} = \{\delta_{(u, v)}: (u, v) \in \mathcal E\}$ of positive numbers, 
define the {\em weighted shift operator} $\mathscr S_{\deltab}$ on ${\mathscr T}$ \index{$\mathscr S_{\deltab}$}
with weights $\deltab$ by
   \begin{align*}
   \begin{aligned}
{\mathscr D}(\mathscr S_{\deltab}) & := \{f \in l^2(V) \colon
\varDelta_{\mathscr T} f \in l^2(V)\},
   \\
\mathscr S_{\deltab} f & := \varDelta_{\mathscr T} f, \quad f \in {\mathscr
D}(\mathscr S_{\deltab}),
   \end{aligned}
   \end{align*}
where $\varDelta_{\mathscr T}$ is the mapping defined on
complex functions $f$ on $V$ by
   \begin{align*}
(\varDelta_{\mathscr T} f) (v) :=
   \begin{cases} \sum_{u \in \mathsf{Par}(v)} \delta_{(u, v)}f(u)
& \text{if } v \in
V \setminus \{\rootb\},
   \\
   0 & \text{otherwise}.
   \end{cases}
   \end{align*}
Let us see the precise relation between $S_{\lambdab}$ and $\mathscr S_{\deltab}$.

\begin{proposition} Let $\mathscr S_{\deltab}$ be the weighted shift operator on the directed Cartesian product of rooted directed trees $\mathscr T_1, \cdots, \mathscr T_d.$  
Set $\lambda^{(j)}_v := \delta_{(u, v)}$ for $v \in \childi{j}{u},$ and let $S_{\lambdab}=(S_1, \cdots, S_d)$ be the multishift on $\mathscr T$ (possibly unbounded). Then $\cap_{j=1}^d\mathscr D(S_{j})\subseteq \mathscr D(\mathscr S_{\deltab}).$ If, in addition, $S_{1}, \cdots, S_d$ are bounded linear operators on $l^2(V)$, then $\mathscr S_{\deltab}$ is bounded. In this case,
$\mathscr S_{\deltab}= \sum_{j=1}^dS_j.$
\end{proposition}
\begin{proof}
Let $f \in \cap_{j=1}^d\mathscr D(S_{j}).$ 
%Set $\lambda^{(j)}_v := \delta_{(u, v)}$ if $u = \parenti{j}{v},$ that is, $v \in \childi{j}{u}.$ 
Then 
\beqn (\mathscr S_{\deltab}f)(v) &=& \sum_{u \in \mathsf{Par}(v)} \delta_{(u, v)}f(u)~\overset{\eqref{semi-parent}}  =~ \sum_{j=1}^d \sum_{u = \mathsf{par}_j(v)} \delta_{(u, v)}f(u) \\ &=& \sum_{j=1}^d \lambda^{(j)}_{v} f(\mathsf{par}_j(v))  ~=~ \sum_{j=1}^d(S_jf)(v). \eeqn
This shows that $f \in \mathscr D(\mathscr S_{\deltab}).$ If $S_{1}, \cdots, S_d$ are bounded on $l^2(V)$, then $\mathscr D(\mathscr S_{\deltab})=l^2(V)$. However, $\mathscr S_{\deltab}$ is always closed \cite[Proposition 5.1]{MS}, and hence $\mathscr S_{\deltab}$ is a bounded linear operator in this case.
This completes the proof.
\end{proof}

The above result shows that $\sum_{j=1}^dS_j$ can be realized as a shift  on $\mathscr T$ endowed with the directed semi-tree structure.

%\chapter{Spectral Properties of Multishifts}

In the remaining part of this chapter, we obtain some basic properties of multishifts $S_{\lambdab}$ on $\mathscr T.$ These include circularity and analyticity. We also obtain a matrix decomposition for $S_{\lambdab}$ in dimension $d=2$.
All these results are then used to examine various spectral parts of $S_{\lambdab}$ such as point spectrum, Taylor spectrum, and essential spectrum. The discussion to follow relies heavily on the multivariable spectral theory as expounded in \cite{Cu}. 

\section{Strong Circularity and Taylor Spectrum}

Let $T=(T_1, \cdots, T_d)$ be a commuting $d$-tuple on $\mathcal H$. We say that $T$ is {\it circular} if for every $\theta=(\theta_1, \cdots, \theta_d) \in \mathbb R^d$, there exists a unitary operator $\varGamma_{\theta}$ on $\mathcal H$ such that $$\varGamma^*_{\theta} T_j \varGamma_{\theta} = \text{exp}(i\theta_j) T_j\ \text{for all}\ j=1, \cdots, d.$$
We say that $T$ is {\it strongly circular} if in addition $\varGamma_{\theta}$ can be chosen to be a strongly continuous unitary representation of $\mathbb R^d$ in the following sense: For every $h \in \mathcal H$, the function
$\theta \longmapsto \varGamma_{\theta}h$ is continuous on $\mathbb R^d$.

The above notion in dimension $d=1$ has been introduced and studied in \cite{AHHK}. These operators have been studied considerably thereafter (refer to \cite{Ge}, \cite{Ml}, \cite{S}, \cite{BM-1}). The fact that any classical multishift is polycircular is first obtained in \cite[Corollary 3]{JL}. This may also be deduced from \cite[Lemma 2.14]{CY}. 

The following generalizes \cite[Corollary 3]{JL}. Unlike the method of proof of \cite[Theorem 3.3.1]{JJS}, where the unitary $\varGamma_{\theta}$ comes from solution of a system of equations, our proof exhibits a formula for $\varGamma_{\theta}$. 
\begin{proposition}\label{circular}
Let $\mathscr T = (V,\mathcal E)$ be the directed Cartesian product of rooted directed trees $\mathscr T_1, \cdots, \mathscr T_d$ and let $S_{\lambdab}$ be a commuting multishift on $\mathscr T$. Then $S_{\lambdab}$ is strongly circular.
\end{proposition}

\begin{proof}
Let $\theta=(\theta_1, \cdots, \theta_d) \in \mathbb R^d$. 
%By Lemma \ref{disjoint}(vi), for each $v \in V$, there is a unique $\alpha_v = (\alpha_1, \cdots, \alpha_d) \in \mathbb{N}^d$ such that $v \in \childnt{\alpha_v}{\rootb}$. 
For $f = \sum_{v \in V} f(v) e_v \in l^2(V)$, define $\varGamma_{\theta} : l^2(V) \rar l^2(V)$ by \beqn \label{U-theta} \varGamma_{\theta} f := \sum_{v \in V} \text{exp}(-i\alpha_v\cdot \theta) f(v) e_v,\eeqn where $\alpha_v$ is the depth of $v$ in $\mathscr T$ and $\alpha \cdot \theta := \sum_{j=1}^d \alpha_j \theta_j$ for $\alpha \in \mathbb N^d$. Clearly, $\varGamma_{\theta}$ is unitary with inverse $\varGamma_{-\theta}$. 
Note that $\alpha_w = \alpha_v + \epsilon_j$ if $w \in \childi{j}{v}$ for any $v \in V$ and $j=1, \cdots, d.$
It follows that
\beqn
(\varGamma^*_{\theta} S_j \varGamma_{\theta}) f &=& (\varGamma^*_{\theta} S_j) \sum_{v \in V} \text{exp}(-i\alpha_v\cdot \theta) f(v) e_v\\
&=& \varGamma^*_{\theta} \Big( \sum_{v \in V} \text{exp}(-i\alpha_v\cdot \theta) f(v) \sum_{w \in \childi{j}{v}} \lambda^{(j)}_w e_w\Big) \\
&=& \sum_{v \in V} \text{exp}(-i\alpha_v \cdot \theta) f(v) \sum_{w \in \childi{j}{v}} \lambda^{(j)}_w \text{exp}(i \alpha_w \cdot \theta) e_w \\
&=& \sum_{v \in V} \text{exp}(-i\alpha_v\cdot \theta) f(v) \sum_{w \in \childi{j}{v}} \lambda^{(j)}_w \text{exp}(i(\alpha_v+\epsilon_j) \cdot \theta) e_w\\
&=& \sum_{v \in V} \text{exp}(i \epsilon_j \cdot \theta)  f(v) \sum_{w \in \childi{j}{v}} \lambda^{(j)}_w e_w\\
%&=& \text{exp}(i\theta_j) \sum_{v \in V} f(v) S_j e_v\\
&=& \text{exp}(i\theta_j) S_j f.
\eeqn 
Now we show that for any $f \in l^2(V),$ $\theta \longmapsto \varGamma_\theta f$ is continuous. Let $\{\theta^{(n)}\}_{n=1}^\infty$ be a sequence in $\mathbb R^d$ which converges to $\theta$. Then for $f = \sum_{v \in V} f(v) e_v \in l^2(V)$, 
$$\|(\varGamma_{\theta^{(n)}} - \varGamma_\theta) f\|^2 = \sum_{v \in V} \big|\exp(-i\alpha_v \cdot \theta^{(n)}) - \exp(-i\alpha_v \cdot \theta)\big|^2 |f(v)|^2 .$$
Let $\epsilon > 0$. Since $f \in l^2(V)$, there is a finite subset $W$ of $V$ such that $\displaystyle\sum_{v \in V \setminus W} |f(v)|^2 < \epsilon$. 
Further, as $\theta \mapsto \exp(-i\alpha_u \cdot \theta)$ is continuous for each $u \in W$, there exists a positive integer $n(u)$ such that $\big|\exp(-i\alpha_u \cdot \theta^{(n)}) - \exp(-i \alpha_u \cdot \theta)\big|^2 < \epsilon$ for all $n \geqslant  n(u)$. Let $n(0) := \max\{n(u) : u \in W\} < \infty.$ Then, {for all} $n \geqslant  n(0)$,
\beqn
\|(\varGamma_{\theta^{(n)}} - \varGamma_\theta) f\|^2 
&=& \sum_{v \in W} \big|\exp(-i\alpha_v \cdot \theta^{(n)}) - \exp(-i\alpha_v \cdot \theta)\big|^2 |f(v)|^2\\
&+& \sum_{v \in V \setminus W} \big|\exp(-i\alpha_v \cdot \theta^{(n)}) - \exp(-i\alpha_v \cdot \theta)\big|^2 |f(v)|^2\\
&<& \epsilon^2 \Big(\sum_{v \in W} |f(v)|^2 \Big) + 4\epsilon. %\\ & \leqslant & \epsilon(\epsilon \|f\|^2 + 4).  
\eeqn
This completes the proof.
\end{proof}

The following is immediate from spectral mapping property of Taylor spectrum and the preceding proposition.
\begin{corollary} \label{c-symmetry}
The Taylor spectrum  of a commuting multishift $S_{\lambdab}$ has poly-circular symmetry, that is, $\zeta\cdot w\in
\sigma(S_{\lambdab})$ for any $w\in \sigma(S_{\lambdab})$ and any $\zeta\in \mathbb T^d$. In particular, the Taylor spectrum of $S_{\lambdab}$ coincides with that of $S^*_{\lambdab}$.
\end{corollary}
\begin{remark}
Note that point spectrum, left spectrum and essential spectrum do also have polycircular symmetry.
\end{remark}

%\begin{proposition}
%Every multishift $S_{\lambdab}$ is unitarily equivalent to the multishift $S_{|\lambdab|}$, where $|\lambdab|:=\{|\lambda^{(j)}_v| : i=1, \cdots, d\}$.
%\end{proposition}
%\begin{proof}
%???
%\end{proof}

A special case of the following result, in which $\mathscr T$ is the directed Cartesian product of $\mathscr T_{1, 0}$ with itself, has been obtained in \cite{CY}.
\begin{proposition}\label{connected}
Let $\mathscr T = (V,\mathcal E)$ be the directed Cartesian product of rooted directed trees $\mathscr T_1, \cdots, \mathscr T_d$ and let $S_{\lambdab}$ be a commuting multishift on $\mathscr T$.  Then the Taylor spectrum of $S_{\lambdab}$ is connected.
\end{proposition}

\begin{proof}
The idea of this proof is similar to that of \cite[Lemma 3.8]{CY}.
By Remark \ref{0-eigen}, $0$ belongs to the point spectrum $\sigma_p(S^*_{\lambdab})$ of
$S^*_{\lambdab}$. Hence $0$ belongs to the Taylor spectrum of $S^*_{\lambdab}$. In view of Corollary \ref{c-symmetry}, 
it suffices to check that $\sigma(S^*_{\lambdab})$ is connected. Let $K_1$
be the connected component of $\sigma(S^*_{\lambdab})$ containing $0$ and let
$K_2=\sigma(S^*_{\lambdab}) \setminus K_1.$ By the Shilov Idempotent Theorem
\cite[Application 5.24]{Cu}, there exist invariant subspaces
$\mathcal M_1, \mathcal M_2$ of $S^*_{\lambdab}$ such that $l^2(V) =
\mathcal M_1 \dotplus \mathcal M_2$ (vector space direct sum of $\mathcal M_1$ and $\mathcal M_2$) and $\sigma({S^*_{\lambdab}}|_{\mathcal M_i}) =K_i$
for $i=1, 2.$

For every for $i \in \mathbb N$, let $T^{(i)}$ denote the commuting $d$-tuple $(S^{*i}_1, \cdots, S^{*i}_d).$
Let $h \in \mbox{ker}(D_{T^{(i)}})$ for fixed $i \in \mathbb N.$ Then $h=x + y$ for $x \in
\mathcal M_1$ and $y \in \mathcal M_2.$ It follows that
$S^{*i}_{j} x=0$ and $S^{*i}_{j}y=0$ for all $j=1, \cdots, d.$
If $y$ is nonzero, then
$0 \in \sigma_p(T^{(i)}|_{\mathcal M_2}) \subseteq \sigma(T^{(i)}|_{\mathcal M_2}),$
and hence by the spectral mapping property \cite{Cu}, $0 \in \sigma(T^{(1)}|_{\mathcal M_2})=\sigma({S^*_{\lambdab}}|_{\mathcal M_2}).$
Since $0 \notin K_2,$ we must have $y =0.$ It follows that $\mathcal
M_1$ contains the linear manifold $\bigcup_{i \in \mathbb N}
\mbox{ker}(D_{T^{(i)}}),$ which is dense in $l^2(V)$ by Corollary \ref{dense}.  Hence 
$\mathcal M_1 = l^2(V).$ Thus
the Taylor spectrum of $S^*_{\lambdab}$ is equal to $K_1.$ In particular, the
Taylor spectrum of $S_{\lambdab}$ is connected.
\end{proof}

A connected subset $\Omega$ of $\mathbb C^d$ is said to be {\it Reinhardt} if it is invariant under the action of the $d$-torus $\mathbb T^d$, that is, $\zeta \cdot z:=(\zeta_1 z_1, \cdots, \zeta_d z_d)$ belongs to $\Omega$ whenever $z \in \Omega$ and $\zeta \in \mathbb T^d.$ 

Combining Proposition \ref{circular} with the preceding result, we obtain the following basic fact.
\begin{corollary}
Let $\mathscr T = (V,\mathcal E)$ be the directed Cartesian product of rooted directed trees $\mathscr T_1, \cdots, \mathscr T_d$. Then
the Taylor spectrum of a commuting multishift on $\mathscr T$ is Reinhardt.
\end{corollary}
\begin{remark}
Suppose that the Taylor spectrum $\sigma(S_{\lambdab})$ of $S_{\lambdab}$ has spherical symmetry in the following sense: $Uz \in  \sigma(S_{\lambdab})$ whenever $z \in \sigma(S_{\lambdab})$ for every $d \times d$ unitary matrix $U$. Then $\sigma(S_{\lambdab})$ must be a closed ball centered at the origin. Indeed, $0 \in \sigma(S_{\lambdab})$ since $e_{\rootb}$ belongs to the joint kernel of $S^*_{\lambdab}$ in view of Remark \ref{0-eigen}. The desired conclusion now follows from the fact that every spherically symmetric, compact Reinhardt set containing $0$ is a closed ball centered at $0$. 
\end{remark}

%\begin{proof}
%Put $W_0:=V$, and for $j \geqslant  1$, let \beqn W_j:=
%%\displaystyle \Big\{v \in V : v \notin \childnt{k}{\rootb}~\mbox{for any~}
%%{k \in \mathbb{N}^d}~\mbox{with~}{0 \leqslant |k|\leqslant j-1} \Big\}
%\bigcup_{\underset{|k|\geqslant  j}{k \in \mathbb{N}^d}}\childnt{k}{\rootb}, ~M_j:=\bigvee \{e_v:v \in W_j\}.\eeqn
%Note that $\{W_j\}_{j\geqslant 0}$ is a strictly decreasing 
%sequence of subsets of $V$. Further, by Proposition \ref{V-sum}, $\displaystyle \cap_{j \in \mathbb N}W_j=\emptyset$. 
%Now, for $v \in V$ and $k = (k_1, \cdots, k_d) \in \mathbb{N}^d$, by \eqref{S-powers},
%\beqn
%S^k e_v = 
%\sum_{w \in \childnt{k}{v}} \prod_{j=1}^d \beta\big(j, \mathsf{par}_1^{\langle k_1 \rangle} \cdots \parentki{j-1}{k_{j-1}}{w}, k_j-1\big)e_w.
%\eeqn
%It follows that
%\[\mbox{ran}\,S^k \subseteq \bigvee \{e_v:v \in W_{|k|}\}=M_{|k|}.\]
%Also, if $f\in M_j$, then $f(u)=0$ for $u\in V\setminus W_j =\displaystyle \bigcup_{\underset{|k|\leqslant j-1}{k \in \mathbb{N}^d}}\childnt{k}{\rootb} 
%$. Thus, if
%$f \in \displaystyle \cap_{j=0}^{\infty} M_j$, then $f(u)=0$ for all
%$u \in \displaystyle \bigcup_{{k \in \mathbb{N}^d}}\childnt{k}{\rootb} = V$. This shows that $f=0$, and hence 
%$\cap_{j=0}^{\infty} M_j=\{0\}.$ It follows that
%\beqn \{0\} \subseteq \displaystyle \cap_{k \in \mathbb N^d}\, \mbox{ran}\, 
%S^k \subseteq \cap_{k \in \mathbb N^d}\, M_{|k|} =\{0\}. \eeqn This shows that 
%$S$ is
%analytic. 
%\end{proof}

\section{Analyticity and Point Spectrum}

A commuting $d$-tuple $T=(T_1, \cdots, T_d)$ on a Hilbert space $\mathcal H$ is called {\it analytic} if
$$\displaystyle\bigcap_{\alpha \in \mathbb N^d}\ran \, T^\alpha = \{0\}.$$

Just like the classical case, the multishifts on $\mathscr T$ are analytic. Indeed, we see that they are separately analytic in the following sense. 
\begin{proposition} \label{sep-ana}
Let $\mathscr T = (V,\mathcal E)$ be the directed Cartesian product of rooted directed trees $\mathscr T_1, \cdots, \mathscr T_d$ and let $S_{\lambdab} = (S_1, \cdots, S_d)$ be a commuting multishift on $\mathscr T$. Then for each $j = 1, \cdots, d$,  $S_j$ is analytic.   
\end{proposition}
\begin{proof}
Let $j = 1, \cdots, d$ be fixed. For $n \in \mathbb N,$ let 
%$$M_n =\bigvee \Big\{e_v : v \in V_1 \times \cdots \times V_{j-1} \times \childn{n}{V_j} \times V_{j +1} \times \cdots \times V_d \Big\},$$ 
$$M_n :=\bigvee \Big\{e_v : v \in \childnt{n\epsilon_j}{V} \Big\},$$
and note that
by \eqref{S-powers}, 
$\ran~{S_j^n} \subseteq M_n.$
It now suffices to check that $\cap_{n=0}^{\infty}M_n=\{0\}.$
To see this, note that if $f \in M_n,$ then
$
f(u)=0~\mbox{for every}~u \in V~\mbox{such~that~}u_j \in \cup_{i=0}^{n-1} \childn{i}{\mathsf{root}_j}.$ However, $\cup_{i=0}^{\infty} \childn{i}{\mathsf{root}_j}=V_j,$ and hence for any $f \in \cap_{n=0}^{\infty}M_n$, we must have $f(u)=0$ for any $u \in V.$
%$$\ran{~S_j} \subseteq \bigvee \Big\{e_v : v \in V_1 \times \cdots \times V_{j-1} \times \child{V_j} \times V_{j +1} \times \cdots \times V_d \Big\}.$$ 
%Similarly,
%$$\ran{~S_j^2} \subseteq \bigvee \Big\{e_v : v \in V_1 \times \cdots \times V_{j-1} \times \childn{2}{V_j} \times V_{j +1} \times \cdots \times V_d \Big\}.$$
%Note that $\{\childn{i}{V_j}\}_{i=1}^\infty$ is a strictly decreasing sequence with $\cap_{i=1}^\infty \childn{i}{V_j} = \emptyset$. Thus, 
%\beqn
%\ran{~S_j} \cap \ran{~S_j^2} &=& \ran{~S_j^2} \\
%&\subseteq & \bigvee \Big\{e_v : v \in V_1 \times \cdots \times V_{j-1} \times \childn{2}{V_j} \times V_{j +1} \times \cdots \times V_d \Big\}\\
%&=& \bigvee \Big\{e_v : v \in V_1 \times \cdots \times V_{j-1} \times \child{V_j} \cap \childn{2}{V_j} \times V_{j +1} \times\\
%&& \cdots \times V_d \Big\}. 
%\eeqn
%Inductively, for each $n \in \mathbb N$, we get
%$$\cap_{i=1}^n \ran~{S_j^i} \subseteq \bigvee \Big\{e_v : v \in V_1 \times \cdots \times V_{j-1} \times \cap_{i=1}^n \childn{i}{V_j} \times V_{j +1} \times \cdots \times V_d \Big\}.$$
%Hence 
%$$\cap_{i=1}^\infty \ran~{S_j^i} \subseteq \bigvee \Big\{e_v : v \in V_1 \times \cdots \times V_{j-1} \times \cap_{i=1}^\infty \childn{i}{V_j} \times V_{j +1} \times \cdots \times V_d \Big\} = \{0\}$$
%by the observation recorded above. 
This completes the proof.
\end{proof}

The next corollary generalizes \cite[Theorem 15]{JL}, where the method of proof relies on the description of the commutant of $S_j$.  
\begin{corollary}
Let $\mathscr T = (V,\mathcal E)$ be the directed Cartesian product of rooted directed trees $\mathscr T_1, \cdots, \mathscr T_d$ and let $S_{\lambdab} = (S_1, \cdots, S_d)$ be a commuting multishift on $\mathscr T$. Then for each $j = 1, \cdots, d$, the spectrum of $S_j$ equals $\mbox{cl}({\mathbb D}_{r(S_j)}),$ where $r(T)$ denotes the spectral radius of a bounded linear operator $T.$
\end{corollary}
\begin{proof}
By \cite[Lemma 5.2]{CT}, the spectrum of $S_j$ is connected. Since $S_j$ is circular (Proposition \ref{circular}) and $0 \in \sigma(S^*_j)$ (Remark \ref{0-eigen}), the spectrum of $S_j$ must be the disc $\mbox{cl}({\mathbb D}_{r(S_j)}).$
\end{proof}

\begin{corollary} \label{analytic}
Let $\mathscr T = (V,\mathcal E)$ be the directed Cartesian product of rooted directed trees $\mathscr T_1, \cdots, \mathscr T_d$. 
Then the commuting multishift $S_{\lambdab}$ on $\mathscr T$ is analytic.
\end{corollary}
\begin{proof}
Note that $\cap_{\alpha \in \mathbb N^d}\, \mbox{ran}\, 
S^{\alpha}_{\lambdab} \subseteq \cap_{k \in \mathbb N}\, \mbox{ran}\, 
S^k_j$ for any $j=1, \cdots, d.$ The desired conclusion now follows from Proposition \ref{sep-ana}. 
\end{proof}

\begin{corollary} \label{p-spectrum}
Let $\mathscr T = (V,\mathcal E)$ be the directed Cartesian product of rooted directed trees $\mathscr T_1, \cdots, \mathscr T_d$ and let $S_{\lambdab} = (S_1, \cdots, S_d)$ be a commuting multishift on $\mathscr T$. Then for each $j = 1, \cdots, d$, the point spectrum of $S_j$ is empty. In particular, the joint kernel of $S_{\lambda}$ is trivial.
\end{corollary}
\begin{proof}
By Lemma \ref{bddness}(ii), $S_j$ is injective. Also, if $S_jf=wf$ for some nonzero $w \in \mathbb C$ then $f \in \bigcap_{k \in \mathbb N} \ran \, S^k_j = \{0\},$ and hence $f=0.$ This proves that the point spectrum of $S_j$ is empty.
\end{proof}
\begin{remark}
Note that none of $S_1, \cdots, S_d$ can be normal, that is, $S^*_jS_j \neq S_jS^*_j$ for every $j =1, \cdots, d.$ In view of Remark \ref{0-eigen}, this may be deduced from the fact that for any normal operator $T$, the kernel of $T$ and the kernel of $T^*$ coincide.
\end{remark}

\begin{corollary}
Let $\mathscr T = (V,\mathcal E)$ be the directed Cartesian product of rooted directed trees $\mathscr T_1, \cdots, \mathscr T_d$ and let $S_{\lambdab} = (S_1, \cdots, S_d)$ be a commuting multishift on $\mathscr T$. Then for each $j = 1, \cdots, d$, 
\beqn
\bigvee_{k \in \mathbb N}\ker S^{*k}_j=l^2(V) =\bigvee_{\alpha \in \mathbb N^d}\ker S^{*\alpha}_{\lambdab}.
\eeqn
\end{corollary}
\begin{proof}
After taking orthogonal complement, first equality may be deduced from the analyticity of $S_j$ while the second one follows from the analyticity of $S_{\lambdab}.$
\end{proof}

\section{A Matrix Decomposition and Essential Spectrum}

In this section, we discuss a matrix decomposition
of multishifts $S_{\lambdab}$ on $\mathscr T$ (cf. \cite[Lemma 5.3]{CT}). 
The building blocks in this decomposition include classical multishifts and tuples with entries as weighted shifts on directed trees.
We will use this decomposition to relate the spectral parts of $S_{\lambdab}$ with the spectral parts of the building blocks appearing in the matrix decomposition of $S_{\lambdab}.$ 
For simplicity, we treat the case $d=2.$  
Let $\mathscr T = (V, \mathcal E)$ be the directed Cartesian product of rooted directed trees $\mathscr T_j = (V_j, \mathcal E_j)$, $j = 1, 2$.
Assume that $\mathscr T$ is locally finite with finite joint branching index $k_{\mathscr T}=(k_{\mathscr T_1}, k_{\mathscr T_2})$. Let us observe the following:
\begin{enumerate}
\item[${\bf (A)}$] Fix $v_1 \in V_1.$ If $W=\{v_1\} \times V_2$, then 
$l^2(W)$ is invariant under $S_2.$ Moreover,
$$P_{l^2(W)}S_1|_{l^2(W)}=0,$$ and $S_2|_{l^2(W)}$ is unitarily equivalent to a weighted shift on the directed tree $\mathscr T_2.$ A similar observation holds for $V_1 \times \{v_2\}$ for any $v_2 \in V_2.$ 
\item[${\bf (B)}$] Fix $j \in \{1, 2\}.$  Define ${G}_j:=\{\mathsf{root}_j\}$ if  $V^{(j)}_{\prec}=\emptyset$. Otherwise, let 
\beqn {G}_j:=\{v_j \in \child{V^{(j)}_{\prec}} : \mbox{card}(\childn{n}{v_j})=1~\mbox{for all~}n \geqslant  1\}.\eeqn
Let $v \in G_1 \times G_2.$ Then $L_{v}:=\bigsqcup_{\alpha \in \mathbb N^2} \childnt{\alpha}{v}$ is a directed graph isomorphic to $\mathscr T_{1, 0} \times \mathscr T_{1, 0}$. Note first that for any two distinct vertices $u_j, v_j \in  G_j,$ there is no positive integer $n_j$ such that $\childn{n_j}{u_j}=\{v_j\}.$ Indeed, if $\childn{n_j}{u_j}=\{v_j\}$ then $\childn{n_j-1}{u_j} \cap V^{(j)}_{\prec}=\{\parent{v_j}\},$ and hence we obtain $\mbox{card}(\childn{n_j}{u_j}) \geqslant  2,$ which is a contradiction.  
We next check that for $v, w \in  G_1 \times  G_2$ such that $v \neq w,$ $L_v \cap L_w = \emptyset.$ Without loss of generality, assume that $v_1 \neq w_1.$ Suppose that $u \in L_v \cap L_w$. Then $\childn{\alpha_1}{v_1} \cap \childn{\beta_1}{w_1}=\{u_1\}$ for some $\alpha_1, \beta_1 \in \mathbb N,$ and hence $\childn{\alpha_1}{v_1} = \childn{\beta_1}{w_1}.$ It follows that either $
\childn{\beta'_1}{w_1}=\{v_1\}$ or $
\childn{\alpha'_1}{v_1}=\{w_1\}$ for some $\alpha'_1, \beta'_1 \in \mathbb N$. This contradicts the above observation.
\item[${\bf (C)}$] Let $ W_j:=\bigcup_{n=1}^{\infty} \parentn{n}{ G_j}$ for $j=1, 2.$ Note that $ W_1,  W_2$ are finite sets.
Consider the disjoint sets
\beqn \label{F1-F2-II}
 F_1:= \bigsqcup_{w_1 \in  W_1}\{w_1\} \times V_2, \quad  F_2:=\bigsqcup_{w_2 \in  W_2}(V_1 \setminus  W_1) \times \{w_2\}.
\eeqn
\item[${\bf (D)}$] Note that $V =  F_1 \sqcup  F_2 \sqcup  F_3$, where
\beqn \label{F3-II}  F_3:= \displaystyle \bigsqcup_{v \in  G_1 \times  G_2}L_v.\eeqn
This gives the decomposition $l^2(V)=l^2( F_1) \oplus l^2( F_2) \oplus l^2( F_3)$,
where \beqn
l^2( F_1) &=& \bigoplus_{w_1 \in  W_1}\mathcal N_{w_1} \quad \mbox{and} ~ \mathcal N_{w_1}:=l^2\big(\{w_1\} \times V_2\big), \\
l^2( F_2) &=& \bigoplus_{w_2 \in  W_2}\mathcal M_{w_2} \quad \mbox{and~}\mathcal M_{w_2}:=l^2\big((V_1 \setminus  W_1)\times \{w_2\}\big), \\
l^2( F_3) &=&  \bigoplus_{v \in  G_1 \times  G_2}l^2(L_v).
\eeqn 
We now decompose $(S_1, S_2)$ with respect to the decomposition $l^2(V)=l^2( F_1) \oplus l^2( F_2) \oplus l^2( F_3).$ Indeed, $S_1=( A_{ij})_{1 \leqslant i, j \leqslant 3}$ and $S_2=( B_{ij})_{1 \leqslant i, j \leqslant 3},$ where 
\begin{enumerate}
\item $ A_{1i}=0~(i=2, 3),$ $ A_{23}=0= A_{32},$
$ B_{1i}=0~(i=2, 3),$ $ B_{2j}=0~(j=1, 3),$ $ B_{31}=0,$
\item $ A_{11}=0$ if $\child{ W_1} \cap  W_1 = \emptyset$ (if and only if $k_{\mathscr T_1} \leqslant 1$) and otherwise of infinite rank, $ B_{22}=0$ if $\child{ W_2} \cap  W_2 = \emptyset$ ((if and only if $k_{\mathscr T_2} \leqslant 1$), and otherwise of infinite rank,
\item 
$ A_{21}$ is the matrix with generic entry $P_{\mathcal M_{w_2}}S_1|_{\mathcal N_{w_1}}$ (finite rank operator), 
\item $ A_{22}$ is the diagonal matrix with generic entry $S_1|_{\mathcal M_{w_2}}$, $ B_{11}$ is the diagonal matrix with generic entry $S_2|_{\mathcal N_{w_1}}$ (one variable shifts on directed trees),
\item $ A_{33}$ is the diagonal matrix with generic entry $S_1|_{l^2(L_v)}$, 
$ B_{33}$ is the diagonal matrix with generic entry $S_2|_{l^2(L_v)}$ (entries of the classical multishift $S_{\bf w}$),
\item $ A_{31}$ is the matrix with generic entry $P_{l^2(L_v)}S_1|_{\mathcal N_{w_1}}$, $ B_{32}$ is the matrix with generic entry $P_{l^2(L_v)}S_2|_{\mathcal M_{w_2}}$ (infinite rank non-shifts).
\end{enumerate}
Since $S_1$ and $S_2$ are commuting, a plain calculation shows that
\beqn
A_{21}B_{11} &=& 0, \quad A_{31}B_{11} = B_{32}A_{21} + B_{33}A_{31}, \\ A_{33}B_{32} &=& B_{32}A_{22}, \quad A_{33}B_{33} = B_{33}A_{33}.
\eeqn
Thus the building blocks for $S_{\lambdab}$ consist of $2$-tuples of the form $(A_{11}, B_{11})$ or $(A_{22}, B_{22})$ for single variable weighted shifts $B_{11}, A_{22}$ on directed trees, commuting classical multishifts $(A_{33}, B_{33})$, finite rank $2$-tuple $(A_{21}, 0)$, and infinite rank non-shifts $(A_{31}, 0),$ $(0, B_{32})$. 
\end{enumerate}

It is worth noting that the situation in case $d=1$ is entirely different in the sense that all non-diagonal entries in the matrix decomposition of $S_{\lambda}$ are of finite rank (see \cite[Lemma 5.3]{CT}).

Before we see applications of the above decomposition, we would like to discuss convergence of nets associated with directed Cartesian product of directed trees. Let $\mathscr T_j = (V_j, \mathcal E_j)~(j=1, \cdots, d)$ be rooted directed trees and
let $\mathscr T = (V,\mathcal E)$ be the directed Cartesian product of $\mathscr T_1, \cdots, \mathscr T_d$. Define the relation $\leq$ on $V$ 
\index{$v \leq w$}
as follows:
\beqn
v \leq w ~\mbox{if}~\alpha_v \Le \alpha_w,
\eeqn
where $\alpha_v$ denotes the depth of $v$ in $\mathscr T.$
Note that $V$ is a partially ordered set with partial order relation $\leq$ (that is, $\leq$ is reflexive and transitive). Note that given two vertices $v, w \in V$, there exists $u \in V$ such that $v \leq u$ and $w \leq u.$ In this text, we will be interested in the nets $\{\lambda_v\}_{v \in V}$ of complex numbers induced by the above partial order
(the reader is referred to \cite{Ke} for the definition and elementary facts pertaining to nets). 
\begin{remark} \label{cgn-net}
One can also endow $V$ with the following partial order relations:
\begin{enumerate}
\item $v \leq w$ if $\alpha_v$ is less than or equal to $\alpha_w$ with respect to the dictionary ordering,
\item $v \leq w$ if $|\alpha_v| \Le |\alpha_w|.$ 
\end{enumerate}
Note that convergence of net in (1) is weaker than and that in (2) is stronger than the convergence defined prior to the remark. All these notions agree in case $d=1$.
\end{remark}

We now see an application of the matrix decomposition of multishifts as discussed above.

\begin{proposition} \label{spectral-i}
Let $\mathscr T = (V,\mathcal E)$ be the directed Cartesian product of locally finite rooted directed trees $\mathscr T_1, \mathscr T_2$ of finite joint branching index $k_{\mathscr T}=(k_{\mathscr T_1}, k_{\mathscr T_2})$. 
Let $S_{\lambdab}$ be the commuting multishift on $\mathscr T.$
Then we have the following statements.
\begin{enumerate}
\item[(i)] $ 
\sigma(S_{\lambdab}) \subseteq \sigma((A_{11}, B_{11})) \cup \sigma((A_{22}, B_{22})) \cup \sigma((A_{33}, B_{33})).$
\item[(ii)] $\sigma_l((A_{33}, B_{33})) \subseteq \sigma_l(S_{\lambdab}).$
\end{enumerate}
Assume further that $\max\{k_{\mathscr T_1}, k_{\mathscr T_2}\} \leqslant 1$. Then 
\begin{enumerate} 
\item[(iii)] $
\sigma(S_{\lambdab}) \subseteq \big(\{0\} \times \sigma(B_{11})\big) \cup \big(\sigma(A_{22}) \times \{0\}\big) \cup \sigma((A_{33}, B_{33})).$
\item[(iv)] If, in addition, \beq \label{cpt-c}
\lim_{u_2 \in \mathsf{Des}(v_2)}\lambda^{(1)}_{(v_1, u_2)}= 0 = \lim_{u_1 \in \mathsf{Des}(v_1)}\lambda^{(2)}_{(u_1, v_2)}
~\mbox{for all~}v \in G_1 \times G_2,
\eeq
then 
$\sigma_e(S_{\lambdab})$ is union of essential spectra of finitely many $2$-tuples of the form $(0, U_{\lambda})$ or $(U_{\lambda}, 0)$ for a weighted shift $U_{\lambda}$ on a directed tree,  and the essential spectra of finitely many commuting classical $2$-variable shifts. 
\end{enumerate}
\end{proposition}
%\begin{corollary}
%Let $\mathscr T = (V,\mathcal E)$ be the directed Cartesian product of rooted, locally finite directed trees $\mathscr T_1$ and $\mathscr T_2$. Let $S_{\lambdab}=(S_1, S_2)$ be a commuting multishift on $\mathscr T$ such that
%\beq \label{cpt-c}
%\lim_{u_2 \in \mathsf{Des}(v_2)}\lambda^{(1)}_{(v_1, u_2)}= 0 = \lim_{u_1 \in \mathsf{Des}(v_1)}\lambda^{(2)}_{(u_1, v_2)}
%~\mbox{for all~}v \in G_1 \times G_2.
%\eeq
%If $\max\{k_{\mathscr T_1}, k_{\mathscr T_2}\} \leqslant 1$, then $S_{\lambdab}$ is a commuting compact perturbation of orthogonal direct sum of finitely many $2$-tuples of the form $(0, U_{\lambda})$ or $(U_{\lambda}, 0)$ for a single variable weighted shift $U_{\lambda}$ on a directed tree,  and finitely many classical $2$-variable shifts.  
%\end{corollary}
\begin{proof}
Note that (i) and (iii) are particular consequences of part (b) of $\bf (D)$ of the previous decomposition and \cite[Lemmas 4.4 and 4.5]{Cu-00}.
To see (ii), note that if any commuting $d$-tuple $T$ on $\mathcal H$ is bounded below then so is its restriction to any joint invariant subspace $\mathcal M$ of $\mathcal H.$ Applying this fact to $T:=S_{\lambdab}-\omega~(\omega \in \mathbb C^2)$ and $\mathcal M := l^2(F_3)$ yields the conclusion in (ii).

To see the remaining part, 
assume further that \eqref{cpt-c} holds.
We first note that $S_{\lambdab}$ is a commuting compact perturbation of orthogonal direct sum of finitely many $2$-tuples of the form $(0, U_{\lambda})$ or $(U_{\lambda}, 0)$ for a single variable weighted shift $U_{\lambda}$ on a directed tree,  and finitely many commuting classical $2$-variable shifts. 
This may be drawn once we observe that $A_{31}=P_{l^2(L_v)}S_1|_{\mathcal N_{w_1}}$ and $B_{32}=P_{l^2(L_v)}S_2|_{\mathcal M_{w_2}}$) are compact for every $v \in G_1 \times G_2$. But this follows from \eqref{cpt-c}.
To complete the proof, in view of Atkinson-Curto Theorem \cite[Theorem 2]{Cu-0}, we need the fact that the essential spectrum $\sigma_e(A \oplus B)$ of orthogonal direct sum of $A$ and $B$ is union of $\sigma_e(A)$ and $\sigma_e(B)$, where $A$ and $B$ denote commuting $d$-tuples of bounded linear operators on $\mathcal H$ and $\mathcal K$ respectively. We include elementary verification of this fact.
Note that the boundary operators $\partial_{A\oplus B}$ appearing in the Koszul complex of $A \oplus B$ are orthogonal direct sum of boundary operators $\partial_A$ and $\partial_B$ appearing in the Koszul complexes of $A$ and $B$ respectively (refer to Section \ref{Spec-th}). That is, $\partial_{A \oplus B} + \partial^*_{A \oplus B} = (\partial_{A} + \partial^*_A) \oplus (\partial_{B} + \partial^*_B)$. On the other hand, by \cite[Theorem 6.2]{Cu}, a $d$-tuple $T$ is Fredholm if and only if $\partial_T + \partial^*_T$ is Fredholm. The desired conclusion is now immediate. 
\end{proof}

%We record the following basic fact for ready reference.
%\begin{lemma}
%Let $A$ and $B$ denote commuting $d$-tuples of bounded linear operators on $\mathcal H$ and $\mathcal K$ respectively. 
%Then the essential spectrum $\sigma_e(A \oplus B)$ of orthogonal direct sum of $A$ and $B$ is union of $\sigma_e(A)$ and $\sigma_e(B)$.
%\end{lemma}
%\begin{proof}
%Note that the boundary operators $D_{A\oplus B}$ appearing in the Koszul complex of $A \oplus B$ are orthogonal direct sum of boundary operators $D_A$ and $D_B$ appearing in the Koszul complexes of $A$ and $B$ respectively. That is, $D_{A \oplus B} + D^*_{A \oplus B} = (D_{A} + D^*_A) \oplus (D_{B} + D^*_B)$. On the other hand, by \cite[Theorem 6.2]{Cu}, a $d$-tuple $T$ is Fredholm if and only if $D_T + D^*_T$ is Fredholm. Now the desired conclusion follows from the same conclusion in the case $d=1.$
%\end{proof}

We illustrate the previous result with the help of an example.
\begin{example}
Let $\mathscr T=\mathscr T_{2, 0} \times \mathscr T_{1, 0}$ be as discussed in Example \ref{classical-mix}. 
Note that $G_1=\{1, 2\}$, $G_2= \{0\},$ $W_1=\{0\},$ $W_2=\emptyset,$ $F_1 = \{0\} \times V_2,$ $F_2 = \emptyset,$ $F_3= L_{(1, 0)} \cup L_{(2, 0)}$.

Let $S_{\lambdab}$ be a multishift on $\mathscr T$ with weights $\lambdab$ such that 
\beqn
\lim_{k \rar \infty} \lambda^{(1)}_{(1, k)}=  \lim_{k \rar \infty} \lambda^{(1)}_{(2, k)}=0.
\eeqn
%By the preceding proposition, $S_{\lambdab}$ is a commuting compact perturbation of $(0, U_{\lambda})$, $S_{\bf w^{(1)}}$, $S_{\bf w^{(2)}}$, where $U_{\lambda}$ is the classical $1$-variable shift with weights $\{\lambda^{(2)}_{(0, n)} : n \in \mathbb N\},$ $S_{\bf w^{(1)}}$ is the classical $2$-variable shift with weights ${\bf w^{(1)}}=\{\lambda^{(j)}_{(2k+1, l)} : k, l \in \mathbb N, j=1, 2\},$ and
%$S_{\bf w^{(2)}}$ is the classical $2$-variable shift with weights ${\bf w^{(2)}}=\{\lambda^{(j)}_{(2k, l)} : k, l \in \mathbb N, j=1, 2\}.$ 
By the above result, the essential spectrum $\sigma_e(S_{\lambdab})$ of $S_{\lambdab}$ is equal to the union of essential spectra of $(0, U_{\lambda}), S_{\bf w^{(1)}}, S_{\bf {w^{(2)}}}$. In particular, this is applicable to the commuting multishift $S_{\lambdab}=(S_1, S_2)$ with weights given by
\beqn
\lambda^{(1)}_{(m, n)} =\frac{1}{\sqrt{\mbox{card}(\mathsf{sib}_1(m, n))}} \sqrt{\frac{[\frac{m}{2}]}{[\frac{m}{2}]+n}}, ~\lambda^{(2)}_{(m, n)} = \sqrt{\frac{n}{[\frac{m}{2}]+n}},
\eeqn 
where $m, n \in  \mathbb N,$ and $$\left[\frac{m}{2}\right]=\begin{cases}\frac{m}{2}~& \mbox{if}~m~\mbox{is an even integer} \\
\frac{m+1}{2}~& \mbox{otherwise.}\end{cases}$$
Note that none of $S_1, S_2$ is compact. Further, $U_{\lambda}$ is the unilateral unweighted shift with essential spectrum the unit circle $\mathbb T$ (refer to \cite{S}). 
By spectral mapping property, the essential spectrum of $(0, U_{\lambda})$ is $\{0\} \times \mathbb T.$
Further, $S_{\bf w^{(1)}}$ is the $2$-variable classical multishift with weights
\beqn
w^{(1)}_{(2k+1, l)} = \sqrt{\frac{k+1}{k+l+1}},~w^{(2)}_{(2k-1, l+1)} = \sqrt{\frac{l+1}{k+l+1}}~(k \geqslant  1, l \geqslant  0).
\eeqn
It is easy to see that $S_{\bf w^{(1)}}$ is unitarily equivalent to the Drury-Arveson $2$-shift, and hence the essential spectrum of $S_{\bf w^{(1)}}$ equals the unit sphere $\partial \mathbb B^2$ in $\mathbb C^2$ (Proposition \ref{sp-th-classical}). Similarly, the essential spectrum of $S_{\bf w^{(2)}}$ is equal to $\partial \mathbb B^2$. It follows that 
$\sigma_e(S_{\lambdab}) = \partial \mathbb B^2.$ Also, since the left-spectrum of Drury-Arveson shift is the unit sphere  (Proposition \ref{sp-th-classical}), it may be concluded from Proposition \ref{spectral-i}(ii) that $\sigma_l(S_{\lambdab})$ contains $\partial \mathbb B^2$. 
%Let us see that 
%$\sigma_l(S_{\lambdab})=\partial \mathbb B^2$. Note that
%\beqn
%I - 2Q_{S_{\lambdab}}(I)+Q^2_{S_{\lambdab}}(I)=0.
%\eeqn
\end{example}
\begin{remark}
We will see later in Chapter 5 that $\sigma_l(S_{\lambdab})$ is contained in the unit sphere in $\mathbb C^2$ (see \eqref{sp-inclusion}).
\end{remark}

%\section{Multishifts with Nowhere Dense Taylor Spectra}

We conclude this section with a recipe to construct non-compact multishifts $S_{\lambdab}$ on $\mathscr T$ having Taylor spectra with empty interior. One such family of classical multishifts has been exhibited in \cite[Example 2]{Cu-2}. 
We now capitalize on it to construct examples of multishifts $S_{\lambdab}$ on $\mathscr T_{2, 0} \times \mathscr T_{2, 0}$ with Taylor spectra of empty interior.
\begin{example} \label{empty-i}
Let $\mathscr T=\mathscr T_{2, 0} \times \mathscr T_{2, 0}$ be as discussed in Example \ref{T1-T1}. 
Note that $G_1=\{1, 2\}=G_2,$ $W_1=\{0\}=W_2,$ $F_1 = \{0\} \times V_2,$ $F_2 = V^{\circ}_1 \times \{0\},$ $F_3= L_{(1, 1)} \cup L_{(1, 2)} \cup L_{(2, 1)} \cup L_{(2, 2)}$. By Proposition \ref{spectral-i}, we have
\beqn
\sigma(S_{\lambdab}) \subseteq \big(\{0\} \times \sigma(B_{11})\big) \cup \big(\sigma(A_{22}) \times \{0\} \big) \cup \sigma((A_{33}, B_{33})).
\eeqn
We now choose weights of $S_{\lambdab}$, so that all four classical multishift appearing in $(A_{33}, B_{33})$ are unitarily equivalent to some classical multishift with Taylor spectrum of empty interior. Since $\big(\{0\} \times \sigma(B_{11})\big) \cup \big(\sigma(A_{22}) \times \{0\} \big),$ being contained in $\big(\{0\} \times \mathbb D_r\big) \cup \big(\mathbb D_s \times \{0\}\big)$ for some $r, s >0$, has always empty interior, we conclude that $\sigma(S_{\lambdab})$ has empty interior.
%Let $S_{\lambdab}$ be a multishift on $\mathscr T$. We now choose weights $\lambdab$ so that 
%\beqn
%\lim_{k \rar \infty} \lambda^{(1)}_{(2k+1, 0)} = 0 =  \lim_{k \rar \infty} \lambda^{(2)}_{(1, k)}, ~\lim_{k \rar \infty} \lambda^{(1)}_{(2k, 0)} = 0 =  \lim_{k \rar \infty} \lambda^{(2)}_{(2, k)}.
%\eeqn
%By the preceding proposition, $S_{\lambdab}$ is a commuting compact perturbation of $(0, U_{\lambda})$, $S_{\bf w^{(1)}}$, $S_{\bf w^{(2)}}$, where $U_{\lambda}$ is the classical $1$-variable shift with weights $\{\lambda^{(2)}_{(0, n)} : n \in \mathbb N\},$ $S_{\bf w^{(1)}}$ is the classical $2$-variable shift with weights ${\bf w^{(1)}}=\{\lambda^{(j)}_{(2k+1, l)} : k, l \in \mathbb N, j=1, 2\},$ and
%$S_{\bf w^{(2)}}$ is the classical $2$-variable shift with weights ${\bf w^{(2)}}=\{\lambda^{(j)}_{(2k, l)} : k, l \in \mathbb N, j=1, 2\}.$ By Atkinson-Curto Theorem \cite[Theorem 2]{Cu}, the essential spectrum $\sigma_e(S_{\lambdab})$ of $S_{\lambdab}$ is equal to that of $(0, U_{\lambda}) \oplus S_{\bf w^{(1)}} \oplus S_{\bf {w^{(2)}}}$.  
\end{example}

%This yields the decomposition of $S_{\bf w^{(i)}}=(S^{(i)}_{1}, S^{(i)}_2)$ on $l^2(L(i, 0))$ as follows:
%\beq \label{deco}
%S^{(i)}_{1}=\left[\begin{array}{ccc}
%A^{(i)}_1 & 0 & 0       \\
%0 & 0 & 0     \\
%0 & A^{(i)}_2 & A^{(i)}_3     
%\end{array}\right] ~\mbox{on~}l^2(L_{(i, 0)}) = l^2(X_i) \oplus  l^2(Y_i) \oplus l^2(Z_i),
%\eeq
%where $A^{(i)}_1=S^{(i)}_{1}|_{l^2(X_i)}$ (a unilateral shift on $X_i$),~$A^{(i)}_2 = P_{l^2(Z_i)}S^{(i)}_1|_{l^2(Y_i)}$, $A^{(i)}_3 =  S^{(i)}_{1}|_{l^2(Z_i)}$ (first entry of classical multishift on $Z_i$).
%Similarly, we obtain
%\beq \label{deco}
%S^{(i)}_{2}=\left[\begin{array}{ccc}
%0 & 0 & 0       \\
%B^{(i)}_1 & B^{(i)}_2 & 0     \\
%B^{(i)}_3 & 0 & B^{(i)}_4     
%\end{array}\right] ~\mbox{on~}l^2(L_{(i, 0)}) = l^2(X_i) \oplus  l^2(Y_i) \oplus l^2(Z_i),
%\eeq
%where $B^{(i)}_1=P_{l^2(Y_i)}S^{(i)}_{2}|_{l^2(X_i)}$ (a rank one operator),~$B^{(i)}_2 = S^{(i)}_2|_{l^2(Y_i)}$ (a unilateral shift on $Y_i$), $B^{(i)}_3 =  P_{l^2(Z_i)}S^{(i)}_{2}|_{l^2(X_i)}$, $B^{(i)}_4 = S^{(i)}_2|_{l^2(Z_i)}$ (second entry of classical multishift on $Z_i$). 

%(0, B_11), (A_33, B_33), (A_31, 0)

\chapter{Wandering Subspace Property}

Let $T=(T_1, \cdots, T_d)$ be a commuting $d$-tuple on a Hilbert space $\mathcal H.$ A subspace $\mathcal W$ of $\mathcal H$ is said to be a {\it wandering subspace} for $T$ if $T^{\alpha}\mathcal W$ is orthogonal to $\mathcal W$ for every $\alpha \in \mathbb N^d \setminus \{0\}.$
Note that the joint kernel $E=\cap_{j=1}^d \ker T^*_j$ of $T^*$ is always a wandering subspace for $T.$ 
Following \cite[Definition 2.4]{Sh}, we say that $T$ possesses {\it wandering subspace property} 
\index{$[E]_T$}
if $\mathcal H =[E]_T$, 
where 
\beqn
[E]_T := \bigvee_{\alpha \in \mathbb N^d} T^{\alpha}E.
\eeqn  

The main result of this chapter ensures the wandering subspace property for multishift $S_{\lambdab}$ on $\mathscr T$ under some modest assumptions. Unlike the cases either of classical multishifts or of one variable weighted shifts on rooted directed trees, this fact lies deeper. 

\begin{thm} \label{wandering}
Let $\mathscr T = (V,\mathcal E)$ be the directed Cartesian product of locally finite, rooted directed trees $\mathscr T_1, \cdots, \mathscr T_d$ and let $S_{\lambdab}$ be a commuting multishift on $\mathscr T$. Then $S_{\lambdab}$ possesses wandering subspace property.
%\beq \label{wanderingeq}
%\bigvee_{\alpha \in \mathbb N^d} S^{\alpha}_{\lambdab}(E) = l^2(V).
%\eeq
\end{thm}

The proof of Theorem \ref{wandering}, as presented below, relies heavily on the analysis of the joint kernel of $S^*_{\lambdab}$ carried out in the next section.  
Before we start preparing for the proof of this result, we would like to discuss Shimorin's approach to the wandering subspace property for left invertible analytic operators.
Note that the wandering subspace property for a left invertible analytic operator $T$ on $\mathcal H$ is a simple consequence of the duality relation \beq \label{d-formula} \Big(\bigcap_{k \in \mathbb N}T'^k\mathcal H\Big)^{\perp}= [\ker T^*]_T,\eeq 
which, in turn, relies on the identity
\beq
\label{S-formula}
I - T^n{T'^{*}}^n = \sum_{k=0}^{n-1}T^k(I-TT'^{*}){T'^{*}}^k,
\eeq
where $T'=T(T^*T)^{-1}$ denotes the Cauchy dual of $T$ and $I-TT'^{*}$ is the orthogonal projection $P_{\ker T^*}$ onto $\ker T^*.$ In order not to distract the reader from the main line of development, we have relegated the discussion on some of the difficulties arising in finding a multivariable counterpart of Shimorin's approach to the Appendix.

%\chapter{The Joint Kernel of $S^*_{\lambdab}$}

\section{The Joint Kernel and a System of Linear Equations}
%\label{section5.1}

%For $u \in V$ and $j = 1, \cdots, d$, set $\mathsf{sib}_j(u) := \childi{j}{\parenti{j}{u}}$. For $W \subseteq V$, we define $\mathsf{sib}_j(W) = \displaystyle\bigcup_{u \in W} \mathsf{sib}_j(u)$.

In this section, we show that finding the joint kernel of $S^*_{\lambdab}$ is equivalent to solving  certain system of linear equations. 
This information is then used to derive wandering subspace property for $S_{\lambdab}$. 

We now introduce a framework suitable for decomposing the joint kernel of $S^*_{\lambdab}$ into smaller subspaces of $l^2(V)$. These subspaces are induced by a system of linear equations arising from the action of $S^*_{\lambdab}$.

For a set $A,$ let $\mathscr P(A)$ 
\index{$\mathscr P(A)$}
denote the collection of all subsets of $A.$ In case $A = \{1, \cdots, d\},$ we sometimes use the simpler notation $\mathscr P$ 
\index{$\mathscr P(\{1, \cdots, d\})=\mathscr P$}
for 
$\mathscr P(\{1, \cdots, d\})$.

Let $\mathscr T = (V,\mathcal E)$ be the directed Cartesian product of rooted directed trees $\mathscr T_1, \cdots, \mathscr T_d$.
Consider the set-valued function $\Phi : \mathscr P \rar \mathscr P(V)$ given by 
\index{$\Phi_F$}
$\Phi(F)=\Phi_F$, where
\beq \label{phi-F-eqn}
\Phi_F := \{v \in V : v_j \in V^{\circ}_j~\mbox{if~}j \in F,~\mbox{and}~v_j =\mathsf{root}_j~\mbox{if}~j \notin F\},~ F \in \mathscr P.
\eeq
Note that $\Phi_F \cap \Phi_G = \emptyset$ if $F \neq G.$ Further, if $v \in V$ then $v \in \Phi_F$ for $$F=\{j \in \{1, \cdots, d\} : v_j \neq \mathsf{root}_j\}.$$ This shows that \beq \label{V-phi-F} V  = \displaystyle \bigsqcup_{F \in \mathscr{P}} \Phi_F.\eeq 
Let $F:= \{i_1, \cdots, i_k \} \subseteq \{1, \cdots, d\}$ be fixed.
For $u \in \Phi_F$, define 
\index{$\mathsf{sib}_F(u)$}
\beq \label{sib}
\mathsf{sib}_F(u) := \mathsf{sib}_{i_1}  \mathsf{sib}_{i_2} \cdots \mathsf{sib}_{i_k}(u).\eeq
As a convention, we set $\mathsf{sib}_{\,\emptyset}(u)=\{u\}$ for all $u \in V.$ 
%\begin{remark}
%\label{sib-card}
%Note that 
%$\mbox{card}(\mathsf{sib}_F(u))=\prod_{j \in F}\mbox{card}(\mathsf{sib}_j(u)).$
%\end{remark}

%Observe the following: $u \in \mathsf{sib}_F(u), ~u \in \mathsf{sib}_F(v)$ if and only if $v \in \mathsf{sib} _F(u)$, and 
%$u \in \mathsf{sib}_F(v), v \in \mathsf{sib}_F(w)$ implies that $u \in \mathsf{sib} _F(w)$.
Define a relation $\sim$ on $\Phi_F$ by $u \sim v$ if $u \in \mathsf{sib}_F(v)$, and note that $\sim$ is an equivalence relation on $\Phi_F$. Moreover, for any $u \in \Phi_F,$ the equivalence class containing $u$ is precisely $\mathsf{sib}_F(u)$. 
An application of axiom of choice \cite{Hal} allows us to form a set $\Omega_F$ (to be referred to as an {\it indexing set corresponding to $F$}) 
\index{$\Omega_F$}
by picking up exactly one element from each of the equivalence classes $\mathsf{sib}_F(u)$.
Thus we have the disjoint union \beq \label{phi-F} \Phi_F = \displaystyle \bigsqcup_{u \in \Omega_F} \mathsf{sib}_F(u).\eeq
This combined with \eqref{V-phi-F} yields
\beqn \label{V-sib-F-u}
V  = \displaystyle \bigsqcup_{F \in \mathscr{P}} \bigsqcup_{u \in \Omega_F} \mathsf{sib}_F(u).
\eeqn
%This yields
%\beq \label{decom-1}
%l^2(V) = \bigoplus_{F  \in \mathscr{P}}l^2(\Phi_F).
%\eeq
As a consequence, we obtain the following decomposition of $l^2(V).$
\begin{proposition} \label{decom}
Let $\mathscr T = (V,\mathcal E)$ be the directed Cartesian product of rooted directed trees $\mathscr T_1, \cdots, \mathscr T_d$. Then
\beqn
l^2(V) = \bigoplus_{F  \in \mathscr{P}}\bigoplus_{u \in \Omega_{F}} l^{2}(\mathsf{sib}_{F}(u)),
\eeqn 
where $\Omega_F$ is the indexing set corresponding to $F$ and $\mathsf{sib}_{F}(u)$ is given by \eqref{sib}.
\end{proposition}
In order to describe the joint kernel $E$ of $S^*_{\lambdab}$, one needs to understand the subspace $l^{2}(\mathsf{sib}_{F}(u))$. Before that, let us see some definitions.
%\begin{proof}
%As a consequence,
%we have \beq \label{decom-2} l^2(\Phi_F) = \bigoplus_{u \in \Omega_{F}} l^{2}(\mathsf{sib}_{F}(u)).\eeq 
%By combining \eqref{decom-1} and \eqref{decom-2}, we have the desired decomposition. 
%\end{proof} 
\begin{definition} \label{dfn4.1.2}
For $F \in \mathscr P$ and $v = (v_1,\cdots, v_d) \in V$, let $v_F$ denote the $d$-tuple with $j^{\mbox{\tiny th}}$ 
\index{$v_F$}
coordinate given by $$ (v_F)_j := \begin{cases} 
v_j & \mbox{~if~} j \in F,  \\
\mathsf{root}_j & \mbox{~if~} j \notin F. \end{cases}$$ 
Further, for fixed $1 \leqslant i \leqslant d$ such that $i \notin F,$ and $u_i \in V_i$, we
define $v_F \vert u_i$ 
\index{$v_F \vert u_i$}
to be $(w_1, \cdots, w_d)$, where  
\beqn
w_j = \begin{cases} 
u_i & \mbox{~if~} j=i,  \\
(v_F)_j & \mbox{~otherwise}. \end{cases}
\eeqn
\end{definition}
\begin{remark} Note that $v_F$ is obtained from $v$ by replacing $j^{\mbox{\tiny th}}$ coordinate by $\mathsf{root}_j$ whenever $j \notin F.$ On the other hand, $v_F|u_i$ is obtained from $v_F$ by replacing its $i^{\mbox{\tiny th}}$ coordinate by $u_i.$ 
\end{remark}
For subsets $F, G$ of $\{1, \cdots, d\}$ such that $G \subseteq F,$ and $u \in \Phi_F,$ \index{$\mathsf{sib}_{F, G}(u)$}
we define 
\beq \label{sib-F-G-u}
\mathsf{sib}_{F, G}(u):=\{v_G : v \in \mathsf{sib}_F(u)\}. 
\eeq
\begin{remark}
Note that different vertices $v$ in $\mathsf{sib}_F(u)$ may correspond to single $v_G \in \mathsf{sib}_{F, G}(u).$
\end{remark}
%In case $G=F \setminus \{i\}$ for some $i \in F$, we use the simpler notation $\mathsf{sib}_{F, i}(u)$ in place of $\mathsf{sib}_{F, G}(u)$.

\begin{lemma} \label{sib_F}
Let $F \in \mathscr P$ and let $i \in F.$
Let $\Omega_F$ be the indexing set corresponding to $F$ and let $\mathsf{sib}_{F}(u)$ be given by \eqref{sib}.
For $u \in \Omega_{F}$ and $G:=F \setminus \{i\}$, we have the following:
\begin{enumerate}
\item[(i)] $
\mathsf{sib}_F (u) = \displaystyle \bigsqcup_{v_G \in \mathsf{sib}_{F, G}(u)} \mathsf{sib}_i (v_G|u_i).$
\item[(ii)] For all $v_G \in \mathsf{sib}_{F, G}(u)$, $\displaystyle \mbox{card}(\mathsf{sib}_i (v_G|u_i))$ is constant.
\item[(iii)] $\mbox{card}(\mathsf{sib}_{F, G}(u))= \displaystyle \prod_{j \in F, j \neq i} \mbox{card}(\mathsf{sib}_j(u))$.
\item[(iv)] $\mbox{card}(\mathsf{sib}_F(u))=\displaystyle \prod_{j \in F} \mbox{card}(\mathsf{sib}_j(u)).$
\end{enumerate}
\end{lemma}
\begin{proof}
Let $v_G, w_G \in \mathsf{sib}_{F, G}(u)$ such that $v_G \neq w_G$. Then there exists $j \in G$ such that $v_j \neq w_j$. Suppose that $\eta \in \mathsf{sib}_i (v_G|u_i) \cap \mathsf{sib}_i (w_G|u_i)$. Then $v_j = \eta_j = w_j$, which is a contradiction. Hence  $\mathsf{sib}_i (v_G|u_i) \cap \mathsf{sib}_i (w_G|u_i) = \emptyset$ if $v_G \neq w_G$. Next, observe that $\mathsf{sib}_i (v_G|u_i) \subseteq \mathsf{sib}_F (u)$ for all $v_G \in \mathsf{sib}_{F, G}(u)$. 
To see the other inclusion in (i), note that if $w \in \mathsf{sib}_F(u)$ then $w \in \mathsf{sib}_i (w_G|u_i).$ This completes the proof of the first part.
The second part follows from the fact that $\mbox{card}(\mathsf{sib}_i (v_G|u_i))=\mbox{card}(\mathsf{sib}(u_i))$ while the third part follows from \eqref{sib-F-G-u}. The last part is immediate from \eqref{sib}.
%We leave the last part to readers.
\end{proof}

%The following is immediate from Proposition \ref{decom} and the preceding lemma.
%\begin{corollary} \label{decom}
%Let $\mathscr T = (V,\mathcal E)$ be the directed Cartesian product of rooted directed trees $\mathscr T_1, \cdots, \mathscr T_d$. For $1 \leqslant i \leqslant d,$ 
%\beqn
%l^2(V) = \bigoplus_{F  \in \mathscr{P}}\bigoplus_{u \in \Omega_{F}} \bigoplus_{v_G \in \mathsf{sib}_{F, G}(u)} l^2(\mathsf{sib}_i (v_G|u_i)),
%\eeqn 
%where $\Omega_F$ is the indexing set corresponding to $F$ and $\mathsf{sib}_{F, G}(u)$ is given by \eqref{sib-F-G-u}.
%\end{corollary}

The following lemma describes the action of $S^*_{\lambdab}$ on $l^2(\mathsf{sib}_F(u)).$
\begin{lemma}
Let $\mathscr T = (V,\mathcal E)$ be the directed Cartesian product of rooted directed trees $\mathscr T_1, \cdots, \mathscr T_d$ and let $S_{\lambdab} = (S_1, \cdots, S_d)$ be a commuting multishift on $\mathscr T$. 
Let $F \in \mathscr P$, $i \in F$ and $G:= F \setminus \{i\}$.
If $u \in \Phi_{F}$,    
then for any $f \in l^2(\mathsf{sib}_F (u)),$ 
\beq \label{S*i-eq}
S_i^*(f) = \sum_{ v_{G} \in \mathsf{sib}_{F, G}(u)} \Big( \sum_{w \in \mathsf{sib}_i(v_{G} | u _i)} f(w) \lambda^{(i)}_w \Big) e_{\parenti{i}{v_{G}|u_i}}.          
\eeq         
%where the sum is independent of choice of $u_i \in  \mathsf{sib}_{F, \{i\}}(u)$. 
\end{lemma}
\begin{proof} Let $u \in \Phi_{F}$.
By (i) of the preceding lemma, we obtain the orthogonal decomposition 
\beqn \label{l-2-sib-F}
l^2(\mathsf{sib}_F (u)) = \bigoplus_{v_{G} \in \mathsf{sib}_{F, G}(u)}l^2(\mathsf{sib}_i (v_G|u_i)).
\eeqn
Let $f \in 
l^2(\mathsf{sib}_F (u)).$ Then $f = \displaystyle \sum_{v_{G} \in \mathsf{sib}_{F, G}(u)}\sum_{w \in \mathsf{sib}_i(v_G | u _i)} f(w)e_w \in l^2(V).$ It follows that
\beqn
S_i^*(f) &=& \sum_{ v_G \in \mathsf{sib}_{F, G}(u)} \sum_{w \in \mathsf{sib}_i(v_G | u _i)} f(w) \lambda^{(i)}_w e_{ \parenti{i}{w}}\\ &=& \sum_{ v_G \in \mathsf{sib}_{F, G}(u)} \Big(\sum_{w \in \mathsf{sib}_i(v_G | u _i)} f(w) \lambda^{(i)}_w \Big)e_{\parenti{i}{v_G|u_i}}, 
\eeqn
where we used the fact that $\parenti{i}{\mathsf{sib}_i(v)}=\parenti{i}{v}$ for any $v \in V$.
\end{proof}
%\begin{lemma}
%%If $v_G, w_G \in \mathsf{sib}_{F, G}(u)$ with $v_G \neq w_G$ then $$\mathsf{sib}_i(v_G | u _i) \cap \mathsf{sib}_i(w_G | u _i) =\emptyset.$$ Further, 
%For any $i \in F$ and $G=F \setminus \{i\},$
%\beqn
%\mathsf{sib}_{F}(u) = \bigsqcup_{v_G \in \mathsf{sib}_{F, G}(u)}\mathsf{sib}_i(v_G | u _i) .
%\eeqn
%\end{lemma}
%\begin{proof}
%...
%\end{proof}

%\section{System of Linear Equations}

Let $i \in F$ be fixed and let $G:=F \setminus \{i\}$.
In view of \eqref{S*i-eq},
finding solution of $S^*_i (f) =0,$ $f \in l^2(\mathsf{sib}_{F}(u))$ amounts to solve the following system of $N_{i, u, F}$ equations in $M_{i, u, F}$ unknowns:
\beq \label{system}
\boxed{\sum_{w \in \mathsf{sib}_i(v_G | u_i)} f(w) \lambda^{(i)}_w =0,~v_G \in \mathsf{sib}_{F, G}(u),}
\eeq
where, in view of Lemma \ref{sib_F}, $N_{i, u, F}, M_{i, u, F}  \in \mathbb N \cup \{\infty\}$ are given by
\beqn N_{i, u, F} &=& \mbox{card}(\mathsf{sib}_{F, G}(u))=\prod_{j \in F, j \neq i} \mbox{card}(\mathsf{sib}_j(u)), \\ M_{i, u, F} &=& \mbox{card}(\mathsf{sib}_i(v_G | u _i))N_{i, u, F}  =\mbox{card}(\mathsf{sib}_F(u))=\prod_{j \in F} \mbox{card}(\mathsf{sib}_j(u)).\eeqn
Note that $M_{u, F}:=M_{i, u, F}$ is independent of $i.$ Further, by Lemma \ref{sib_F}(i), the set of unknowns in the system \eqref{system} is equal to $\{f(w) : w \in \mathsf{sib}_F(u)\}$ for each $i \in F.$ Thus varying $i$ over $F$, we get the following system of $N(u, F):=\sum_{i \in F}N_{i, u, F}$ equations in $M_{u, F}$ number of unknowns:
\beq \label{system-main} 
\boxed{\sum_{w \in \mathsf{sib}_i(v_G | u_i)} f(w) \lambda^{(i)}_w =0,~i \in F~\mbox{and~}v_G \in \mathsf{sib}_{F, G}(u).} 
\eeq

%Let $L_{u, F}$ denote the associated linear transformation (possibly unbounded) in $l^2(\mathsf{sib}_F (u))$. 
%By Proposition \ref{decom}, the joint kernel $E$ of $S^*_{\lambdab}$ is given by
%\beq \label{j-kernel}
%E = [e_\rootb] \oplus \bigoplus_{\underset{F \neq \emptyset}{F  \in \mathscr{P}}} \bigoplus_{u \in \Omega_{F}} \ker L_{u, F}.
%\eeq 
Let $\mathcal L_{u, F}$ \index{$\mathcal L_{u, F}$} denote the linear manifold of $l^2(\mathsf{sib}_F (u))$ given by 
\beq \label{L-u-F-subspace}
\mathcal L_{u, F} :=\{f \in l^2(\mathsf{sib}_F (u)) : f \mbox{~is a solution of~}\eqref{system-main}\}.
\eeq
If $\mathscr T_1, \cdots, \mathscr T_d$ are locally finite then $\mathcal L_{u, F}$ is a subspace. In this case,
by Proposition \ref{decom}, the joint kernel $E$ of $S^*_{\lambdab}$ is given by
\beq \label{j-kernel}
E = [e_\rootb] \oplus \bigoplus_{\underset{F \neq \emptyset}{F  \in \mathscr{P}}} \bigoplus_{u \in \Omega_{F}} \mathcal L_{u, F}.
\eeq 
\begin{remark}
Let us discuss the system \eqref{system-main} in following special cases:
\begin{enumerate}
\item In case $S_{\lambdab}$ is the classical multishift $S_{\bf w}$, the system \eqref{system-main} has only trivial solution, and hence $E=[e_{0}].$
\item In case $d=1,$ $\mathscr P=\{\emptyset, \{1\}\},$ and hence the system \eqref{system-main} takes the form 
\beqn
{\sum_{w \in \mathsf{sib}(u)} f(w) \lambda_w =0~(u \in V^{\circ}).} 
\eeqn
However, linear equations associated with vertices outside $\child{V_{\prec}}$ have trivial solutions, and hence
\beqn
E  = [e_\rootb] \oplus \bigoplus_{u \in \child{V_{\prec}}} \mathcal L_{u, \{1\}}.
\eeqn
This expression should be compared with \eqref{formula-k}. 
\end{enumerate}
\end{remark}
To understand the above description of the joint kernel of $S^*_{\lambdab},$ we include a couple of instructive examples.

\begin{example} \label{j-k-mixed}
Let $\mathscr T$ be the directed Cartesian product of rooted directed trees $\mathscr T_1=\mathscr T_{2, 0}, \mathscr T_2=\mathscr T_{1, 0}$ as described in Example \ref{classical-mix}. Note that $$\mathscr P=\{\emptyset, \{1\}, \{2\}, \{1, 2\}\}.$$
It follows from \eqref{phi-F-eqn} that 
\beqn
\Phi_{\emptyset} & = & \{(0, 0)\}, \Phi_{\{1\}}=\{(i, 0) : i \geqslant  1\}, \\ \Phi_{\{2\}}& = & \{(0, j) : j \geqslant  1\}, \Phi_{\{1, 2\}}=\{(i, j) : i, j \geqslant  1\}.
\eeqn
Let us now understand $\mathsf{sib}_F(u)$ for $F \in \mathscr P$ and $u \in \Phi_F.$ 
By convention, $\mathsf{sib}_{\emptyset}((0, 0))=\{(0, 0)\}.$
Note that
\beqn
\mathsf{sib}_{\{1\}}((1, 0)) &=& \{(1, 0), (2, 0)\} = \mathsf{sib}_{\{1\}}((2, 0)),~ \mathsf{sib}_{\{1\}}((i, 0))=\{(i, 0)\}~(i \geqslant  3),\\
\mathsf{sib}_{\{2\}}((0, j)) &=& \{(0, j)\}~\mbox{for all~} j \geqslant  1,~ \mbox{and}\\
\mathsf{sib}_{\{1, 2\}}((i, j)) &=& \begin{cases} 
\{(1, j), (2, j)\} &~\mbox{if~}i \in \{1, 2\}~\mbox{and~}j \geqslant  1, \\
\{(i, j)\} &~\mbox{if~}i \geqslant  3, j \geqslant 1. \end{cases}
\eeqn
One may form $\Omega_F$ by picking up one element from each of the equivalence classes $\mathsf{sib}_F(u)$ as follows:
\beqn
&& \Omega_{\emptyset}=\{(0, 0)\}, \Omega_{\{1\}} = \{(1, 0)\}\cup\{(i, 0) : i \geqslant  3\}, \Omega_{\{2\}}=\{(0, j) : j \geqslant  1\}, \\
&& \Omega_{\{1, 2\}}=\{(1, j), (i, j) : i \geqslant  3, j \geqslant  1\}.
\eeqn
Let us calculate $\mathsf{sib}_{F, G}(u)$ for possible choices of $F, G,$ and $u \in \Omega_F.$
If $F=\{1\}$, then $G = \emptyset$. In this case, 
\beqn
\mathsf{sib}_{\{1\}, \emptyset}(1, 0)=\{(0, 0)\} = \mathsf{sib}_{\{1\}, \emptyset}(i, 0)~(i \geqslant  3).
\eeqn 
This together with \eqref{system} yields the following equations:
\beqn
f(1, 0) \lambda^{(1)}_{(1, 0)} + f(2, 0) \lambda^{(1)}_{(2, 0)} &=& 0, \\
f(i, 0) \lambda^{(1)}_{(i, 0)} &=& 0~( i \geqslant  3).
\eeqn
In case $F=\{2\},$ 
$G=\emptyset$ and $\mathsf{sib}_{\{2\}, \emptyset}(0, j)=\{(0, 0)\}$ for $j \geqslant 1,$ and hence
we obtain the equations
\beqn
f(0,j) \lambda^{(2)}_{(0,j)} = 0~(j \geqslant  1).
\eeqn
In case $F = \{1,2\}$, $G = \{1\}$ or $\{2\}$. Then for all $i \geqslant  3$ and $j \geqslant  1$,
\beqn
\mathsf{sib}_{\{1, 2\}, \{2\}}(1, j) = \{(0, j)\},~
\mathsf{sib}_{\{1, 2\}, \{1\}}(1, j) = \{(1, 0), (2, 0)\}, \\ \mathsf{sib}_{\{1, 2\}, \{2\}}(i, j) = \{(0, j)\},  
~\mathsf{sib}_{\{1, 2\}, \{1\}}(i, j) = \{(i, 0)\}.
\eeqn
This gives following equations for $i \geqslant  3$ and $j \geqslant  1$,
\beqn
f(1,j) \lambda^{(1)}_{(1,j)} + f(2,j) \lambda^{(1)}_{(2,j)} = 0,\\
f(1,j) \lambda^{(2)}_{(1,j)} = 0, ~ f(2,j) \lambda^{(2)}_{(2,j)} = 0,\\
f(i,j) \lambda^{(1)}_{(i,j)} = 0, ~ f(i,j) \lambda^{(2)}_{(i,j)} = 0.
\eeqn
Solving above, we get that $f(1,0) = \alpha \lambda^{(1)}_{(2,0)}$, $f(2,0) = -\alpha \lambda^{(1)}_{(1,0)}$ for $\alpha \in \mathbb C$, $f(i,j) = 0$ for all $i,j \geqslant  1$, $f(0,j) = 0$ for all $j \geqslant  1$ and $f(i,0) = 0$ for all $i \geqslant  3$. Thus,
\beqn
E = [e_{\rootb}] \oplus [\lambda^{(1)}_{(2,0)} e_{(1,0)} - \lambda^{(1)}_{(1,0)} e_{(2,0)}].
\eeqn
\end{example}

The situation in the preceding example resembles more like the situation occurring in dimension $d=1$ (cf. \eqref{formula-k}). We see below an example which gives an idea of the complicacies which can occur in  dimension more than $1$.
\begin{example} \label{j-k-T1T1}
Consider the directed Cartesian product $\mathscr T=(V, \mathcal E)$ of the directed tree $\mathscr T_{2, 0}$ with itself (see Example \ref{T1-T1}). Note that $$\mathscr P=\{\emptyset, \{1\}, \{2\}, \{1, 2\}\}.$$
It follows that 
\beqn
\Phi_{\emptyset} & = & \{(0, 0)\}, \Phi_{\{1\}}=\{(i, 0) : i \geqslant  1\}, \\ \Phi_{\{2\}}& = & \{(0, i) : i \geqslant  1\}, \Phi_{\{1, 2\}}=\{(i, j) : i, j \geqslant  1\}.
\eeqn
Let us now understand $\mathsf{sib}_F(u)$ for $F \in \mathscr P$ and $u \in \Phi_F.$ 
By convention, $\mathsf{sib}_{\emptyset}((0, 0))=\{(0, 0)\}.$
Note that
\beqn
\mathsf{sib}_{\{1\}}((1, 0)) = \{(1, 0), (2, 0)\} = \mathsf{sib}_{\{1\}}((2, 0)), \mathsf{sib}_{\{1\}}((i, 0))=\{(i, 0)\}~(i \geqslant  3).
\eeqn
Similarly, 
\beqn
\mathsf{sib}_{\{2\}}((0, 1)) = \{(0, 1), (0, 2)\} = \mathsf{sib}_{\{2\}}((0, 2)), \mathsf{sib}_{\{2\}}((0, i))=\{(0, i)\}~(i \geqslant  3).
\eeqn
Further,
\beqn
\mathsf{sib}_{\{1, 2\}}((i, j)) =\begin{cases} \{(1, 1), (1, 2), (2, 1), (2, 2)\}& ~\mbox{if~}i, j \in \{1, 2\}, \\ 
\{(1, j), (2, j)\}&~\mbox{if~}i \in \{1, 2\}~\mbox{and~}j \geqslant  3, \\
\{(i, 1), (i, 2)\}&~\mbox{if~}i \geqslant  3 ~\mbox{and~}j \in \{1, 2\},\\
\{(i, j)\}&~\mbox{if~}i, j \geqslant  3. \end{cases}
\eeqn
%(see Figure 4.1).

One may form $\Omega_F$ by picking up one element from each of the equivalence classes $\mathsf{sib}_F(u)$ as follows:
\beqn
&& \Omega_{\emptyset}=\{(0, 0)\}, \Omega_{\{1\}} = \{(1, 0)\}\cup\{(i, 0) : i \geqslant  3\}, \Omega_{\{2\}}=\{(0, 1)\}\cup\{(0, j) : j \geqslant  3\}, \\
&& \Omega_{\{1, 2\}}=\{(1, 1)\}\cup\{(i, 1), (1, j), (i, j) : i, j \geqslant  3\}.
\eeqn
Let us calculate $\mathsf{sib}_{F, G}(u)$ for possible choices of $F, G,$ and $u \in \Omega_F.$
If $F=\{1\}$, then $G = \emptyset$. In this case, 
\beqn
\mathsf{sib}_{\{1\}, \emptyset}(1, 0) = \{(0, 0)\}, \mathsf{sib}_{\{1\}, \emptyset}(i, 0)=\{(0, 0)\}~(i \geqslant  3).
\eeqn 
This yields the following equation:
\beqn
\sum_{w \in \mathsf{sib}_1(v_{\emptyset} |1)} f(w) \lambda^{(1)}_w &=& 0,~v_{\emptyset} \in \mathsf{sib}_{\{1\}, \emptyset}(1, 0) = \{(0, 0)\},\\
\sum_{w \in \mathsf{sib}_1(v_{\emptyset} |i)} f(w) \lambda^{(1)}_w &=& 0,~v_{\emptyset} \in \mathsf{sib}_{\{1\}, \emptyset}(i, 0) = \{(0, 0)\}~( i \geqslant  3) 
\eeqn
which is same as 
\beqn
f(1, 0) \lambda^{(1)}_{(1, 0)} + f(2, 0) \lambda^{(1)}_{(2, 0)} &=& 0, \\
f(i, 0) \lambda^{(1)}_{(i, 0)} &=& 0~( i \geqslant  3).
\eeqn
Similarly, in case $F=\{2\},$ we obtain the equations
%\beqn
%\mathsf{sib}_{\{2\}, \emptyset}(0, 1) = \{(0, 1), (0, 2)\}, \mathsf{sib}_{\{2\}, \emptyset}(0, j)=\{(0, j)\}~(j \geqslant  3).
%\eeqn 
\beqn
f(0, 1) \lambda^{(2)}_{(0, 1)} + f(0, 2) \lambda^{(2)}_{(0, 2)} &=& 0, \\
f(0, j) \lambda^{(2)}_{(0, j)} &=& 0~( j \geqslant  3).
\eeqn
In case $F=\{1, 2\}$ then either $G=\{1\}$ or $\{2\},$ and hence for $i, j \geqslant  3,$
\beqn
\mathsf{sib}_{\{1, 2\}, \{2\}}(1, 1) = 
\{(0, 1), (0, 2)\}, ~\mathsf{sib}_{\{1, 2\}, \{1\}}(1, 1) = \{(1, 0), (2, 0)\} \\
\mathsf{sib}_{\{1, 2\}, \{2\}}(i, 1) = \{(0, 1), (0, 2)\}, 
~\mathsf{sib}_{\{1, 2\}, \{1\}}(i, 1) = \{(i, 0)\}, \\
\mathsf{sib}_{\{1, 2\}, \{2\}}(1, j) = \{(0, j)\},~
\mathsf{sib}_{\{1, 2\}, \{1\}}(1, j) = \{(1, 0), (2, 0)\}, \\ \mathsf{sib}_{\{1, 2\}, \{2\}}(i, j) = \{(0, j)\},  
~\mathsf{sib}_{\{1, 2\}, \{1\}}(i, j) = \{(i, 0)\}.
\eeqn
Thus we obtain the equations for $i, j \geqslant  3,$
\beqn
\sum_{w \in \mathsf{sib}_1(v_{\{2\}} | 1)} f(w) \lambda^{(1)}_w &=& 0,~v_{\{2\}} \in \{(0, 1), (0, 2)\}, 
\\
\sum_{w \in \mathsf{sib}_2(v_{\{1\}} | 1)} f(w) \lambda^{(2)}_w &=& 0,~v_{\{1\}} \in \{(1, 0), (2, 0)\}, \\
\sum_{w \in \mathsf{sib}_1(v_{\{2\}} | i)} f(w) \lambda^{(1)}_w &=& 0,~v_{\{2\}} \in \{(0, 1), (0, 2)\}, \\
\sum_{w \in \mathsf{sib}_2(v_{\{1\}} | 1)} f(w) \lambda^{(2)}_w &=& 0,~v_{\{1\}} \in \{(i, 0)\}, \\
\sum_{w \in \mathsf{sib}_1(v_{\{2\}} | 1)} f(w) \lambda^{(1)}_w &=& 0,~v_{\{2\}} \in \{(0, j)\} \\
\sum_{w \in \mathsf{sib}_2(v_{\{1\}} | j)} f(w) \lambda^{(2)}_w &=& 0,~v_{\{1\}} \in \{(1, 0), (2, 0)\}, \\
\sum_{w \in \mathsf{sib}_1(v_{\{2\}} | i)} f(w) \lambda^{(1)}_w &=& 0,~v_{\{2\}} \in \{(0, j)\}, \\
 \sum_{w \in \mathsf{sib}_2(v_{\{1\}} | j)} f(w) \lambda^{(2)}_w &=& 0,~v_{\{1\}} \in \{(i, 0)\},
\eeqn
which is same as
\beqn
f(1, 1) \lambda^{(1)}_{(1, 1)} + f(2, 1) \lambda^{(1)}_{(2, 1)} &=& 0, ~
f(1, 2) \lambda^{(1)}_{(1, 2)} + f(2, 2) \lambda^{(1)}_{(2, 2)} = 0. \\
f(1, 1) \lambda^{(2)}_{(1, 1)} + f(1, 2) \lambda^{(2)}_{(1, 2)} &=& 0, ~
f(2, 1) \lambda^{(2)}_{(2, 1)} + f(2, 2) \lambda^{(2)}_{(2, 2)} = 0 \\
f(i, 1) \lambda^{(1)}_{(i, 1)} &=& 0, ~
f(i, 2) \lambda^{(1)}_{(i, 2)} = 0 \\
f(i,1) \lambda^{(2)}_{(i,1)} + f(i,2) \lambda^{(2)}_{(i,2)} &=& 0 \\
f(1,j) \lambda^{(1)}_{(1,j)} + f(2,j) \lambda^{(1)}_{(2,j)} &=& 0 \\
f(1, j) \lambda^{(2)}_{(1, j)} &=& 0, ~
f(2, j) \lambda^{(2)}_{(2, j)} = 0  \\
f(i, j) \lambda^{(1)}_{(i, j)} &=& 0 \\
f(i, j) \lambda^{(2)}_{(i, j)} &=& 0.
\eeqn
Let $W:=\{(i, j) \in V : i \geqslant  3~\mbox{or~}j\geqslant  3\} \cup \{(0, 0)\}.$ 
Then $f \in E \ominus [e_\rootb]$ if and only if $f \in l^2(V \setminus W)$ satisfies the following systems of equations:
\beqn
L_{(1, 0), \{1\}}[f(1, 0), f(2, 0)]^{\mathsf{T}}&=& 0, \\ L_{(0, 1), \{2\}}[f(0, 1), f(0, 2)]^{\mathsf{T}}&=& 0, \\
L_{(1, 1), \{1, 2\}}[f(1, 1), f(2, 1), f(1, 2), f(2, 2)]^{\mathsf{T}}&=& 0,
\eeqn
where 
$X^{\mathsf{T}}$ denotes the transpose of a column vector $X$. Here $L_{(1, 0), \{1\}} = [\lambda^{(1)}_{(1, 0)}, \lambda^{(1)}_{(2, 0)}]$,
$L_{(0, 1), \{2\}} = [\lambda^{(2)}_{(0, 1)}, \lambda^{(2)}_{(0, 2)}]$, and 
\beqn 
L_{(1, 1), \{1, 2\}}=\left[\begin{array}{cccc}
\lambda^{(1)}_{(1, 1)} & \lambda^{(1)}_{(2, 1)} & 0 & 0      \\ \\
0 & 0 & \lambda^{(1)}_{(1, 2)} & \lambda^{(1)}_{(2, 2)}   \\  \\
\lambda^{(2)}_{(1, 1)} & 0 & \lambda^{(2)}_{(1, 2)} & 0   \\  \\
0 & \lambda^{(2)}_{(2, 1)} & 0 & \lambda^{(2)}_{(2, 2)}  
\end{array}\right].
\eeqn
Note that the rank of $L_{(1, 1), \{1, 2\}}$ is at least $3.$
By Schur's formula \cite[Theorem 1.1]{Z}, the determinant of $L_{(1, 1), \{1, 2\}}$ is zero if and only if $$\lambda^{(1)}_{(1, 2)}\lambda^{(1)}_{(2, 1)}\lambda^{(2)}_{(1, 1)}\lambda^{(2)}_{(2, 2)} = \lambda^{(1)}_{(1, 1)}\lambda^{(1)}_{(2, 2)}\lambda^{(2)}_{(1, 2)}\lambda^{(2)}_{(2, 1)}.$$
Thus any $f \in E $ takes the form
\beqn f &=& f(0, 0)e_{(0, 0)} + \sum_{v \in V \setminus W}f(v)e_v \\
&=& f(0, 0)e_{(0, 0)} + f(1, 0)(e_{(1, 0)} + a_1 e_{(2, 0)}) \\ &+& f(0, 1)(e_{(0, 1)} + a_2 e_{(0, 2)}) + g_{(1, 1)},
\eeqn 
where $g_{(1, 1)}$ is given by
\beqn g_{(1, 1)} = \begin{cases} 
0 & \mbox{if}~\mbox{rank}\,L_{(1, 1), \{1, 2\}}=4, \\  f(1, 1)(e_{(1, 1)} + b_1 e_{(2, 1)} + b_2 e_{(1, 2)}  + b_3 e_{(2, 2)}) & ~\mbox{otherwise.} \end{cases}
\eeqn
Further,  the scalars $a_1, a_2, b_1, b_2, b_3$ are given by
\beqn
a_1 &=& -\frac{\lambda^{(1)}_{(1, 0)}}{\lambda^{(1)}_{(2, 0)}}, \quad
a_2 = -\frac{\lambda^{(2)}_{(0, 1)}}{\lambda^{(2)}_{(0, 2)}}, \\
b_1 &=& -\frac{\lambda^{(1)}_{(1, 1)} }{\lambda^{(1)}_{(2, 1)}}, \quad b_2=-\frac{\lambda^{(2)}_{(1, 1)}}{\lambda^{(2)}_{(1, 2)}}, \quad
b_3=\frac{\lambda^{(2)}_{(2, 1)}}{\lambda^{(2)}_{(2, 2)}}\frac{\lambda^{(1)}_{(1, 1)} }{\lambda^{(1)}_{(2, 1)}}.
\eeqn
%for some scalars $a_1, a_2, b_1, b_2, b_3$ in $\mathbb C$. 
%These scalars are governed by
%\beqn
%\lambda^{(1)}_{(1, 0)} + a_1\lambda^{(1)}_{(2, 0)}=0, \quad
%\lambda^{(2)}_{(0, 1)} + a_2\lambda^{(2)}_{(0, 2)}=0
%\eeqn
%\uwam{specify conditions on $a_i$ and $b_i$'s}
Thus the dimension of $E$ is either $3$ or $4.$
\end{example}

%\section{Wandering Subspace Property}
%\section{Main Theorem}

As a key step in the proof of Theorem \ref{wandering}, we need to understand the orthogonal complement of the subspace $\mathcal L_{u, F}$ of $l^2(\mathsf{sib}_{F}(u))$. 
%It would be helpful to the reader to invoke all notations introduced in the previous section.
\begin{lemma} \label{kernel-perp}
Let $\mathscr T = (V,\mathcal E)$ be the directed Cartesian product of locally finite, rooted directed trees $\mathscr T_1, \cdots, \mathscr T_d$ and let $S_{\lambdab} = (S_1, \cdots, S_d)$ be a commuting multishift on $\mathscr T$. 
Let $\mathcal L_{u, F}$ be the subspace as given in \eqref{L-u-F-subspace}. Then
\beqn
l^2(\mathsf{sib}_{F}(u)) &\ominus & \mathcal L_{u, F} \\ &=& \bigvee \{S_i e_{v_{G}|\parent{u_i}} : v_{G} \in \mathsf{sib}_{F, G}(u)~\mbox{with}~G=F \setminus \{i\}, i\in F\}.
\eeqn
\end{lemma}
\begin{proof} For $v_{G} \in \mathsf{sib}_{F, G}(u)$, 
note that 
$$S_i e_{v_{G}|\parent{u_i}}=\sum_{w \in \childi{i}{v_G | \parent{u_i}}} \lambda^{(i)}_w e_w = \sum_{w \in \mathsf{sib}_i(v_G | u_i)} \lambda^{(i)}_w e_w.$$
Thus, for $f \in l^2(\mathsf{sib}_{F}(u)),$ by Lemma \ref{sib_F}(i), 
\beqn \inp{f}{S_i e_{v_{G}|\parent{u_i}}} &=& \Big\langle {\sum_{\eta_G \in \mathsf{sib}_{F, G}(u)} \sum_{w \in \mathsf{sib}_i(\eta_G | u_i)} f(w) e_w},~{\sum_{w \in \mathsf{sib}_i(v_G | u_i)} \lambda^{(i)}_w e_w}\Big \rangle \\ &=& \sum_{w \in \mathsf{sib}_i(v_G | u_i)} f(w)\lambda^{(i)}_w.\eeqn In particular,
$f \in l^2(\mathsf{sib}_{F}(u))$ is orthogonal to $S_i e_{v_{G}|\parent{u_i}}$ for every $v_{G} \in \mathsf{sib}_{F, G}(u)$ and $i \in F$ if and only if
$f$ satisfies the system \eqref{system-main}. The latter one holds if and only if $f \in \mathcal L_{u, F}.$ This yields the desired formula.
%\beqn
%\bigvee \{S_i e_{(v_{G}|\parent{u_i})} : v_{G} \in \mathsf{sib}_{F, G}(u)~\mbox{with}~G=F \setminus \{i\}, i\in F\} = (\ker L_{u, F})^{\perp}.
%\eeqn
%This completes the proof.
\end{proof}

%Note that $f \in l^2(\mathsf{sib}_{F}(u))$ can be written as
%$$f=\sum_{v_G \in \mathsf{sib}_{F, G}(u)} \sum_{w_i \in \mathsf{sib}(u_i)} f((v_G|w_i)) e_{(v_G|w_i)}.$$ 
%Define now a linear transformation $L_{u, F} : l^2(\mathsf{sib}_{F}(u)) \rar l^2(\mathsf{sib}_{F}(u))$ by
%\beqn
%L_{u, F}(e_{(v_G|w_i)})=
%%\sum_{v_G \in \mathsf{sib}_{F, G}(u)} \sum_{w_i \in \mathsf{sib}(u_i)} f((v_G|w_i)) \lambda^{(i)}_{(v_G|w_i)} e_{(v_G|w_i)}.
%\sum_{v_i \in \mathsf{sib}(w_i)} \lambda^{(i)}_{(v_G|v_i)} e_{(v_G|v_i)}
%\eeqn 
%%for $v_G \in \mathsf{sib}_{F, G}(u)$ and $w_i \in \mathsf{sib}(u_i).$
%\begin{lemma}
%Let $f \in l^2(\mathsf{sib}_{F}(u)).$ Then $f$ is a solution of the system \eqref{system-main} if and only if $f$ is orthogonal to the range of $L_{u, F}.$
%\end{lemma}
%\begin{proof}
%
%It follows that $f$ is a solution of \eqref{system} if and only if
%\beqn 
%\sum_{w_i \in \mathsf{sib}(u_i)} f((v_G|w_i)) \lambda^{(i)}_{(v_G|w_i)} =0,~v_G \in \mathsf{sib}_{F, G}(u),
%\eeqn
%which is same as 
%\beqn 
%\inp{f}{L_{u, F}(e_{(v_G|w_i)})} &=& \Big \langle \sum_{w_i \in \mathsf{sib}(u_i)} f((v_G|w_i)) e_{(v_G|w_i)}, \sum_{w_i \in \mathsf{sib}(u_i)} \lambda^{(i)}_{(v_G|w_i)}e_{(v_G|w_i)} \Big \rangle \\ &=& 0,~v_G \in \mathsf{sib}_{F, G}(u).
%\eeqn
%This completes the proof.
%\end{proof}

We are now ready to complete the derivation of the wandering subspace property for $S_{\lambdab}$. 
\begin{proof}[Proof of Theorem \ref{wandering}]
Let $E$ denote the joint kernel of $S^*_{\lambdab}$.
Since $e_{\rootb} \in E,$
it is enough to see that for every nonempty $F \subseteq \{1, \cdots, d\},$ $$l^2(V_F) \subseteq [E]_{S_{\lambdab}}=\bigvee_{\alpha \in \mathbb N^d}\{S^{\alpha}_{\lambda}f : f \in E\},$$ where $V_F = \bigsqcup_{G \in \mathscr P(F)}\Phi_G.$ Fix a nonempty subset $F$ of $\{1, \cdots, d\}$. For $l = 0, \cdots, \mbox{card}(F),$ set
$$\mathscr F_l := \{ e_v \in l^2(V) : v \in \Phi_G~\mbox{with~} G \in \mathscr P(F)~\mbox{and~}\mbox{card}(G) \leqslant l\}.$$
Since $\bigvee \mathscr F_{\tiny \mbox{card}(F)} = l^2(V_F),$
it suffices to check that $$\mathscr F_{l-1} \subseteq [E]_{S_{\lambdab}} \Longrightarrow \mathscr F_{l} \subseteq [E]_{S_{\lambdab}},~ l=1, \cdots, \mbox{card}(F).$$ 
To this end, fix $1 \leqslant l \leqslant \mbox{card}(F)$, and assume that  
$\mathscr F_{l-1} \subseteq [E]_{S_{\lambdab}}$. Let $G \in \mathscr P(F)$ with $\mbox{card}(G)=l$. In particular, \beq \label{G-i} e_{v_{G \setminus \{i\}}} \in [E]_{S_{\lambdab}}~\mbox{for every}~i \in G ~\mbox{and~}v \in \Phi_G.\eeq
We must check that $e_v \in [E]_{S_{\lambdab}}$ for every $v \in \Phi_G.$

We prove by induction on $k \in \mathbb N$ the following statement: For every $i \in G$ and $v \in \Phi_G$, 
$e_{v_{G \setminus \{i\}}|w_i} \in [E]_{S_{\lambdab}}$ 
for all $w_i \in \childn{k}{\mathsf{root}_i}$. 
In view of \eqref{G-i}, this statement holds trivially for $k=0$ since $v_{G \setminus \{i\}}|\mathsf{root}_i=v_{G \setminus \{i\}}.$ Let us assume the inductive statement for an integer $k \geqslant  0$ and let $w_i \in \childn{k+1}{\mathsf{root}_i}$. 
By induction hypothesis, $e_{v_{G \setminus \{i\}|\parent{w_i}}} \in [E]_{S_{\lambdab}}$ for every $i \in G$ and $v \in \Phi_G$. It follows from Lemma \ref{kernel-perp} with $u:= v_{G \setminus \{i\}}|{w_i}$ that
\beqn
l^2(\mathsf{sib}_G(u)) \ominus \mathcal L_{u, G}  &=&
\bigvee \Big\{S_i e_{v_{G \setminus \{i\}}|\parent{w_i}} : v_{G \setminus \{i\}} \in \mathsf{sib}_{G, G \setminus \{i\}}(u), i \in G \Big\} \\ &\subseteq & [E]_{S_{\lambdab}}.
\eeqn
%\beqn
%l^2(\mathsf{sib}_{G}(u)) \ominus \ker L_{u, G} = \bigvee \{S_i e_{(v_{G \setminus \{i\}}|\parent{u_i})} : v_{G \setminus \{i\}} \in \mathsf{sib}_{G, G \setminus \{i\}}(u),~ i\in G\}.
%\eeqn
But we already know that $\mathcal L_{u, G} \subseteq E,$ and hence 
\beqn
e_{v_{G \setminus \{i\}}|w_i} \in l^2(\mathsf{sib}_G(u)) = \mathcal L_{u, G} \oplus (\mathcal L_{u, G})^{\perp} \subseteq [E]_{S_{\lambdab}}.
\eeqn 
This completes the proof of induction on $k \in \mathbb N.$

To complete the proof, let $v \in \Phi_G.$ 
%By \eqref{phi-F}, there exists a unique $u \in \Omega_G$ such that $v \in \mathsf{sib}_G(u)$. 
Thus $v=(v_1, \cdots, v_d)$ with $v_j \in V^{\circ}_j$ for $j \in G$ and $v_j=\mathsf{root}_j$ for $j \notin G.$ 
Since $v_i \in \childn{\alpha_{v_i}}{\mathsf{root}_i}$, we obtain
$e_{v}=e_{v_{G \setminus \{i\}}|v_i} \in [E]_{S_{\lambdab}}$. This completes the proof of the theorem.
\end{proof}

In the remaining part of this section, we discuss some immediate consequences of Theorem \ref{wandering}.

Let $T=(T_1, \cdots, T_d)$ be a commuting $d$-tuple on
a Hilbert space $\mathcal H.$ A subspace $\mathcal M$ of
$\mathcal H$ is said to be {\it cyclic} for $T$ if \beqn \mathcal H
= \bigvee \{T^\alpha h : h \in \mathcal M,~ \alpha \in \mathbb N^d\}. \eeqn We
say that $T$ is {\it finitely multicyclic} if there exists a finite dimensional 
cyclic subspace for $T.$

\begin{corollary} \label{f-cyclic}
Let $\mathscr T = (V,\mathcal E)$ be the directed Cartesian product of locally finite, rooted directed trees $\mathscr T_1, \cdots, \mathscr T_d$ and let $S_{\lambdab}$ be a commuting multishift on $\mathscr T$. 
If $\mathscr T$ has finite joint branching index, 
then $S_{\lambdab}$ is finitely multicyclic with cyclic subspace the joint kernel $E$ of $S^*_{\lambdab}$.
In this case, $\dim{\ker({S^*_{\lambda}-\omega})} \leqslant \dim E$ for every $\omega \in \mathbb C^d.$ 
\end{corollary}
\begin{proof}
Assume that $\mathscr T$ has finite joint branching index.
The first part follows from Theorem \ref{wandering} and Corollary \ref{dimE-bound}. The idea of proof of second part seems to be known (see, for instance, \cite{He} for the case $d=1$).  Assume that $\mathscr T$ has finite joint branching index and let $\omega \in \mathbb C^d$.  
By Theorem \ref{wandering}, \beqn 
\bigvee_{\alpha \in \mathbb N^d} S^{\alpha}_{\lambdab}(E) = l^2(V).
\eeqn
Since $S^{\alpha}_{\lambdab}$ is a finite linear combination of terms of the form $(S_{\lambdab} - \omega)^{\beta}$ for $\beta \in \mathbb N^d$, we must have
\beqn 
\bigvee_{\alpha \in \mathbb N^d} (S_{\lambdab}-\omega)^{\alpha}(E) = l^2(V).
\eeqn
If $P_{\omega}$ denotes the orthogonal projection of $l^2(V)$ onto $\ker(S^*_{\lambdab}-\omega)$, then $P_{\omega}(S_j-\overline{\omega}_j)=0$ for any $j=1, \cdots, d.$ It follows that
\beqn
\ker(S^*_{\lambdab}-\omega) &=& P_{\omega}l^2(V) = P_{\omega} \bigvee_{\alpha \in \mathbb N^d} (S_{\lambdab}-\overline{\omega})^{\alpha}(E)  \\ &=& P_{\omega}E + P_{\omega} \Big( \bigvee_{\underset{\alpha \neq 0}{\alpha \in \mathbb N^d}} (S_{\lambdab}-\overline{\omega})^{\alpha}(E) \Big) \\ &=& P_{\omega}E +
\bigvee_{j=1}^d \bigvee_{\underset{\alpha_j \neq 0}{\alpha \in \mathbb N^d}} P_{\omega}(S_j-\overline{\omega}_j)^{\alpha_j}
 (S_{\lambdab}-\overline{\omega})^{\alpha-\alpha_j\epsilon_j}(E) = P_{\omega}E,
\eeqn
since $P_{\omega}(S_j-\overline{\omega}_j)^{k}=0$ for $k \neq 0$ and $j=1, \cdots, d$. Hence the dimension of $\ker(S^*_{\lambdab}-\omega)$ is at most $\dim E.$
\end{proof}

%\chapter{Analytic Model}

Recall that $l^2_{\mathcal M}(\mathbb N^d)$ is defined as the Hilbert space of square-summable multisequence $\{h_{\alpha}\}_{\alpha \in \mathbb N^d}$ in $\mathcal M$, where $\mathcal M$ is a nonzero complex
Hilbert space. If $\{W^{(j)}_\alpha\}_{\alpha \in \mathbb N^d} \subseteq
B(\mathcal M)$ for $j=1, \cdots, d$, then the linear operator $W_j$
in $l^2_{\mathcal M}(\mathbb N^d)$ is defined by $W_j(h_\alpha)_{\alpha \in \mathbb N^d}  =(k_{\alpha})_{\alpha \in \mathbb N^d}$ for $(h_\alpha)_{\alpha \in \mathbb N^d} \in \mathcal D,$ where 
\beqn
k_{\alpha} = \begin{cases}  W^{(j)}_{\alpha - \epsilon_j}h_{\alpha - \epsilon_j} & \mbox{if~}\alpha_j \geqslant 1, \\
0 & \mbox{if~}\alpha_j=0 \end{cases}
\eeqn
and $\mathcal D:=\{(h_\alpha)_{\alpha \in \mathbb N^d} \in l^2_{\mathcal M}(\mathbb N^d) : (k_\alpha)_{\alpha \in \mathbb N^d} \in  l^2_{\mathcal M}(\mathbb N^d)\}$. 
%   \begin{align} \label{opvuws}
%W_j(h_\alpha)_{\alpha \in \mathbb N^d}  = {(\hspace{-4mm}\underbrace{W^{(j)}_{\alpha}h_{\alpha}}_{(\alpha + \epsilon_j)\mbox{th}~\mbox{place}}\hspace{-4mm})_{{\alpha \in \mathbb N^d}}},
%\quad \oplus_{\alpha \in \mathbb N^d}h_\alpha \in l^2_{\mathcal M}(\mathbb N^d).
%   \end{align}
If we use the convention that $W^{(j)}_{\alpha}=0=h_{\alpha}$ whenever $\alpha_j < 0$, then the definition of $W_j$ can be rewritten as $W_j(h_\alpha)_{\alpha \in \mathbb N^d}  =(W^{(j)}_{\alpha - \epsilon_j}h_{\alpha - \epsilon_j} )_{\alpha \in \mathbb N^d}$. 
We refer to the $d$-tuple $W=(W_1, \cdots, W_d)$ as an {\em operator valued multishift with operator weights $\{W^{(j)}_\alpha: {\alpha \in \mathbb N^d}, j=1, \cdots, d\}$} (for one variable counter-part of operator valued multishift, the reader is referred to \cite{L}). Note that $W$ is the classical multishift $S_{\bf w}$ in case $\mathcal M$ is the one dimensional complex Hilbert space.
%Putting $\mathcal M=\C$, we arrive at the well-known
%notion of a unilateral weighted shift in $\ell^2$. Let
%$W$ be as in \eqref{opvuws}. Given integers $m\Ge n
%\Ge 0$, we set (cf.\ \cite[p.\ 409]{Ja-3})
%   \begin{align} \label{wmndef}
%W_{m,n} =
%   \begin{cases}
%W_{m-1} \cdots W_n & \text{if } m>n,
%   \\[.5ex]
%I & \text{if } m=n.
%   \end{cases}
%   \end{align}
%It follows from \eqref{opvuws} that
%   \begin{align} \label{aopws}
%W^*(h_0, h_1, \ldots) &= (W_0^*h_1, W_1^*h_2, \ldots),
%\quad (h_0, h_1, \ldots) \in \ell^2_{\mathcal M},
%   \\   \label{aopws2}
%W^*W(h_0, h_1, \ldots) &= (W_0^*W_0h_0, W_1^*W_1h_1,
%\ldots), \quad (h_0, h_1, \ldots) \in \ell^2_{\mathcal
%M}.
%   \end{align}

\begin{corollary}
Let $\mathscr T = (V,\mathcal E)$ be the directed Cartesian product of locally finite, rooted directed trees $\mathscr T_1, \cdots, \mathscr T_d$.
Let $S_{\lambdab}$ be a toral left invertible multishift on $\mathscr T$ and let $E$ denote the joint kernel of $S^*_{\lambdab}$.
%Let $U_{\theta}$ be the unitary as given by \eqref{U-theta} in the proof of Theorem \ref{circular}. If $U_{\theta}f=f$ %for every $f \in E$ and every $\theta \in \mathbb R^d$ 
If the multisequence $\{S^{\alpha}_{\lambdab}E\}_{\alpha \in \mathbb N^d}$ of linear manifolds of $l^2(V)$ is mutually orthogonal, then $S_{\lambdab}$ is unitarily equivalent to a commuting operator valued multishift $W$ on $l^2_{E}(\mathbb N^d)$ with invertible weights. 
\end{corollary}
\begin{proof} 
%Suppose that $U_{\theta}f=f$ for every $f \in E$ and every $\theta \in \mathbb R^d$. By Theorem \ref{circular},  $S_{j}U_{\theta}=\exp(i \theta_j)U_{\theta} S_j$ for $j=1, \cdots, d$ and every $\theta \in \mathbb R^d.$ Let $f, g \in E$ and $\alpha, \beta \in \mathbb N^d.$ Then
%\beqn
%\inp{S^{\alpha}_{\lambdab}f}{S^{\beta}_{\lambdab}g} = \inp{U_{\theta}S^{\alpha}_{\lambdab}f}{U_{\theta}S^{\beta}_{\lambdab}g}  =
%%\exp(i (\alpha - \beta)\cdot \theta)
%%\inp{S^{\alpha}_{\lambdab}U_{\theta}f}{S^{\beta}_{\lambdab}U_{\theta}g} \\ &=& 
%\exp(i (\alpha - \beta)\cdot \theta)
%\inp{S^{\alpha}_{\lambdab}f}{S^{\beta}_{\lambdab}g},
%\eeqn 
%which implies that $\inp{S^{\alpha}_{\lambdab}f}{S^{\beta}_{\lambdab}g}  = 0$ whenever $\alpha \neq \beta.$ 
%This shows that
The proof relies on the technique employed in \cite[Theorem 3.3]{ACJS}.
Assume that the multisequence $\{S^{\alpha}_{\lambdab}E\}_{\alpha \in \mathbb N^d}$ is mutually orthogonal.
Since the
operator $S_{\lambdab}$ is toral left invertible, $S^{\alpha}_{\lambdab}$ is left invertible for
every $\alpha \in \mathbb N^d$. Hence
we deduce that $\mathcal M_{\alpha}:=S_{\lambdab}^\alpha(E)~( \alpha \in \mathbb N^d)$ is
a subspace of $l^2(V)$ and $\dim\, \mathcal M_{\alpha} = \dim\, E$ for every $\alpha \in \mathbb N^d.$
%Hence for every
%$\alpha \in \mathbb N^d$, the Hilbert spaces $\mathcal %M_{\alpha}$ and
%$\mathcal M_0$ are isometrically isomorphic. 
For $\alpha \in
\mathbb N^d$, let $U_{\alpha} \colon \mathcal M_\alpha \to E$ be
any isometric isomorphism. 
By Theorem \ref{wandering},
$\displaystyle l^2(V) =
\bigoplus_{\alpha \in \mathbb N^d} S_{\lambdab}^\alpha(E)=\bigoplus_{\alpha \in \mathbb N^d} \mathcal M_\alpha$.  
We can now define
the isometric isomorphism $U :  l^2(V) \to
l^2_{E}(\mathbb N^d)$ by
   \begin{align*}
U(\oplus_{\alpha \in \mathbb N^d} h_\alpha) := (U_{\alpha}h_{\alpha})_{\alpha \in \mathbb N^d},~  \oplus_{\alpha \in \mathbb N^d} h_\alpha \in
l^2(V).
  \end{align*}
Consider the operator
valued multishift $W$ on $l^2_{\mathcal M_0}(\mathbb N^d)$ with weights
$\{W^{(j)}_{\alpha}:=U_{\alpha + \epsilon_j} S_j|_{\mathcal M_{\alpha}} U_{\alpha}^{-1}\}_{\alpha \in \mathbb N^d}
\subseteq B(E)$ for $j=1, \cdots, d$. 
%Then
%   \begin{align*}
%VS_{j}(\oplus_{\alpha \in \mathbb N^d} h_\alpha) & = 
%V(\oplus_{\alpha \in \mathbb N^d} \hspace{-3mm} \underbrace{S_j|_{\mathcal M_{\alpha}}h_\alpha}_{(\alpha + \epsilon_j)\mbox{th}~\mbox{place}}\hspace{-3mm}) =
% (\underbrace{V_{\alpha + \epsilon_j}S_j|_{\mathcal M_{\alpha}}h_\alpha}_{(\alpha + \epsilon_j)\mbox{th}~\mbox{place}})_{\alpha \in \mathbb N^d} \\
%& = 
% (\hspace{-2mm} \underbrace{W^{(j)}_{\alpha}V_{\alpha}h_\alpha}_{(\alpha + \epsilon_j)\mbox{th}~\mbox{place}}\hspace{-2mm})_{\alpha \in \mathbb N^d} = W_jV(\oplus_{\alpha \in \mathbb N^d} h_\alpha)
%   \end{align*}
Then for $\oplus_{\alpha \in \mathbb N^d} h_\alpha \in l^2(V)$ with at most finitely many nonzero terms $h_{\alpha}$,
   \begin{align*}
US_{j}(\oplus_{\alpha \in \mathbb N^d} h_\alpha) & = 
U(\oplus_{\alpha \in \mathbb N^d} {S_j|_{\mathcal M_{\alpha}}h_\alpha}) =
 ({U_{\alpha}}S_j|_{\mathcal M_{\alpha- \epsilon_j}}h_{\alpha- \epsilon_j})_{\alpha \in \mathbb N^d} \\
& = 
 ({W^{(j)}_{\alpha -  \epsilon_j}U_{\alpha -  \epsilon_j}h_{\alpha -  \epsilon_j}})_{\alpha \in \mathbb N^d} = W_jU(\oplus_{\alpha \in \mathbb N^d} h_\alpha).
   \end{align*}
   This shows that $US_jU^*$ agrees with $W_j$ on a dense linear manifold of $l^2_{E}(\mathbb N^d)$, and hence $W_j$ must be a bounded linear operator on $l^2_{E}(\mathbb N^d)$.
  Since $S_{\lambdab}$ is commuting, so is $W.$ Finally, since $S_{\lambdab}$ is toral left-invertible, each $W^{(j)}_{\alpha}$ is invertible in $B(E)$.
\end{proof}
\begin{remark}
The converse of the above result is trivially true. 
We note further that in case $E$ is finite dimensional, the conclusion of the corollary holds without the assumption of toral left invertibility of $S_{\lambdab}$. A slight modification of the argument above together with the injectivity of $S_1, \cdots, S_d$, as ensured by Corollary \ref{p-spectrum}, yields the desired conclusion.
%If one takes $d=1$ then for any directed tree with a branch, there exists $f \in E$ such that $U_{\theta}f \neq f$ for any $\theta \in \mathbb R \setminus \{0\}.$ Indeed, for any element in $E$ of the form $\alpha e_v + \beta e_w$ (so that $v, w$ have same parent), by \eqref{U-theta},
%$U_{\theta}(\alpha e_v + \beta e_w)=\alpha e_v + \beta e_w~\mbox{for some}~\theta \in \mathbb R \setminus \{0\}$ if and only if $\alpha = 0 = \beta.$ 
\end{remark}

%\begin{proposition}\label{wandering}
%Let $\mathscr T_i = (V_i,\mathcal E_i)$, $1 \leqslant i \leqslant n$, be a collection of rooted directed trees and  $\mathscr T = (V,\mathcal E)$ be the directed Cartesian product  of $\mathscr T_i$'s. Let $S=(S_1, \cdots, S_d)$ be a commuting multishift on $\mathscr T$ and $E = \cap_{i=1}^n \ker S_i^*$ be the joint kernel  of $S$. Then 
%\beq \label{wanderingeq}
%\bigvee_{K \geqslant  0} S^K(E) = l^2(V).
%\eeq
%\end{proposition}
%
%\begin{proof}
%???
%\end{proof}

\section{Multishifts admitting Shimorin's Analytic Model}
In this section, we discuss a subclass of multishifts $S_{\lambdab}$ on $\mathscr T$ satisfying a kernel condition (cf. \cite[Equation (5.3)]{CC}). In particular, we obtain an analytic model for this class and discuss its applications to multivariable spectral theory. 
It turns out that Shimorin's analytic model \cite{Sh} does not naturally extend to all toral left invertible multishifts. Indeed, the toral left invertible multishifts admitting Shimorin's model must belong to the aforementioned class (see Remark \ref{k-c-rem}).
Before we introduce this class, we need a lemma pertaining to toral Cauchy dual of multishifts.

\begin{lemma}\label{Cauchy-dual}
Let $\mathscr T = (V,\mathcal E)$ be the directed Cartesian product of rooted directed trees $\mathscr T_1, \cdots, \mathscr T_d$. 
Let $S_{\lambdab}=(S_1, \cdots, S_d)$ be a toral left invertible multishift on $\mathscr T$. Then the toral Cauchy dual $S^{\mf t}_{\lambdab} = (S_1^{\mf t}, \cdots, S_d^{\mf t})$ of $S_{\lambdab}$ is given by 
$$S_j^{\mf t} e_v = \frac{1}{\|S_j e_v\|^2} \sum_{w \in \childi{j}{v}} \lambda_w^{(j)} e_w\ \text{for all}\ v \in V~\mbox{and}~j=1, \cdots, d.$$
In particular, $S^{\mf t}_{\lambdab}$ is a multishift with weights 
\beqn
\Big\{\frac{\lambda^{(j)}_w}{\|S_j e_{{v}}\|^2} : w \in \childi{j}{v}, ~v \in V, ~ j=1, \cdots, d \Big\}.
\eeqn
%Moreover, $S^{\mf t}_{\lambdab}$ is commuting if and only
%\beq \label{t-commuting} \|S_j e_{\parenti{j}{v}}\|^2 \|S_i e_{\mathsf{par}_i{\parenti{j}{v}}}\|^2=\|S_i e_{\parenti{i}{v}}\|^2 \|S_j e_{\mathsf{par}_j{\parenti{i}{v}}}\|^2\eeq
%for all $v \in V$ and $i, j=1, \cdots, d.$
\end{lemma}

\begin{proof}
This follows from the definition of $S_j^{\mf t}$ (see \eqref{toral-dual}). 
%while the second follows from \eqref{commuting}.
\end{proof}

\begin{definition} \label{kernel-c}
%Let $\mathscr T = (V,\mathcal E)$ be the directed Cartesian product of rooted directed trees $\mathscr T_1, \cdots, \mathscr T_d$. 
Let $T=(T_1, \cdots, T_d)$ be a toral left invertible commuting $d$-tuple on $\mathcal H$ and let $E$ denote the joint kernel of $T^*$. Let $T^{\mf t}$ be the toral Cauchy dual of $T$.  
We say that $T$ satisfies kernel condition $(\mf K)$ 
\index{$(\mf K)$}
if
\beqn  E \subseteq \ker T_j^* {T^{\mf t}}^{\alpha}_{[j]}~\mbox{ for all~} j = 1, \cdots, d~\mbox{and for all}~ \alpha \in \mathbb{N}^d,\eeqn 
where ${T^{\mf t}}_{[j]}^{\alpha}:=\prod_{i \neq j}{T^{\mf t}_i}^{\alpha_i}$ for $\alpha \in \mathbb N^d$ and $j=1, \cdots, d.$
\end{definition}
\begin{remark} \label{ex-k-c}
Note that the kernel condition $(\mf K)$ is satisfied in any one of the following cases:
\begin{enumerate}
\item[(i)] the dimension $d=1$.
\item[(ii)] $T$ is {doubly commuting}.
\item[(iii)] $T$ is a commuting operator valued multishift $W$.
\end{enumerate}
In case dimension $2,$ $S_{\lambdab}$ satisfies kernel condition $(\mf K)$ if and only if 
\beq \label{k-c-two} E \subseteq \big(\ker S_1^* {S^{{\mf t}^{\alpha_2}}_2}\big) \bigcap \big(\ker S_2^* {S^{{\mf t}^{\alpha_1}}_1}\big)~ \mbox{for all}~ (\alpha_1, \alpha_2) \in \mathbb{N}^2. \eeq
In general, $S_{\lambdab}$ does not satisfy the kernel condition $(\mf K).$ Indeed,  a rather tedious calculation shows that for $S_{\lambdab}$ on $\mathscr T_{2, 0} \times \mathscr T_{2, 0}$, as discussed in Example \ref{T1-T1}, $f:=\lambda^{(1)}_{(2, 0)}e_{(1, 0)} - \lambda^{(1)}_{(1, 0)}e_{(2, 0)} \in E$ does not belong to $\ker S^*_1S^{\mf t}_2$ for suitable choice of weights $\lambdab$.
\end{remark}

%\begin{example}
%Let $\mathscr T$ be the directed Cartesian product of  $\mathscr T_{2,0}$ and $\mathscr T_{1,0}$ as described in Example \ref{classical-mix}. Then by  Example \ref{j-k-mixed}, the joint kernel  $E$ of $S^*_{\lambdab}$ is given by 
%\beqn
%E = [e_{\rootb}] \oplus [\lambda^{(1)}_{(2,0)} e_{(1,0)} - \lambda^{(1)}_{(1,0)} e_{(2,0)}].
%\eeqn
%Now suppose that $\lambda^{(2)}_{(0,1)} \neq \lambda^{(2)}_{(1,1)}$. Then
%\beqn
%S_1^* S^{{\mf t}^2}_2 \Big(\lambda^{(1)}_{(2,0)} e_{(1,0)} - \lambda^{(1)}_{(1,0)} e_{(2,0)} \Big) &=& \left( \frac{\lambda^{(1)}_{(2,0)} \lambda^{(1)}_{(1,2)}}{\lambda^{(2)}_{(1,1)} \lambda^{(2)}_{(1,2)}} - \frac{\lambda^{(1)}_{(1,0)} \lambda^{(1)}_{(2,2)}}{\lambda^{(2)}_{(2,1)} \lambda^{(2)}_{(2,2)}} \right) e_{(0,2)}\\
%&=& \frac{\lambda^{(2)}_{(2,1)} \lambda^{(1)}_{(2,0)} \lambda^{(1)}_{(1,2)} \lambda^{(2)}_{(2,2)} - \lambda^{(2)}_{(1,2)} \lambda^{(1)}_{(1,0)} \lambda^{(1)}_{(2,2)} \lambda^{(2)}_{(1,1)}}{\lambda^{(2)}_{(1,1)} \lambda^{(2)}_{(1,2)} \lambda^{(2)}_{(2,1)} \lambda^{(2)}_{(2,2)}} e_{(0,2)}\\
%&\overset{\eqref{commuting}}=& \frac{\lambda^{(1)}_{(2,1)} \lambda^{(2)}_{(0,1)} \lambda^{(1)}_{(1,2)} \lambda^{(2)}_{(2,2)} - \lambda^{(1)}_{(1,2)} \lambda^{(1)}_{(1,0)} \lambda^{(1)}_{(2,2)} \lambda^{(2)}_{(1,1)}}{\lambda^{(2)}_{(1,1)} \lambda^{(2)}_{(1,2)} \lambda^{(2)}_{(2,1)} \lambda^{(2)}_{(2,2)}} e_{(0,2)}
%\eeqn
%
%\end{example}

The following provides a multivariable counterpart of \cite[Theorem 2.2]{CT}.
\begin{theorem}\label{model}
Let $\mathscr T = (V,\mathcal E)$ be the directed Cartesian
product of locally finite, rooted directed trees $\mathscr T_1, \cdots, \mathscr T_d$ and let $S_{\lambdab}=(S_1, \cdots, S_d)$ be a toral left invertible multishift on $\mathscr T$. Let $E$ be the joint kernel of $S^*_{\lambdab}$. 
Assume that the toral Cauchy dual $S^{\mf t}_{\lambdab}=(S^{\mf t}_1, \cdots, S^{\mf t}_d)$ of $S_{\lambdab}$ is commuting and let $$r:=(r(S^{\mf t}_1)^{-1},\cdots, r(S^{\mf t}_d)^{-1}),$$
where $r(T)$ denotes the spectral radius of a bounded linear operator $T.$
If $S_{\lambdab}$ satisfies the kernel condition $(\mf K)$,
then there exist a reproducing kernel Hilbert space $\mathscr H$ of $E$-valued holomorphic functions defined on the polydisc $\mathbb D^d_r$ and a unitary $U : l^2(V) \rar \mathscr H$ such that $U S_j = \mathscr M_{z_j} U$ for $j=1, \cdots, d$. If, in addition, $\mathscr T$ has finite joint branching index $k_{\mathscr T}$, then the reproducing kernel $\kappa$ of $\mathscr H$ is given by
\beq \label{kappa}
\kappa_\mathscr H(z, w)=\sum_{\underset{\underset{{j=1, \cdots, d}}{|\alpha_j - \beta_j| \leqslant k_{\mathscr T_j}}}{\alpha, \beta \in \mathbb N^d}} P_E {S^{{\mf t}^{*\alpha}}_{\lambdab}}{S^{{\mf t}^{\beta}}_{\lambdab}}|_E ~z^{\alpha} \overline{w}^{\beta} \quad (z, w \in \mathbb D^d_r).
\eeq
\end{theorem}

\begin{proof}
The proof relies on Shimorin's technique as presented in \cite{Sh} and the wandering subspace property of $S_{\lambdab}$ as obtained in Theorem \ref{wandering}. Assume that $S_{\lambdab}$ satisfies the kernel condition $(\mf K)$.
For $f \in l^2(V)$, define $$U_f(z) := \sum_{\alpha \in \mathbb{N}^d} (P_E S^{{\mf t}^{*\alpha}}_{\lambdab} f) z^\alpha, \quad z \in \mathbb C^d.$$ Then the power series $U_f$ converges absolutely on the polydisc $\mathbb D^d_r$ of polyradius $r.$ Let $\mathscr H$ denote the complex vector space of $E$-valued holomorphic functions of the form $U_f$. Thus $U : l^2(V) \rar \mathscr H$ defines a map from $l^2(V)$ onto $\mathscr H$ given by $U(f)=U_f$. Now we show that $U$ is injective.

To this end, let $U_f = 0$ for some $f \in l^2(V)$. Then $\sum_{\alpha \in \mathbb{N}^d} (P_E S^{{\mf t}^{*\alpha}}_{\lambdab} f) z^\alpha = 0$ which implies that $P_E S^{{\mf t}^{*\alpha}}_{\lambdab} f = 0$ for all $\alpha \in \mathbb{N}^d$. Note that $S^{\mf t}_{\lambdab}$ is also a multishift and the joint-kernel of $S^{\mf t}_{\lambdab}$ is equal to $E$. Hence by Theorem \ref{wandering}, we get that $\bigvee_{\alpha \in \mathbb N^d} {S^{\mf t^\alpha}_{\lambdab}}(E) = l^2(V).$ By taking orthogonal complement on both sides, we get 
$\bigcap_{\alpha \in \mathbb N^d} ({S^{\mf t^\alpha}_{\lambdab}}(E))^{\perp} = \{0\}$. 
It is easy to see that $({S^{\mf t^\alpha}_{\lambdab}}(E))^{\perp}=\ker P_E S^{{\mf t}^{*\alpha}}_{\lambdab}$ for any $\alpha \in \mathbb N^d.$ Hence
$\bigcap_{\alpha \in \mathbb N^d} \ker P_E S^{{\mf t}^{*\alpha}}_{\lambdab} = \{0\}$. Since $P_E S^{{\mf t}^{*\alpha}}_{\lambdab} f = 0$ for all $\alpha \in \mathbb{N}^d,$ we must have $f \in \bigcap_{\alpha \in \mathbb N^d} \ker P_E S^{{\mf t}^{*\alpha}}_{\lambdab}$. This shows that $f = 0$, and hence $U$ is injective.

We now define the inner product on $\mathscr H$ as $\langle U_f, U_g \rangle = \langle f, g \rangle_{l^2(V)}$ for all $f, g \in l^2(V)$. Then $\mathscr H$ becomes a Hilbert space and $U$ a unitary. Also, for $f \in l^2(V)$, 
\beq \label{intertwining}
(U S_j f)(z) &=& \sum_{\alpha \in \mathbb{N}^d} (P_E S^{{\mf t}^{*\alpha}}_{\lambdab} S_j f) z^\alpha \nonumber \\
&=&\sum_{\underset{\alpha_j = 0}{\alpha \in \mathbb{N}^d}} (P_E S^{{\mf t}^{*\alpha}}_{\lambdab} S_j f) z^\alpha + \sum_{\underset{\alpha_j \geqslant  1}{\alpha \in \mathbb{N}^d}} (P_E S^{{\mf t}^{*\alpha}}_{\lambdab} S_j f) z^\alpha \nonumber \\
&=& \sum_{\underset{\alpha_j = 0}{\alpha \in \mathbb{N}^d}} (P_E S^{{\mf t}^{*\alpha}}_{\lambdab{[j]}} S_j f) z^\alpha + \sum_{\alpha \in \mathbb{N}^d} (P_E S^{{\mf t}^{*\alpha + \epsilon_j}}_{\lambdab} S_j f) z^{\alpha + \epsilon_j} \notag \\ &=& \sum_{\alpha \in \mathbb{N}^d} (P_E S^{{\mf t}^{*\alpha + \epsilon_j}}_{\lambdab} S_j f) z^{\alpha + \epsilon_j},
\eeq
where we used the kernel condition $(\mf K)$ to get the last equality. Since the toral Cauchy dual $S^{\mf t}_{\lambdab}$ is commuting and $S^{{\mf t}*}_jS_j=I$, the sum on the right hand side of \eqref{intertwining} is equal to $z_j \sum_{\alpha \in \mathbb{N}^d} (P_E S^{{\mf t}^{*\alpha}}_{\lambdab} f) z^\alpha = z_j U_f(z)$. Thus we get $U S_j = \mathscr M_{z_j} U$.

We skip the verification of \beqn 
\kappa_\mathscr H(z, w) &=& \sum_{{\alpha, \beta \in \mathbb N^d}} P_E S^{{\mf t}^{*\alpha}}_{\lambdab}S^{{\mf t}^{\beta}}_{\lambdab}|_E z^{\alpha} \overline{w}^{\beta} \\ &=& P_E \prod_{i=1}^d(I-z_iS^{\mf t^*}_i)^{-1}\prod_{j=1}^d(I-\bar{w}_jS^{\mf t}_j)^{-1}|_E,
\eeqn
for $z, w \in \mathbb D^d_r,$ since it is along the lines of \cite[Proposition 2.13]{Sh}.
It is now easy to see that
\beqn
\inp{U_f}{\kappa_\mathscr H(\cdot, w)g}_{\mathscr H} = \inp{U_f(w)}{g}~(f, g \in E, ~w \in \mathbb D^d_r).
\eeqn
Thus $\mathscr H$ is a reproducing kernel Hilbert space with kernel $\kappa.$ 

Assume further that $\mathscr T$ has finite joint branching index $k_{\mathscr T}$.
To check that $\kappa$ has the form given in \eqref{kappa}, let $\alpha, \beta \in \mathbb N^d$ be such that $|\alpha_j - \beta_j| > k_{\mathscr T_j}$ for some $j=1, \cdots, d.$
In view of Proposition \ref{joint-k}, it suffices to check that $P_E S^{{\mf t}^{*\alpha}}_{\lambdab}S^{{\mf t}^{\beta}}_{\lambdab}e_v=0$ for all $v \in F_1 \times \cdots \times F_d,$ where $F_j:=\child{V^{(j)}_{\prec}} \cup \{\mathsf{root}_j\}~(j=1, \cdots, d).$ To see this, let $v \in F_1 \times \cdots \times F_d.$ 
Since the depth $\alpha_v$ of $v$ equals $(\alpha_{v_1}, \cdots, \alpha_{v_d})$ with $\alpha_{v_j}$ being the depth of $v_j$ in $\mathscr T_j$,  we obtain
$0 \leqslant \alpha_{v_j} \leqslant k_{\mathscr T_j}$ for every $j=1, \cdots, d.$ 
An application of Proposition \ref{shift-prop}(vi) shows that 
\beqn
S^{{\mf t}^{*\alpha}}_{\lambdab}S^{{\mf t}^{\beta}}_{\lambdab}e_v = \sum_{u \in \parentnt{\alpha}{\childnt{\beta}{v}}}\gamma_u e_u
\eeqn
for some scalars $\gamma_u \in \mathbb C.$
It follows that $\alpha_u = \alpha_v + \beta - \alpha$ for $u \in \parentnt{\alpha}{\childnt{\beta}{v}}.$ 
Note that $\beta_j - \alpha_j = \alpha_{u_j} - \alpha_{v_j} \geqslant  \alpha_{u_j} - k_{\mathscr T_j}$, and hence 
$$\beta_j - \alpha_j + |\beta_j - \alpha_j| > \beta_j - \alpha_j +  k_{\mathscr T_j}  \geqslant  \alpha_{u_j} \geqslant  0.$$
This shows that $\beta_j - \alpha_j >  0$, and consequently,
$$\alpha_{u_j} = \alpha_{v_j} + \beta_j - \alpha_j =\alpha_{v_j} + |\alpha_j - \beta_j| > k_{\mathscr T_j}.$$ Thus $u \notin   F_1 \times \cdots \times F_d$, and hence by Proposition \ref{joint-k}, $P_E S^{{\mf t}^{*\alpha}}_{\lambdab}S^{{\mf t}^{\beta}}_{\lambdab}e_v=0$.
\end{proof}
\begin{remark} \label{k-c-rem}
The kernel condition $(\mf K)$ is used only in obtaining the intertwining relation $U S_j = \mathscr M_{z_j} U$.
Conversely, if one assumes the above intertwining relation then the calculations in \eqref{intertwining} shows that the kernel condition $(\mf K)$ is also necessary. 
\end{remark}

It is interesting to know the maximum value of $r$ for which $\kappa_\mathscr H(z, w)$ converges on $\mathbb D^d_r \times \mathbb D^d_r$. Unfortunately, we do not know this even in the one dimensional case (see \cite[Equation (2.3)]{CT}). Before we proceed to the next result, it is convenient to introduce some terminology.

Let $S_{\lambdab}$ be a toral left invertible multishift on $\mathscr T$.
We refer to the pair 
\index{$(\mathscr M_z, \kappa_\mathscr H)$}
$(\mathscr M_z, \kappa_\mathscr H)$ as {\it Shimorin's analytic model}.

If one relaxes the toral left invertibility of $S_{\lambdab}$ then it may happen that the interior of the Taylor spectrum of $S_{\lambdab}$ is empty (Example \ref{empty-i}). This is not possible otherwise.

\begin{corollary}
If $S_{\lambdab}$ has Shimorin's analytic model, then the polydisc $\mathbb D^d_r$ is contained in the point spectrum of $S^*_{\lambdab}$, where $r:=(r(S^{\mf t}_1)^{-1},\cdots, r(S^{\mf t}_d)^{-1}).$
\end{corollary}
\begin{remark}
Since $\mbox{cl}({\sigma_p(S^*_{\lambdab})}) \subseteq \sigma(S^*_{\lambdab})=\sigma(S_{\lambdab})$, we must have 
\beqn
\Big(r(S^{\mf t}_1)^{-2} + \cdots +  r(S^{\mf t}_d)^{-2}\Big)^{\frac{1}{2}} \Le r(S_{\lambdab}),
\eeqn
where $r(S_{\lambdab})$ is the spectral radius of $S_{\lambdab}.$
\end{remark}

Let us analyze Theorem \ref{model} in case $S_{\lambdab}$ is a doubly commuting toral isometry.
\begin{corollary}
Let $\mathscr T = (V,\mathcal E)$ be the directed Cartesian
product of locally finite, rooted directed trees $\mathscr T_1, \cdots, \mathscr T_d$ and let $S_{\lambdab}$ be a toral isometry multishift on $\mathscr T$ and let $E$ denote the joint kernel of $S^*_{\lambdab}$. Then $S_{\lambdab}$ is doubly commuting if and only if $S_{\lambdab}$ is unitarily equivalent to the multiplication $d$-tuple $\mathscr M_z$ on the $E$-valued Hardy space of the unit polydisc $\mathbb D^d$. In particular, \beqn 
\kappa_\mathscr H(z, w)=\prod_{j=1}^d \frac{I_E}{1-z_j\overline{w}_j}~(z, w \in \mathbb D^d),
\eeqn
where $I_E$ denotes the identity operator.
\end{corollary}
\begin{proof}
Assume that $S_{\lambdab}$ is doubly commuting.
Since $S_{\lambdab}$ is a toral isometry, $S^{\mf t}_{\lambdab}=S_{\lambdab}.$ Thus
\beqn
P_E S^{{\mf t}^{*\alpha}}_{\lambdab}S^{{\mf t}^{\beta}}_{\lambdab}|_E = \delta_{\alpha \beta}I_E~(\alpha, \beta \in \mathbb N^d),
\eeqn
where $\delta_{\alpha \beta}$ denotes the Kronecker delta. It now follows from \eqref{kappa} that $\kappa$ has the desired form. The conclusion can now be drawn from Theorem \ref{model} and the fact that  the reproducing kernel uniquely determines the reproducing kernel Hilbert space \cite{Aro}. We leave the converse to the interested reader.
\end{proof}

%\uwam{k-diagonal kernels must be computed in some examples}

Now we discuss a large class of multishifts $S_{\lambdab}$ (not covered by Remark \ref{ex-k-c}), which always satisfy the kernel condition $(\mf K)$.
\begin{corollary}
Let $\mathscr T_1=(V_1, \mathcal E_1)$ be a locally finite, rooted directed tree and let $\mathscr T_2=\mathscr T_{1, 0}$ be the rooted directed trees as described in Example \ref{classical}. Consider the directed Cartesian product $\mathscr T$ of $\mathscr T_1$ and $\mathscr T_2.$ Let $S_{\lambdab}=(S_1, S_2)$ be a toral left invertible multishift on $\mathscr T$ and let $E$ be the joint kernel of $S^*_{\lambdab}$. 
Assume that the toral Cauchy dual $S^{\mf t}_{\lambdab}=(S^{\mf t}_1, S^{\mf t}_2)$ of $S_{\lambdab}$ is commuting and let $r:=(r(S^{\mf t}_1)^{-1},r(S^{\mf t}_2)^{-1}),$
where $r(T)$ denotes the spectral radius of a bounded linear operator $T.$
Then $S_{\lambdab}$ has Shimorin's analytic model $(\mathscr M_z, \kappa_\mathscr H).$ If, in addition, $\mathscr T_1$ has finite branching index $k_{\mathscr T_1}$, then the reproducing kernel $\kappa$ of $\mathscr H$ is given by
\beqn
\kappa_\mathscr H(z, w) &=& \sum_{\alpha \in \mathbb N^2} P_E S^{{\mf t}^{*\alpha}}_{\lambdab}S^{{\mf t}^{\alpha}}_{\lambdab}|_E ~z^{\alpha} \overline{w}^{\alpha} \\ &+& 
\sum_{\underset{\alpha_2=\beta_2}{\underset{0 < |\alpha_1 - \beta_1| \leqslant k_{\mathscr T_1}}{\alpha, \beta \in \mathbb N^2}}} P_E S^{{\mf t}^{*\alpha}}_{\lambdab}S^{{\mf t}^{\beta}}_{\lambdab}|_E ~z^{\alpha} \overline{w}^{\beta} \quad (z, w \in \mathbb D^2_r).
\eeqn
\end{corollary}
\begin{proof} 
We first compute the joint kernel $E$ of $S^*_{\lambdab}$.
The argument is similar to that of Example \ref{j-k-mixed}.
 Note that $\mathscr P=\{\emptyset, \{1\}, \{2\}, \{1, 2\}\}.$
Also, 
\beqn
\Phi_{\emptyset} & = & \{(\rootb_1, 0)\}, \Phi_{\{1\}}=\{(v, 0) : v \in V^{\circ}_1\}, \\ \Phi_{\{2\}}& = & \{(\rootb_1, j) : j \geqslant  1\}, \Phi_{\{1, 2\}}=\{(v, j) : v \in V^{\circ}_1, ~j \geqslant  1\}.
\eeqn
Note further that
\beqn
\mathsf{sib}_{\emptyset}((\rootb_1, 0)) &=&  \{(\rootb_1, 0)\}, \\
\mathsf{sib}_{\{1\}}((v, 0)) &=& \{(w, 0) : w \in \mathsf{sib}(v)\}~(v \in V^{\circ}_1),\\
\mathsf{sib}_{\{2\}}((\rootb_1, j)) &=& \{(\rootb_1, j)\}~\mbox{for all~} j \geqslant  1,~ \mbox{and}\\
\mathsf{sib}_{\{1, 2\}}((v, j)) &=&\{(w, j) : w \in \mathsf{sib}(v)\}~(v \in V^{\circ}_1, ~j \Ge 1). 
\eeqn
Let us form $\Omega_F$ by picking up one element from each of the equivalence classes $\mathsf{sib}_F(u)$ for every $F \in \mathscr P$. 
%as follows: $\Omega_{\emptyset} = \{(\rootb_1, 0)\} $, and
%\beqn
%\Omega_{\emptyset} &=& \{(\rootb_1, 0)\} \\
%\Omega_{\{1\}} &=& \bigcup\{(w, 0) : w \in \mathsf{sib}(v)~\mbox{is a fixed child~}v \in \child{V^{(1)}_{\prec}}\}  \\ && \bigcup\,  \{(w, 0) : w \in \child{V_1 \setminus V^{(1)}_{\prec}}\}, \\ \Omega_{\{2\}} &=& 
%\{(\rootb_1, j) : j \geqslant  1\}, \\
%\Omega_{\{1, 2\}} &=& \{(w, j) : w \in \child{V^{(1)}_{\prec}}~\mbox{is a fixed child},~j \Ge 1\}  \\  && \bigcup\,   \{(w, j) : w \in \child{V_1 \setminus V^{(1)}_{\prec}}, j \Ge 1\}. 
%\eeqn
We next calculate $\mathsf{sib}_{F, G}(u)$ for possible choices of $F, G,$ and $u \in \Omega_F.$
If $F=\{1\}$, then $G = \emptyset$. In this case, 
\beqn
\mathsf{sib}_{\{1\}, \emptyset}(v, 0)=\{(\rootb_1, 0)\}~ (v \in  V^{\circ}_1).
\eeqn 
This together with \eqref{system} yields the following equations:
\beqn
\sum_{w \in \child{v}} f(w, 0) \lambda^{(1)}_{(w, 0)} &=& 0~(v \in V^{(1)}_{\prec}), \\
f(w, 0) \lambda^{(1)}_{(w, 0)} &=& 0~(w \in \child{V_1 \setminus V^{(1)}_{\prec}}).
\eeqn
In case $F=\{2\},$ 
$G=\emptyset$ and $\mathsf{sib}_{\{2\}, \emptyset}(\rootb_1, j)=\{(\rootb_1, 0)\}$ for $j \geqslant 1,$ and hence
we obtain the equations
\beqn
f(\rootb_1,j) \lambda^{(2)}_{(\rootb_1, j)} = 0~(j \geqslant  1).
\eeqn
In case $F = \{1,2\}$, $G = \{1\}$ or $\{2\}$. Then for all $w \in V^{\circ}_1$ and $j \geqslant  1$,
\beqn
\mathsf{sib}_{\{1, 2\}, \{2\}}(w, j) = \{(\rootb_1, j)\},~
\mathsf{sib}_{\{1, 2\}, \{1\}}(w, j) = \{(u, 0) : u \in \mathsf{sib}(w)\}.
\eeqn
This gives following equations for $j \geqslant  1$,
\beqn
\sum_{u \in \child{w}} f(u,j) \lambda^{(1)}_{(u,j)} &=& 0~(w \in V^{(1)}_{\prec}),~
f(w,j) \lambda^{(1)}_{(w,j)} = 0~(w \in \child{V_1 \setminus V^{(1)}_{\prec}}), \\
f(w,j) \lambda^{(2)}_{(w,j)} &=& 0~(w \in \child{V^{(1)}_{\prec}}), \\
f(w,j) \lambda^{(1)}_{(w,j)} &=& 0, ~ f(w,j) \lambda^{(2)}_{(w,j)} = 0~(w \in \child{V_1 \setminus V^{(1)}_{\prec}}).
\eeqn
Solving all the above equations, we get that 
\beqn
E = [e_{(\rootb_1, 0)}] \bigoplus_{w \in V^{(1)}_{\prec}} \Big(l^2(\child{w} \times \{0\}) \ominus [\Gamma^{(1)}_{(w, 0)}]\Big),
\eeqn
where $\Gamma^{(1)}_{(w, 0)} : \childi{1}{(w, 0)} \rar \mathbb{C}$ defined as $\Gamma^{(1)}_{(w, 0)} (u, 0) = \lambda^{(1)}_{(u, 0)}.$

We next check that $S_{\lambdab}$ satisfies the kernel condition $(\mf K)$. Since $E \subseteq \ker S_2^* {S^{{\mf t}^{\alpha_1}}_1}$ for all $\alpha_1 \in \mathbb{N},$
in view of \eqref{k-c-two}, it suffices to check that $$E \subseteq \ker S_1^* {S^{{\mf t}^{\alpha_2}}_2}~\mbox{for all}~\alpha_2 \in \mathbb{N}.$$
Clearly, $e_{(\rootb_1, 0)} \in \ker S_1^* {S^{{\mf t}^{\alpha_2}}_2}.$ For $w \in V^{(1)}_{\prec}$, let $f = \sum_{u \in \child{w}}f((u, 0))e_{(u, 0)} \in l^2(\child{w} \times \{0\})$ be such that 
\beq \label{Gamma-ortho} \sum_{u \in \child{w}}f((u, 0))\lambda^{(1)}_{(u, 0)}=0~(w \in V^{(1)}_{\prec}). \eeq Note that
\beqn
S_1^* {S^{{\mf t}^{\alpha_2}}_2}f &=& S_1^*\sum_{u \in \child{w}} \frac{f((u, 0))}{\prod_{k=1}^{\alpha_2}\lambda^{(2)}_{(u, k)}}e_{(u, \alpha_2)} \\ &=& \left(\sum_{u \in \child{w}} \frac{f((u, 0))}{\prod_{k=1}^{\alpha_2}\lambda^{(2)}_{(u, k)}}\lambda^{(1)}_{(u, \alpha_2)}\right)e_{(w, \alpha_2)}.
\eeqn
Thus $S_1^* {S^{{\mf t}^{\alpha_2}}_2}f =0$ if and only if 
\beqn
\sum_{u \in \child{w}} \frac{f((u, 0))}{\prod_{k=1}^{\alpha_2}\lambda^{(2)}_{(u, k)}}\lambda^{(1)}_{(u, \alpha_2)}=0.
\eeqn
Note that any general solution of \eqref{Gamma-ortho} is of the form $$f=\sum_{\underset{u \neq v}{u \in \child{w}}}f((u, 0))\left(e_{(u, 0)} - {e_{(v, 0)}}\frac{\lambda^{(1)}_{(u, 0)}}{\lambda^{(1)}_{(v, 0)}}\right)$$ for some fixed $v \in \child{w}$. 
It follows that \beqn
\sum_{u \in \child{w}} \frac{f((u, 0))}{\prod_{k=1}^{\alpha_2}\lambda^{(2)}_{(u, k)}}\lambda^{(1)}_{(u, \alpha_2)} = \sum_{\underset{u \neq v}{u \in \child{w}}}f((u, 0))\left(\frac{\lambda^{(1)}_{(u, \alpha_2)}}{\prod_{k=1}^{\alpha_2}\lambda^{(2)}_{(u, k)}} - \frac{\lambda^{(1)}_{(u, 0)}}{\lambda^{(1)}_{(v, 0)}}\frac{\lambda^{(1)}_{(v, \alpha_2)}}{\prod_{k=1}^{\alpha_2}\lambda^{(2)}_{(v, k)}}\right), 
%\\ = \sum_{\underset{u \neq v}{u \in \child{w}}} \frac{f((u, 0))}{\prod_{k=1}^{\alpha_2}\lambda^{(2)}_{(u, k)}\lambda^{(1)}_{(v, 0)}\prod_{k=1}^{\alpha_2}\lambda^{(2)}_{(v, k)}}\left({\lambda^{(1)}_{(u, \alpha_2)}}{\lambda^{(1)}_{(v, 0)}\prod_{k=1}^{\alpha_2}\lambda^{(2)}_{(v, k)}} - \lambda^{(1)}_{(u, 0)}{\lambda^{(1)}_{(v, \alpha_2)}}{\prod_{k=1}^{\alpha_2}\lambda^{(2)}_{(u, k)}}\right)
\eeqn
and hence it suffices to see that for every $u \in \child{w} \setminus \{v\},$ $$\frac{\lambda^{(1)}_{(u, \alpha_2)}}{\prod_{k=1}^{\alpha_2}\lambda^{(2)}_{(u, k)}} - \frac{\lambda^{(1)}_{(u, 0)}}{\lambda^{(1)}_{(v, 0)}}\frac{\lambda^{(1)}_{(v, \alpha_2)}}{\prod_{k=1}^{\alpha_2}\lambda^{(2)}_{(v, k)}}=0.$$ However, by repeated applications of \eqref{commuting}, we obtain
\beqn
{\lambda^{(1)}_{(u, \alpha_2)}}\Big({\lambda^{(1)}_{(v, 0)}}{\prod_{k=1}^{\alpha_2}\lambda^{(2)}_{(v, k)}} \Big) 
&=& {\lambda^{(1)}_{(u, \alpha_2)}}\Big({\prod_{k=1}^{\alpha_2}\lambda^{(2)}_{(w, k)}} {\lambda^{(1)}_{(v, \alpha_2)}}\Big), \\ 
{\lambda^{(1)}_{(v, \alpha_2)}}\Big({\lambda^{(1)}_{(u, 0)}} {\prod_{k=1}^{\alpha_2}\lambda^{(2)}_{(u, k)}}\Big) &=& {\lambda^{(1)}_{(v, \alpha_2)}}\Big( {\prod_{k=1}^{\alpha_2}\lambda^{(2)}_{(w, k)}}{\lambda^{(1)}_{(u, \alpha_2)}}\Big),
\eeqn
which shows that $S_{\lambdab}$ satisfies the kernel condition $(\mf K).$ The desired conclusion now follows from Theorem \ref{model} once we observe that $k_{\mathscr T_2}=0.$
\end{proof}

\begin{remark}
We discuss two special cases of the preceding corollary.
\begin{enumerate}
\item[(i)] In case $\mathscr T_1$ is $\mathscr T_{1, 0}$, $S_{\lambdab}$ is nothing but the classical multishift and the associated kernel $\kappa_\mathscr H(z, w)$ is diagonal. 
\item[(ii)] In case $\mathscr T_1$ is $\mathscr T_{2, 0}$,  $k_{\mathscr T_1}=1,$ and hence the kernel $\kappa(z, w)$ is given by 
\beqn
\kappa_\mathscr H(z, w) &=& \sum_{\alpha \in \mathbb N^2} P_E S^{{\mf t}^{*\alpha}}_{\lambdab}S^{{\mf t}^{\alpha}}_{\lambdab}|_E ~z^{\alpha} \overline{w}^{\alpha} + \sum_{\alpha \in \mathbb N^2} P_E S^{{\mf t}^{*\alpha}}_{\lambdab}S^{{\mf t}^{\alpha + \epsilon_1}}_{\lambdab}|_E ~z^{\alpha} \overline{w}^{\alpha + \epsilon_1}    \\ &+& 
\sum_{\alpha \in \mathbb N^2} P_E S^{{\mf t}^{*\alpha + \epsilon_1}}_{\lambdab}S^{{\mf t}^{\alpha}}_{\lambdab}|_E ~z^{\alpha + \epsilon_1} \overline{w}^{\alpha}
%\sum_{\underset{\alpha_2=\beta_2}{\underset{0 < |\alpha_1 - \beta_1| \leqslant k_{\mathscr T_1}}{\alpha, \beta \in \mathbb N^2}}} P_E S^{{\mf t}^{*\alpha}}_{\lambdab}S^{{\mf t}^{\beta}}_{\lambdab}|_E ~z^{\alpha} \overline{w}^{\beta} 
\quad (z, w \in \mathbb D^2_r).
\eeqn
\end{enumerate}
\end{remark}

We conclude this chapter with one application to the Cowen-Douglas theory.

 Let $\Omega$ be an open connected subset of $\mathbb C^d.$
  For a positive integer $n$, let ${\bf B}_n(\Omega)$ denote the set
  of all commuting $d$-tuples $T$ on $\mathcal H$ satisfying the following \index{${\bf B}_n(\Omega)$}
  conditions:
  \begin{enumerate}
  \item  For every point $\omega=(\omega_1, \cdots, \omega_d) \in \Omega$, we have
    \begin{enumerate}
    \item  the map $D_{T- \omega}(x) = ((T_j-\omega_j)x)$ from $\mathcal H$ into $\mathcal H^{\oplus d}$ has closed range.
    \item  $\dim{\ker({T-\omega})} = n$.
    \end{enumerate}
  \item The subspace
    $\bigvee_{\omega \in \Omega}  \ker({T-\omega})$ of $\mathcal H$ equals $\mathcal H$.
  \end{enumerate}
  We will call the set ${\bf B}_n(\Omega)$ the {\it Cowen-Douglas class
    of degree $n$ with respect to $\Omega$} (refer to \cite{CD}, \cite{CuS} for the basic theory of Cowen-Douglas class in one and several variables).
    
\begin{corollary}
Let $\mathscr T = (V,\mathcal E)$ be the directed Cartesian
product of locally finite, rooted directed trees $\mathscr T_1, \cdots, \mathscr T_d$ with finite joint branching index.  If $S_{\lambdab}$ is a toral left invertible multishift which satisfies the kernel condition $(\mf K)$, then
$S^*_{\lambdab}$ belongs to Cowen-Douglas class ${\bf B}_{\dim E}(\mathbb D^d_r)$, where $E$ denotes the joint kernel of $S^*_{\lambdab}$ and $r :=(\|S^{\mf t}_1\|^{-1},\cdots, \|S^{\mf t}_d\|^{-1})$.
\end{corollary}
\begin{proof} Suppose that $S_{\lambdab}$ is a toral left invertible multishift satisfying the kernel condition $(\mf K)$.
By Corollary \ref{dimE-finite}, $\dim E$ is finite.
Also, by Theorem \ref{wandering}, the toral Cauchy dual $S^{\mathfrak t}_{\lambdab}$ possesses wandering subspace property. It may now be concluded from \cite[Theorem 5.4]{CC} that for any $s \leqslant r$,
\beqn
l^2(V) = \bigvee_{\omega \in \mathbb D^d_s} \ker (S^*_{\lambdab}-\omega)~\mbox{and}~\dim \ker(S^*_{\lambdab}-\omega) \Ge \dim E.
\eeqn 
However, by Corollary \ref{f-cyclic}, we get $\dim \ker(S^*_{\lambdab}-\omega) = \dim E.$
Thus it remains only to check (1)(a) of the definition of the Cowen-Douglas class. 
Also, $\sigma_{l}(S_j) \cap \mathbb D_{\frac{1}{\|S^{\mf t}_j\|}} = \emptyset$ for $j=1, \cdots, d$. Thus for any $\omega \in \mathbb D^d_r,$ $S_j - \omega_j$ has closed range, and hence range of $S^*_j - \overline{\omega}_j$ is closed for $j=1, \cdots, d.$ This shows that the range of $D_{S^*_{\lambdab}-{\omega}}$ is also closed for every $\omega \in \mathbb D^d_r.$
%we obtain that $\sigma_{ap}(S_{\lambdab})$ is contained in $\prod_{j=1}^d \mathbb C \setminus \mathbb D_{\frac{1}{\|S^{\mf t}_j\|}}.$  In particular, $\mathbb D^d_r$ is disjoint from $\sigma_{ap}(S_{\lambdab}),$ and hence $D_{S_{\lambdab}-w}$ is bounded from below. 
\end{proof}

\chapter{Special Classes of Multishifts}
%\chapter{Balanced Multishifts}

In this chapter, we discuss two classes of so-called balanced multishifts, namely torally balanced multishifts and spherically balanced multishifts (cf. \cite{CK}, \cite{K}). These generalize largely the classes of toral and spherical isometries (\cite{At}, \cite{DE}, \cite{EP}, \cite{AKM}). In particular, we introduce tree analogs of the classical multishifts $S_{{\bf w}, a}$ as discussed in Example \ref{1.1.1}. 
We show that these multishifts are unitarily equivalent to multiplication tuples acting on reproducing kernel Hilbert spaces of vector valued holomorphic functions defined on the unit ball. We also provide a compact formula for the associated reproducing kernels involving finitely many hypergeometric functions.
We further investigate some known classes of multishifts which include
mainly the well-studied class of joint subnormal tuples (\cite{CS-1} \cite{At-0}, \cite{At}, \cite{St}, \cite{Co-1}, \cite{At-1}, \cite{EP}, \cite{AZ}, \cite{GHX}), and comparatively less understood class of joint hyponormal tuples \cite{At-00}, \cite{CMX}, \cite{CP}, \cite{Cu-1}, \cite{CL}. We emphasize, in particular, on characterizations of these classes within spherically balanced multishifts.

\section{Torally Balanced Multishifts}

Before we introduce the class of torally balanced multishifts, let us understand somewhat related class of multishifts with commuting toral Cauchy dual. The later class admits a polar decomposition in the following sense. 
\begin{proposition} \label{t-p-decom} 
Let $\mathscr T = (V,\mathcal E)$ be the directed Cartesian product of rooted directed trees $\mathscr T_1, \cdots, \mathscr T_d$.
Let $S_{\lambdab}=(S_1, \cdots, S_d)$ be a toral left invertible multishift on $\mathscr T$ with toral Cauchy dual $S^{\mf t}_{\lambdab}$.  Then $S^{\mf t}_{\lambdab}$ is commuting if and only if there exist a toral isometry multishift $U_{\thetab}=(U_1, \cdots, U_d)$ on $\mathscr T$ and a commuting $d$-tuple $D=(D_1, \cdots, D_d)$ of diagonal, positive, invertible bounded linear operators on $l^2(V)$ such that
\beq \label{t-p-d}
S_j = U_jD_j,~j=1, \cdots, d.
\eeq
Further, this decomposition is unique.
\end{proposition}
\begin{proof} 
Let us first see the uniqueness of the above decomposition.
Indeed, if \eqref{t-p-d} holds then for $j=1, \cdots, d,$
\beqn
S^*_j S_j =  D_jU^*_jU_j D_j = D^2_j, 
\eeqn
and hence $D_j$ must be the positive square root of $S^*_j S_j.$
Also, since $D_j$ is invertible, $U_j=S_j D^{-1}_j$ for $j=1, \cdots, d$. 

%Let us now see the sufficiency part.
By Proposition \ref{shift-prop}(i) and Lemma \ref{Cauchy-dual}, 
$S^{\mf t}_{\lambdab}$ is commuting if and only if
\beq \label{t-commuting} \|S_j e_{\parenti{j}{v}}\| \|S_i e_{\mathsf{par}_i{\parenti{j}{v}}}\|=\|S_i e_{\parenti{i}{v}}\| \|S_j e_{\mathsf{par}_j{\parenti{i}{v}}}\| \eeq
for all $v \in V$ and $i, j=1, \cdots, d.$
Consider now the multishift $U_{\thetab}=(U_1, \cdots, U_d)$ with weights given by
\beq \label{t-theta} \theta^{(j)}_w := \frac{\lambda^{(j)}_w}{\|S_j e_{{v}}\|}, ~\mbox{for~} w \in \childi{j}{v}, ~v \in V~\mbox{and~} j=1, \cdots, d.\eeq
Since $S_iS_j=S_jS_i~(i, j=1, \cdots, d)$, by an application of Proposition \ref{shift-prop}(i), $S^{\mf t}_{\lambdab}$ is commuting if and only if $U_iU_j=U_jU_i~(i, j=1, \cdots, d)$. 

To see the sufficiency part, assume that $S^{\mf t}_{\lambdab}$ is commuting. Thus $U_{\thetab}$ is commuting.
Also, since $U^*_jU_j=I$ for $j=1, \cdots, d,$
%If $S^{\mf t}_{\lambdab}$ is commuting, then 
$U_{\thetab}$ is a toral isometry. Thus $S_{\lambdab}$ admits the decomposition \eqref{t-p-d}, 
where $D_j~(j=1, \cdots, d)$ is given by
\beqn
D_j e_v :=\|S_j e_{v}\| e_v,~v \in V.
\eeqn
Further, by Proposition \ref{shift-prop}(vii), $D_1, \cdots, D_d$ are diagonal, positive, invertible bounded linear operators. 

Finally, since $S^{\mf t}_{\lambdab}$ is commuting if and only if $U_{\thetab}$ is commuting, the necessary part follows from the
uniqueness of \eqref{t-p-d}.
%This proves the sufficiency part. 
%Conversely, 
%if we assume \eqref{t-p-d} then 
%$U_{\thetab}$ is commuting, and hence by the commutativity of $S_{\lambdab},$ the toral Cauchy dual $S^{\mf t}_{\lambdab}$ must be commuting. 
\end{proof}

For the sake of convenience, we refer to $U_{\thetab}$ and $D=(D_1, \cdots, D_d)$ 
\index{$U_{\thetab}$} 
as {\it toral isometry part} and {\it diagonal part} 
%\index{$D=(D_1, \cdots, D_d)$}
of the multishift $S_{\lambdab}$ respectively.
\begin{remark}
Let $S_{\bf w}$ be a toral left invertible classical multishift.
Then the operator tuple $S^{\mathfrak{t}}_{\bf w}$ toral Cauchy dual of $S_{\bf w}$ is given
by 
\beqn 
S^{\mathfrak{t}}_j e_\alpha = \frac{1}{w^{(j)}_\alpha} e_{\alpha + \epsilon_j}~(1 \leqslant j \leqslant
d). 
\eeqn 
Note that $S^{\mathfrak{t}}_{\bf w}$ is also a commuting $d$-variable weighted shift with
weight  multisequence 
\begin{equation*}
\Big\{\frac{1}{w^{(j)}_\alpha} : 1 \leqslant j  \leqslant d, ~\alpha \in {\mathbb{N}}^d\Big\}.
\end{equation*}
Thus the toral Cauchy dual of a classical multishift $S_{\bf w}$ always commutes. Indeed, \eqref{t-commuting} is equivalent to the commutativity of $S_{\bf w}$ in this case. Moreover, the toral isometry part can be identified with the multiplication tuple $\mathscr M_z$ of the Hardy space of the unit polydisc (commonly known as {\it Cauchy $d$-shift}). Indeed, $\theta^{(j)}_v = 1$ for all $v \in V^{\circ}$ and $j=1, \cdots, d.$ 
\end{remark}

Since toral isometry part and diagonal part in the
above decomposition need not commute, prima facie the relation between the moments $\{\|S^{\alpha}_{\lambdab}e_v\|^2\}_{\alpha \in \mathbb N^d}$ of $S_{\lambdab}$ and that of toral isometry part $U_{\thetab}$ is not visible. 
Nevertheless, there is a subclass of multishifts in which we are facilitated to get a nice formula for $\{\|S^{\alpha}_{\lambdab}e_v\|^2\}_{\alpha \in \mathbb N^d}$ (see \eqref{moment-t-1}).

\begin{definition}
Let $\mathscr T = (V,\mathcal E)$ be the directed Cartesian product of rooted directed trees $\mathscr T_1, \cdots, \mathscr T_d$ and
let $S_{\lambdab}=(S_1, \cdots, S_d)$ be a commuting multishift on $\mathscr T$. 
%with commuting toral Cauchy dual $S^{\mf t}_{\lambdab}$. 
For $j=1, \cdots, d,$ define $\mf C_j : V \rar (0, \infty)$ by \beqn \mf C_j(v) :=\|S_j e_{v}\| ~(v \in V, ~ i, j=1, \cdots, d). \eeqn
We say that $S_{\lambdab}$ is {\it torally balanced} if for each $j=1, \cdots, d$, $\mf C_j$ is constant on \index{$\mf C_j$}
every generation $\mathcal G_t$, $t \in \mathbb N$.
%\beq \label{generation} \mathcal G_t := \big\{v \in V : v \in \childnt{\alpha}{\rootb}, \alpha \in \mathbb N^d, |\alpha|=t\big\}, ~t \in \mathbb N.\eeq
We denote the constant value of $\mf C_j(v)$ by 
\index{$c^{(j)}_{|\alpha_v|}$}
$c^{(j)}_{|\alpha_v|}$, where $\alpha_v$ is the depth of $v$ in $\mathscr T$ (see Definition \ref{depth-gen}). 
%\uwam{can be made weaker- constant on child}
\end{definition}
\begin{remark}
Note that $S_{\lambdab}$ is a toral isometry if and only if 
\beqn \sum_{w \in \childi{j}{v}} (\lambda_{w}^{(j)})^2 = 1~\mbox{for all~}v \in V~\mbox{and~}j=1, \cdots, d.\eeqn
Clearly, any toral isometry multishift is torally balanced with $\mf C_j$ being constant function with value $1$ for every $j=1, \cdots, d$.
\end{remark}

The following proposition leads to an interesting family of torally balanced multishifts. In particular, any directed Cartesian product of locally finite rooted directed trees supports a toral isometry (cf. \cite[Proposition 8.1.3]{JJS}).

\begin{proposition}
Let $\mathscr T = (V,\mathcal E)$ be the directed Cartesian product of locally finite rooted directed trees $\mathscr T_1, \cdots, \mathscr T_d$.
For $i=1, \cdots, d,$ let $\{c(t, i)\}_{t \in \mathbb N}$ be a bounded sequence of positive real numbers such that
\beq \label{c(t, i)}
c(t, i)c(t-1, j) = c(t, j) c(t-1, i)~\mbox{for all integers~} t \geqslant  1~\mbox{and~} i, j=1, \cdots, d.
\eeq
%e.g. $c(t, i)=1/i(t+1)$
Consider the multishift $S_{\lambdab_{\bf c}}=(S_1, \cdots, S_d)$  with weights \index{$S_{\lambdab_{\bf c}}$}
\beqn
\lambda^{(j)}_w =\sqrt{\frac{c(|\alpha_v|, j)}{\mbox{card}(\childi{j}{v})}}~ \mbox{for~} w \in \childi{j}{v}, ~v \in V~\mbox{and~} j=1, \cdots, d.
\eeqn
Then $S_{\lambdab_{\bf c}}$ defines a torally balanced multishift.
In case $c(t, j) = 1$ for all $j=1, \cdots, d$ and $t \in \mathbb N$, $S_{\lambdab_{\bf c}}$ is a toral isometry.
\end{proposition}

\begin{proof}
Since $\{c(t, j)\}_{t \in \mathbb N}$ is a bounded sequence, by Lemma \ref{bddness}(i), $S_j$ defines a bounded linear operator on $l^2(V)$ for every $j=1, \cdots, d.$
Let $w \in V$ and $i, j = 1, \cdots, d.$
By Proposition \ref{sib-id}, for every $v \in \childi{i}{\childi{j}{w}}$,
we get
\beq \label{sib-1}
\mbox{card}(\mathsf{sib}_i(v)) \mbox{card}(\mathsf{sib}_j(\parenti{i}{v})) =  \mbox{card}(\mathsf{sib}_j(v)) \mbox{card}(\mathsf{sib}_i(\parenti{j}{v})).
\eeq
We now check the commutativity of $S_{\lambdab_{\bf c}}$. 
Note that for $w \in \childi{j}{v}$, $\lambda^{(j)}_w$ can be rewritten as $\lambda^{(j)}_w = \sqrt{\frac{c(|\alpha_w|-1, j)}{\mbox{card}(\mathsf{sib}_j(w))}}$. Now for $u \in \mathsf{Chi}_i \childi{j}{v}$, we have
\beqn
\lambda^{(i)}_u \lambda^{(j)}_{\parenti{i}{u}} &=& \sqrt{\frac{c(|\alpha_u|-1, i)}{\mbox{card}(\mathsf{sib}_i(u))}} \sqrt{\frac{c(|\alpha_{\parenti{i}{u}}|-1, j)}{\mbox{card}(\mathsf{sib}_j(\parenti{i}{u}))}} \\ &=& \sqrt{\frac{c(|\alpha_u|-1, i)}{\mbox{card}(\mathsf{sib}_i(u))}} \sqrt{\frac{c(|\alpha_{{u}}|-2, j)}{\mbox{card}(\mathsf{sib}_j(\parenti{i}{u}))}},\\
\lambda^{(j)}_u \lambda^{(i)}_{\parenti{j}{u}} &=& \sqrt{\frac{c(|\alpha_u|-1, j)}{\mbox{card}(\mathsf{sib}_j(u))}} \sqrt{\frac{c(|\alpha_{\parenti{j}{u}}|-1, i)}{\mbox{card}(\mathsf{sib}_i(\parenti{j}{u}))}}\\ &=& \sqrt{\frac{c(|\alpha_u|-1, j)}{\mbox{card}(\mathsf{sib}_j(u))}} \sqrt{\frac{c(|\alpha_{{u}}|-2, i)}{\mbox{card}(\mathsf{sib}_i(\parenti{j}{u}))}}.
\eeqn
This together with Proposition \ref{shift-prop}(i), \eqref{c(t, i)} and \eqref{sib-1} shows that $S_{\lambdab_{\bf c}}$ is commuting. Further,  
\beqn
\|S_j e_v \|^2 = \sum_{w \in \childi{j}{v}} (\lambda^{(j)}_w)^2 = \sum_{w \in \childi{j}{v}} \frac{c(|\alpha_v|, j)}{\mbox{card}(\childi{j}{v})} = c(|\alpha_v|, j),
\eeqn
which is a function of $|\alpha_v|$. This shows that $S_{\lambdab_{\bf c}}$ is torally balanced. The last assertion is immediate from the above equality.
\end{proof}
\begin{remark} \label{S-L-C-comm}
Assume that $S_{\lambdab_{\bf c}}$ is toral left invertible.
Then  $S^{\mf t}_{\lambdab_{\bf c}}$ is commuting. Indeed, the weights of $S^{\mf t}_{\lambdab_{\bf c}}$ are $$\Big\{\frac{1}{\sqrt{c(|\alpha_v|, j)}}{\frac{1}{\sqrt{\mbox{card}(\mathsf{Chi}_j(v))}}} : w \in \childi{j}{v}, ~v \in V~\mbox{and~} j=1, \cdots, d \Big\}.$$ By arguing as above, it can be seen that $S^{\mf t}_{\lambdab_{\bf c}}$ is commuting.
\end{remark}

\begin{example}
Let $\mathscr T = (V,\mathcal E)$ be the directed Cartesian product of locally finite rooted directed trees $\mathscr T_1, \cdots, \mathscr T_d$.
For positive numbers $a, b \in \mathbb R,$ let
\beqn
c_{a, b}(t, j) = \frac{t + b}{t+a}~(t \in \mathbb N, ~j=1, \cdots, d).
\eeqn
Then $S_{\lambdab_{\bf c}}$ is a torally balanced multishift on $\mathscr T$. In case $\mathscr T$ is the $d$-fold directed Cartesian product of $\mathscr T_{1, 0}$ with itself,  the choice $a=1=b$ yields the Cauchy $d$-shift. In case $b=1,$ we denote $c_{a, b}$ by a simpler notation $c_a.$
\end{example}
%\begin{remark}
%It can be seen that the multiplication $d$-tuple $\mathscr M_z$ on the Bergman space of the unit polydisc is not torally balanced. 
%\end{remark}

In dimension $d=1,$ the multishifts $S_{\lambdab_{\bf c}}$ with $b=1$ can be realized as multiplication operators on reproducing kernel Hilbert spaces. Indeed, these shifts can be looked upon as tree counter-part of Agler-type shifts \cite{Ag}.  
This is made precise in the following proposition.

\begin{proposition} \label{S-c-a-kernel-dim1}
Let $\mathscr T = (V, \mathcal E)$ be a locally finite rooted directed tree of finite branching index. For a positive integer $a,$ let $S_{\lambda_{c_a}}$ denote the weighted shift on $\mathscr T$ with weights given by \index{$S_{\lambda_{c_a}}$}
\beqn
\lambda_u = \sqrt{\frac{\alpha_v+1}{\alpha_v + a}} \frac{1}{\sqrt{\mbox{card}(\child{v})}}~\mbox{for~}u \in \child{v},~v \in V,
\eeqn 
where $\alpha_v$ denotes the depth of $v$ in $\mathscr T$. Then $S_{ \lambda_{c_a}}$ is unitarily equivalent to the multiplication operator $\mathscr M_{z, a}$ on a reproducing kernel \index{$\mathscr H_a$}
Hilbert space $\mathscr H_a$ of $E$-valued holomorphic functions on unit disc $\mathbb D$, where $E:=\ker S^*_{\lambda_{c_a}}$ (see \eqref{formula-k}). Moreover, the reproducing kernel $\kappa_{\mathscr H_a}$ associated with $\mathscr H_a$ 
\index{$\kappa_{\mathscr H_a}$}
is given by 
\beqn
\kappa_{\mathscr H_a}(z, w) &=& \sum_{n=0}^{\infty}{n+a-1 \choose n}~ {z^n \overline{w}^n} ~P_{[e_\rootb]} \\  &+& 
\sum_{v \in V_{\prec}} \sum_{n=0}^{\infty}  \frac{(\alpha_v + n + a)! (\alpha_v +1)!}{(\alpha_v +a)! (\alpha_v + n+1)!} {z^n \overline{w}^n}~P_{l^2(\child{v}) \ominus [\Gamma_v]}~(z, w \in \mathbb D),
\eeqn
where $P_{\mathcal M}$ denotes the orthogonal projection of $\mathcal H$ onto a subspace $\mathcal M$ of $\mathcal H$.
%where $\omega$ is arbitrarily chosen but fixed child of $v \in V_\prec$ (see \eqref{formula-k}).
%\beqn
%\kappa_{\mathscr{H}_a}(z,w) = \sum_{k=0}^\infty D_{k, a} z^k \overline{w}^k  \quad  (z, w \in \mathbb{D}),
%\eeqn
%where $D_{k, a}$ is the block diagonal operator on $E :=\ker S^*_{\lambdab, c_a}$ 
%with diagonal entries ${k+a-1 \choose k}$ (corresponding to $\rootb$), and $\beta_k(v, a)$ (corresponding to $l^2(\child{v}) \ominus [\lambda^v]$, $v \in V_{\prec}$) given by
%\beq \label{beta-k-v-a}
%\beta_{k}(v,a) :=
%\frac{(\alpha_v + k + a)! (\alpha_v +1)!}{(\alpha_v +a)! (\alpha_v + k+1)!}~\mbox{for every~}k \in \mathbb N.
%\eeq
%\beqn
%\kappa_{a}(z,w) = \sum_{k=0}^\infty C_{k, a}\, z^k \overline{w}^k  \quad  (z, w \in \mathbb{D}_{r_\lambda})
%\eeqn
%%\beq \label{repkernel-1-dim} 
%%\kappa_{\mathscr{H}}(z,w) = A\sum_{k=0}^\infty  z^k \overline{w}^k + B \sum_{k=0}^\infty(k+1)  z^k \overline{w}^k  \quad  (z, w \in \mathbb{D}_{r_\lambda}),
%%\eeq
%where $C_{k, a}$ is a bounded linear operator on $E=\ker S^*_{\lambda, c_a}$ given by
%\beqn
%C_{k, a}f = 
%\eeqn
%\beqn
%\beta_{k,k}(v,a) = \frac{(\alpha_v + k + a -1)! \alpha_v !}{(\alpha_v +a -1)! (\alpha_v + k)!}.
%\eeqn
%where \beqn A\Big(a\, e_{\mathsf{root}} \bigoplus_{v \in  V_{\prec}} f_v \Big) &=& \bigoplus_{v \in  V_{\prec}} \frac{\alpha_v + 1}{\alpha_v + 2} f_v, \\     B\Big(a\, e_{\mathsf{root}} \bigoplus_{v \in  V_{\prec}} f_v\Big) &=& a\, e_{\mathsf{root}} \bigoplus_{v \in  V_{\prec}} \frac{1}{\alpha_v + 2} f_v,
%\eeqn
%$a \in \mathbb{C}$ and $f_v = \sum_{w \in \child{v}} f(w)e_w.$
\end{proposition}

\begin{proof}
Note that
\beqn
 \inf_{v \in V}\sum_{u \in \child{v}} \lambda^2_u = \inf_{v \in V}\frac{\alpha_v+1}{\alpha_v + a} =\frac{1}{a}.
\eeqn
It now follows from Proposition \ref{shift-prop}(vii) that $S_{\lambda_{c_a}}$ is left-invertible. Hence the first part follows from Remark \ref{ex-k-c} and Theorem \ref{model}. To see the remaining part, we need the following identity:
\beq \label{card-ide}
\sum_{u \in \childn{k}{v}} \prod_{l=0}^{k-1} \frac{1}{\mbox{card} (\sib{\parentn{l}{u}})} = 1~\mbox{for~}v \in V~\mbox{and~}k \Ge 1.
\eeq
We prove this by induction on integers $k \Ge 1$. For $k = 1$, the identity \eqref{card-ide} follows from the fact that $\mbox{card}(\child{v}) = \mbox{card}(\sib{u})$, where $u \in \child{v}$. Suppose that \eqref{card-ide} holds for some integer $k \Ge 1$. Then
\beqn
&&\sum_{u \in \childn{k+1}{v}} \prod_{l=0}^{k} \frac{1}{\mbox{card} (\sib{\parentn{l}{u}})} \\ &=& \sum_{\eta \in \childn{k}{v}} \sum_{u \in \child{\eta}}  \prod_{l=0}^{k} \frac{1}{\mbox{card} (\sib{\parentn{l}{u}})} \\ 
\\ &=& \sum_{\eta \in \childn{k}{v}} \sum_{u \in \child{\eta}}  \frac{1}{\mbox{card} (\sib{{u}})} \prod_{l=1}^{k} \frac{1}{\mbox{card} (\sib{\parentn{l-1}{\parent{u}}})} \\ 
&=&  \sum_{\eta \in \childn{k}{v}} \left( \sum_{u \in \child{\eta}}  \frac{1}{\mbox{card} (\sib{u})} \right) \prod_{l=0}^{k-1}  \frac{1}{\mbox{card} (\sib{\parentn{l}{\eta}})}\\  
&=&\sum_{\eta \in \childn{k}{v}} \prod_{l=0}^{k-1} \frac{1}{\mbox{card}(\sib{\parentn{l}{\eta}})} =1,
\eeqn
where the last equality  follows from the induction hypothesis.

Note that the weights of $S_{\lambda_{c_a}}$ can be rewritten as\beqn 
\lambda_v = \sqrt{\frac{\alpha_v}{\alpha_v + a-1}} \frac{1}{\sqrt{\mbox{card}(\sib{v})}}~\mbox{for~}v \in V^{\circ}.
\eeqn 
Let $S_{\lambda'_{c_a}}$ denote the Cauchy dual of $S_{\lambda_{c_a}}$. By Lemma \ref{Cauchy-dual}, the weights $\lambda'$ of $S_{\lambda'_{c_a}}$ are given by
\beqn \label{weights-dual}
\lambda'_v = \sqrt{\frac{\alpha_v + a - 1}{\alpha_v}} \frac{1}{\sqrt{\mbox{card}(\sib{v})}} ~ \mbox{for all}~ v \in V^\circ.
\eeqn
Now for $v \in V$ and $j,k \Ge 1$, an application of parts (iv) and (v) of Proposition \ref{shift-prop} yields
\beqn \label{power-S-lambda}
S^k_{\lambda'_{c_a}} e_v &=& \displaystyle \sqrt{\frac{(\alpha_v + k + a -1)! \alpha_v !}{(\alpha_v +a -1)! (\alpha_v + k)!}} \sum_{u \in \childn{k}{v}}  \prod_{l=0}^{k-1} \frac{1}{\sqrt{\mbox{card} (\sib{\parentn{l}{u}})}} e_u, \\
S^{*j}_{\lambda'_{c_a}} e_v  &=& \sqrt{\frac{(\alpha_v +a-1)!(\alpha_v -j)!}{(\alpha_v +a - j -1)! \alpha_v !}} ~\prod_{i=0}^{j-1} \frac{1}{\sqrt{\mbox{card} (\sib{\parentn{i}{v}})}} e_{\parentn{j}{v}}.\quad
\eeqn
For a positive integer $i$ and $v \in V,$ set $s_{i, v}:=\mbox{card} (\sib{\parentn{i}{v}})$. Also, for positive integers $j, k$ and $v \in V$ such that $\parentn{j-k}{v}$ is nonempty, set
\beqn
\beta_{j, k}(v, a) &:=& \prod_{i=0}^{j-k-1} \frac{1}{\sqrt{s_{i, v}}} \sqrt{\frac{(\alpha_v + k + a -1)! \alpha_v !}{(\alpha_v +a -1)! (\alpha_v + k)!}} \\ &\times& \sqrt{\frac{(\alpha_v +k +a-1)!(\alpha_v +k -j)!}{(\alpha_v +k +a - j -1)! (\alpha_v +k) !}}.
\eeqn
%Note that for every $v \in V,$ $\beta_{j, k}(\cdot, a)$ is constant on $\child{v}.$
Let $v \in V$ and $j \Ge k$. It is easily seen that if $\parentn{j-k}{v}$ is empty, then
$S^{*j}_{\lambda'_{c_a}} S^k_{\lambda'_{c_a}} e_v =0.$ Otherwise
\beqn
S^{*j}_{\lambda'_{c_a}} S^k_{\lambda'_{c_a}} e_v &=& \sqrt{\frac{(\alpha_v + k + a -1)! \alpha_v !}{(\alpha_v +a -1)! (\alpha_v + k)!}} \sum_{u \in \childn{k}{v}}  \prod_{l=0}^{k-1} \frac{1}{\sqrt{s_{l, u}}} S^{*j}_{\lambda'_{c_a}}e_u\\
&=& \sqrt{\frac{(\alpha_v + k + a -1)! \alpha_v !}{(\alpha_v +a -1)! (\alpha_v + k)!}} \sum_{u \in \childn{k}{v}}  \prod_{l=0}^{k-1} \frac{1}{\sqrt{s_{l, u}}}\\ &&\sqrt{\frac{(\alpha_u +a-1)!(\alpha_u -j)!}{(\alpha_u +a - j -1)! \alpha_u !}} ~\prod_{i=0}^{j-1} \frac{1}{\sqrt{s_{i, u}}} e_{\parentn{j}{u}}\\
&=& \beta_{j, k}(v, a) \sum_{u \in \childn{k}{v}} \prod_{l=0}^{k-1} \frac{1}{s_{l, u}} e_{\parentn{j-k}{v}}\\
&=&  \beta_{j,k}(v,a) e_{\parentn{j-k}{v}},
\eeqn
where the last equality follows from \eqref{card-ide}.

Let $E$ be the kernel of $S^*_{\lambda, c_a}$ and let
$f \in E$. Note that $f$ takes the form
$f = f_{\rootb} + \sum_{v \in  V_{\prec}} f_v $, 
where $f_{\rootb} = \gamma e_{\mathsf{root}}$ for some $\gamma \in \mathbb C$ and $f_v = \sum_{u \in \child{v}} f(u)e_u$ such that $\sum_{u \in \child{v}} f(u) \lambda_u = 0~\mbox{for~} v \in V_\prec.$
Since $\lambda_u$ is constant on $\child{v}$, we obtain that \beq\label{sum-0}  \sum_{u \in \child{v}} f(u)  = 0~\mbox{for all~} v \in V_\prec. \eeq 
%Also, note that $\sum_{w \in \child{v}} f(w)\lambda_w = 0$ for all $v \in V_\prec$. Since $\lambda_w$ is constant on siblings, it follows that $\sum_{w \in \child{v}} f(w) = 0$. 
It follows that for $v \in V_\prec$ and $j >k$, 
\beqn
S^{*j}_{\lambda'_{c_a}} S^k_{\lambda'_{c_a}} \sum_{u \in \child{v}} f(u) e_u &=& \sum_{u \in \child{v}} f(u) S^{*j}_{\lambda'_{c_a}} S^k_{ \lambda', c_a} e_u \\
&=& \sum_{u\in \child{v}} f(u) \beta_{j,k}(u, a) e_{\parentn{j-k}{u}}\\
%&=& \beta_{j,k}(v,a) \Big( \sum_{w \in \child{v}} f(w) \Big) e_{\parentn{j-k}{w}} 
&=& 0,
\eeqn
where we used \eqref{sum-0} and the fact that $\beta_{j,k}(u,a)$ is constant on $\child{v}$.
Almost the same calculations show that 
\beqn \label{beta-k-v-a}
S^{*k}_{\lambda'_{c_a}} S^k_{\lambda'_{c_a}} f_v =
\frac{(\alpha_v + k + a)! (\alpha_v +1)!}{(\alpha_v +a)! (\alpha_v + k+1)!}~f_v \quad \mbox{for every~}k \in \mathbb N.
\eeqn
It may now be concluded from \eqref{kappa} that the reproducing kernel $\kappa_{\mathscr H_a}$ takes the required form. Finally, since $E$ is finite dimensional (Corollary \ref{dimE-finite}), we note that for every $w \in \mathbb D$, $\kappa_{\mathscr H_a}(\cdot, w)$ is a sum of finitely many power series in $z$ converging on the unit disc.
\end{proof}

We now classify all torally balanced multishifts. For this, we must understand their moments.
\begin{lemma}
Let $\mathscr T = (V,\mathcal E)$ be the directed Cartesian product of rooted directed trees $\mathscr T_1, \cdots, \mathscr T_d$. Let $S_{\lambdab}$ be a toral left invertible multishift on $\mathscr T$ with commuting toral Cauchy dual $S^{\mf t}_{\lambdab}$  
and let $U_{\thetab}$ be the toral isometry part of $S_{\lambdab}$ governed by \eqref{t-p-d}. If
$S_{\lambdab}$ is torally balanced, 
then, for any $v \in V$ and $\alpha \in \mathbb N^d,$
\beq \label{moment-t}
S^{\alpha}_{\lambdab} e_v = \prod_{j=1}^d \Big(\prod_{k=0}^{\alpha_j-1}c^{(j)}_{|\alpha_v| + \sum_{i=1}^{j-1}\alpha_i + k}\Big)  U^{\alpha}_{\thetab}e_v,
\eeq
where $c^{(j)}_{|\alpha_v|}$ denotes the constant value of $\mf C_j(v)$ for $v \in V$ and $j=1, \cdots, d.$
In this case,
\beq \label{moment-t-1}
\|S^{\alpha}_{\lambdab} e_v\|^2 = \prod_{j=1}^d \Big(\prod_{k=0}^{\alpha_j-1}c^{(j)}_{|\alpha_v| + \sum_{i=1}^{j-1}\alpha_i + k}\Big)^2 ~(v \in V, ~\alpha \in \mathbb N^d).
\eeq
%where the inner most product is understood to be $1$ if $\alpha_j=0$ and the sum $\sum_{i=1}^{j-1}\alpha_i$ is understood to be $0$ if $j=1.$
\end{lemma}
\begin{proof}
Assume that $S_{\lambdab}$ is torally balanced.
Let us first verify $$S^n_je_v = \Big(\prod_{k=0}^{n-1}c^{(j)}_{|\alpha_v| + k}\Big)U^n_je_v~(n \geqslant  1, ~j=1, \cdots, d)$$  by induction on integers $n \geqslant  1.$ Note that
\beq \label{eq-t-b}
S_j e_v &=& \sum_{w \in \childi{j}{v}}\lambda^{(j)}_w e_w = \sum_{w \in \childi{j}{v}}\frac{\lambda^{(j)}_w}{\mf C_{j}({\parenti{j}{w}})}\mf C_{j}({\parenti{j}{w}})  e_w \nonumber \\ 
&\overset{\eqref{t-theta}}=& \mf C_{j}(v) \sum_{w \in \childi{j}{v}}{\theta^{(j)}_w}  e_w = 
c^{(j)}_{|\alpha_v|} U_j e_v.
\eeq
This verifies the induction statement for $n=1.$ Assume the induction hypothesis for $n \geqslant  1,$ and consider
\beqn
S^{n+1}_je_v &=& \prod_{k=0}^{n-1}c^{(j)}_{|\alpha_v| + k}S_jU^n_je_v = \prod_{k=0}^{n-1}c^{(j)}_{|\alpha_v| + k }S_j \sum_{w \in \childki{j}{n}{v}}{\theta^{(j)}_w \cdots \theta^{(j)}_{\parentki{j}{n-1}{w}}}  e_w \\ &\overset{\eqref{eq-t-b}}=& \prod_{k=0}^{n-1}c^{(j)}_{|\alpha_v| + k} \sum_{w \in \childki{j}{n}{v}}{\theta^{(j)}_w \cdots \theta^{(j)}_{\parentki{j}{n-1}{w}}}  c^{(j)}_{|\alpha_w|} U_j e_w \\
&=& \prod_{k=0}^{n}c^{(j)}_{|\alpha_v| + k} \sum_{w \in \childki{j}{n}{v}}{\theta^{(j)}_w \cdots \theta^{(j)}_{\parentki{j}{n-1}{w}}}  U_j e_w ~ (\mbox{since}~\alpha_w = \alpha_v + n \epsilon_j) \\
&=& \Big(\prod_{k=0}^{n}c^{(j)}_{|\alpha_v| + k}\Big)U^{n+1}_je_v.
\eeqn 
This completes the inductive argument. A similar inductive argument on $m \geqslant  1$ together with the commutativity of $S^{\mf t}_{\lambdab}$ yields 
\beqn
S^m_iS^n_je_v = \prod_{k=0}^{m-1}c^{(i)}_{|\alpha_v| + k} \prod_{k=0}^{n-1}c^{(j)}_{|\alpha_v| + m + k}U^m_iU^n_je_v
\eeqn
for $m, n \in \mathbb N$ and $1 \leqslant i, j \leqslant d$.
The desired conclusion may now be deduced from the above identity by a finite inductive argument. Finally, we note that \eqref{moment-t-1} follows from \eqref{moment-t} and the fact that $U_{\thetab}$ is a toral isometry.
\end{proof}

\begin{definition}
Let ${\bf c}:=\{c(t, j) : t \in \mathbb N, ~j=1, \cdots, d\}$ be a bounded multisequence of positive real numbers such that \eqref{c(t, i)} holds.
For $s \in \mathbb N$, set
\beqn
\gamma_{\alpha, s} :=\prod_{j=1}^d 
\prod_{k=0}^{\alpha_j-1}c\big({s + \alpha_1 + \cdots + \alpha_{j-1} + k }, j\big)~(\alpha \in \mathbb N^d).
\eeqn
We refer to the Hilbert space of formal series $H^2(\gamma_s)$ as {\it the Hilbert space associated with $\bf c$}. 
\end{definition}
\begin{remark}
Note that the multiplication $d$-tuple on $H^2(\gamma_s)$ is a commuting $d$-tuple of bounded linear operators $M_{z_1}, \cdots, M_{z_d}$.
\end{remark}

We are now in a position to present the main result of this section.
\begin{theorem}
Let $\mathscr T = (V,\mathcal E)$ be the directed Cartesian product of rooted directed trees $\mathscr T_1, \cdots, \mathscr T_d$. For $v \in V$, 
let $$f = \sum_{\beta \in \mathbb N^d}a_\beta S^{\beta}_{\lambdab} e_v \in l^2(V), \quad \tilde{f}(w)=\sum_{\beta \in \mathbb N^d}a_\beta w^{\beta}.$$
Let $S_{\lambdab}$ be a toral left invertible multishift on $\mathscr T$ with commuting toral Cauchy dual $S^{\mf t}_{\lambdab}$. 
Then $S_{\lambdab}$ is a torally balanced multishift on $\mathscr T$ if and only if for every $v \in V,$ there exists a Hilbert space $H^2(\gamma_{|\alpha_v|})$ of formal power series in the variables $w_1, \cdots, w_d$ associated with a bounded multisequence $\bf c$ such that $$\|f\|_{l^2(V)} = \|\tilde{f}\|_{H^2(\gamma_{|\alpha_v|})}.$$ 
%where
%\beqn
%\gamma_{\alpha, v} :=\prod_{j=1}^d 
%\Big(\prod_{k=0}^{\alpha_j-1}c^{(j)}_{|\alpha_v|+ \sum_{i=1}^{j-1}\alpha_i + k }\Big)
%\eeqn
%for a bounded multi-sequence $\{c^{(j)}_t : {t \in \mathbb N}, ~j=1, \cdots, d\}$.
\end{theorem}
\begin{proof}
Suppose that $S_{\lambdab}$ is a torally balanced multishift.
Set $c(t, j):=c^{(j)}_t$ for $t \in \mathbb N$ and $j=1, \cdots, d,$
where $c^{(j)}_{t}$ denotes the constant value of $\mf C_j(v)$ on the generation $\mathcal G_t$.
Consider the Hilbert space $H^2(\gamma_{|\alpha_v|})$ of formal power series in $w_1, \cdots, w_d$.
By taking norms on both sides of \eqref{moment-t}, we obtain for every $\alpha \in \mathbb N^d,$
\beq \label{t-b-mono}
\|S^{\alpha}_{\lambdab} e_v\|_{l^2(V)} &=& \prod_{j=1}^d \Big(\prod_{k=0}^{\alpha_j-1}c^{(j)}_{|\alpha_v| + \sum_{i=1}^{j-1}\alpha_i + k}\Big)  \|U^{\alpha}_{\thetab}e_v\|_{l^2(V)} \notag \\ &=& \gamma_{\alpha, v} = \|w^{\alpha}\|_{H^2(\gamma_{|\alpha_v|})},
\eeq 
where we used that $U_{\theta}$ is a toral isometry.
By orthogonality of $\{S^{\alpha}_{\lambdab}e_v\}_{\alpha \in \mathbb N^d}$ (Proposition \ref{shift-prop}(ix)), the above formula holds for all pairs $f$ and $\tilde{f}.$ To see the converse, 
let $f=S_je_v$ and $\tilde{f}=w_j$ in $\|f\|_{l^2(V)} = \|\tilde{f}\|_{H^2(\gamma_{|\alpha_v|})}$ to obtain
 $\|S_je_v\|= c({|\alpha_v|}, j)$, which is clearly constant on $\mathcal G_{|\alpha_v|}.$ 
\end{proof}

%Recall that a bounded linear operator $T$ on $\mathcal H$ is a { contraction} if $T^*T \leqslant I.$

Here we present a local analog of von Neumann's inequality for torally balanced multishifts. 
\begin{corollary}
Let $\mathscr T = (V,\mathcal E)$ be the directed Cartesian product of rooted directed trees $\mathscr T_1, \cdots, \mathscr T_d$ and let $S_{\lambdab}$ be a toral left invertible, torally balanced multishift with commuting toral Cauchy $S^{\mf t}_{\lambdab}$. If $S_{\lambdab}$ is a toral contraction, then for any positive integer $k$ and scalars $a_{\beta}, ~|\beta| \leqslant k$, 
\beqn
\sup_{v \in V}\Big \|\sum_{|\beta| \leqslant k}a_\beta S^{\beta}_{\lambdab}e_v\Big\| \leqslant 
\sup_{z \in \mathbb D^d}\Big|\sum_{|\beta| \leqslant k}a_\beta z^{\beta}\Big|.
\eeqn
\end{corollary}
\begin{proof}
Assume that $S_{\lambdab}$ is a toral contraction.
Fix $v \in V.$ By the preceding theorem, 
there exists a Hilbert space $H^2(\gamma_{|\alpha_v|})$ of formal power series in the variables $w_1, \cdots, w_d$ such that $$\Big\|\sum_{|\beta| \leqslant k}a_\beta S^{\beta}_{\lambdab}e_v\Big\|_{l^2(V)} = \Big\|\sum_{|\beta| \leqslant k}a_\beta w^{\beta}\Big\|_{H^2(\gamma_{|\alpha_v|})}=\Big\|\sum_{|\beta| \leqslant k}a_\beta M^{\beta}_w1\Big\|_{H^2(\gamma_{|\alpha_v|})},$$
where $M_w$ denotes the $d$-tuple of operators of multiplication by the coordinate functions $w_1, \cdots, w_d$. However, since $S_{\lambdab}$ is a toral contraction, by \eqref{t-b-mono}, so is $M_{w}.$
It follows from $\|1\|_{H^2(\gamma_{|\alpha_v|})}=1$ that 
$$\Big\|\sum_{|\beta| \leqslant k}a_\beta S^{\beta}_{\lambdab}e_v\Big\|_{l^2(V)} \leqslant \Big\|\sum_{|\beta| \leqslant k}a_\beta M^{\beta}_w\Big\|\|1\|_{H^2(\gamma_{|\alpha_v|})}=\Big\|\sum_{|\beta| \leqslant k}a_\beta M^{\beta}_w\Big\|.$$ By a result of M. Hartz \cite[Theorem 1.1]{Hz},
von Neumann's inequality holds for any torally contractive classical multishift. Hence $$\Big\|\sum_{|\beta| \leqslant k}a_\beta S^{\beta}_{\lambdab}e_v\Big\|_{l^2(V)} \leqslant \sup_{z \in \mathbb D^d}\Big|\sum_{|\beta| \leqslant k}a_\beta z^{\beta}\Big|.$$ Taking supremum over $v \in V$ on the left hand side,  we get the desired inequality.
\end{proof}

\section{Spherically Balanced Multishifts}

\begin{definition}
Let $\mathscr T = (V,\mathcal E)$ be the directed Cartesian product of rooted directed trees $\mathscr T_1, \cdots, \mathscr T_d$ and
let $S_{\lambdab}=(S_1, \cdots, S_d)$ be a commuting multishift on $\mathscr T$. 
Define the function $\mf C : V \rar (0, \infty)$ by
\beq \label{constant-gen} \mf C(v) := \sum_{j=1}^d \|S_j e_v\|^2~\mbox{for~}v \in V. \eeq
We say that $S_{\lambdab}$ is {\it spherically balanced} if 
$\mf C$ is \index{$\mf C$}
constant on every generation $\mathcal G_t$, $t \in  \mathbb N$. 
We denote the constant value of $\mf C(v)$ 
\index{$\mf C_{|\alpha_v|}$}
by $\mf C_{|\alpha_v|}$, where $\alpha_v$ is the depth of $v$ in $\mathscr T$.
%where
%$$\mathcal G_t := \big\{v \in V : v \in \childn{t}{\rootb}\big\}.$$
\end{definition}
\begin{remark}
Note that $S_{\lambdab}$ is a joint isometry if and only if 
\beqn \sum_{j=1}^d \|S_je_v\|^2 = \sum_{j=1}^d \sum_{w \in \childi{j}{v}} (\lambda_{w}^{(j)})^2 = 1~\mbox{for all~}v \in V.\eeqn
It is now clear that every joint isometry multishift is spherically balanced with $\mf C$ being the constant function $1.$ 
\end{remark}

The following proposition yields examples of spherically balanced multishifts apart from joint isometries.

\begin{proposition} 
Let $\mathscr T = (V,\mathcal E)$ be the directed Cartesian product of locally finite rooted directed trees $\mathscr T_1, \cdots, \mathscr T_d$ and let $\{c_{t}\}_{t \in \mathbb N}$ be a bounded sequence of positive real numbers. 
Consider the multishift \index{$S_{\lambdab_{\mf C}}$}
$S_{\lambdab_{\mf C}}=(S_1, \cdots, S_d)$ with weights
\beqn
\lambda^{(i)}_w = \sqrt{\frac{c_{|\alpha_v|}}{{\mbox{card}(\childi{i}{v})}}} \sqrt{\frac{\alpha_{v_i}  + 1}{|\alpha_v| + d}}~\mbox{for~}w \in \childi{i}{v},~v \in V\mbox{~and~}i=1, \cdots, d.
\eeqn
Then $S_{\lambdab_{\mf C}}$ defines a spherically balanced multishift. In case $c_t=1$ for all $t \in \mathbb N,$ $S_{\lambdab_{\mf C}}$ is a joint isometry.
\end{proposition}

\begin{proof}
One may conclude from Lemma \ref{bddness}(i) that $S_1, \cdots, S_d$ are bounded linear operators on $l^2(V)$ whenever the sequence $\{c_t\}_{t \in \mathbb N}$ is bounded.
Let $v \in V$ and $i, j =1, \cdots, d.$
First we check the commutativity of $S_{\lambdab_{\mf C}}$. 
Note that for $w \in \childi{i}{v}$, $\lambda^{(i)}_w$ can be rewritten as $$\lambda^{(i)}_w = \sqrt{\frac{c_{|\alpha_w|-1}}{{\mbox{card}(\mathsf{sib}_i(w))}}} \sqrt{\frac{\alpha_{w_i} }{|\alpha_w| + d-1}}.$$ Now for $u \in \mathsf{Chi}_i \childi{j}{v}$, we have
\beqn
\lambda^{(i)}_u \lambda^{(j)}_{\parenti{i}{u}} &=& \sqrt{\frac{c_{|\alpha_u|-1}}{{\mbox{card}(\mathsf{sib}_i(u))}}} \sqrt{\frac{\alpha_{u_i} }{|\alpha_u| + d-1}} \\ & \times & \sqrt{\frac{c_{|\alpha_u|-2}}{{\mbox{card}(\mathsf{sib}_j(\parenti{i}{u}))}}} \sqrt{\frac{\alpha_{\parenti{i}{u}_j} }{|\alpha_{\parenti{i}{u}}| + d-1}},\\
\lambda^{(j)}_u \lambda^{(i)}_{\parenti{j}{u}} &=& \sqrt{\frac{c_{|\alpha_u|-1}}{{\mbox{card}(\mathsf{sib}_j(u))}}} \sqrt{\frac{\alpha_{u_j}}{|\alpha_u| + d-1}} \\ & \times & \sqrt{\frac{c_{|\alpha_u|-2}}{{\mbox{card}(\mathsf{sib}_i(\parenti{j}{u}))}}} \sqrt{\frac{\alpha_{\parenti{j}{u}_i}}{|\alpha_{\parenti{j}{u}}| + d-1}}.
\eeqn
Since $\alpha_{u_i} =\alpha_{\parenti{j}{u}_i}$ for $i \neq j$, by \eqref{sib-1}, $S_{\lambdab_{\mf C}}$ is commuting. Note further that
\beqn
\sum_{j=1}^d \|S_j e_v\|^2 &=& \sum_{j=1}^d \sum_{w \in \childi{j}{v}} (\lambda^{(j)}_w)^2 = \sum_{j=1}^d \sum_{w \in \childi{j}{v}} \frac{c_{|\alpha_v|}}{{\mbox{card}(\childi{j}{v})}} {\frac{\alpha_{v_j}  + 1}{|\alpha_v| + d}} \\
&=& c_{|\alpha_v|}\sum_{j=1}^d {\frac{\alpha_{v_j}  + 1}{|\alpha_v| + d}} = c_{|\alpha_v|}, 
\eeqn
which is a function of $|\alpha_v|$. Thus $S_{\lambdab_{\mf C}}$ is spherically balanced. The above calculation also shows that $S_{\lambdab_{\mf C}}$ is joint isometry if and only if $c_t=1$ for all $t \in \mathbb N$.
\end{proof}
\begin{remark}
The choice $c_{t}=\frac{t+d}{t+1}\frac{t+2}{t+1}~(t \in \mathbb N)$ with $\mathscr T$ being the $d$-fold directed Cartesian product of $\mathscr T_{1, 0}$ with itself yields the Dirichlet $d$-shift on the unit ball \cite[Example 1]{GHX}. 
\end{remark}

We discuss below a family of examples of spherically balanced multishifts, which is a tree analog of the multiplication $d$-tuples on the reproducing kernel Hilbert spaces associated with the reproducing kernels $\frac{1}{(1 -\inp{z}{w})^a}$ defined on the unit ball in $\mathbb C^d,$ where $a$ is a positive number.
%(cf. \cite[Theorem 2.1]{CY}). 

\begin{example} \label{ex-sp-bal}
Let $\mathscr T = (V,\mathcal E)$ be the directed Cartesian product of locally finite rooted directed trees $\mathscr T_1, \cdots, \mathscr T_d$.
For a positive real number $a,$ consider the sequence $\{c_{a, t}\}_{t \in \mathbb N}$ given by $$c_{a, t}=\frac{t+d}{t+a}~(t \in \mathbb N).$$ 
Then the multishift $S_{\lambdab_{\mf C_a}}$ 
\index{$S_{\lambdab_{\mf C_a}}$}
on $\mathscr T$ is spherically balanced. In case $\mathscr T$ is the $d$-fold directed Cartesian product of $\mathscr T_{1, 0}$ with itself then the choices $a=d$, $a=d+1$, $a=1$ yield Szeg$\ddot{\mbox{o}}$ $d$-shift, Bergman $d$-shift, Drury-Arveson $d$-shift respectively on the unit ball (see Example \ref{1.1.1}). 
\end{example}

We refer to the multishifts $S_{\lambdab_{\mf C_a}}$ on $\mathscr T$ as the {\it tree analog of  Szeg$\ddot{\mbox{o}}$ $d$-shift, Bergman $d$-shift, Drury-Arveson $d$-shift} respectively in case $a=d, a=d+1, a=1$. It is worth noting that  $S_{\lambdab_{\mf C_d}}$ and $S_{\lambdab_{\mf C_{d+1}}}$ are joint contractions while $S_{\lambdab_{\mf C_1}}$ is a {\it row contraction} (that is, $S^*_{\lambdab_{\mf C_1}}$ is a joint contraction).

Although there is no known satisfactory counter-part of Shimorin's analytic model for joint left invertible analytic tuples, we are able to show,
as in the classical case, that the multishifts $S_{\lambdab_{\mf C_a}}$ on $\mathscr T$ can be realized as multiplication tuples $\mathscr M_z$ on reproducing kernel Hilbert spaces at least in case the joint kernel $E$ of $S^*_{\lambdab_{\mf C_a}}$ is finite dimensional. Before we make  this precise, we recall from \eqref{j-kernel} that $E$ is given by
\beq \label{j-kernel-1}
E = [e_\rootb] \oplus \bigoplus_{\underset{F \neq \emptyset}{F  \in \mathscr{P}}} \bigoplus_{u \in \Omega_{F}} \mathcal L_{u, F}
\eeq
(refer to Section 4.1 for the definitions of $\Omega_{F}$ and $L_{u, F}$). We record that $\Omega_F$ is finite whenever $V_{\prec}$ is finite. For instance, this happens whenever $E$ is finite dimensional.

\begin{theorem} \label{S-c-a-kernel}
Let $\mathscr T = (V,\mathcal E)$ be the directed Cartesian product of locally finite rooted directed trees $\mathscr T_1, \cdots, \mathscr T_d$ of finite joint branching index. Let $S_{\lambdab_{\mf C_a}}$ be as introduced in Example \ref{ex-sp-bal} and let $E$ denote the joint kernel of $S^*_{\lambdab_{\mf C_a}}$. Then $S_{\lambdab_{\mf C_a}}$ is unitarily equivalent to the multiplication $d$-tuple $\mathscr M_{z, a}=(\mathscr M_{z_1}, \cdots, \mathscr M_{z_d})$ on a reproducing kernel Hilbert space $\mathscr H_{a, d}$ 
\index{$\mathscr H_{a, d}$}
of $E$-valued holomorphic functions defined on the open unit ball $\mathbb B^d$ in $\mathbb C^d$. Further, the reproducing kernel 
\index{$\kappa_{\mathscr H_{a, d}}(z, w)$}
$\kappa_{\mathscr H_{a, d}} : \mathbb B^d \times \mathbb B^d \rar B(E)$ associated with $\mathscr H_{a, d}$ is given by 
\beqn
 \kappa_{\mathscr H_{a, d}}(z, w) = \sum_{\alpha \in \mathbb N^d}  \frac{a(a+1) \cdots (a+|\alpha|-1)}{\alpha!}  z^{\alpha} \overline{w}^{\alpha} ~P_{[e_\rootb]} + \sum_{\underset{F \neq \emptyset}{F  \in \mathscr{P}}} \sum_{u \in \Omega_F} \kappa_{u, F}(z, w),
\eeqn
where $\kappa_{u, F}(z, w)$ is given by
\beq \label{kappa-u-F}
\kappa_{u, F}(z, w)=\sum_{\alpha \in \mathbb N^d}  \frac{\alpha_{u}!}{(\alpha_{u}+\alpha)!}\, {\prod_{j=0}^{|\alpha|-1}(|\alpha_u|+a + j)}\,   z^{\alpha} \overline{w}^{\alpha}~P_{\mathcal L_{u, F}},
\eeq
with $P_{\mathcal M}$ being the orthogonal projection of $\mathcal H$ onto a subspace $\mathcal M$ of $\mathcal H$.
\end{theorem}

The proof of the above theorem, as presented below, turns out to be more involved as compared to that of Proposition \ref{S-c-a-kernel-dim1}.
Perhaps one reason may be that no counter-part of Shimorin's model is known for joint left invertible analytic tuples.
This proof relies on the description of joint kernel $E$ as provided in Chapter 4. 
A key observation required is the following lemma.

\begin{lemma} \label{ortho-powers}
Let $\mathscr T = (V,\mathcal E)$ be the directed Cartesian product of locally finite rooted directed trees $\mathscr T_1, \cdots, \mathscr T_d$ of finite joint branching index. Let $S_{\lambdab_{\mf C_a}}$ be as introduced in Example \ref{ex-sp-bal} and let $E$ denote the joint kernel of $S^*_{\lambdab_{\mf C_a}}$. Then the following are true:
\begin{enumerate}
\item[(i)] $E$ is invariant under $S^{*\alpha}_{\lambdab_{\mf C_a}} S^{\alpha}_{\lambdab_{\mf C_a}}$ and $S^{*\alpha}_{\lambdab_{\mf C_a}} S^{\alpha}_{\lambdab_{\mf C_a}}|_E$ is boundedly invertible for every $\alpha \in \mathbb N^d$. 
\item[(ii)] The multisequence $\{S^{\alpha}_{\lambdab_{\mf C_a}}E\}_{\alpha \in \mathbb N^d}$ of subspaces of $l^2(V)$ is mutually orthogonal.
\end{enumerate}
\end{lemma}
\begin{proof}
(i)
In view of \eqref{j-kernel-1} and $S^{*\alpha}_{\lambdab_{\mf C_a}} S^{\alpha}_{\lambdab_{\mf C_a}}e_v = \|S^{\alpha}_{\lambdab_{\mf C_a}}e_v\|^2e_v~(v \in V)$, it suffices to check that for every $F \in \mathscr P$ and $u \in \Omega_F,$ the function $v \longmapsto \|S^{\alpha}_{\lambdab_{\mf C_a}}e_v\|^2$ is constant on ${\mathsf{sib}_F(u)}$ with value $\|S^{\alpha}_{\lambdab_{\mf C_a}}e_u\|^2$.
For $j=1, \cdots, d$, $w \in V$,  let $\beta (j, w, 0)=1$ and 
\beqn \beta (j, w, n) = \lambda_w^{(j)} \lambda_{\parenti{j}{w}}^{(j)} \cdots \lambda_{\parentki{j}{n-1}{w}}^{(j)}~(n \geqslant  1).\eeqn 
It is easy to see using $\lambda^{(i)}_w= {\frac{1}{\sqrt{\mbox{card}(\mathsf{sib}_{i}(w))}}} \sqrt{\frac{\alpha_{w_i} }{|\alpha_w| + a-1}}$ that
\beq \label{beta-sp-bal} 
\beta(j, w, n)^2 = \left(\prod_{k=0}^{n-1} \frac{1}{\mbox{card}(\mathsf{sib}_j{\parentki{j}{k}{w}})}\right) \frac{\alpha_{w_j}!}{(\alpha_{w_j}-n)!} \frac{(|\alpha_w|-n+a-1)!}{(|\alpha_w|+a-1)!}
\eeq
for $n \Ge 1$ and $j=1, \cdots, d.$
Fix $\alpha \in \mathbb N^d$ and let $1 \Le i_1 < i_2 < \cdots < i_k \Le d$ be positive integers such that 
$\alpha_{i_1}, \alpha_{i_2}, \cdots, \alpha_{i_k}$ are the only nonzero entries in $\alpha.$
A routine verification using Proposition \ref{shift-prop}(vi) and \eqref{beta-sp-bal} shows that
\beqn
S^{*\alpha}_{\lambdab_{\mf C_a}} S^{\alpha}_{\lambdab_{\mf C_a}}e_v &=& \sum_{w \in \childnt{\alpha}{v}} ~\prod_{j=1}^k \big|\beta \big(i_j, \parentnt{\alpha^{(i_j-1)}}{w}, \alpha_{i_j}\big)\big|^2e_v \\
%&=& \sum_{w \in \childnt{\alpha}{v}} ~ \prod_{j=1}^k \frac{\alpha_{w_{i_j}}!}{(\alpha_{w_{i_j}}-\alpha_{i_j})!} \frac{(|\alpha_w|-(\alpha_{i_1} + \cdots +\alpha_{i_j})+a-1)!}{(|\alpha_w| - \sum_{m=1}^{i_{j-1}}\alpha_m + \alpha_{i_j}-1)!} \\ 
&=& \sum_{w \in \childnt{\alpha}{v}} ~\Big( \prod_{j=1}^k \frac{\alpha_{w_{i_j}}!}{(\alpha_{w_{i_j}}-\alpha_{i_j})!} \frac{(|\alpha_{\parentnt{\alpha^{(i_j-1)}}{w}}|-\alpha_{i_j} + a-1)!}{(|\alpha_{\parentnt{\alpha^{(i_j-1)}}{w}}| +a-1)!} \\
&\times & \prod_{j=1}^k \prod_{l=0}^{\alpha_{i_j}-1} \frac{1}{\mbox{card}(\mathsf{sib}_{i_j}(\mathsf{par}^{\langle l \rangle }_{i_j} \mathsf{par}^{\langle \alpha_{i_{j-1}} \rangle}_{i_{j-1}} \cdots \mathsf{par}^{\langle \alpha_{i_1}\rangle}_{i_1}(w)))}\Big)~e_v \\ &=& 
\prod_{j=1}^k \frac{(\alpha_{v_{i_j}}+\alpha_{i_j})!}{\alpha_{v_{i_j}}!} \frac{(|\alpha_v| + \sum_{m=j+1}^k \alpha_{i_{m}} +a-1)!}{(|\alpha_v| + \sum_{m=j}^{k}\alpha_{i_m} + a -1)!} \\ &\times & \sum_{w \in \childnt{\alpha}{v}} ~\prod_{j=1}^k \prod_{l=0}^{\alpha_{i_{j}}-1} \frac{1}{\mbox{card}(\mathsf{sib}_{i_j}(\mathsf{par}^{\langle l \rangle}_{i_j} \mathsf{par}^{\langle \alpha_{i_{j-1}} \rangle}_{i_{j-1}} \cdots \mathsf{par}^{\langle \alpha_{i_1} \rangle}_{i_1}(w)))}~e_v,
\eeqn
where we used that $|\alpha_{\parentnt{\alpha^{(i_j-1)}}{w}}|=|\alpha_w|-\sum_{l=1}^{j-1}\alpha_{i_l},$ $\alpha_w = \alpha_v + \alpha$ for any
$w \in \childnt{\alpha}{v}.$
As in the proof of Proposition \ref{S-c-a-kernel-dim1}, one can verify by induction on $|\alpha| \Ge1$ that
\beqn
\sum_{w \in \childnt{\alpha}{v}} ~\prod_{j=1}^d \prod_{l=0}^{\alpha_j-1} \frac{1}{\mbox{card}(\mathsf{sib}_j(\mathsf{par}^{\langle l \rangle}_j \mathsf{par}^{\langle \alpha_{j-1} \rangle}_{j-1} \cdots \mathsf{par}^{\langle \alpha_1 \rangle}_1(w)))} = 1.
\eeqn
This yields the formula
\beq \label{s-ca-moment-f}
S^{*\alpha}_{\lambdab_{\mf C_a}} S^{\alpha}_{\lambdab_{\mf C_a}}e_v = \prod_{j=1}^k \frac{(\alpha_{v_{i_j}}+\alpha_{i_j})!}{\alpha_{v_{i_j}}!} \frac{(|\alpha_v| + \sum_{m=j+1}^k \alpha_{i_{m}} +a-1)!}{(|\alpha_v| + \sum_{m=j}^{k}\alpha_{i_m} + a -1)!} ~e_v~(v \in V). \quad
\eeq
Since depth is constant on $\mathsf{sib}_F(u)$ and $E$ is finite dimensional (Corollary \ref{dimE-finite}), the desired conclusion in (i) is immediate from this formula. For future reference, we also note the following expression for moments of $S_{\lambdab_{\mf C_a}}$ deduced from \eqref{s-ca-moment-f}:
\beq \label{s-ca-moment-f-2}
\|S^{\alpha}_{\lambdab_{\mf C_a}}e_v\|^2 = \displaystyle \prod_{j=1}^d \frac{(\alpha_{v_{j}}+\alpha_{j})!}{\alpha_{v_{j}}!} \frac{(|\alpha_v| + \sum_{i=j+1}^d \alpha_{i} +a-1)!}{(|\alpha_v| + \sum_{i=j}^{d}\alpha_{i} + a -1)!} \notag \\ 
= \Big(\displaystyle \prod_{j=1}^d \frac{(\alpha_{v_{j}}+\alpha_{j})!}{\alpha_{v_{j}}!} \Big)\frac{1}{(|\alpha_v|+a)(|\alpha_v|+a+1) \cdots (|\alpha_v| + a + |\alpha|-1)}
\eeq
 for all $v \in V.$ 

(ii) It suffices to check that $S^{*\beta_j+1}_jS^{\beta}_{\lambdab_{\mf C_a}}|_E=0$ for $j=1, \cdots, d.$ We will verify this only for $j=d.$ The verification for the rest of the coordinates is invariably the same. Let $f \in E.$ 
Clearly, for $f =e_\rootb$,  we have $S^{*\beta_d+1}_dS^{\beta}_{\lambdab_{\mf C_a}}f=0$. Let $F \in \mathscr P$ such that 
$F \neq \emptyset$ and let $f \in \mathcal L_{u, F}$ for $u \in \Omega_F \subseteq \Phi_F$ (see \eqref{phi-F-eqn} and \eqref{phi-F}). If $d \notin F$ then once again $S^{*\beta_d+1}_dS^{\beta}_{\lambdab_{\mf C_a}}f=0$ since $f \in l^2(\mathsf{sib}_F(u))$ is supported on a subset of $V_1 \times \cdots V_{d-1} \times \{\rootb_d\}.$
Hence we may assume that $d \in F$. Since $f \in \mathcal L_{u, F}$, by \eqref{system-main}, $f=\sum_{v \in \mathsf{sib}_F(u)}f(v)e_v$ satisfies
\beqn
\sum_{w \in \mathsf{sib}_i(v_G | u_i)} f(w) \lambda^{(i)}_w =0,~i \in F~\mbox{and~}v_G \in \mathsf{sib}_{F, G}(u).
\eeqn 
However, since $\lambda^{(i)}_w= {\frac{1}{\sqrt{\mbox{card}(\mathsf{sib}_{i}(w))}}} \sqrt{\frac{\alpha_{w_i} }{|\alpha_w| + a-1}}$ is constant on $w \in \mathsf{sib}_i(v_G | u_i)$, we obtain
\beq \label{system-main-sp}
\sum_{w \in \mathsf{sib}_i(v_G | u_i)} f(w) =0,~i \in F~\mbox{and~}v_G \in \mathsf{sib}_{F, G}(u).
\eeq 
Let $\alpha=\beta - \beta_d \epsilon_d$.
One may now argue as in (i) to see that 
\beq \label{comp-kernel-sb}
&& S^{*\beta_d+1}_dS^{\beta}_{\lambdab_{\mf C_a}}f = S^{*}_d\sum_{v \in \mathsf{sib}_F(u)}f(v)S^{*\beta_d}_d S^{\beta_d}_d S^{\alpha}_{\lambdab_{\mf C_a}}e_v \notag \\ 
&&= \sum_{v \in \mathsf{sib}_F(u)} \sqrt{\prod_{j=1}^d \frac{(\alpha_{v_{j}}+\alpha_{j})!}{\alpha_{v_{j}}!} } \frac{f(v)}{\sqrt{(|\alpha_v|+a)(|\alpha_v|+a+1) \cdots (|\alpha_v| + a + |\alpha|-1)}} \notag \\  && \times  \sum_{w \in \childnt{\alpha}{v}} ~ \prod_{j=1}^d \prod_{l=0}^{\alpha_j-1}{\frac{S^{*}_d \big(S^{*\beta_d}_d S^{\beta_d}_d e_w\big)}{\sqrt{\mbox{card}(\mathsf{sib}_j(\mathsf{par}^{\langle l \rangle}_j \mathsf{par}^{\langle \alpha_{j-1} \rangle}_{j-1} \cdots \mathsf{par}^{\langle \alpha_1 \rangle}_1(w)))}}} \notag \\
&& \overset{\eqref{s-ca-moment-f-2}} =  \sum_{v \in \mathsf{sib}_F(u)}f(v)  \|S^{\alpha}_{\lambdab_{\mf C_a}}e_v\| \sum_{w \in \childnt{\alpha}{v}} ~ \gamma(w, \alpha)~S^{*}_d \big(S^{*\beta_d}_d S^{\beta_d}_d e_w\big), 
\eeq
where $\gamma(w, \alpha)$ is given by
\beq \label{gamma-w-alpha}
\gamma(w, \alpha):=\prod_{j=1}^d \prod_{l=0}^{\alpha_j-1}{\frac{1}{\sqrt{\mbox{card}(\mathsf{sib}_j(\mathsf{par}^{\langle l \rangle}_j \mathsf{par}^{\langle \alpha_{j-1} \rangle}_{j-1} \cdots \mathsf{par}^{\langle \alpha_1 \rangle}_1(w)))}}}.
\eeq
However, by \eqref{s-ca-moment-f} and $\alpha_{w_d}=\alpha_{v_d}$, 
\beqn
S^{*\beta_d}_d S^{\beta_d}_d e_w &=& \frac{(\alpha_{w_d} + \beta_d)!}{\alpha_{w_d}!} \frac{(|\alpha_w| + a-1)!}{(|\alpha_w|+\beta_d + a-1)!} ~e_w\\ &=& \frac{(\alpha_{v_d} + \beta_d)!}{\alpha_{v_d}!} \frac{(|\alpha_v| +|\alpha| + a-1)!}{(|\alpha_v|+ | \alpha| + \beta_d + a-1)!}~e_w.
\eeqn
This combined with \eqref{comp-kernel-sb} yields
\beqn
S^{*\beta_d+1}_dS^{\beta}_{\lambdab_{\mf C_a}}f 
&=& \sum_{v \in \mathsf{sib}_F(u)}f(v)  ~\|S^{\alpha}_{\lambdab_{\mf C_a}}e_v\|~  \frac{(\alpha_{v_d} + \beta_d)!}{\alpha_{v_d}!} \frac{(|\alpha_v| +|\alpha| + a-1)!}{(|\alpha_v|+ | \alpha| + \beta_d + a-1)!} \\ & \times & \sum_{w \in \childnt{\alpha}{v}} ~\gamma(w, \alpha)~S^{*}_d e_w.
\eeqn
Since $w_d=v_d$ and $\alpha_w = \alpha_v + \alpha$, we have
\beqn S^{*}_d e_w &=& {\frac{1}{\sqrt{\mbox{card}(\mathsf{sib}_{d}(w))}}} \sqrt{\frac{\alpha_{w_d} }{|\alpha_w| + a-1}} e_{\mathsf{par}_d(w)} \\ &=& {\frac{1}{\sqrt{\mbox{card}(\mathsf{sib}_{d}(v))}}} \sqrt{\frac{\alpha_{v_d} }{|\alpha_v| + |\alpha| + a-1}} e_{\mathsf{par}_d({w})}. \eeqn
%where $\tilde{F}:= \{1, \cdots, d-1\}$ and $\tilde{w}:=(w_{\tilde{F}}|v_d)$
%(see Definition \ref{dfn4.1.2} for the definition of $(w_{\tilde{F}}|v_d)$).
This gives
\beqn
S^{*\beta_d+1}_dS^{\beta}_{\lambdab_{\mf C_a}}f 
&=& \sum_{v \in \mathsf{sib}_F(u)}f(v)  ~\|S^{\alpha}_{\lambdab_{\mf C_a}}e_v\|~  \frac{(\alpha_{v_d} + \beta_d)!}{\alpha_{v_d}!} \frac{(|\alpha_v| +|\alpha| + a-1)!}{(|\alpha_v|+ | \alpha| + \beta_d + a-1)!} \\ 
& \times & {\frac{1}{\sqrt{\mbox{card}(\mathsf{sib}_{d}(v))}}} \sqrt{\frac{\alpha_{v_d} }{|\alpha_v| + |\alpha| + a-1}} \\ 
& \times & \sum_{w \in \childnt{\alpha}{v}} ~\gamma(w, \alpha)~ e_{\mathsf{par}_d({w})}.
\eeqn
It is clear from the definition of depth and siblings that the expression $\Gamma(v, \alpha)$ below is independent of $v \in \mathsf{sib}_F(u)$: \beqn \Gamma(v, \alpha) & := & \|S^{\alpha}_{\lambdab_{\mf C_a}}e_v\|~  \frac{(\alpha_{v_d} + \beta_d)!}{\alpha_{v_d}!} \frac{(|\alpha_v| +|\alpha| + a-1)!}{(|\alpha_v|+ | \alpha| + \beta_d + a-1)!} \\ &\times &
  {\frac{1}{\sqrt{\mbox{card}(\mathsf{sib}_{d}(v))}}} \sqrt{\frac{\alpha_{v_d} }{|\alpha_v| + |\alpha| + a-1}}.\eeqn 
 Further, since $\alpha_d=0$, we conclude from \eqref{gamma-w-alpha} that $\gamma(w, \alpha)$ is independent of $w_d=v_d.$ 
Also, 
 since $v=v_G|v_d$ for $G=F \setminus \{d\}$ and $v \in \mathsf{sib}_F(u)$, 
 it follows that
  \beqn
&& \frac{1}{\Gamma(u, \alpha)} S^{*\beta_d+1}_dS^{\beta}_{\lambdab_{\mf C_a}}f 
=   \sum_{v \in \mathsf{sib}_F(u)}f(v)  \sum_{w \in \childnt{\alpha}{v}} \gamma(w, \alpha) e_{\mathsf{par}_d({w})} \\ &=& \sum_{v_G|v_d \in \mathsf{sib}_F(u)}  \left(\sum_{{w} \in \childnt{\alpha}{v}} f(v_G|v_d) \gamma({w}, \alpha) e_{\mathsf{par}_d({w})} \right)\\
&=& \sum_{v_G \in \mathsf{sib}_{F, G}(u)}  \sum_{ v_d \in \mathsf{sib}(u_d)} \left(\sum_{{w} \in \childnt{\alpha}{v}} f(v_G|v_d) \gamma({w}, \alpha) e_{\mathsf{par}_d({w})} \right)\\
&=& \sum_{v_G \in \mathsf{sib}_{F, G}(u)}  \sum_{{w} \in \childnt{\alpha}{v}} \left(\sum_{ v_d \in \mathsf{sib}(u_d)} f(v_G|v_d) \gamma({w}, \alpha) e_{\mathsf{par}_d({w})}  \right) \\
&=& \sum_{v_G \in \mathsf{sib}_{F, G}(u)}  \sum_{{w} \in \childnt{\alpha}{v}} \left(\sum_{ v_d \in \mathsf{sib}(u_d)} f(v_G|v_d)  \right) \gamma({w}, \alpha) e_{\mathsf{par}_d({w})} 
\eeqn
where the sum in the inner bracket is $0$ in view of \eqref{system-main-sp}. 
\end{proof}

We are now in a position to complete the proof of Theorem \ref{S-c-a-kernel}.
\begin{proof}[Proof of Theorem \ref{S-c-a-kernel}] 
We divide the proof into several steps.

\vskip.2cm

{\bf Step I.} In this step, we prove that $S_{\lambdab_{\mf C_a}}$ can be modeled as a multiplication tuple on a Hilbert space of $E$-valued formal power series. Note that by Theorem \ref{wandering} and Lemma \ref{ortho-powers}(ii), we have
\beqn
l^2(V) = \bigoplus_{\alpha \in \mathbb N^d}S^{\alpha}_{\lambdab_{\mf C_a}}E.
\eeqn
Thus for any $f \in l^2(V),$ there exists a multisequence $\{f_{\alpha}\}_{\alpha \in \mathbb N^d}$ in $E$ such that 
\beqn
f = \sum_{\alpha \in \mathbb N^d} S^{\alpha}_{\lambdab_{\mf C_a}}f_{\alpha}.
\eeqn
Also, since $S_1, \cdots, S_d$ are injective (Corollary \ref{p-spectrum}), the multisequence $\{f_{\alpha}\}_{\alpha \in \mathbb N^d}$ with the above property is unique. 
This unique representation allows us to form the inner product  space $\mathscr H_{a, d}$ of $E$-valued formal power series by
\beqn
\mathscr H_{a, d} := \Big\{F(z)=\sum_{\alpha \in \mathbb N^d}f_{\alpha} z^{\alpha} : f_{\alpha} \in E~(\alpha \in \mathbb N^d),~\sum_{\alpha \in \mathbb N^d}\|S^{\alpha}_{\lambdab_{\mf C_a}} f_{\alpha}\|^2 < \infty \Big\}
\eeqn
endowed with the inner product 
\beqn
\inp{F(z)}{G(z)} := \sum_{\alpha \in \mathbb N^d}\inp{S^{\alpha}_{\lambdab_{\mf C_a}} f_{\alpha}}{S^{\alpha}_{\lambdab_{\mf C_a}} g_{\alpha}},
\eeqn
where $G(z)=\sum_{\alpha \in \mathbb N^d}g_{\alpha} z^{\alpha}$. 
Since $S^{*\alpha}_{\lambdab_{\mf C_a}} S^{\alpha}_{\lambdab_{\mf C_a}}$ is bounded below on $E$ (Lemma \ref{ortho-powers}(i)), $\mathscr H_{a, d}$ is a Hilbert space. 
We now define unitary $U : l^2(V) \rar \mathscr H_{a, d}$ by $U(f) = F.$ Further, if $\mathscr M_{z, a}$ denotes the $d$-tuple of (densely defined) multiplication operators $\mathscr M_{z_1}, \cdots, \mathscr M_{z_d}$ in $\mathscr H_{a, d}$ then 
\beqn
US_j(S^{\alpha}_{\lambdab_{\mf C_a}} f_{\alpha})=US^{\alpha + \epsilon_j}_{\lambdab_{\mf C_a}} f_{\alpha} = f_{\alpha} z^{\alpha + \epsilon_j} = \mathscr M_{z_j}f_{\alpha} z^{\alpha} =\mathscr M_{z_j}U(S^{\alpha}_{\lambdab_{\mf C_a}} f_{\alpha})~(j=1, \cdots, d).
\eeqn
Note that $US_j=\mathscr M_{z_j}U$ holds on a dense set. Since $S_j$ is bounded, it follows that $\mathscr M_{z_j}$ is bounded for every $j=1, \cdots, d.$ 

\vskip.2cm

{\bf Step II.} In this step, we check that $\mathscr H_{a, d}$ is a reproducing kernel Hilbert space associated with the reproducing kernel 
\beqn
\kappa_{\mathscr H_{a, d}}(z, w) = \sum_{\alpha \in \mathbb N^d} D_{\alpha} z^{\alpha} \overline{w}^{\alpha}~(z, w \in \Omega),
\eeqn
where $D_{\alpha}$ is the inverse of $S^{*\alpha}_{\lambdab_{\mf C_a}} S^{\alpha}_{\lambdab_{\mf C_a}}|_E$ on $E$ as ensured by Lemma \ref{ortho-powers}(i), and
%given by
%$D_{\alpha}f = \|S^{\alpha}_{c_a, \lambdab}f\|^{-2}f$ for $\alpha \in \mathbb N^d.$ 
$\Omega$ denotes the domain of convergence of $\kappa_{\mathscr H_{a, d}}$ (possibly $\{0\}$). 
Note that for $F(z)=\sum_{\alpha \in \mathbb N^d}f_{\alpha} z^{\alpha}  \in \mathscr H_{a, d}$, $g \in E$ and $w \in \Omega,$
\beqn
\inp{F}{\kappa_{\mathscr H_{a, d}}(\cdot, w)g} &=& \sum_{\alpha \in \mathbb N^d}\inp{S^{\alpha}_{\lambdab_{\mf C_a}} f_{\alpha}}{S^{\alpha}_{\lambdab_{\mf C_a}} (D_{\alpha}g) \overline{w}^{\alpha}} \\ &=& \sum_{\alpha \in \mathbb N^d}
%\frac{1}{\|S^{\alpha}_{c_a, \lambdab}g\|^{2}} 
\inp{f_{\alpha} w^{\alpha}}{S^{*\alpha}_{\lambdab_{\mf C_a}} S^{\alpha}_{\lambdab_{\mf C_a}} D_{\alpha}g} \\ &=& \inp{F(w)}{g}_E,
\eeqn
where we used that $S^{*\alpha}_{\lambdab_{\mf C_a}} S^{\alpha}_{\lambdab_{\mf C_a}} D_{\alpha}=I|_E$ for every $\alpha \in \mathbb N^d.$
This completes the verification of {\bf Step II}.

\vskip.2cm

{\bf Step III.} We note that $D_{\alpha}$ is a block diagonal operator on $E$ with diagonal entries $\frac{a(a+1) \cdots (a+|\alpha|-1)}{\alpha!}$  (corresponding to $e_{\rootb}$) and 
$$\Big(\prod_{j=1}^d \frac{\alpha_{u_{j}}!}{(\alpha_{u_{j}}+\alpha_{j})!} \Big){(|\alpha_u|+a)(|\alpha_u|+a+1) \cdots (|\alpha_u| + a + |\alpha|-1)}$$ (corresponding to the component of $E$ from $l^2(\mathsf{sib}_F(u))$).
This is immediate from \eqref{s-ca-moment-f-2} and the definition of $D_{\alpha}$.

\vskip.2cm

{\bf Step IV.} We verify that the domain of convergence of $\kappa_{\mathscr H_{a, d}}$ equals the open unit ball $\mathbb B^d$ in $\mathbb C^d.$ Indeed, by the preceding two steps, $\kappa_{\mathscr H_{a, d}}$ takes the form
\beqn
\kappa_{\mathscr H_{a, d}}(z, w) = \sum_{\alpha \in \mathbb N^d}  \frac{a(a+1) \cdots (a+|\alpha|-1)}{\alpha!}  z^{\alpha} \overline{w}^{\alpha} ~P_{[e_\rootb]} + \sum_{\underset{F \neq \emptyset}{F  \in \mathscr{P}}} \sum_{u \in \Omega_F} \kappa_{u, F}(z, w),
\eeqn
where $\kappa_{u, F}(z, w)$ is given by \eqref{kappa-u-F}.
Since the first series in the expression for $\kappa_{\mathscr H_{a, d}}(z, w)$ is precisely $\frac{P_{[e_\rootb]}}{(1-\inp{z}{w})^a}$, it suffices to check that the domain of convergence of $\kappa_{u, F}(\cdot, w)$ equals $\mathbb B^d$ for every $w \in \mathbb B^d.$ However, the coefficients in this series are obtained (modulo some scalars) by adding the constant $d$-tuple $\alpha_u$ to the coefficients of  $\frac{1}{(1-\inp{z}{w})^a}$, and hence the domain of convergence is $\mathbb B^d$.
\end{proof}
\begin{remark}
In case $\mathscr T_j=\mathscr T_{1, 0}$ for $j=1, \cdots, d$, then the reproducing kernel spaces $\mathscr H_{a, d}$ are precisely the spaces appearing in \cite[(1.11)]{BV} (see Table 1 below).
\end{remark}

{\small
\begin{table}[H]
\caption{Tree analogs of reproducing kernels $\kappa_{\mathcal H_{a, d}}(z, w)$}
% title of Table
\begin{center}
% used for centering table
\begin{tabular}{| l | l |}
% centered columns (4 columns)
\hline
%inserts double horizontal lines
Kernel $\kappa_{\mathscr H_{a, d}}(z, w)$ & $\mathscr T_1 \times \mathscr T_2$ 
 \\ \hline
% inserts table
%heading
% inserts single horizontal line
$\displaystyle \frac{1}{(1-\inp{z}{w})^a} P_{[e_{(0, 0)}]}$ & $\mathscr T_{1, 0} \times \mathscr T_{1, 0}$ \\ \hline $\displaystyle \frac{1}{(1-\inp{z}{w})^a} P_{{[e_{(0, 0)}]}} $ $+  \displaystyle \sum_{\alpha \in \mathbb N^d}  \frac{1}{(\alpha + \epsilon_1)!} {\prod_{j=0}^{|\alpha|-1}(a + 1 + j)}  z^{\alpha} \overline{w}^{\alpha}\,P_{\mathcal L_{(1, 0), \{1\}}} $ & $\mathscr T_{2, 0} \times \mathscr T_{1, 0}$ \\ \hline $\displaystyle \frac{1}{(1-\inp{z}{w})^a} P_{[e_{(0, 0)}]} + \sum_{\alpha \in \mathbb N^d}  \frac{1}{(\alpha + \epsilon_1)!} {\prod_{j=0}^{|\alpha|-1}(a +1 + j)} z^{\alpha} \overline{w}^{\alpha}\, P_{\mathcal L_{(1, 0), \{1\}}} $
&  \\  $+ \displaystyle \sum_{\alpha \in \mathbb N^d}  \frac{1}{(\alpha + \epsilon_2)!} {\prod_{j=0}^{|\alpha|-1}(a + 1 + j)} z^{\alpha} \overline{w}^{\alpha}\, P_{\mathcal L_{(0, 1), \{2\}}}$
& $ \mathscr T_{2, 0} \times \mathscr T_{2, 0}$ \\ $+ \displaystyle \sum_{\alpha \in \mathbb N^d}  \frac{1}{(\alpha+\epsilon_1 + \epsilon_2)!} {\prod_{j=0}^{|\alpha|-1}(a + 2+ j)} z^{\alpha} \overline{w}^{\alpha}\, P_{\mathcal L_{(1, 1), \{1, 2\}}} $
&  \\ \hline
\end{tabular}
\end{center}

%\label{table:nonlin}
% is used to refer this table in the text
\end{table}
}

\begin{corollary} \label{p-spec-S-c-a}
Let $\mathscr T = (V,\mathcal E)$ be the directed Cartesian product of locally finite rooted directed trees $\mathscr T_1, \cdots, \mathscr T_d$ of finite joint branching index and let $S_{\lambdab_{\mf C_a}}$ be as introduced in Example \ref{ex-sp-bal}. Then the point spectrum of $S^*_{\lambdab_{\mf C_a}}$ contains the open unit ball $\mathbb B^d$ in $\mathbb C^d.$
\end{corollary}
\begin{proof}
By Theorem \ref{S-c-a-kernel}, $S_{\lambdab}$ is unitarily equivalent to multiplication $d$-tuple $\mathscr M_z=(\mathscr M_{z_1}, \cdots, \mathscr M_{z_d})$ acting on the reproducing kernel Hilbert space $\mathscr H_{a, d}$ of $E$-valued holomorphic functions defined on unit ball $\mathbb B^d.$
The desired conclusion now follows from the fact that $\mathscr M^*_{z_j}(\kappa_{\mathscr H_{a, d}}(\cdot, w)f)=\overline{w}_j \kappa_{\mathscr H_{a, d}}(\cdot, w)f$ for any $f \in E$, $w \in \mathbb B^d$ and $j=1, \cdots, d,$ where $\kappa_{\mathscr H_{a, d}}$ denotes the reproducing kernel associated with $\mathscr H_{a, d}.$
\end{proof}

%We will revisit the last example in the context joint subnormal and other subclasses of multishifts later in this chapter.  

We now turn our attention to the classification of spherically balanced multishifts.
The notion of spherically balanced multishifts is closely related to that of spherical Cauchy dual tuple. Before we make this precise, note that by \eqref{spherical-dual}, $S_i^{\mf s} = S_i \big(\sum_{i=1}^d S_i^* S_i \big)^{-1}$. It follows that
the spherical Cauchy dual $S^{\mathfrak s}_{\lambdab} = (S_1^{\mathfrak s}, \cdots, S_d^{\mathfrak s})$ of a joint left invertible multishift $S_{\lambdab}$ is given by
$$S_i^{\mf s} e_v = \Big(\sum_{i=1}^d \|S_i e_v\|^2 \Big)^{-1} \sum_{w \in \childi{i}{v}} \lambda_w^{(i)} e_w\ \text{for all}\ v \in V,~i=1, \cdots, d.$$ 
Note that $S^{\mathfrak s}_{\lambdab}$ is also a multishift on $\mathscr T$ with weights $${\lambda_w^{(i)}}\Big({\sum_{i=1}^d \|S_i e_{\parenti{i}{w}}\|^2}\Big)^{-1}, w \in V^{\circ}, ~i=1, \cdots, d.$$

In the next proposition, we show that every joint left invertible spherically balanced multishift admits a polar decomposition in the following sense.
%, the spherical Cauchy dual turns out to be commuting.
\begin{proposition}\label{sphericallybal}
Let $\mathscr T = (V,\mathcal E)$ be the directed Cartesian product of rooted directed trees $\mathscr T_1, \cdots, \mathscr T_d$.
Let $S_{\lambdab}=(S_1, \cdots, S_d)$ be a joint left invertible multishift on $\mathscr T$ and let $S^{\mathfrak s}_{\lambdab}$ denote the spherical Cauchy dual of $S_{\lambdab}$. 
Then 
the following statements are equivalent:
\begin{enumerate}
\item[(i)] $S^{\mf s}_{\lambdab}$ is commuting. 
\item[(ii)] For every $v \in V^{\circ}$, 
$\mf C$ is constant on $\mathsf{Par}(v)$, where $\mf C$ is as defined in \eqref{constant-gen}.
\item[(iii)] There exists a joint isometry multishift $T_{\lambdab}=(T_1, \cdots, T_d)$ and a block diagonal, positive, invertible bounded linear operator $D_c$ on $l^2(V)$
such that $$S_j = T_jD_c,~j=1, \cdots, d.$$
\end{enumerate}
In this case, the above decomposition is unique.
\end{proposition}

\begin{proof} 
We do not include the verification of the uniqueness part as it is similar to that of Proposition \ref{t-p-decom}.

(i) $\Longleftrightarrow$ (ii):
Fix $v \in V.$
By the discussion prior to the proposition, we have 
$$S_i^{\mf s} e_v = \mf C(v)^{-1} \sum_{w \in \childi{i}{v}} \lambda_w^{(i)} e_w.$$
Therefore,
\beqn
S_j^{\mf s} S_i^{\mf s} e_v &=& \mf C(v)^{-1} \sum_{w \in \childi{i}{v}} \lambda_w^{(i)} S^{\mf s}_j e_w\\
&=& \mf C(v)^{-1} \sum_{w \in \childi{i}{v}} \lambda_w^{(i)} \mf C(w)^{-1} \sum_{u \in \childi{j}{w}} \lambda_u^{(j)} e_u\\
&=& \mf C(v)^{-1} \sum_{u \in \mathsf{Chi}_j \childi{i}{v}} \lambda_{\parenti{j}{u}}^{(i)} \lambda_u^{(j)} 
\mf C(\parenti{j}{u})^{-1} e_u.
\eeqn
Similarly,
$$S_i^{\mf s} S_j^{\mf s} e_v = \mf C(v)^{-1} \sum_{u \in \mathsf{Chi}_i \childi{j}{v}} \lambda_{\parenti{i}{u}}^{(j)} \lambda_u^{(i)} \mf C(\parenti{i}{u})^{-1} e_u.$$
Since $S_{\lambdab}$ is commuting, by Proposition \ref{shift-prop}(i), $S_i^{\mf s} S_j^{\mf s} e_v = S_j^{\mf s} S_i^{\mf s} e_v$ for all $i,j =1, \cdots, d$ if and only if \beqn \label{C-constant} \mf C(\parenti{i}{u}) = \mf C(\parenti{j}{u})~\mbox{for all~} u \in \mathsf{Chi}_i \childi{j}{v}~\mbox{and~} i,j =1, \cdots, d.\eeqn 
This yields the desired equivalence of (i) and (ii).
%Since $\mathsf{par}_i \mathsf{Chi}_i \childi{j}{v}=\childi{j}{v},$ \eqref{C-constant} is equivalent to 
 %$\mf C$ being constant on $\childi{i}{v}$ for all $i=1, \cdots, d$, that is, $\mf C$ being constant on $\child{v}$. 
%% Hence $\mf C$ is constant on $\{w \in \childnt{\alpha}{v} : \alpha \in \mathbb N^d, |\alpha|=1\}$. Since this holds for all $v$, it follows that $\mf C$ is constant on each set  $\mathcal G_t$, $t \geqslant  1$.

(ii) $\Longleftrightarrow$ (iii):
We introduce a multishift $T_{\lambdab} = (T_1, \cdots, T_d)$ on $\mathscr T$ with weights
\beqn
\frac{{\lambda_w^{(i)}}}{\sqrt{\mf C({v})}}, ~w \in \childi{i}{v}, ~v \in V, ~\mbox{and~}i=1, \cdots, d.
\eeqn
Note that for $v \in V$,
\beqn
\sum_{j=1}^d \|T_j e_v\|^2 = \mf C(v)^{-1} \sum_{j=1}^d \| S_j e_v \|^2 =1.
\eeqn
One may argue as in the previous paragraph to see that $T_{\lambdab}$ is commuting if and only (ii) holds.
We now define a block diagonal operator $D_c$ on $l^2(V)$ as follows:
\beq \label{b-diagonal}
D_c e_v := \sqrt{\mf C({v})}\, e_v~\mbox{for~any}~v \in V.
\eeq
Since $S_{\lambdab}$ is joint left invertible, by Proposition \ref{shift-prop}(viii), $D_c$ is a block diagonal, positive, invertible bounded linear operator. If (ii) holds then by the above argument $S_{\lambdab}$ has the decomposition given in (iii). Conversely, if (iii) holds then by the uniqueness of the decomposition, $T_{\lambdab}$ and $D_c$ must be of the form as defined above. The desired conclusion in (ii) now follows from the commutativity of $T_{\lambdab}.$  
\end{proof}
\begin{remark}
The spherical Cauchy dual $S^{\mathfrak s}_{\bf w}$ of a joint left invertible classical multishift $S_{\bf w}$
 is the $d$-variable weighted
shift given by \beqn S^{\mathfrak{s}}_je_\alpha =
\frac{w^{(j)}_\alpha}{\delta_{\alpha, S_{\bf w}}} \;\, e_{\alpha + \epsilon_j}~(1 \leqslant j
\leqslant d), \eeqn where
\beqn
\label{delta_n} \delta_{\alpha, S_{\bf w}} := {\sum_{j=1}^d
\left(w^{(j)}_{\alpha}\right)^2} \enspace ~(\alpha \in {\mathbb{N}}^d).
\eeqn
It is easily seen that that $S^{\mathfrak{s}}_{\bf w}$ is commuting if and
only if $\delta_{\alpha + \epsilon_j, S_{\bf w}} = \delta_{\alpha + \epsilon_k, S_{\bf w}}$
for all $1 \leqslant j, k \leqslant d$ \cite[Section 6]{CC}.
In this case, the function $\mf C$ turns out to be a function of $|\alpha|$, $\alpha \in \mathbb N^d$ \cite[Lemma 3.1]{CK}. This may also be deduced from the result above once we note that for any integer $t \Ge 1 ,$ one can order $\mathcal G_t$ as $\{u_1, \cdots, u_{\tiny \mbox{card}(\tiny \mathcal G_t)}\}$ such that $\mathsf{Par}(u_i) \cap \mathsf{Par}(u_{i+1})  \neq \emptyset$ for every $i=1, \cdots, \mbox{card}(\mathcal G_t)-1.$
\end{remark}

We refer to $T_{\lambdab}$ \index{$T_{\lambdab}$} and $D_c$ as {\it joint isometry part} and {\it diagonal part} of the 
%\index{$D_c$}
multishift $S_{\lambdab}$ respectively.

\begin{corollary} \label{p-decom}
Let $\mathscr T = (V,\mathcal E)$ be the directed Cartesian product of rooted directed trees $\mathscr T_1, \cdots, \mathscr T_d$ and
let $S_{\lambdab}$ be a joint left invertible multishift on $\mathscr T$. 
%and let $S^{\mathfrak s}_{\lambdab}$ denote the spherical Cauchy dual of $S_{\lambdab}$. If $S^{\mathfrak s}_{\lambdab}$ is commuting, then
If $S_{\lambdab}$ is spherically balanced, 
%if and only if
%there exists a joint isometry multishift $T_{\lambdab}$ and a block diagonal, positive, invertible bounded linear operator $D_c$ on $l^2(V)$, as given by \eqref{b-diagonal}, 
%such that $$S_j = T_jD_c,~j=1, \cdots, d.$$ 
then
for any $\beta \in \mathbb N^d$ and any $v \in V$, we have
\begin{enumerate}
\item[(i)] $T^\beta_{\lambdab} e_v = \Big(\prod_{p=0}^{|\beta|-1} \frac{1}{\sqrt{\mf C_{|\alpha_v|+p}}}\Big) S^\beta_{\lambdab} e_v,$
\item[(ii)] $\| T^\beta_{\lambdab} e_v \|^2 = \Big(\prod_{p=0}^{|\beta|-1} \frac{1}{\mf C_{|\alpha_v|+p}} \Big)\| S^\beta_{\lambdab} e_v \|^2$,
\end{enumerate}
where $\mf C_t$ denotes the constant value of $\mf C$ on the generation $\mathcal G_t$ and $T_{\lambdab}$ is the joint isometry part of $S_{\lambdab}.$
\end{corollary}
\begin{proof} Assume that $S_{\lambdab}$ is spherically balanced.
Note that the weights of $T_{\lambdab}$ takes the form
\beqn
\frac{{\lambda_w^{(i)}}}{\sqrt{\mf C_{|\alpha_v|}}}, ~w \in \childi{i}{v}, ~v \in V, ~\mbox{and~}i=1, \cdots, d.
\eeqn
To see (i), let $v \in V.$ 
Using induction on $k \in \mathbb N$, one can verify that for $i=1, \cdots, d$, 
$$T_i^{k} e_v = \prod_{p=0}^{k-1} \frac{1}{\sqrt{\mf C_{|\alpha_v|+p}}} S_i^{k} e_v.$$
Further, in a similar fashion one can get
$$T_j^{l} T_i^{k} e_v = \prod_{p=0}^{k+l-1} \frac{1}{\sqrt{\mf C_{|\alpha_v|+p}}} S_j^{l} S_i^{k} e_v.$$
Continuing this, we obtain (i). 
Finally, (ii) is immediate from (i). 
\end{proof}

Note that part (ii) above gives precise relation between the moments of $S_{\lambdab}$ and that of $T_{\lambdab}$. The multiplicative factor $\Big(\prod_{p=0}^{|\beta|-1} \frac{1}{\mf C_{|\alpha_v|+p}} \Big)$ appearing in (ii) suggests one to introduce a classical unilateral weighted shift $S_{\theta}$ on some Hilbert space $H^2(\gamma)$ of formal power series, so that $\| T^\beta_{\lambdab} e_v \| = \|S^{|\beta|}_{\theta}f_{0}\|\| S^\beta_{\lambdab} e_v \|$ for some orthonormal basis $\{f_{k}\}_{k \in \mathbb N}$ of $H^2(\gamma)$.
Unfortunately, in this formulation, the tree like structure of $\mathscr T$ does not reflect in the construction of $S_{\theta}$ on $H^2(\gamma).$ However, there is an alternate way to construct a shift on a directed tree arising naturally from $\mathscr T$ which at the same time gives the above relation between moments of  $S_{\lambdab}$ and that of $T_{\lambdab}$.
\begin{definition} \label{shift-tensor-dfn}
Let $\mathscr T = (V,\mathcal E)$ be the directed Cartesian product of locally finite, rooted directed trees $\mathscr T_1, \cdots, \mathscr T_d$. Consider the component $\mathscr T^{\otimes}_{\rootb}=(V^{\otimes}, \mathcal F)$ of the tensor product $\mathscr T^{\otimes}$ of $\mathscr T_1, \cdots, \mathscr T_d,$ which contains $\rootb$ (see Theorem \ref{tensor-prop}).  
%Let $S_{\lambdab}$ be a commuting multishift on $\mathscr T$ and let $c_t$ be as appearing in \eqref{C-v}. 
For a bounded sequence $\{c_t\}_{t \in \mathbb N}$ of positive real numbers, consider the (one variable) weighted shift $S_{\theta}$ 
\index{$S_{\theta}$}
on the rooted directed tree $\mathscr T^{\otimes}_{\rootb}$ with weights given by
\beq \label{theta}
\theta_{\mf w} := \frac{{\sqrt{c_{\alpha_{\mf w}-1}}}}{\sqrt{\mbox{card}(\mathsf{sib}(\mf w))}}~(\mf w \in V^{\otimes} \setminus \rootb),
\eeq
where $\alpha_{\mf w}$ is the depth of $\mf w$ in $\mathscr T^{\otimes}_{\rootb}$ and $\mathsf{sib}(\mf w)$ is the set of siblings of $\mf w$ in the directed tree $\mathscr T^{\otimes}_{\rootb}$.
We refer to the weighted shift $S_{\theta}$ on $\mathscr T^{\otimes}_{\rootb}$ as the {\it balanced weighted shift associated with $\{c_t\}_{t \in \mathbb N}$}.
\end{definition}

\begin{example}
If $\mathscr T_1 = \mathscr T_{1, 0} = \mathscr T_2$, then as seen in Example \ref{tensor-c}, $\mathscr T^{\otimes}_{\rootb}$ is (isomorphic to) $\mathscr T_{1, 0}$, and hence $S_{\theta}$ is (unitarily equivalent to) classical unilateral shift on $l^2(\mathbb N).$

If $\mathscr T_1 = \mathscr T_{2, 0}, \mathscr T_2 = \mathscr T_{1, 0}$, then as seen in Example \ref{tensor-m}, $\mathscr T^{\otimes}_{\rootb}$ is $\mathscr T_{2, 0}$, and hence $S_{\theta}$ is a weighted shift on the rooted directed tree $\mathscr T_{2, 0}$ (see Figure 5.1).

Finally, if $\mathscr T_1 = \mathscr T_{2, 0}=\mathscr T_2$ then as seen in Example \ref{tensor-T1-T1}, $\mathscr T^{\otimes}_{\rootb}$ is $\mathscr T_{4, 0}$, and hence $S_{\theta}$ is a weighted shift on the rooted directed tree $\mathscr T_{4, 0}$ (see Figure 5.2).
\end{example}

\begin{figure}
\includegraphics[scale=.5]{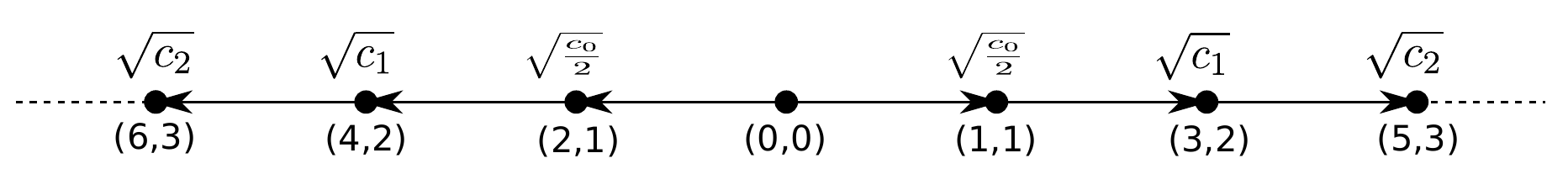} \caption{The weights of $S_{\theta}$ on $\mathscr T^{\otimes}_{\rootb}=\mathscr T_{2, 0}$}
\end{figure}

\begin{figure}
\includegraphics[scale=.45]{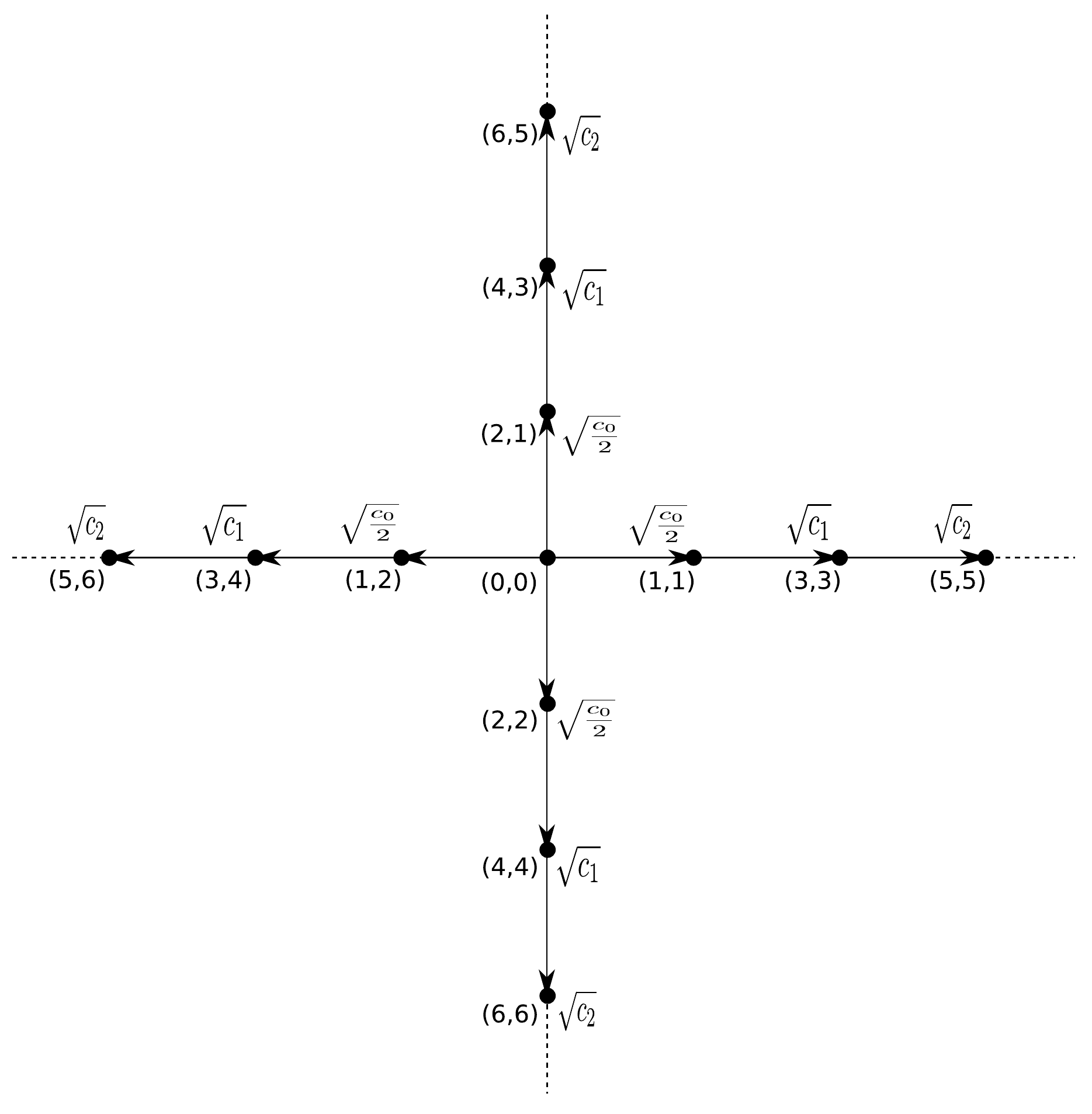} \caption{The weights of $S_{\theta}$ on $\mathscr T^{\otimes}_{\rootb}=\mathscr T_{4, 0}$}
\end{figure}

Before we present the main result of this section, recall that a finite positive
Borel measure $\mu$ supported in the unit sphere $\partial \mathbb
B^d$ in $\mathbb C^d$ is said to be the $\mathbb T^d$-{\it invariant} if for any Borel
measurable subset $\Delta$ of $\partial \mathbb B^d,$
$$\mu(\zeta \cdot \Delta)=\mu(\Delta)~\mbox{for~all}~\zeta \in \mathbb T^d,$$ where $\zeta \cdot \Delta = \{\zeta \cdot z : z
\in \Delta\}$ with $\zeta \cdot z$ denoting the dot product of $\zeta$ and $z.$ A {\it Reinhardt measure} is a $\mathbb
T^d$-invariant probability measure on the unit sphere.

%Recall further that
%a commuting $d$-tuple $T=(T_1, \cdots, T_d)$ on a
%Hilbert space ${\mathcal H}$ is said to be {\it joint subnormal}
%if there exist a Hilbert space ${\mathcal K} \supseteq \mathcal H$
%and a commuting $d$-tuple $N=(N_1, \cdots, N_d)$ of normal
%operators $N_i$ in ${\mathcal B}({\mathcal K})$ such that
%\begin{eqnarray*}
%N_ih = T_ih~\mbox{for~ every~}h \in {\mathcal H}~ \mbox{and}~1\leqslant
%i\leqslant d.
%\end{eqnarray*}

The following generalizes \cite[Theorem 1.10]{CK}, \cite[Theorem 1.3]{K} (the case in which $\mathscr T=\mathscr T^d_{1, 0}$, as described in Example \ref{classical}, with $v=\rootb$).
\begin{theorem}   \label{sp-balanced} 
Let $\mathscr T = (V,\mathcal E)$ be the directed Cartesian product of locally finite, rooted directed trees $\mathscr T_1, \cdots, \mathscr T_d$ and let $\mathscr T^{\otimes}_{\rootb}=(V^{\otimes}, \mathcal F)$ be the component of $\mathscr T^{\otimes}$ containing $\rootb.$
For $v \in V$, choose $\mf v \in V^{\otimes}$ so that $|\alpha_v| = \alpha_{\mf v}$ (see the last part of Theorem \ref{tensor-prop}).
If $S_{\lambdab}$ is a joint left invertible multishift on $\mathscr T$, 
%and let $S_{\theta}$ be the weighted shift on the rooted directed tree $\mathscr T^{\otimes}_{\rootb}=(V^{\otimes}, \mathcal F)$ associated with $S_{\lambdab}.$ 
then 
the following statements are equivalent:
\begin{enumerate}
\item[(i)] The multishift $S_{\lambdab}$ is spherically balanced.
%\item[(ii)] The spherical Cauchy dual $S^{\mf s}_{\lambdab}$ of $S_{\lambdab}$ is commuting.
%\item[(ii)]
%The multishift $S_{\lambdab}$ admits the decomposition $S_j = T_jD_c,~j=1, \cdots, d$, where $T_j$ and $D_c$ are given by \eqref{sp-iso-part} and \eqref{b-diagonal} respectively. 
\item[(ii)] 
For every $v \in V$, there exists a Reinhardt measure $\mu_v$ supported in the unit sphere $\partial \mathbb B^d$ and a balanced weighted shift $S_{\theta}$ on $\mathscr T^{\otimes}_{\rootb}$ associated with a bounded sequence $\{c_t\}_{t \in \mathbb N}$
such that 
\beq \label{int-rep-main-1}
\|f_k\|^2_{l^2(V)} =\int_{\partial \mathbb B^d}\|f_{\theta, k}(z)\|^2_{l^2(V^{\otimes})}d\mu_v(z),
\eeq
where
$$f_k = \sum_{\underset{|\beta| \leqslant k}{\beta \in \mathbb N^d}}a_\beta S^{\beta}_{\lambdab} e_v \in l^2(V), \quad f_{\theta, k}(z)=\sum_{\underset{|\beta| \leqslant k}{\beta \in \mathbb N^d}}a_\beta z^{\beta} S^{|\beta|}_{\theta} e_{\mf v} \in l^2(V^{\otimes})~(k \in \mathbb N, z \in \mathbb C^d).$$
%Let $$f = \sum_{\beta \in \mathbb N^d}a_\beta S^{\beta}_{\lambdab} e_v \in l^2(V), \quad f_{\theta}(z)=\sum_{\beta \in \mathbb N^d}a_\beta (z \cdot S_{\theta})^{\beta} e_{\mf v},$$ 
%and $z \cdot S_{\theta}$ denotes the $d$-tuple $(z_1S_{\theta}, \cdots, z_d S_{\theta})$ for $z \in \mathbb C^d$. 
\end{enumerate}
If any of the above equivalent statements holds, then we have the following:
\begin{enumerate}
\item[(a)] If $\{f_k\}_{k \in \mathbb N}$ converges to $f$ in $l^2(V)$, then $\{f_{\theta, k}(z)\}_{k \in \mathbb N}$ converges to some $f_{\theta}(z)$ in $l^2(V^{\otimes})$ for $\mu_v$-\mbox{a.e.} $z \in \partial \mathbb B^d$, and
\beq \label{int-rep-main}
\|f\|^2_{l^2(V)} =\int_{\partial \mathbb B^d}\|f_{\theta}(z)\|^2_{l^2(V^{\otimes})}d\mu_v(z). 
\eeq
\item[(b)] If $Q^n_{T}(I)$ is as defined in \eqref{sp-gen-powers}, then for $n \in \mathbb N,$
\beq \label{norm-sp-gen}
\inp{Q^n_{S_{\lambdab}}(I)e_v}{e_v} &=& \big\|S^n_{\theta}e_{\mf v}\big\|^2, \\
 \label{norm-Q(I)}
\|Q^n_{S_{\lambdab}}(I)\| &=& \|S^n_{\theta}\|^2.
\eeq
\end{enumerate}
\end{theorem}
%\begin{remark}
%If $f \in l^2(V)$, then $f_{\theta}(z) \in l^2(V^{\otimes})$ for $\mu_v$-\mbox{a.e.} $z \in \partial \mathbb B^d$.
%\end{remark}
\begin{proof}
%We have already seen the equivalence of (i) and (ii) in Lemma \ref{p-decom}.
Suppose that $S_{\lambdab}$ is a joint left invertible multishift on $\mathscr T$.
To see the implication (i) $\Longrightarrow$ (ii), 
assume that $S_{\lambdab}$ is spherically balanced and let $c_t:=\mf C_t,$ the constant value of $\mf C$ on the generation $\mathcal G_t$ for $t \in \mathbb N$. Consider the balanced weighted shift $S_{\theta}$ on $\mathscr T^{\otimes}_{\rootb}$ associated with $\{c_t\}_{t \in \mathbb N}$.
We first prove the formula 
\beq \label{moment-S-theta} \|S^{k}_\theta e_{\mf v}\|^2 = \prod_{p=0}^{k-1} {c_{\alpha_{\mf v}+p}}~(k \Ge 1, ~\mf v \in V^{\otimes}) \eeq 
by induction on integers $k \Ge 1.$
If $k=1,$ then \beqn \|S_\theta e_{\mf v}\|^2 &=& \sum_{\mf u \in \child{\mf v}}\theta^2_{\mf u} \overset{\eqref{theta}} = \sum_{\mf u \in \child{\mf v}}\frac{{c_{\alpha_{\mf u}-1}}}{\mbox{card}(\mathsf{sib}(\mf u))} \\ &=& {c_{\alpha_{\mf v}}}\sum_{\mf u \in \child{\mf v}}\frac{1}{\mbox{card}(\mathsf{sib}(\mf u))} ={c_{\alpha_{\mf v}}}.\eeqn
Assume the formula for some integer $k \geqslant  1$. 
%Note that if $\mf u \in \child{\mf v}$ then $$|\alpha_u|=n_u =n_{\mf v} + 1 =  |\alpha_v| + 1.$$
By the induction hypothesis, we obtain
\beqn
\|S^{k+1}_\theta e_{\mf v}\|^2 &=& \Big\|\sum_{\mf u \in \child{\mf v}} \theta_{\mf u} S^{k}_\theta e_{\mf u}\Big\|^2 = \sum_{\mf u \in \child{\mf v}} \theta^2_{\mf u} \|S^{k}_\theta e_{\mf u}\|^2 \\ &=& \sum_{\mf u \in \child{\mf v}} \frac{{c_{\alpha_{\mf u}-1}}}{{\mbox{card}(\mathsf{sib}(\mf u))}} \Big(\prod_{p=0}^{k-1} {c_{\alpha_{\mf u}+p}} \Big) \\ &=& {c_{\alpha_{\mf v}}} \Big(\prod_{p=0}^{k-1} {c_{\alpha_{\mf v}+1+p}} \Big) \sum_{\mf u \in \child{\mf v}} \frac{1}{{\mbox{card}(\mathsf{sib}(\mf u))}} \\ &=& \prod_{p=0}^{k} {c_{\alpha_{\mf v}+p}}.
\eeqn
This completes the induction.
We now verify that there exists a Reinhardt measure $\mu_v$ supported in the unit sphere $\partial \mathbb B^d$ 
such that 
\beq \label{int-rep}
\|S^{\beta}_{\lambdab}e_v\|^2_{l^2(V)} = \|S^{|\beta|}_\theta e_{\mf v}\|^2_{l^2(V^{\otimes})} \int_{\partial \mathbb B^d} |z^{\beta}|^2d \mu_v,~\beta \in \mathbb N^d.
\eeq
In view of Corollary \ref{p-decom}(ii), it suffices to find a Reinhardt measure $\mu_v$ supported in the unit sphere $\partial \mathbb B^d$ such that $\| T^\beta_{\lambdab} e_v \|^2=\int_{\partial \mathbb B^d} |z^{\beta}|^2d \mu_v,$ 
where $T_{\lambdab}$ is the joint isometry part of $S_{\lambdab}$.
Consider the $T_{\lambdab}$-invariant subspace $\mathcal M:=\bigvee \{T^{\beta}_{\lambdab}e_v : {\beta \in \mathbb N^d}\}$ of $l^2(V)$. 
By Proposition \ref{shift-prop}(ix), $\inp{T^{\beta}_{\lambdab} e_v}{T^{\gamma}_{\lambdab}e_v}=0$ if $\beta \neq \gamma.$ It follows that
${T_{\lambdab}}|_{\mathcal M}$ is a joint isometry classical multishift (up to unitary equivalence). The existence of the desired measure is now
a consequence of the well-known fact that any joint isometry is joint subnormal \cite[Proposition 2]{At} (see the proof of \cite[Theorem 1.10]{CK} for more details).

We now check the integral representation appearing in \eqref{int-rep-main-1}.
%Consider
%$$f_k = \sum_{\underset{|\beta| \leqslant k}{\beta \in \mathbb N^d}}a_\beta S^{\beta}_{\lambdab} e_v, \quad f_{\theta, k}(z)=\sum_{\underset{|\beta| \leqslant k}{\beta \in \mathbb N^d}}a_\beta (z \cdot S_{\theta})^{\beta} e_{\mf v} \in l^2(V^{\otimes})~(k \in \mathbb N).$$
Note that $\inp{S^{\beta}_{\lambdab} e_v}{S^{\gamma}_{\lambdab}e_v}=0$ if $\beta \neq \gamma$ and
$\inp{S^{k}_{\theta}e_{\mf v}}{S^{l}_{\theta}e_{\mf v}}=0$ if $k \neq l$ (see Proposition \ref{shift-prop}(ix)).
It follows that
\beq \label{eq-proof}
\|f_k\|^2_{l^2(V)} &=& \sum_{\underset{|\beta| \leqslant k}{\beta \in \mathbb N^d}}|a_\beta|^2
 \|S^\beta_{\lambdab} e_v \|^2_{l^2(V)}  \nonumber \\ &\overset{\eqref{int-rep}}=& \sum_{\underset{|\beta| \leqslant k}{\beta \in \mathbb N^d}}|a_\beta|^2 \|S^{|\beta|}_\theta e_{\mf v}\|^2_{l^2(V^{\otimes})} \int_{\partial \mathbb B^d} |z^{\beta}|^2d \mu_v. 
\eeq
Since $\mu_v$ is a Reinhardt measure, by \cite[Lemma 2.3]{CK}, the monomials $\{z^{\alpha}\}_{\alpha \in \mathbb N^d}$ are orthogonal in $L^2(\partial \mathbb B^d, \mu_v)$. Thus
\beqn
&& \int_{\partial \mathbb B^d}\|f_{\theta, k}(z)\|^2_{l^2(V^{\otimes})}d\mu_v(z)  \\ &=& \int_{\partial \mathbb B^d} \sum_{\underset{|\beta|, |\gamma| \leqslant k}{\beta, \gamma \in \mathbb N^d}} a_{\beta}\overline{a}_{\gamma} \inp{S^{|\beta|}_{\theta}e_{\mf v}}{S^{|\gamma|}_{\theta}e_{\mf v}} z^{\beta}\overline{z}^{\gamma} d\mu_v(z) \\ &=& \sum_{\underset{|\beta|\leqslant k}{\beta \in \mathbb N^d}} |a_{\beta}|^2 \|{S^{|\beta|}_{\theta}e_{\mf v}}\|^2_{l^2(V^{\otimes})} \int_{\partial \mathbb B^d}  |z^{\beta}|^2 d\mu_v(z) \\ &\overset{\eqref{eq-proof}}=& \|f_k\|^2_{l^2(V)}.
\eeqn
This proves that (i) implies (ii). To see that (ii) implies (i), let 
$f_k=S_je_v$ and $f_{\theta, k}=z_jS_{\theta}e_{\mf v}$ in \ref{int-rep-main-1}
and sum over $j=1, \cdots, d$ to see that \beqn \sum_{j=1}^d\|S_je_v\|^2_{l^2(V)} &=& \|S_{\theta}e_{\mf v}\|^2_{l^2(V^{\otimes})}\sum_{j=1}^d\int_{\partial \mathbb B^d}|z_j|^2d\mu_v \\ &=& \|S_{\theta}e_{\mf v}\|^2_{l^2(V^{\otimes})} = c_{\alpha_{\mf v}}=c_{|\alpha_v|},\eeqn
which is constant on $\mathcal G_{|\alpha_v|}.$

To see the remaining part of the proof, assume that $S_{\lambdab}$ is spherically balanced.

(a) Note that $\|f_k\|_{l^2(V)} \uparrow \|f\|_{l^2(V)}$ and $\|f_{\theta, k}(z)\|_{l^2(V^{\otimes})} \uparrow g(z)$ (possibly in the extended real line) as $k \rar \infty,$ where
\beqn
%\|f_{\theta}(z)\|_{l^2(V^{\otimes})}
g(z):=\Big(\sum_{{\beta \in \mathbb N^d}}|a_\beta|^2 |z^{\beta}|^2 \|S^{|\beta|}_{\theta} e_{\mf v}\|^2 \Big)^{\frac{1}{2}}~(z \in \partial \mathbb B^d).
\eeqn
Applying monotone convergence theorem to \eqref{int-rep-main-1}, we obtain 
\beqn
\|f\|^2_{l^2(V)} =\int_{\partial \mathbb B^d}g(z)^2 d\mu_v(z).
\eeqn
Since the left hand side is a finite positive number, $0 \Le g(z) < \infty$ except for $z$ in a set of $\mu_v$ measure $0.$ 
It follows that 
 $f_{\theta}(z):=\sum_{\beta \in \mathbb N^d} a_\beta z^{\beta}  S^{|\beta|}_{\theta} e_{\mf v} \in l^2(V^{\otimes})$ except for $z$ in a set of $\mu_v$ measure $0$ and $g(z)=\|f_{\theta}(z)\|_{l^2(V^{\otimes})}.$ This also yields \eqref{int-rep-main}.
 
(b) Note that
\beqn \sum_{\underset{|\alpha| = n}{\alpha \in \mathbb N^d}} \frac{n!}{\alpha!} \big \|S^{\alpha}_{\lambdab}e_v\big \|^2
&\overset{\eqref{int-rep}}=&  \sum_{\underset{|\alpha| = n}{\alpha \in \mathbb N^d}} \frac{n!}{\alpha!} \big\|S^{|\alpha|}_\theta e_{\mf v}\big\|^2_{l^2(V^{\otimes})} \int_{\partial \mathbb B^d} |z^{\alpha}|^2d \mu_v \\
&=& \big\|S^{n}_\theta e_{\mf v}\big\|^2_{l^2(V^{\otimes})}   \int_{\partial \mathbb B^d} \sum_{\underset{|\alpha| = n}{\alpha \in \mathbb N^d}} \frac{n!}{\alpha!}|z^{\alpha}|^2d \mu_v, 
\eeqn
which is same as $\big\|S^{n}_\theta e_{\mf v}\big\|^2_{l^2(V^{\otimes})}$ since $$\sum_{\underset{|\alpha| = n}{\alpha \in \mathbb N^d}} \frac{n!}{\alpha!}|z^{\alpha}|^2 = \|z\|^{2n}_2~(z \in \mathbb C^d)$$ and $\mu$ is a probability measure supported in the unit sphere. It follows that 
$\inp{Q^n_{S_{\lambdab}}(I)e_v}{e_v} = \big\|S^n_{\theta}e_{\mf v}\big\|^2.$
Note further that
\beqn
\|Q^n_{S_{\lambdab}}(I)\| =\sup_{v \in V} \inp{Q^n_{S_{\lambdab}}(I)e_v}{e_v} = \sup_{\mf v \in V^{\otimes}} \big\|S^n_{\theta}e_{\mf v}\big\|^2=\|S^n_{\theta}\|^2.
\eeqn
This completes the verification of (b).
\end{proof}

For convenience, we refer to the balanced weighted shift $S_{\theta}$ on $\mathscr T^{\otimes}_{\rootb}$ as the {\it shift associated with the multishift $S_{\lambdab}$ on $\mathscr T$} (cf. \cite[Definition 2.3]{CY}).

Here we discuss some consequences of the preceding theorem.
The first of which is a local spherical analog of von Neumann's inequality (cf. \cite[Proposition 2.5]{K} and \cite[Theorem 7.6]{Ka}).
\begin{corollary}
Let $\mathscr T = (V,\mathcal E)$ be the directed Cartesian product of locally finite, rooted directed trees $\mathscr T_1, \cdots, \mathscr T_d$.
Let $S_{\lambdab}$ be a joint left invertible, spherically balanced multishift on $\mathscr T$. If $S_{\lambdab}$ is a joint contraction, then for any positive integer $k$ and finite sequence $\{a_{\beta} : {\beta \in \mathbb N^d, ~|\beta| \Le k}\}$, we have
\beqn
\sup_{v \in V}\Big\|\sum_{\underset{|\beta| \leqslant k}{\beta \in \mathbb N^d}}a_\beta S^{\beta}_{\lambdab} e_v\Big\| \leqslant \sup_{z \in \mathbb B^d} \Big|\sum_{\underset{|\beta| \leqslant k}{\beta \in \mathbb N^d}}a_\beta z^{\beta}\Big|.
\eeqn
%\beqn
%\Big\|\sum_{k=0}^n a_k Q^k_{S_{\lambdab}}(I)e_v\Big\| \leqslant \sup_{w \in \mathbb T} \Big|\sum_{k=0}^n a_k w^k\Big|,
%\eeqn
%where $Q^n_{T}(I)$ is as defined in \eqref{sp-gen-powers}.
\end{corollary}
\begin{proof}
Suppose that $S_{\lambdab}$ is a joint contraction. Let $S_{\theta}$ be the balanced weighted shift associated with $S_{\lambdab}$ and fix $v \in V.$ Note that by \eqref{norm-Q(I)},
$S_{\theta}$ is a contraction. 
%and hence by von Neumann's inequality \cite{SF}, 
%\beqn
%\Big\|\sum_{k=0}^n a_k S^k_{\theta}\Big\| \leqslant \sup_{w \in \mathbb T} \Big|\sum_{k=0}^n a_k w^k\Big|.
%\eeqn
On the other hand, by the preceding theorem, there exists a Reinhardt measure $\mu_v$ supported in $\partial \mathbb B^d$ such that 
%by Proposition \ref{shift-prop}(ix),
\beqn
\Big\|\sum_{\underset{|\beta| \leqslant k}{\beta \in \mathbb N^d}}a_\beta S^{\beta}_{\lambdab} e_v\Big\|^2_{l^2(V)}  &\overset{\eqref{int-rep-main-1}}=& 
\int_{\partial \mathbb B^d}\Big\|\sum_{\underset{|\beta| \leqslant k}{\beta \in \mathbb N^d}}a_\beta z^{\beta} S^{|\beta|}_{\theta} e_{\mf v}\Big\|^2_{l^2(V^{\otimes})}d\mu_v(z) \\
&=& \int_{\partial \mathbb B^d}\Big\|\sum_{l=0}^k \sum_{\underset{|\beta| = l}{\beta \in \mathbb N^d}}a_\beta z^{\beta} S^{l}_{\theta} e_{\mf v}\Big\|^2_{l^2(V^{\otimes})}d\mu_v(z) \\
&=& \int_{\partial \mathbb B^d}\sum_{l=0}^k \Big|\sum_{\underset{|\beta| = l}{\beta \in \mathbb N^d}}a_\beta z^{\beta}\Big|^2 \Big\|  S^{l}_{\theta} e_{\mf v}\Big\|^2_{l^2(V^{\otimes})}d\mu_v(z)
 \\
 &\Le& \int_{\partial \mathbb B^d}\sum_{l=0}^k \Big|\sum_{\underset{|\beta| = l}{\beta \in \mathbb N^d}}a_\beta z^{\beta}\Big|^2 d\mu_v(z) \\ 
&=& \int_{\partial \mathbb B^d}|\sum_{\underset{|\beta| \leqslant k}{\beta \in \mathbb N^d}} a_\beta z^{\beta}|^2d\mu_v(z) 
\\ & \Le & \sup_{z \in \partial \mathbb B^d} \Big|\sum_{\underset{|\beta| \leqslant k}{\beta \in \mathbb N^d}}a_\beta z^{\beta}\Big|,
\eeqn
where we used that monomials are orthogonal in $L^2(\partial \mathbb B^d, \mu_v)$. 
After taking supremum over $v \in V,$ the desired conclusion now follows from the maximum modulus principle in several complex variables \cite[Pg 5]{Ru}.
\end{proof}
%where $Q^n_{T}(I)$ is as defined in the discussion following \eqref{sp-gen}.

\begin{corollary}  \label{sp-rd}
Let $\mathscr T = (V,\mathcal E)$ be the directed Cartesian product of locally finite, rooted directed trees $\mathscr T_1, \cdots, \mathscr T_d$.
Let $S_{\lambdab}=(S_1, \cdots, S_d)$ be a joint left invertible, spherically balanced multishift on $\mathscr T$ and let $S_{\theta}$ be the weighted shift on the rooted directed tree $\mathscr T^{\otimes}_{\rootb}=(V^{\otimes}, \mathcal F)$ associated with $S_{\lambdab}.$ 
Let $\mf C_t$ be the constant value of $\sum_{j=1}^d \|S_j e_v\|^2$ on the generation $\mathcal G_t$ of $\mathscr T.$
Then
we have the following:
\begin{enumerate}
\item[(i)] $r(S_{\lambdab})=\lim_{n \rar \infty}\sup_{k \in \mathbb N} \Big(\prod_{p=0}^{n-1} {\mf C_{k+p}}\Big)^{\frac{1}{2n}}=r(S_{\theta}),$ where $r(T)$ denotes the spectral radius of any commuting $d$-tuple $T$. In particular, the Taylor spectrum of $S_{\lambdab}$ is a Reinhardt set containing $0$ and contained in the closed ball centered at the origin and of radius $r(S_{\theta})$.
\item[(ii)] $m_{\infty}(S_{\lambdab})=\sup_{n \geqslant  1}\inf_{k \in \mathbb N}\Big(\prod_{p=0}^{n-1} {\mf C_{k+p}}\Big)^{\frac{1}{2n}}=m_{\infty}(S_{\theta}).$ In particular, the left spectrum $\sigma_l(S_{\lambdab})$ of $S_{\lambdab}$ is contained in the closed ball shell centered at the origin with inner radius $m_{\infty}(S_{\theta})$ and outer radius $r(S_{\theta})$.
\end{enumerate}
\end{corollary}
\begin{proof}
The first part follows from
\eqref{norm-Q(I)}, \eqref{moment-S-theta}, and 
the  spectral radius formula \eqref{sp-rad} for the Taylor spectrum: \beqn r(S_{\lambdab}) = \lim_{n \rar \infty} \|Q^n_{S_{\lambdab}}(I)\|^{1/2n} = \lim_{n \rar \infty} \|S^n_{\theta}\|^{1/n}=r(S_{\theta}). \eeqn
%The remaining part in (i) follows from the fact that the Taylor spectrum of $S_{\lambdab}$ is contained in the ball $\overline{\mathbb B}^d_{r(S_{\lambdab})}.$
To see (ii), note that by \eqref{l-sp-rad},
$$m_{\infty}(S_{\lambdab}) \leqslant \sup_{n \geqslant  1} \inf_{v \in V} \inp{Q^n_{S_{\lambdab}}(I)e_v}{e_v}^{\frac{1}{2n}}.$$ Let $M_n:= \inf_{v \in V} \inp{Q^n_{S_{\lambdab}}(I)e_v}{e_v}^{\frac{1}{2n}}$ for $n \geqslant  1$. Then for any $f=\sum_{v \in V}f(v)e_v \in l^2(V)$ of unit norm, by parts (vi) and (x) of Proposition \ref{shift-prop},
\beqn
\inp{Q^n_{S_{\lambdab}}(I)f}{f}^{\frac{1}{2n}} = \Big(\sum_{v \in V}|f(v)|^2\inp{Q^n_{S_{\lambdab}}(I)e_v}{e_v}\Big)^{\frac{1}{2n}} \geqslant  M_n,
\eeqn
and hence $m_{\infty}(S_{\lambdab}) = \sup_{n \geqslant  1} \inf_{v \in V} \inp{Q^n_{S_{\lambdab}}(I)e_v}{e_v}^{\frac{1}{2n}}.$ Similar observation holds for $S_{\theta}.$ The desired conclusion in (ii) may now be drawn from \eqref{norm-sp-gen} and \eqref{moment-S-theta}.
\end{proof}

\begin{example}
Consider the multishift $S_{\lambdab_{\mf C_a}}$ as discussed in Example 
\ref{ex-sp-bal}. 
Recall that $\{c_{a, t}\}_{t \in \mathbb N}$ is given by $$c_{a, t}=\frac{t+d}{t+a}~(t \in \mathbb N).$$ 
Thus weights of $S_{\lambdab_{\mf C_a}}=(S_1, \cdots, S_d)$ are given by 
\beq \label{weights-S-c-lambdab}
\lambda^{(j)}_w=\frac{1}{\sqrt{{\mbox{card}(\childi{j}{v})}}} \sqrt{\frac{\alpha_{v_j}  + 1}{|\alpha_v| + a}}~\mbox{for~}w \in \childi{j}{v}~\mbox{and~}j=1, \cdots, d.
\eeq
By the preceding corollary,
\beqn
r(S_{\lambdab_{\mf C_a}})=\lim_{n \rar \infty}\sup_{k \in \mathbb N}\Big(\prod_{p=0}^{n-1} {c_{a, k+p}}\Big)^{\frac{1}{2n}}=\lim_{n \rar \infty}\sup_{k \in \mathbb N}\Big(\prod_{p=0}^{n-1} {\frac{k+p+d}{k+p+a}}\Big)^{\frac{1}{2n}}.
\eeqn
We follow the argument in \cite[Lemma 3.9]{CY} to see that $r(S_{\lambdab_{\mf C_a}})=1.$
Let $F(n, k)=\prod_{p=0}^{n-1} {\frac{k+p+d}{k+p+a}}~(k \in \mathbb N, n \geqslant  1)$, and note that $F(n, k)$ is increasing (resp. decreasing) in $k$ if and only if $a \geqslant  d$ (resp. $a \leqslant d$). Thus the following possibilities occur:
\beqn \prod_{p=0}^{n-1} \Big({\frac{p+d}{p+a}}\Big)^{\frac{1}{2n}}=F(n, 0)^{\frac{1}{2n}} \leqslant  F(n, k)^{\frac{1}{2n}} \leqslant 1 \quad \mbox{or}\\ 1 \leqslant F(n, k)^{\frac{1}{2n}} \leqslant  F(n, 0)^{\frac{1}{2n}}=\prod_{p=0}^{n-1} \Big({\frac{p+d}{p+a}}\Big)^{\frac{1}{2n}}.\eeqn 
In either case, $\lim_{n \rar \infty}\sup_{k \in \mathbb N}\Big(\prod_{p=0}^{n-1} {\frac{k+p+d}{k+p+a}}\Big)^{\frac{1}{2n}}=1.$ This shows that $r(S_{\lambdab_{\mf C_a}})=1.$ One can argue similarly to see that $m_{\infty}(S_{\lambdab_{\mf C_a}})=1.$ In particular, 
\beq \label{sp-inclusion}
\sigma(S_{\lambdab_{\mf C_a}}) \subseteq \mbox{cl}({\mathbb B}^d), \quad 
\sigma_l(S_{\lambdab_{\mf C_a}}) \subseteq \partial {\mathbb B}^d.
\eeq
Finally, by Corollary \ref{p-spec-S-c-a}, we must have $\sigma(S_{\lambdab_{\mf C_a}}) = \mbox{cl}({\mathbb B}^d).$
\end{example}

We conclude this section with a brief discussion on the essential spectrum of the multishift $S_{\lambdab_{\mf C_a}}.$ Recall first that a commuting $d$-tuple $T$ is {\it essentially normal} if $[T^*_i, T_j]=T^*_iT_j-T_jT^*_i$ is compact for every $i, j=1, \cdots, d$.

\begin{proposition}
Let $S_{\lambdab_{\mf C_a}}$ be multishift as discussed in Example 
\ref{ex-sp-bal}.
If $\mathscr T$ is of finite joint branching index $k_{\mathscr T}$, then
$S_{\lambdab_{\mf C_a}}=(S_1, \cdots, S_d)$ is essentially normal.
\end{proposition}
\begin{proof}
Assume that $\mathscr T$ is of finite joint branching index $k_{\mathscr T}$. By the Putnam-Fuglede Theorem \cite{Co-0}, it suffices to check that $[S^*_j, S_j]$ is compact for every $j=1, \cdots, d$. For fixed $j=1, \cdots, d$, define
$$W_j :=\{v \in V : \mbox{card}(\childi{j}{v})=1~\mbox{and~}\mbox{card}(\mathsf{sib}_j(v))=1 \}.$$ 
Note that $[S^*_j, S_j]$ decomposes into $A_j \oplus B_j$ on $l^2(V)=l^2(W_j) \oplus l^2(V \setminus W_j),$ where $A_j, B_j$ are block diagonal operators given by
\beqn
A_je_v = ((\lambda^{(j)}_w)^2 - (\lambda^{(j)}_v)^2) e_v~(v \in W_j, \childi{j}{v}=\{w\}),
\eeqn
\beqn 
B_je_v = \sum_{w \in \mathsf{Chi}_j(v)}(\lambda^{(j)}_w)^2 e_v - \sum_{{u \in \mathsf{sib}_j(v)}}\lambda^{(j)}_v \lambda^{(j)}_u e_u~(v \in V \setminus W_j).
\eeqn
By \eqref{weights-S-c-lambdab}, $A_j$ is a diagonal operator with diagonal entries $$(\lambda^{(j)}_w)^2 - (\lambda^{(j)}_v)^2=\frac{|\alpha_v|-\alpha_{v_j} + a-1}{(|\alpha_v|+a)(|\alpha_v|+a-1)},$$ which tends to $0$ as $|\alpha_v| \rar \infty$. This shows that $A_j$ is a compact operator. To see that $B_j$ is a compact operator, note first that
\beq \label{Bj}
B_je_v &\overset{\eqref{weights-S-c-lambdab}}=& \sum_{u \in \mathsf{Chi}_j(v)}  \frac{1}{{{\mbox{card}(\childi{j}{v})}}} {\frac{\alpha_{v_j}  + 1}{|\alpha_v| + a}}\,e_v \nonumber \\ &-& \sum_{{w \in \mathsf{sib}_j(v)}} \frac{1}{{{\mbox{card}(\mathsf{sib}_j(v))}}} {\frac{\alpha_{v_j}}{|\alpha_v| + a-1}}\,e_w \nonumber \\
&=& {\frac{\alpha_{v_j}  + 1}{|\alpha_v| + a}}\,e_v - \frac{1}{{{\mbox{card}(\mathsf{sib}_j(v))}}} {\frac{\alpha_{v_j}}{|\alpha_v| + a-1}} \sum_{{w \in \mathsf{sib}_j(v)}} e_w.
\eeq
We next decompose $V \setminus W_j$ as
$\sqcup_{v \in \Omega} \mathsf{sib}_j(v),$ where $\Omega$ is formed by picking up only one element (as ensured by axiom of choice) from every $\mathsf{sib}_j(v)$ for $j=1, \cdots, d.$ Note that 
$l^2(\mathsf{sib}_j(v))$ is reducing for $B_j$ for every $v \in \Omega,$.  This immediately yields the decomposition \beqn B_j=\bigoplus_{v \in \Omega} B_{jv}~\mbox{on}~ l^2(V \setminus W_j)=\bigoplus_{v \in \Omega} l^2(\mathsf{sib}_j(v)),\eeqn where $B_{jv}$ is a finite rank operator (since $\mathscr T_j$ is locally finite) for $j=1, \cdots, d$.
It now suffices to check that $\|B_{jv}\| \rar 0$ as $|\alpha_v| \rar \infty$ (see Remark \ref{cgn-net}).
Before proceeding to this end, 
observe that \beq \label{Mj} \sup_{v \in V \setminus W_j}\mbox{card}(\mathsf{sib}_j(v)) &\leqslant & M_j:=\mbox{card}(\childn{k_{\mathscr T_j}}{\mathsf{root}_j}) < \infty,\\
\label{alphaj} \sup_{v \in V \setminus W_j}\alpha_{v_j} &\leqslant & k_{\mathscr T_j}.  
\eeq
Let $f=\sum_{u \in \mathsf{sib}_j(v)}f(u)e_u \in l^2(\mathsf{sib}_j(v))$ and let $\Upsilon_v:=\sum_{u \in \mathsf{sib}_j(v)} f(u).$
Note that
\beqn
B_{jv}f &\overset{\eqref{Bj}}=& \sum_{u \in \mathsf{sib}_j(v)} f(u) \Big( {\frac{(\alpha_u)_j  + 1}{|\alpha_u| + a}}\,e_u - \frac{1}{{{\mbox{card}(\mathsf{sib}_j(u))}}} {\frac{(\alpha_u)_j}{|\alpha_u| + a-1}} \sum_{{w \in \mathsf{sib}_j(u)}} e_w\Big)\\
&=& {\frac{\alpha_{v_j}  + 1}{|\alpha_v| + a}} \sum_{u \in \mathsf{sib}_j(v)} f(u)e_u  - \frac{\Upsilon_v}{{{\mbox{card}(\mathsf{sib}_j(v))}}} {\frac{\alpha_{v_j}}{|\alpha_v| + a-1}} \Big(\sum_{{w \in \mathsf{sib}_j(v)}} e_w\Big) \\
&=& \sum_{{u \in \mathsf{sib}_j(v)}} \beta_u e_u,
\eeqn
where $\beta_u$ is given by
\beqn
\beta_u = {\frac{\alpha_{v_j}  + 1}{|\alpha_v| + a}}f(u) - \frac{\Upsilon_v}{{{\mbox{card}(\mathsf{sib}_j(v))}}} {\frac{\alpha_{v_j}}{|\alpha_v| + a-1}}.
\eeqn
Since $|\Upsilon_v| \Le\|f\|M_j,$ by \eqref{alphaj} and Cauchy-Schwarz inequality, 
\beqn |\beta_u| &\leqslant & {\frac{(k_{\mathscr T_j}+1)|f(u)|}{|\alpha_v| + a}} + \frac{|\Upsilon_v|}{{{\mbox{card}(\mathsf{sib}_j(v))}}} {\frac{k_{\mathscr T_j}}{|\alpha_v| + a-1}} \\ &\Le& \frac{k_{\mathscr T_j}+1}{|\alpha_v| + a-1} \Big(1 + \frac{M_j}{{{\mbox{card}(\mathsf{sib}_j(v))}}} \Big)\|f\|.
\eeqn
It follows from \eqref{Mj} and  that 
\beqn
\|B_{jv}f\|^2 =\sum_{{u \in \mathsf{sib}_j(v)}} |\beta_u|^2 \leqslant \frac{(k_{\mathscr T_j}+1)^2(1+M_j)^3}{(|\alpha_v| + a-1)^2}\|f\|^2.
\eeqn
This shows that $\|B_{jv}\| \rar 0$ as $|\alpha_v| \rar \infty$.
\end{proof}

The conclusion of the preceding proposition no more holds true if we relax the assumption of finite joint branching index. 
\begin{example} \label{B-S-ph}
Consider the {\it $n$-ary tree $\mathscr T^{(n)}$} given by
\beqn
V^{(n)}=\{v_{k,l} : k \in \mathbb N, ~l=1, \cdots, 2^k\}, ~ \child{v_{k,l}} = \{v_{k+1, j} : n(l-1) + 1 \leqslant j \leqslant nl\}.  
\eeqn 
Let $\mathscr T=(V, \mathcal E)$ denote the directed product of $\mathscr T^{(n)}$ with itself. 
Note that for $v_{k, l}, v_{p, q} \in V^{(n)},$
\beqn
\childi{1}{(v_{k, l}, v_{p, q})} &=& \{(v_{k+1, j}, v_{p, q}) \in V : n(l-1) + 1 \leqslant j \leqslant nl\}, \\
\childi{2}{(v_{k, l}, v_{p, q})} &=& \{(v_{k, l}, v_{p+1, j}) \in V : n(q-1) + 1 \leqslant j \leqslant nq\},
\eeqn
so that $\mbox{card}(\childi{j}{(v_{k, l}, v_{p, q})})=n$ for $j=1, 2.$
Let $S_{\lambdab_{\mf C_a}}$ be as discussed in Example \ref{ex-sp-bal} with $d=2$.
Note that the system $\lambdab$ is given by
%\beqn
%\lambda^{(j)}_w = \frac{1}{\sqrt{2n}},~w \in \childi{j}{(v_{k, l}, v_{p, q})},~j=1, 2.
%\eeqn
\beqn
\lambda^{(1)}_w = \frac{1}{\sqrt{n}}\sqrt{\frac{k+1}{k+p+a}},~w \in \childi{1}{(v_{k, l}, v_{p, q})}.
\eeqn
\beqn
\lambda^{(2)}_w = \frac{1}{\sqrt{n}}\sqrt{\frac{p+1}{k+p+a}},~w \in \childi{2}{(v_{k, l}, v_{p, q})}.
\eeqn
%Then $S^{(n, r)}_{\lambdab}$ is commuting. Indeed, for $k, p \geqslant  1,$ \beqn  \lambda^{(2)}_{(v_{k, l}, v_{p, q})} \lambda^{(1)}_{\parenti{2}{(v_{k, l}, v_{p, q})}} &=& \frac{1}{\sqrt{n}}\sqrt{\frac{p}{k+p+r-1}}
%\frac{1}{\sqrt{n}}\sqrt{\frac{k}{k+p + r-2}} \\
%\lambda^{(1)}_{(v_{k, l}, v_{p, q})} \lambda^{(2)}_{\parenti{1}{(v_{k, l}, v_{p, q})}} &=& \frac{1}{\sqrt{n}}\sqrt{\frac{k}{k+p+r-1}} \frac{1}{\sqrt{n}}\sqrt{\frac{p}{k+p+r-2}}
% \eeqn
% (see \eqref{commuting}).
We claim that $S_{\lambdab_{\mf C_a}}$ on $\mathscr T^{(n)}$ is essentially normal if and only if $n=1$. In case $n=1,$ $S_{\lambdab_{\mf C_a}}$ are classical multishifts.
The essential normality in this case is well-known (see, for example, \cite{CY}). To see the converse, assume that $n \Ge2.$ 
%We use the notations used in the proof of the preceding proposition. Since 
%$\mbox{card}(\childi{j}{(v_{k, l}, v_{p, q})})=n$,
%$W_j=\emptyset,$ and hence
%$S^*_jS_j-S_jS^*_j=B_j$  for $j=1, 2.$ 
Let $B_j:=[S^*_j, S_j]$ for $j=1, 2.$
It suffices to check that $\|B_{1}e_{(v_{k, l}, v_{p, q})}\| \nrightarrow 0$ as $k =p \rar \infty.$ For $v= (v_{k, l}, v_{p, q}),$ note that
\beqn
\inp{B_{1}e_{v}}{e_v} &=& \sum_{u \in \mathsf{Chi}_1(v)}  \frac{1}{n} \frac{k + 1}{k+p + a} - \sum_{w \in \mathsf{sib}_j(v)} \frac{1}{n} \frac{k}{k+p + a-1}\,\inp{e_w}{e_v} \\
&=& \frac{k + 1}{k+p + a} - \frac{1}{n} \frac{k}{k+p + a-1},
\eeqn
which converges to $\frac{1}{2}\Big(1-\frac{1}{n}\Big)$ as $k=p \rar \infty.$
%Moreover, $S^{(n, 2)}_{\lambdab}$ is a joint isometry if and only if $r=2$. 
%In fact, for every $v=(v_{k, l}, v_{p, q}) \in V,$
%\beqn \sum_{j=1}^2\|S_je_{v}\|^2 &=& \sum_{w \in \childi{1}{v}} \frac{1}{{n}}{\frac{k+1}{k+p+r}} ~+ \sum_{w \in \childi{2}{v}}  \frac{1}{{n}}{\frac{p+1}{k+p+r}} \\ &=& {\frac{k+1}{k+p+r}} + {\frac{p+1}{k+p+r}} = \frac{k+p+2}{k+p+r},\eeqn
%which is $1$ if and only if $r=2.$
\end{example}
%The preceding example shows that unlike the case classical multishifts, the Berger-Shaw phenomenon may fail for the tree analog of $S_{{\bf w}, a}$.

\begin{corollary}
Let $S_{\lambdab_{\mf C_a}}$ be multishift as discussed in Example 
\ref{ex-sp-bal}.
Assume that $\mathscr T$ is of finite joint branching index $k_{\mathscr T}$. Then
\beqn \label{e-sp-inclusion} \sigma_e(S_{\lambdab_{\mf C_a}}) \subseteq \partial \mathbb B^d.\eeqn
\end{corollary}
\begin{proof}
It may be concluded from the proof of \cite[Lemma 3.5]{CY} that for any {essentially normal} $d$-tuple $T$, the essential spectrum $\sigma_e(T)$ is contained in $$\{w \in \mathbb C^d : \|w\|^2_2 \in \sigma_e(Q_T(I))\},$$
where $Q_T(\cdot)$ is as defined in \eqref{sp-gen-powers}.
In view of the last result, it now suffices to check that $\sigma_e(Q_{S_{\lambdab_{\mf C_a}}}(I))$ is equal to $\{1\}.$ However, $Q_{S_{\lambdab_{\mf C_a}}}(I)$ is diagonal operator with diagonal entries  $\frac{|\alpha_v|+d}{|\alpha_v|+a}$ (repeated ${\mbox{card}(\childnt{\alpha_v}{\rootb})}$ times) for $v \in V$. Since the only limit point of these eigenvalues of $Q_{S_{\lambdab_{\mf C_a}}}(I)$ is $1,$ the essential spectrum of $Q_{S_{\lambdab_{\mf C_a}}}(I)$ must be $\{1\}$ \cite{Co-0}.
\end{proof}
\begin{remark} 
%Assume that $S_{\lambdab}-\omega$ is Fredholm.
Since the point spectrum of $S_{\lambdab}$ is empty (Corollary \ref{p-spectrum}), in dimension $d=2,$ the dimension of the cohomology group at the middle stage in the Koszul complex of $S^*_{\lambdab_{\mf C_a}}-\omega$ is same for every $\omega \in \mathbb B^d$ (see \eqref{index}).
\end{remark}
%We do not know whether equality holds, in general, in \eqref{sp-inclusion} and \eqref{e-sp-inclusion}. 

%\chapter{Special Classes of Multishifts}

\section{Joint Subnormal Multishifts}

We begin this section with a simple characterization of joint subnormal multishift in terms of complete monotonicity of its moments.

\begin{proposition}
Let $\mathscr T = (V,\mathcal E)$ be the directed Cartesian product of rooted directed trees $\mathscr T_1, \cdots, \mathscr T_d$. 
Let $S_{\lambdab}$ be a toral contractive multishift on $\mathscr T$. Then $S_{\lambdab}$ is  joint subnormal if and only if
for every $v \in V,$ the multisequence $\{\|S^{\alpha}_{\lambdab}e_v\|^2\}_{\alpha \in \mathbb N^d}$ is completely monotone. 
\end{proposition}
\begin{proof}
As recorded earlier, by \cite[Theorem 4.4]{At-0}, a toral contractive $d$-tuple $T$ on a complex Hilbert space $\mathcal H$ is joint subnormal if and only if for every $h \in \mathcal H,$ the multisequence $\{\|T^{\alpha}h\|^2\}_{\alpha \in \mathbb N^d}$ is completely monotone. 
Since $\{S^{\alpha}_{\lambdab}e_v\}_{v \in V}$ is mutually orthogonal (Proposition \ref{shift-prop}(x)), for $f=\sum_{v \in V}f(v)e_v,$ 
\beqn \|S^{\alpha}_{\lambdab}f\|^2 = \sum_{v \in V}|f(v)|^2\|S^{\alpha}_{\lambdab}e_v\|^2.
\eeqn
By the general theory \cite[Chapter 4]{BCR}, we conclude that
$S_{\lambdab}$ is joint subnormal if and only if
for every $v \in V,$ $\{\|S^{\alpha}_{\lambdab}e_v\|^2\}_{\alpha \in \mathbb N^d}$ is completely monotone. 
%The remaining part follows from the Hausdorff's solution to multidimensional Hausdorff moment problem (refer to \cite{BCR}).  
\end{proof}

Although the preceding result characterizes all joint subnormal contractive multishifts on $\mathscr T$, the necessary and sufficient conditions include information at all vertices. On the other hand, information at single vertex (namely, $\{\|S^{\alpha}_{\lambdab}e_{\rootb}\|^2\}_{\alpha \in \mathbb N^d}$ is completely monotone) 
is sufficient to 
ensure joint subnormality in the context of classical multishifts.
Thus a natural question arises whether joint subnormality of $S_{\lambdab}$ can be recovered from complete monotonicity at finitely many vertices. This question has an affirmative answer in case each $\mathscr T_j$ is locally finite with finite branching index.

\begin{theorem} \label{subnormal-minimal}
Let $\mathscr T = (V,\mathcal E)$ be the directed Cartesian product of rooted directed trees $\mathscr T_1, \cdots, \mathscr T_d$. 
Let $S_{\lambdab}$ be a toral contractive multishift on $\mathscr T$. 
 Let $$W :=\bigcup_{\underset{\alpha \leqslant k_{\mathscr T}}{\alpha \in \mathbb N^d}}\childnt{\alpha}{\rootb}$$ and let $\tilde{W}$ be the set $\tilde{W} := W_1 \times \cdots \times W_d$, where $$W_j := \child{V^{(j)}_\prec} \cup \{\mathsf{root}_j\},~j=1, \cdots, d.$$ 
Then the following statements are equivalent:
\begin{enumerate}
\item[(i)] $S_{\lambdab}$ is joint subnormal. 
\item[(ii)] For every $v \in W,$ $\{\|S^{\alpha}_{\lambdab}e_v\|^2\}_{\alpha \in \mathbb N^d}$ is completely monotone. 
\item[(iii)] For every $v \in \tilde{W},$ $\{\|S^{\alpha}_{\lambdab}e_v\|^2\}_{\alpha \in \mathbb N^d}$ is completely monotone. 
\end{enumerate}
\end{theorem}
\begin{proof}
The implication that (i) implies (ii) is clear from the previous result while that (ii) implies (iii) is obvious in view of the inclusion $\tilde{W} \subseteq W$.
Let us check that (ii) implies (i).
Assume that
for every $v \in W,$ the multisequence $\{\|S^{\alpha}_{\lambdab}e_v\|^2\}_{\alpha \in \mathbb N^d}$ is completely monotone. Fix $v \in V \setminus W$. We contend that there exist $w \in W$, $\tilde \alpha \in \mathbb N^d$ and $\gamma \in \mathbb C$ such that \beq \label{ev-gamma} e_v =\gamma S^{\tilde \alpha}_{\lambdab}e_w.\eeq
Note that there exists a subset $\{i_1, \cdots, i_k\}$ of $\{1, \cdots, d\}$ such that $\alpha_{v_{i_j}} > k_{\mathscr T_{i_j}}$ for every $j=1, \cdots, k$. Let $l_j = \alpha_{v_{i_j}} - k_{\mathscr T_{i_j}}~(j=1, \cdots,  k)$ and set $$w := \mathsf{par}_{i_1}^{\langle l_1 \rangle} \cdots \parentki{i_k}{l_k}{v}.$$ Then $\alpha_{w_j} \Le k_{\mathscr T_j}$ for every $j = 1, \cdots, d$, and hence $w \in W$. Now put $$\tilde{\alpha} := l_1 \epsilon_{i_1} + \cdots + l_k \epsilon_{i_k}.$$ Since $\alpha_{w_{i_j}} = k_{\mathscr T_{i_j}}$ for $j=1, \cdots, k$, we must have $\gamma S_{\lambdab}^{\tilde{\alpha}} e_w = e_v$ for some scalar $\gamma \in \mathbb C$. This completes the verification of \eqref{ev-gamma}.
It follows that for any $\beta \in \mathbb N^d,$
\beqn
\sum_{\underset{|\alpha|\leqslant n}{\alpha \in \mathbb N^d}} (-1)^{|\alpha|} {n \choose \alpha} \|S^{\alpha + \beta}_{\lambdab}e_v\|^2 = \sum_{\underset{|\alpha|\leqslant n}{\alpha \in \mathbb N^d}} (-1)^{|\alpha|} {n \choose \alpha} |\gamma|^2 \|S^{\alpha + \beta + \tilde \alpha}_{\lambdab}e_w\|^2 \geqslant  0,
\eeqn
since $\{\|S^{\alpha}_{\lambdab}e_w\|^2\}_{\alpha \in \mathbb N^d}$ is given to be completely monotone. Now apply the preceding proposition to complete the verification of (ii) implies (i). 

Finally, we check the implication that (iii) implies (ii). Let $v \in W \setminus \tilde{W}$. Define 
\beqn
\mathscr F_v := \big\{w \in \tilde{W} : \childnt{\alpha^{(w)}}{w}~\mbox{contains $v$ for some~} \alpha^{(w)} \in \mathbb N^d \big\}.
\eeqn
Then $\mathscr F_v$ is nonempty as $\rootb \in \mathscr F_v$. Now consider the set
\beqn
\mathscr G_v := \big\{w \in \mathscr F_v : |\alpha_w| \geqslant  |\alpha_u|~\mbox{for all~} u \in \mathscr F_v \big\}.
\eeqn
%Note that $\mathscr G_v$ is well defined as $|\alpha_w| \leqslant |k_\mathscr T|$ for all $w \in W$. 
We claim that for all $w \in \mathscr G_v$, there exists $\alpha^{(w)}=\big(\alpha^{(w)}_1, \cdots, \alpha^{(w)}_d\big) \in \mathbb N^d$ such that $\childnt{\alpha^{(w)}}{w} = \{v\}$. If possible, suppose that there are distinct vertices $v, v' \in \childnt{\alpha^{(w)}}{w}$ for some $w \in \mathscr G_v$. Without loss of generality, assume that $v_1 \neq v'_1$. As $v_1, v'_1 \in \childn{\alpha^{(w)}_1}{w_1}$, there exists a positive integer $k$,  $1 \leqslant k \leqslant \alpha^{(w)}_1$, such that $u_1:= \parentn{k}{v_1} \in V^{(1)}_\prec$. Let $\dot{u_1} \in \child{u_1}$ be such that $v_1 \in \childn{k-1}{\dot{u_1}}$. Note that $\dot{u_1} \in W_1$. Consider $w':= (\dot{u_1}, w_2, \cdots, w_d)$. Then $w' \in \tilde{W}$ and $v \in \childnt{\beta}{w'}$, where $\beta = (k-1, \alpha^{(w)}_2, \cdots, \alpha^{(w)}_d)$. Thus $w' \in \mathscr F_v$, and hence $|\alpha_{w}| \geqslant  |\alpha_{w'}|$. On the other hand, $|\alpha_{w'}| = |\alpha_w|+\alpha^{(w)}_1-k+1 >|\alpha_w|$, which is a contradiction. This proves the claim that $\childnt{\alpha^{(w)}}{w} = \{v\}$ for all $w \in \mathscr G_v$. Now one may argue as in the preceding paragraph to see that for every $v \in W \setminus \tilde{W}$, there exist $w \in \tilde{W}$, $\tilde \alpha \in \mathbb N^d$ and $\gamma \in \mathbb C$ such that $e_v =\gamma S^{\tilde \alpha}_{\lambdab}e_w.$ This gives immediately the complete monotonicity of $\{\|S^{\alpha}_{\lambdab}e_v\|^2\}_{\alpha \in \mathbb N^d}$.
\end{proof}
\begin{remark}
If $\mathscr T$ is locally finite with finite joint branching index, then $W$ (and hence $\tilde{W}$) is finite. 
\end{remark}

It is well-known that there is a class of tuples antithetical to joint subnormal tuples commonly known as {\it  (toral or joint) completely hyperexpansive tuples} (refer to \cite{At-2} and \cite{CC} for definitions and basic properties). A characterization similar to one given above can be obtained for toral completely hyperexpansive multishifts as well, where, as expected, the moments being completely monotone is replaced by moments being completely alternating (refer to \cite{BCR} for the definition of completely alternating functions). Similar characterization can be obtained for the class of joint $q$-isometries as introduced and studied in \cite{GR}.

The class of joint subnormal multishifts within the class of spherically balanced multishifts admits a handy characterization (cf. \cite[Theorem 5.3(1)]{CY}).
\begin{proposition}
Let $\mathscr T = (V,\mathcal E)$ be the directed Cartesian product of rooted directed trees $\mathscr T_1, \cdots, \mathscr T_d$.
Let $S_{\lambdab}=(S_1, \cdots, S_d)$ be a joint left invertible, spherically balanced multishift on $\mathscr T$ and let $S_{\theta}$ be the weighted shift on the rooted directed tree $\mathscr T^{\otimes}_{\rootb}=(V^{\otimes}, \mathcal F)$ associated with $S_{\lambdab}.$ If $S_{\lambdab}$ is a joint contraction, then the following statements are equivalent:
\begin{enumerate}
\item[(i)] $S_{\lambdab}$ is joint subnormal. 
\item[(ii)] $\{\prod_{p=0}^{n} {\mf C_{p}}\}_{n \in \mathbb N}$ is completely monotone, where $\mf C_t$ is the constant value of $\sum_{j=1}^d \|S_j e_v\|^2$ on the generation $\mathcal G_t$ of $\mathscr T.$
\item[(iii)] $S_{\theta}$ is subnormal.
\end{enumerate}
\end{proposition}
\begin{proof} 
Assume that $S_{\lambdab}$ is a joint contraction.
We proved the following formula in \eqref{norm-sp-gen}: \beqn 
\inp{Q^n_{S_{\lambdab}}(I)e_v}{e_v}  = \big\|S^n_{\theta}e_{\mf v}\big\|^2~(n \in \mathbb N),
\eeqn
where $Q^n_T(\cdot)$ is as given by \eqref{sp-gen-powers}.
By \cite[Theorem 5.2]{At-1}, $S_{\lambdab}$ is joint subnormal if and only if
$\big\{\inp{Q^n_{S_{\lambdab}}(I)e_v}{e_v}\big\}_{n \in \mathbb N}$ is completely monotone for every $v \in V$, and hence by the formula above, this is equivalent to the complete monotonicity of $\{\big\|S^n_{\theta}e_{\mf v}\big\|^2\}_{n \in \mathbb N}$ for every $\mf v \in V^{\otimes}.$ 
This yields the equivalence of (i) and (iii).
The equivalence of (ii) and (iii) is immediate from \eqref{moment-S-theta}. This completes the proof.
\end{proof}

Let us illustrate the previous result with the help of the family of multishifts discussed in Example \ref{ex-sp-bal}.
\begin{example} \label{j-s-bal}
Let $S_{\lambdab_{\mf C_a}}$ be as introduced in Example \ref{ex-sp-bal}.
Note that $S_{\lambdab_{\mf C_a}}$ is a joint contraction if and only if 
$d \leqslant a.$ Assume that $a$ is an integer such that $d \leqslant a.$ By the preceding proposition, $S_{\lambdab_{\mf C_a}}$ is joint subnormal if and only if $\{\prod_{p=0}^{n} {c_{a, p}}\}_{n \in \mathbb N}$ is completely monotone, where $$c_{a, t}=\frac{t+d}{t+a}~(t \in \mathbb N).$$
Let us verify the last statement.
Recall the fact that the product of completely monotone sequences ($\{\frac{i}{i+n}\}_{n \in \mathbb N},$ $i=d, \cdots, d+k$) is completely monotone.
Since for $k=a-d \in \mathbb N,$
\beqn
\prod_{p=0}^{n} {c_{a, p}} = \begin{cases} 1~&\mbox{if~}k=0, \\
\frac{d(d+1) \cdots (d+k-1)}{(d+n+1)(d+n+2) \cdots (d + k + n)}~& \mbox{otherwise,}
\end{cases}  
\eeqn
$\{\prod_{p=0}^{n} {c_{a, p}}\}_{n \in \mathbb N}$ is completely monotone.
\end{example}

\section{Joint Hyponormal Multishifts}

In this short section, we discuss the class of joint hyponormal tuples.

\begin{proposition}
\label{j-h-gen}
Let $\mathscr T = (V,\mathcal E)$ be the directed Cartesian product of rooted directed trees $\mathscr T_1, \cdots, \mathscr T_d$.
Let $S_{\lambdab}=(S_1, \cdots, S_d)$ be a multishift on $\mathscr T$. Then $S_{\lambdab}$ is joint hyponormal if and only if for every $t \in \mathbb N$ and for every $f_1, \cdots, f_d \in l^2(V)$ supported on generation $\mathcal G_t$, 
\beqn
\sum_{i, j=1}^d \inp{[S^*_j, S_i]f_j}{f_i} \geqslant  0.
\eeqn
\end{proposition}
\begin{proof}
Note that $\inp{[S^*_j, S_i]e_v}{e_w}=0~(i, j = 1, \cdots, d)$ for any $v, w \in V$ such that $|\alpha_v| \neq |\alpha_w|$ (see Lemma \ref{disjoint}(vi)). 
It follows that for $f_t, g_s \in l^2(V)$ with supports on $\mathcal G_t$ and $\mathcal G_s$ respectively with $s \neq t$,
\beq \label{gen-ortho}
\inp{[S^*_j, S_i]f_t}{g_s}=0~\mbox{for every}~i, j=1, \cdots, d.
\eeq
For $j=1, \cdots, d,$ let $f_j\in l^2(V)$ and write $f_j = \sum_{t \in \mathbb N}f_{j, t},$ where $f_{j, t}$ is supported on $\mathcal G_t.$ Then
\beqn \sum_{i, j=1}^d \inp{[S^*_j, S_i]f_j}{f_i} &=& \sum_{t, s \in \mathbb N} \sum_{i, j=1}^d \inp{[S^*_j, S_i]f_{j, t}}{f_{i, s}} \\ &\overset{\eqref{gen-ortho}}=& \sum_{t \in \mathbb N} \sum_{i, j=1}^d \inp{[S^*_j, S_i]f_{j, t}}{f_{i, t}}. \eeqn
The desired equivalence is now immediate. 
\end{proof}

In case of spherically balanced multishifts, the preceding characterization can be made more explicit (cf. \cite[Theorem 5.3(5)]{CY}).
\begin{theorem}
Let $\mathscr T = (V,\mathcal E)$ be the directed Cartesian product of rooted directed trees $\mathscr T_1, \cdots, \mathscr T_d$.
Let $S_{\lambdab}=(S_1, \cdots, S_d)$ be a joint left invertible, spherically balanced multishift on $\mathscr T$ and let $S_{\theta}$ be the weighted shift on the rooted directed tree $\mathscr T^{\otimes}_{\rootb}=(V^{\otimes}, \mathcal F)$ associated with $S_{\lambdab}.$ Then the following statements are equivalent:
\begin{enumerate}
\item[(i)] $S_{\lambdab}$ is joint hyponormal. 
\item[(ii)] $\{{\mf C_{t}}\}_{t \in \mathbb N}$ is monotonically increasing, where $\mf C_t$ is the constant value of $\sum_{j=1}^d \|S_j e_v\|^2$ on the generation $\mathcal G_t$ of $\mathscr T.$
\item[(iii)] $S_{\theta}$ is hyponormal.
\end{enumerate}
\end{theorem}
\begin{proof}
We first verify the implication that (i) implies (ii). Recall from \cite[Lemma 4.10]{CS} that any joint hyponormal $d$-tuple $T$ satisfies the inequality
\beqn
Q^2_T(I) \geqslant  Q_T(I)^2,
\eeqn
where $Q^n_T(\cdot)$ is as given in \eqref{sp-gen-powers}. In particular, for any $v \in V,$
\beqn
\inp{Q^2_{S_{\lambdab}}(I)e_v}{e_v} \geqslant  \|Q_{S_{\lambdab}}(I)e_v\|^2.
\eeqn
However, by the polar decomposition obtained in Proposition \ref{sphericallybal},
\beq \label{p-d-proof}
S_j = T_jD_c,~j=1, \cdots, d,
\eeq
where $T_{\lambdab}=(T_1, \cdots, T_d)$ and $D_c$ are joint isometry and diagonal part of $S_{\lambdab}$ respectively. Since $T_{\lambdab}$ is a joint isometry,  
$Q_{S_{\lambdab}}(I)=D^2_c$. It is now easy to see that
\beqn
\mf C_{|\alpha_v|}\mf C_{|\alpha_v|+1} &=& \inp{Q_{S_{\lambdab}}(D^2_c)e_v}{e_v} = \inp{Q^2_{S_{\lambdab}}(I)e_v}{e_v} \\ & \geqslant & \|Q_{S_{\lambdab}}(I)e_v\|^2 = \|D^2_ce_v\|^2=\mf C^2_{|\alpha_v|},
\eeqn
which yields that $\{{\mf C_{t}}\}_{t \in \mathbb N}$ is monotonically increasing. 

We next verify the implication that (ii) implies (i). In view of Proposition \ref{j-h-gen}, it suffices to check that for every $t \in \mathbb N$ and for every $f_1, \cdots, f_d \in l^2(V)$ supported on $\mathcal G_t$, 
\beqn
\sum_{i, j=1}^d \inp{[S^*_j, S_i]f_j}{f_i} \geqslant  0.
\eeqn
To see this, let $f_1, \cdots, f_d \in l^2(V)$ be supported on $\mathcal G_t$ for some $t \in \mathbb N.$ A routine verification using \eqref{p-d-proof} shows that 
\beqn
[S^*_j, S_i]e_v = \mf C_{|\alpha_v|} [T^*_j, T_i]e_v + (\mf C_{|\alpha_v|}-\mf C_{|\alpha_v|-1})T_iT^*_je_v~\mbox{for any~}v \in V.
\eeqn
Hence
\beq \label{(b)}
\sum_{i, j=1}^d \inp{[S^*_j, S_i]f_j}{f_i} &=& \mf C_{t} \sum_{i, j=1}^d \inp{[T^*_j, T_i]f_j}{f_i} \notag \\ &+& (\mf C_{t}-\mf C_{t-1})\sum_{i, j=1}^d \inp{T_iT^*_jf_j}{f_i},
\eeq
where we used the convention that $\mf C_{-1}=0$.  
However, $$\sum_{i, j=1}^d \inp{[T^*_j, T_i]f_{j}}{f_{i}} \geqslant  0$$ since $T_{\lambdab}$, being joint subnormal, is joint hyponormal, and $$\sum_{i, j=1}^d \inp{T_iT^*_jf_{j}}{f_{i}}=\Big\|\sum_{j=1}^dT^*_jf_{j}\Big\|^2.$$ Since $\{{\mf C_{t}}\}_{t \in \mathbb N}$ is monotonically increasing, we conclude from \eqref{(b)} that $$\sum_{i, j=1}^d \inp{[S^*_j, S_i]f_j}{f_i} \geqslant  0.$$

We finally check the equivalence of (ii) and (iii). In view of \cite[Theorem 5.1.2]{JJS}, it suffices to check that 
\beqn
\sum_{\mf w \in \child{\mf v}} \frac{\theta^2_{\mf w}}{\|S_{\theta}e_{\mf w}\|^2} \leqslant 1~\mbox{for every}~\mf v \in V^{\otimes}
\eeqn
if and only if $\{{\mf C_{t}}\}_{t \in \mathbb N}$ is monotonically increasing.
Since $\|S_\theta e_{\mf w}\|^2 =  \mf C_{\alpha_{\mf w}}$
and $\theta_{\mf w} = \frac{{\sqrt{\mf C_{\alpha_{\mf w}-1}}}}{\sqrt{\mbox{card}(\mathsf{sib}(\mf w))}}~(\mf w \in V^{\otimes} \setminus \rootb)$ (see \eqref{moment-S-theta} and \eqref{theta}), we obtain 
\beqn
\sum_{\mf w \in \child{\mf v}} \frac{\theta^2_{\mf w}}{\|S_{\theta}e_{\mf w}\|^2} &=& \sum_{\mf w \in \child{\mf v}} \frac{1}{\mbox{card}(\mathsf{sib}(\mf w))}\frac{\mf C_{\alpha_{\mf w}-1}}{\mf C_{\alpha_{\mf w}}}\\ &=& \frac{\mf C_{\alpha_{\mf v}}}{\mf C_{\alpha_{\mf v}+1}}\sum_{\mf w \in \child{\mf v}} \frac{1}{\mbox{card}(\mathsf{sib}(\mf w))}=\frac{\mf C_{\alpha_{\mf v}}}{\mf C_{\alpha_{\mf v}+1}}.
\eeqn
The equivalence of (ii) and (iii) is now clear.
This completes the proof.
\end{proof}
\begin{remark}
Assume that $S_{\lambdab}$ is a joint hyponormal multishift. 
It is well-known that the spectral radius and norm of a hyponormal operator coincide \cite{Co}. 
One may now conclude from Corollary \ref{sp-rd}(i) and \eqref{norm-Q(I)} that the spectral radius of $S_{\lambdab}$ equals $\|S^*_1S_1 + \cdots + S^*_dS_d\|^{\frac{1}{2}}$.
\end{remark}

\begin{example}
Let $S_{\lambdab_{\mf C_a}}$ be as introduced in Example \ref{ex-sp-bal}.
Note that $S_{\lambdab_{\mf C_a}}$ is a joint hyponormal if and only if 
$\{{c_{a, t}}\}_{t \in \mathbb N}$ is monotonically increasing, where $$c_{a, t}=\frac{t+d}{t+a}~(t \in \mathbb N).$$
This holds if and only if $d \leqslant a.$ In view of Example \ref{j-s-bal}, we have the following equivalent statements (cf. \cite[Lemma 3.3]{AZ}):
\begin{enumerate}
\item $S_{\lambdab_{\mf C_a}}$ is joint subnormal.
\item $S_{\lambdab_{\mf C_a}}$ is joint hyponormal.
\item $S_{\lambdab_{\mf C_a}}$ is joint contraction.
\end{enumerate}
\end{example}

\section*{Afterword}

Needless to say, the work presented in this paper provides a framework to unify the theories of classical multishifts and weighted shifts on rooted directed trees. 
This framework also enables one to pose and peruse a diverse range of problems.
More importantly, this framework allows the rich interplay of graph theory, complex function theory, and operator theory.
One of the important outcomes of these investigations is perhaps the tree analogs $S_{\lambdab_{\mf C_a}}$ of extensively studied classical multishifts like Szeg$\ddot{\mbox{o}}$, Bergman, and Drury-Arveson $d$-shifts.
On one hand, the tree analogs of these $d$-shifts share many important properties of their classical counterparts like $S_{\lambdab_{\mf C_d}}$ (Szeg$\ddot{\mbox{o}}$ $d$-shift) is a joint isometry, $S_{\lambdab_{\mf C_{d+1}}}$ (Bergman $d$-shift) is joint subnormal, and $S_{\lambdab_{\mf C_1}}$ (Drury-Arveson $d$-shift) is a row contraction. On the other hand, due to abundance of directed tree structures, various complicacies may arise. For instance, unlike their classical counterparts, for suitable choice of $\mathscr T,$ the defect operator $\displaystyle \sum_{k=0}^{a} (-1)^k {a \choose k}Q^k_{S_{\lambdab_{\mf C_a}}}(I)$ fails to be an orthogonal projection.
Further,  in the matrix decomposition of  $S^*_{\lambdab_{\mf C_a}}$,  non-diagonal tuples of infinite rank operators appear naturally.  We believe that the class of multishifts $S_{\lambdab_{\mf C_a}}$ warrants further attention as they may play the role of building blocks in the classification of $\mathcal G$-homogeneous tuples associated with the action of various linear groups $\mathcal G$ such as diagonal unitaries and unitaries.

%\section*{Acknowledgment}
%\begin{acknowledgement}
\chapter*{Acknowledgement}
%\medskip \textit{Acknowledgment}. \
We express our gratitude to Dmitry Yakubovich for a careful reading and several useful suggestions. 
We also convey our sincere thanks to H. Turgay Kaptano$\check{\mbox{g}}$lu, especially, for drawing our attention to \cite{AV}, where certain Hardy-type spaces are associated with some infinite acyclic, undirected, connected graphs. Further, we are thankful to Akash Anand for indicating a way to embed three dimensional directed graphs into the plane making the presentation of diagrams more accessible. 
Finally, we  thank the faculty and the administrative unit of Department of Mathematics and Statistics, IIT Kanpur and School of Mathematics, HRI, Allahabad for their warm hospitality during the preparation of the paper.

%\section{Joint $q$-isometry Multishifts}
%
%Fix an integer $q \geqslant  1$ and put
%\beq
%\label{5.26}
%B_q(Q_T)
%\mathrel{\mathop:}=
%\sum_{s=0}^q
%(-1)^{s} {q \choose s} Q^{s}_T(I).
%\eeq
%We say that $T$ is a
%{\it joint $q$-contraction} (respectively,
%{\it joint $q$-expansion}) if
%$B_q(Q_T) \ge 0$ (respectively,
%$B_q(Q_T) \leqslant 0$).
%
%We say that $T$ is a {\it joint $q$-hyperexpansion} if $T$ is a
%joint $k$-expansion for all $k=1, \cdots, q.$ Also, $T$ is said
%to be a {\it joint complete hyperexpansion} if $T$ is a
%joint $q$-hyperexpansion for all $q \geqslant  1.$
%If $B_q(Q_T)=0$, then $T$ is a {\it joint $q$-isometry}.
%If $m=1$ then we drop the prefix $1$- and term joint in all the above definitions.
%
%\section{Joint Complete Hyperexpansive Multishifts}

%\printglossaries

\chapter*{Appendix}

We are thankful to V. M. Sholapurkar for kindly providing a multivariable analog of the identity given in 
\eqref{S-formula} (along with proof). 
\begin{nono-lemma}  Let $n$ be a positive integer and let 
$X=(x_1, \cdots , x_d), Y=(y_1,\cdots , y_d)$ be $d$-tuples such that the variables $x_i, y_i~(i=1, \cdots, d)$ belong to a unital complex algebra.
Then $$  1- \sum_{\underset{|\alpha |=n}{\alpha \in \mathbb N^d}}  {|\alpha| \choose \alpha}X^{\alpha } Y^{\alpha }= \sum_{\underset{|\beta|\leqslant n-1}{\beta \in \mathbb N^d}} {|\beta|\choose  \beta }X^{\beta}(1-x_1y_1-x_2y_2-\cdots -x_dy_d)Y^{\beta},$$
where ${\displaystyle \ {|\alpha| \choose \alpha}= \frac{|\alpha|!}{\alpha_1 !\cdots \alpha_m !}}~(\alpha \in \mathbb N^d).$
\end{nono-lemma}
\begin{proof} We prove the result by induction on $n.$ For $n=1,$ both the sides of the identity reduce to $(1-x_1y_1-x_2y_2-\cdots -x_dy_d) $ and hence the result holds. Suppose the result holds for some $n \geqslant  1.$ We now prove the identity for $n+1.$ Starting with the right hand side, we split the sum $$A :=\sum_{\underset{|\beta|\leqslant n-1}{\beta \in \mathbb N^d}} {|\beta|\choose  \beta}X^\beta(1-x_1y_1-x_2y_2-\cdots -x_dy_d)Y^\beta$$ as $A_1+A_2$, where  \beqn A_1 &:=& \sum_{\underset{|\beta|\leqslant n-1}{\beta \in \mathbb N^d}} {|\beta|\choose  \beta}X^\beta(1-x_1y_1-x_2y_2-\cdots -x_dy_d)Y^\beta, \\ A_2 &:=& \sum_{\underset{|\beta|= n}{\beta \in \mathbb N^d}} {|\beta|\choose  \beta}X^\beta(1-x_1y_1-x_2y_2-\cdots -x_dy_d)Y^\beta.\eeqn
Now by induction hypothesis, $A_1= 1- \sum_{|\beta|=n}  {n\choose \beta}X^{\beta} Y^{\beta}.$
Observe also that $$A_2= \sum_{\underset{|\beta|= n}{\beta \in \mathbb N^d}}  {n\choose \beta}X^\beta  Y^\beta- \sum_{\underset{|\beta|= n}{\beta \in \mathbb N^d}}  {n\choose \beta}\left[ \sum_{i=1}^d X^{\beta+\epsilon_i}Y^{\beta+\epsilon_i}\right].$$
Thus
\begin{eqnarray*}
 A_1 + A_2&=& 1-\sum_{\underset{|\beta| = n}{\beta \in \mathbb N^d}}  {n\choose \beta}\left[\sum_{i=1}^d X^{\beta+\epsilon_i}Y^{\beta+\epsilon_i} \right]\\
  &=& 1- \sum_{\underset{|\beta| = n+1}{\beta \in \mathbb N^d}} \left[ \sum_{i=1}^d  {n \choose \beta - \epsilon_i}\right]X^\beta Y^\beta\\
    &=& 1-\sum_{\underset{|\beta| = n+1}{\beta \in \mathbb N^d}}  {n+1 \choose \beta }X^\beta Y^\beta.
   \end{eqnarray*}
   This is the left hand side of the identity with $n$ replaced by $n+1$.
Thus the result holds for all integers $n \Ge 1.$
\end{proof}

Let us discuss some of the difficulties in obtaining appropriate analog of Shimorin's formula \eqref{d-formula} in several variables.
The very first difficulty which arises is the appropriate notion of Cauchy dual $S'$ in several variables, where $S=(S_1, \cdots, S_d)$ is the given $d$-tuple of bounded linear operators on $\mathcal H$. To see what a correct choice would be,  note that \eqref{S-formula} may be rewritten as
\beqn
I - T^n{T'^{*}}^n = \sum_{k=0}^{n-1}T^k P_E {T'^{*}}^k,
\eeqn  
where $T$ is a left invertible operator on $\mathcal H$ and $P_E = I-TT'^{*}$ is the orthogonal projection onto the kernel of $T^*.$ Hence, in view of the preceding lemma, the choice of Cauchy dual in several variables should necessarily ensure that $I - \sum_{i=1}^d S_iS'^*_i$ an orthogonal projection.
Examples show that none of the notions of Cauchy dual (toral and spherical) apply with success in this context.  
Keeping this aside and assuming this condition for a moment, what we obtain is the following formula:
\beqn
\bigcap_{|\alpha|=n} \ker S'^{*\alpha} \subseteq \bigvee \{S^{\beta}f : f \in \ker S^*, |\beta| \leqslant n-1\},
\eeqn 
where it is not clear whether equality holds.
In case of single operator, we obtain equality, which can be then used to derive the duality formula \eqref{d-formula} crucial in obtaining the wandering subspace property for $T$.  

%\begin{proposition} Let $T=(T_1, \cdots, T_m)$ be a commuting
%$m$-tuple such that $\sum_{i=1}^mT_iT^*_i$ is an orthogonal
%projection. If $T$ is analytic then $T$ admits the wandering
%subspace property.
%\end{proposition}

%\newpage
%
%\pagenumbering{gobble}
%
%\thispagestyle{empty}
%
%\begin{table*}[!t]   

%\begin{framed}

%\index{$abbreviation$}{explanation for the abbreviation}

%\printnomenclature
\printindex

%\end{framed}

%\end{table*}

%\chapter*{index}

%\printnomenclature

%\printglossaries

%\newpage

%$\pagestyle{empty}

%
%{\small
%\address{\noindent Sameer Chavan and Deepak Kumar Pradhan \\ Department of Mathematics and Statistics\\ 
%Indian Institute of Technology Kanpur, India \\}
%   \email{chavan@iitk.ac.in \\ dpradhan@iitk.ac.in}
%   \\ 
%   
%   
% % \author[S. Trivedi]{Shailesh Trivedi}
%   \address{\noindent Shailesh Trivedi \\ School of Mathematics \\ Harish-Chandra Research Institute\\ 
%Chhatnag Road, Jhunsi, Allahabad 211019, India \\}
%\email{shaileshtrivedi@hri.res.in}   
%}

\end{document}